\documentclass[phd,lfcs,logo,twoside]{infthesis}

\usepackage[british]{babel}
\usepackage{imakeidx}
\usepackage{amsthm,amsmath,amssymb,stmaryrd,tikz,rotating,xspace,hyperref,enumerate,haskell,bibentry,bussproofs}
\usepackage{nameref}
\EnableBpAbbreviations
\nobibliography* 
\makeindex[name=word,title=Index of concepts,intoc,noautomatic]
\makeindex[name=symb,title=Index of symbols,intoc,noautomatic]

\DeclareRobustCommand{\coprod}{\mathop{\text{\fakecoprod}}}
\newcommand{\fakecoprod}{%
  \sbox0{$\prod$}%
  \smash{\raisebox{\dimexpr.9625\depth-\dp0}{\scalebox{1}[-1]{$\prod$}}}%
  \vphantom{$\prod$}%
}

\numberwithin{equation}{section}
\theoremstyle{plain}
\newtheorem{theorem}{Theorem}
\numberwithin{theorem}{section}
\newtheorem{lemma}[theorem]{Lemma}
\newtheorem{proposition}[theorem]{Proposition}
\newtheorem{corollary}[theorem]{Corollary}
\theoremstyle{definition}
\newtheorem{definition}[theorem]{Definition}
\newtheorem{example}[theorem]{Example}
\newtheorem{remark}[theorem]{Remark}

\usetikzlibrary{matrix,decorations.markings} 
\tikzstyle{dot}=[circle, draw=black, fill=black!25, inner sep=.4ex]
\newif\ifvflip\pgfkeys{/tikz/vflip/.is if=vflip}
\newif\ifhflip\pgfkeys{/tikz/hflip/.is if=hflip}
\newif\ifhvflip\pgfkeys{/tikz/hvflip/.is if=hvflip}
\newlength\morphismheight
\newlength\wedgewidth
\tikzset{width/.initial=1mm}
\makeatletter
\tikzstyle{morphism}=[font=\small,morphismshape]
\pgfdeclareshape{morphismshape}
{
    \savedanchor\centerpoint
    {
        \pgf@x=0pt
        \pgf@y=0pt
    }
    \anchor{center}{\centerpoint}
    \anchorborder{\centerpoint}
    \saveddimen\overallwidth
    {
        \pgfkeysgetvalue{/tikz/width}{\minwidth}
        \pgf@x=\wd\pgfnodeparttextbox
        \ifdim\pgf@x<\minwidth
            \pgf@x=\minwidth
        \fi
    }
    \savedanchor{\upperrightcorner}
    {
        \pgf@y=.5\ht\pgfnodeparttextbox
        \advance\pgf@y by -.5\dp\pgfnodeparttextbox
        \pgf@x=.5\wd\pgfnodeparttextbox
    }
    \anchor{north}
    {
        \pgf@x=0pt
        \pgf@y=0.5\morphismheight
    }
    \anchor{north east}
    {
        \pgf@x=\overallwidth
        \multiply \pgf@x by 2
        \divide \pgf@x by 5
        \pgf@y=0.5\morphismheight
    }
    \anchor{east}
    {
        \pgf@x=\overallwidth
        \divide \pgf@x by 2
        \advance \pgf@x by 5pt
        \pgf@y=0pt
    }
    \anchor{west}
    {
        \pgf@x=-\overallwidth
        \divide \pgf@x by 2
        \advance \pgf@x by -5pt
        \pgf@y=0pt
    }
    \anchor{north west}
    {
        \pgf@x=-\overallwidth
        \multiply \pgf@x by 2
        \divide \pgf@x by 5
        \pgf@y=0.5\morphismheight
    }
    \anchor{south east}
    {
        \pgf@x=\overallwidth
        \multiply \pgf@x by 2
        \divide \pgf@x by 5
        \pgf@y=-0.5\morphismheight
    }
    \anchor{south west}
    {
        \pgf@x=-\overallwidth
        \multiply \pgf@x by 2
        \divide \pgf@x by 5
        \pgf@y=-0.5\morphismheight
    }
    \anchor{south}
    {
        \pgf@x=0pt
        \pgf@y=-0.5\morphismheight
    }
    \anchor{text}
    {
        \upperrightcorner
        \pgf@x=-\pgf@x
        \pgf@y=-\pgf@y
    }
    \backgroundpath
    {
    \begin{scope}
        \pgfkeysgetvalue{/tikz/fill}{\morphismfill}
        \pgfsetstrokecolor{black}
        \pgfsetlinewidth{.7pt}
        \begin{scope}
        \pgfsetstrokecolor{black}
        \pgfsetfillcolor{white}
        \pgfsetlinewidth{.7pt}
                \ifhflip
                    \pgftransformyscale{-1}
                \fi
                \ifvflip
                    \pgftransformxscale{-1}
                \fi
                \ifhvflip
                    \pgftransformxscale{-1}
                    \pgftransformyscale{-1}
                \fi
                \pgfpathmoveto{\pgfpoint
                    {-0.5*\overallwidth-5pt}
                    {0.5*\morphismheight}}
                \pgfpathlineto{\pgfpoint
                    {0.5*\overallwidth+5pt}
                    {0.5*\morphismheight}}
                \pgfpathlineto{\pgfpoint
                    {0.5*\overallwidth + \wedgewidth}
                    {-0.5*\morphismheight}}
                \pgfpathlineto{\pgfpoint
                    {-0.5*\overallwidth-5pt}
                    {-0.5*\morphismheight}}
                \pgfpathclose
                \pgfusepath{fill,stroke}
        \end{scope}
    \end{scope}
    }
}
\makeatother
\newcommand{\tinymult}[1][dot]{
\smash{\raisebox{-2pt}{\hspace{-5pt}\ensuremath{\begin{pic}[scale=0.4,yscale=-1]
    \node (0) at (0,0) {};
    \node[#1, inner sep=1.5pt] (1) at (0,0.55) {};
    \node (2) at (-0.5,1) {};
    \node (3) at (0.5,1) {};
    \draw (0.center) to (1.center);
    \draw (1.center) to [out=left, in=down, out looseness=1.5] (2.center);
    \draw (1.center) to [out=right, in=down, out looseness=1.5] (3.center);
    \node[#1, inner sep=1.5pt] (1) at (0,0.55) {};
\end{pic}
}\hspace{-3pt}}}}
\newcommand{\tinyunit}[1][dot]{
\smash{\raisebox{1pt}{\hspace{-3pt}\ensuremath{\begin{pic}[scale=0.4,yscale=-1]
    \node (0) at (0,0) {};
    \node[#1, inner sep=1.5pt] (1) at (0,0.55) {};
    \draw (0.center) to (1.north);
\end{pic}
}\hspace{-1pt}}}}
\newcommand{\tinycomult}[1][dot]{
\smash{\raisebox{-2pt}{\hspace{-5pt}\ensuremath{\begin{pic}[scale=0.4,yscale=1]
    \node (0) at (0,0) {};
    \node[#1, inner sep=1.5pt] (1) at (0,0.55) {};
    \node (2) at (-0.5,1) {};
    \node (3) at (0.5,1) {};
    \draw (0.center) to (1.center);
    \draw (1.center) to [out=left, in=down, out looseness=1.5] (2.center);
    \draw (1.center) to [out=right, in=down, out looseness=1.5] (3.center);
    \node[#1, inner sep=1.5pt] (1) at (0,0.55) {};
\end{pic}
}\hspace{-3pt}}}}
\newcommand{\tinycounit}[1][dot]{
\smash{\raisebox{-3pt}{\hspace{-3pt}\ensuremath{\begin{pic}[scale=0.4,yscale=1]
    \node (0) at (0,0) {};
    \node[#1, inner sep=1.5pt] (1) at (0,0.55) {};
    \draw (0.center) to (1.south);
\end{pic}
}\hspace{-1pt}}}}
\newcommand\sxto[1]{\mathbin{\smash{
\begin{tikzpicture}[baseline={([yshift=-1pt]
current bounding box.south)}]
    \node (A) at (0,0) [inner xsep=0pt, inner ysep=1pt, minimum width=0.15cm] {\ensuremath{\scriptstyle #1}};
    \draw [->, line width=0.4pt, line cap=round]
        ([xshift=-2.5pt] A.south west)
        to ([xshift=3pt] A.south east);
\end{tikzpicture}}}}
\newcommand\sxfrom[1]{\mathbin{\smash{
\begin{tikzpicture}[baseline={([yshift=-1pt]
current bounding box.south)}]
    \node (A) at (0,0) [inner xsep=0pt, inner ysep=1pt, minimum width=0.15cm] {\ensuremath{\scriptstyle #1}};
    \draw [<-, line width=0.4pt, line cap=round]
        ([xshift=-2.5pt] A.south west)
        to ([xshift=3pt] A.south east);
\end{tikzpicture}}}}
\newenvironment{pic}[1][]
{\begin{aligned}\begin{tikzpicture}[font=\tiny,#1]}
{\end{tikzpicture}\end{aligned}}
\tikzset{dagger/.style={
    decoration={markings, mark=at position 0.5 with {
        \draw [-] ++ (0,-.75mm) -- (0,.75mm);}
    }, postaction={decorate},
}}

\usepackage{xparse}
\NewDocumentCommand{\xrightarrows}{ O{}O{} }{%
\mathrel{%
\vcenter{\hbox{%
\begin{tikzpicture}
  \node[minimum width=1cm,minimum height=1ex,anchor=south,align=center] (a){\text{\vphantom{hg}#1}\\[0.5ex] \vphantom{hg}#2};
  \draw[<-] ([yshift=0.35ex]a.west) -- ([yshift=0.35ex]a.east);
  \draw[->] ([yshift=-0.35ex]a.west) -- ([yshift=-0.35ex]a.east);
\end{tikzpicture}
}}%
}%
}

\newcommand{\cat}[1]{\ensuremath{\mathbf{#1}}}
\newcommand{\op}{\ensuremath{{}^{\mathrm{op}}}}
\newcommand{\id}[1][]{\ensuremath{\mathrm{id}_{#1}}}
\newcommand{\norm}[1]{\ensuremath{\left\Vert {#1}\right\Vert}}
\newcommand{\abs}[1]{\ensuremath{\left\vert {#1}\right\vert}}
\newcommand{\sem}[1]{\ensuremath{\llbracket {#1}\rrbracket}}
\DeclareMathOperator{\ev}{ev}
\DeclareMathOperator{\st}{st}
\DeclareMathOperator{\dst}{dst}
\DeclareMathOperator{\Tr}{Tr}
\DeclareMathOperator{\Ima}{Im}
\DeclareMathOperator{\FEM}{FEM}
\DeclareMathOperator{\EM}{EM}
\DeclareMathOperator{\DFMonad}{DFMonad}
\DeclareMathOperator{\arr}{arr}
\DeclareMathOperator{\first}{first}
\DeclareMathOperator{\second}{second}
\DeclareMathOperator{\inv}{inv}
\DeclareMathOperator{\inl}{inl}
\DeclareMathOperator{\inr}{inr}
\DeclareMathOperator{\fix}{fix}
\DeclareMathOperator{\dom}{dom}
\DeclareMathOperator{\cod}{cod}
\DeclareMathOperator{\Inv}{Inv}

\DeclareMathOperator{\sym}{Sym}

\DeclareMathOperator{\dcolim}{dcolim} 
\DeclareMathOperator{\dlim}{dlim} 
\DeclareMathOperator{\colim}{colim} 
\DeclareMathOperator{\coker}{coker}
\DeclareMathOperator{\im}{im}
\DeclareMathOperator{\assocrplus}{assocr^\oplus}
\DeclareMathOperator{\assoclplus}{assocl^\oplus}
\DeclareMathOperator{\assocrtimes}{assocr^\otimes}
\DeclareMathOperator{\assocltimes}{assocl^\otimes}
\DeclareMathOperator{\swapplus}{swap^\oplus}
\DeclareMathOperator{\swaptimes}{swap^\otimes}
\DeclareMathOperator{\unitlplus}{unitl^\oplus}
\DeclareMathOperator{\unitrplus}{unitr^\oplus}
\DeclareMathOperator{\unitltimes}{unitl^\otimes}
\DeclareMathOperator{\unitrtimes}{unitr^\otimes}
\DeclareMathOperator{\distr}{distr}
\DeclareMathOperator{\factor}{factor}
\DeclareMathOperator{\absorb}{absorb}
\DeclareMathOperator{\unabsorb}{unabsorb}
\DeclareMathOperator{\fold}{fold}
\DeclareMathOperator{\unfold}{unfold}
\newcommand{\prof}{\ensuremath{[\cat{C}\op\times\cat{C},\cat{Set}]}}
\newcommand{\Set}{\ensuremath{\mathbf{Set}}}
\newcommand{\acmp}{\ensuremath{\mathrel{>\!\!>\!\!>}}}
\newcommand{\afanout}{\ensuremath{\mathrel{\&\!\!\&\!\!\&}}}
\newcommand{\scolon}{\ensuremath{\mathrel{;}}}
\newcommand{\PInj}{\cat{PInj}}

\newcommand{\Pfn}{\cat{Pfn}}
\newcommand{\Kl}{\ensuremath{\mathcal{K}\!\ell}}
\newcommand{\ie}{\text{i.e.}\xspace}
\newcommand{\eg}{\text{e.g.}\xspace}
\newcommand{\cut}[1]{}
\newcommand{\N}{\mathbb{N}}
\newcommand{\C}{\mathbb{C}} 
\newcommand{\Ban}{\cat{Ban_{\leq 1}}}
\newcommand{\Hilb}{\cat{Hilb_{\leq 1}}}
\newcommand{\inprod}[2]{\ensuremath{\langle #1 \,,\, #2 \rangle}}

\title{The Way of the Dagger}
\author{Martti Karvonen}

 \submityear{2018}

 \graduationdate{July 2019}

\abstract{
A dagger category is a category equipped with a functorial way of reversing morphisms, \ie a contravariant involutive identity-on-objects endofunctor. Dagger categories with additional structure have been studied under different names in categorical quantum mechanics, algebraic field theory and homological algebra, amongst others. In this thesis we study the dagger in its own right and show how basic category theory adapts to dagger categories.

We develop a notion of a dagger limit that we show is suitable in the following ways: it subsumes special cases known from the literature; dagger limits are unique up to unitary isomorphism; a wide class of dagger limits can be built from a small selection of them; dagger limits of a fixed shape can be phrased as dagger adjoints to a diagonal functor; dagger limits can be built from ordinary limits in the presence of polar decomposition; dagger limits commute with dagger colimits in many cases.

Using cofree dagger categories, the theory of dagger limits can be leveraged to provide an enrichment-free understanding of limit-colimit coincidences in ordinary category theory. We formalize the concept of an ambilimit, and show that it captures known cases. As a special case, we show how to define biproducts up to isomorphism in an arbitrary category without assuming any enrichment. Moreover, the limit-colimit coincidence from domain theory can be generalized to the unenriched setting and we show that, under suitable assumptions, a wide class of endofunctors has canonical fixed points.

The theory of monads on dagger categories works best when all structure respects the dagger: the monad and adjunctions should preserve the dagger, and the monad and its algebras should satisfy the so-called Frobenius law. Then any monad resolves as an adjunction, with extremal solutions given by the categories of Kleisli and Frobenius-Eilenberg-Moore algebras, which again have a dagger. 

We use dagger categories to study reversible computing. Specifically, we model reversible effects by adapting Hughes' arrows to dagger arrows and inverse arrows. This captures several fundamental reversible effects, including serialization and mutable store computations. Whereas arrows are monoids in the category of profunctors, dagger arrows are involutive monoids in the category of profunctors, and inverse arrows satisfy certain additional properties. These semantics inform the design of functional reversible programs supporting side-effects.

}

\begin{document}

\begin{preliminary}

\maketitle


\begin{layman}

This thesis is primarily about the mathematics of so-called dagger categories. For now, think of a category as a collection of systems $A,B,C,\dots$ and processes between them. Processes can be done after each other, and for any system $A$ there should be a trivial process $\id[A]$ that amounts to doing nothing. In a dagger category there is an additional piece of structure -- any process $f\colon A\to B$ has a  counterpart process $f^\dag\colon B\to A$ going in the opposite direction. One might think of $f^\dag$ as some kind of a reversal of the process $f$, and this motivates the rules this reversal is supposed to obey, namely 
	\begin{itemize}
		\item reversing the trivial process of any system $A$ results in the trivial process, \ie $\id[A]^\dag=\id[A]$
		\item reversing a process twice results in the original process, \ie $f^{\dag\dag}=f$
		\item if one reverses the process ``$g$ after $f$'', one might as well have reversed each process individually and done them after each other in the opposite order. More succinctly, $(g\circ f)^\dag=f^\dag\circ g^\dag$
	\end{itemize}
Any mathematical structure obeying these rules is a dagger category. For a concrete example, let us consider the alphabet.  Letters can be written after one another, resulting in strings such as ``asdf'' and ``cat''. Strings can be composed by writing them one after the other, so ``world'' composed after ``hello'' becomes `helloworld''. The empty string is a neutral element with respect to composition, \ie appending or prepending it to a string doesn't change the string in question. What makes this into a dagger category is the fact that strings can be reversed. Reversing the empty string results in an empty string. Reversing a string twice results in the original string. Reversing a composite string is the same thing as composing the reversals in the opposite order. Hence the rules above are indeed obeyed by strings of letters. 

Dagger categories arise independently both in physics and computation, and also at their intersection in quantum computing. However, so far a systematic study of dagger categories has been missing, and this thesis fills the gap. There are various mathematical questions and notions people study in the context of ordinary categories, such as (co)limits, which consider well-behaved ways of building new systems from old ones, or monads and arrows, which can be used to enlarge the collection of processes available. This thesis shows how to understand such topics in the context of dagger categories.

\end{layman}


\begin{acknowledgements}

First and foremost I need to thank my first supervisor Chris Heunen for his guidance throughout the years. Besides helping in choosing a topic, nudging me in the right direction when I got stuck, giving advice on literature to read and advising with numerous smaller issues that come about in academic life, he has helped me to grow into a researcher who can hopefully make all these decisions alone. Without him you would be reading an empty thesis\footnote{A variant of~\cite{upper:writersblock}}. 

I must also thank my second supervisor Tom Leinster. His intuition, whether pertaining to topics described in this thesis or to mathematics in general, has not only been helpful but also a delight to observe. I hope that I have been able to absorb even a fraction of his way of thinking. 

I am very thankful for my examiners Peter Selinger and David Jordan. The time they invested into reading this thesis with care resulted in insightful comments and feedback that has greatly improved this thesis.

I wish to thank Gordon Plotkin for very helpful comments during my yearly reviews and on an earlier draft. It is truly a pleasure to see how a computer scientist thinks about category theory. Similarly, I must thank Robin Kaarsgaard not only for his helpful comments but also for the work we have done together -- hopefully this was only a start. To the extent that the thesis falls under theoretical computer science, it is because of Gordon and Robin.

I am indebted for Jamie Vicary both for coining the name ``way of the dagger'' to describe the general philosophy when working with dagger categories, and for letting me use it as a title. More widely, I am thankful for the whole community centred around the conference Quantum Physics and Logic -- it is a great research community to be a part of, and a great source of friends and collaborators.

Before moving to friends and family, I wish to thank various organizations that have helped me. I wish to thank the Osk. Huttunen foundation for financially supporting me during my PhD. The money that V\"ais\"al\"a foundation kindly let me use as a travel grant was very useful for letting me become a part of an international scientific community and I learned a lot on the way. COST Action IC1405 made possible a trip to Copenhagen, without which this thesis would be shorter. My host institute LFCS was not only a great academic community, but also helped me with a funding gap at the start and some travel money during the later years. 

Life in Edinburgh would have been much duller without Toms, Kima and their friends. Climbing with Orfeas has provided me with a much needed counterbalance to research. My summer breaks in Finland were mostly timed so that I could visit KKT, so I have to thank everyone involved for that particular getaway and the support therein. I must also thank Tuomo, Sakari and IHRA more generally for their sustained friendship. 

I am grateful for Pentti, Pirjo and Pekka, not only for the usual things loving families provide, but also for being crucial for my intellectual development during my formative years and supportive once I got interested in mathematics. All of my siblings, half or otherwise, have been crucial to me growing up to be the person I am.

There is no way words can express the gratitude I have for Esma, so I will not even try.

\end{acknowledgements}

\standarddeclaration


\tableofcontents


\end{preliminary}


%


\chapter{Introduction}

\section{Why dagger categories}

In fundamental physical theories, the time evolution of a system is in principle reversible. That is, if the clock was run backwards\footnote{and charge conjugation \& parity transformation are applied as well, see the CPT theorem.}, everything would still comply with the laws of physics. But what is the nature of this ``time-reversal'' itself? In a sense, this thesis exists in order to study such reversal operations in an abstract setting. 

Such reversals occur not only in physics but also in computation. In a reversible computer, each computation step is logically reversible by design. There are many more phenomena that are \emph{weakly reversible}, in the sense that any process $f\colon A\to B$ is associated with a corresponding counterpart process $f^\dag\colon B\to A$ going the other way. For instance, in a nondeterministic computer program, for any state of the system there are in principle several possible states to enter next. Such a computer program can naturally be ran in the other direction: given a state of system, just go to one of the possible preceding states. However, doing a computation step back and forth in this manner need not result in nothing happening, hence ``weak'' reversibility.

One way of studying a concept is to study it together with its environment. Hence physicists study time-evolution in the context of the theory they are working on and say a computer scientist might study reversible computing with all the tools needed for the task. Another, complementary way of understanding a concept is to try to isolate it and abstract away the unnecessary details, studying fundamental properties of the concept on its own. This latter route is the one taken in this thesis, at least for the most part.

First, we abstract away from a particular theoretical context by using the language of category theory. Then ``weak reversibility'' is modelled by a dagger category: a way of associating, to every morphism $f\colon A\to B$ a morphism $f^\dag\colon B\to A$ going in the opposite direction. If one thinks of the category in question as modelling physical processes, then the dagger of $f$ might correspond to running the evolution described by $f$ in reverse. Similarly, if the category is modelling reversible computation, a dagger of a computation $f$ would then correspond to running the computation backwards. However, one should not always think of $f^\dag$ as something that undoes $f$, but rather as a canonical counterpart of $f$. For instance, the binary relation ``$x$ is the parent of $y$'' has a dagger ``$y$ is the offspring of $x$'', but the composite relation ``$x$ is the parent of an offspring of $y$'' is not the identity relation.

Dagger categories have been studied in several schools of thought as discussed in the following section. Usually, the dagger plays a minor part, being merely one of several interlocking pieces of structure assumed. We deviate from tradition by studying the dagger in isolation, and thus studying piece-by-piece, what ``cooperation with the dagger'' means. 

From the point of view of a mathematician, the need to study dagger categories is obvious -- if the theory of dagger categories has not been worked out, how sure can we be that we have the correct viewpoint when studying categories where the dagger is merely one of several parts of the structure? Maybe there is an underlying unity behind various kinds of dagger categories which can only be found by going through the foundational topics of category theory with a dagger at hand.






\section{Historical Background}

The tradition from which I enter dagger categories is called categorical quantum mechanics (see \eg \cite{abramskycoecke:cqm,abramskycoecke:categorical} for the founding papers and \cite{coeckekissinger:picturing} or the upcoming \cite{heunenvicary:cqm} for a textbook), 
and indeed it is within this context that the name ``dagger categories'' came up~\cite{selinger:completelypositive}. In this tradition, one usually uses categories that are either dagger compact or at least dagger monoidal and are often enriched in commutative monoids \ie sums of morphisms make sense. Such categories are then used to illuminate the structure of quantum computing protocols and results used in quantum computing -- it turns out that a surprising amount of quantum computing is not about finite-dimensional complex linear algebra but works in any dagger compact category.

However, this is not the first school of thought to use categories with a dagger. In fact, categories of relations~\cite{puppe:abelsche,puppe:korrespondenzen} and categories with involution~\cite{burgin:gamma,burgin:involution} were already studied in the 60s in the context of homological algebra~\cite{maclane:additiverelations}. This resulted in the notion of a regular category~\cite{grillet:regularcategories,vanosdol:regularcats,barr:exact}, which has sparked considerable amounts of attention (\eg\cite{kawahara:relations,kawahara:matrix} and~\cite{day:localisation,chikhladze:barrsembeddingthm} in the enriched setting). In this setting the hom-sets are partially ordered~\cite{lambek:diagram} like in the standard category of sets and binary relations. See~\eg\cite{freyd:catsandalligators,butz:regular} or~\cite[Chapter 2]{borceux:vol2} for a modern treatment of the topic.

Another use of dagger categories, perhaps closer to categorical quantum mechanics, comes from the study of operator algebras~\cite{ghezlimaroberts:wstarcategories,henry:completecstarcats}, unitary group representations~\cite{doplicher:newduality} and algebraic quantum field theory~\cite{halvorsonmuger:algebraicqft}. In these settings, one typically assumes the sets of morphisms to be at least complex vector spaces, or maybe even Banach spaces.

Finally, dagger categories come up in the study of reversible computing~\cite{glueckkaarsgaard:rfcl,bowman:dagger,jamessabry:infeff,carette:computingwithsemirings}. Here one often uses inverse categories~\cite{kastl:inversecategories}, which are closely related to inverse semigroups~\cite{lawson:inversesemigroups}. Once again, one often wants the categories in question to have considerable additional structure, \eg DCPO-enrichment~\cite{kaarsgaardaxelsengluck:joininversecategories} or various kinds of monoidal products~\cite{giles:thesis}.

In each of these schools of thought, the dagger is merely a small piece of the structure that is assumed, but dagger categories have not been studied on their own right -- it is almost as if mathematicians had studied particular kinds of enriched categories without ever developing ordinary unenriched category theory. This thesis is an attempt to rectify this omission. The main philosophy when working with dagger categories, called ``the way of the dagger'' by Jamie Vicary, is that all structure in sight should cooperate with the dagger. What this means for basic notions such as (monoidal) products and isomorphisms has been known for a while, but the dagger counterparts for many standard categorical notions have either been unstudied or only understood in special cases. We fix this by discussing the dagger versions of the topics one would expect to meet in a typical introductory textbook on category theory: equivalences, adjunctions, monads, limits and the like.

Given the power of (enriched) category theory to generalize and unify, one would expect this to fall under already existing theory and to be trivial, or at the very least trivially trivial, to borrow a phrase from Freyd. However, this turns out not to be the case: dagger categories are not enriched categories in any obvious and interesting sense. Similarly, so-called formal category theory~\cite{gray:formalcategorytheory,street1972formal}, \ie working in a 2-category, possibly with additional structure such as Yoneda structures~\cite{street:fibrations,streetwalters:yoneda} or proarrow equipments~\cite{wood:prowarrowsI,wood:proarrowsII} fails to apply for two reasons: 
	\begin{enumerate}[(i)]
		\item The 2-category \cat{DagCat} of dagger categories, dagger functors and natural transformations is not just a 2-category, it is in fact a \emph{dagger 2-category}, see Definition~\ref{def:dagtwocat} below. This means that the problem is merely shifted a level up: to understand dagger category theory formally, one needs to understand dagger 2-categories.
		\item Moreover, even when studying dagger categories, one often wants to also use non-dagger categories. For instance, dagger equalizers are an important kind of a dagger limit, and require stepping outside of \cat{DagCat}.
	\end{enumerate}

The main issue, directly resulting in (i) above, is that unlike in ordinary category theory, isomorphic objects are not always similar enough. In a dagger category, the correct notion of an isomorphism is given by a unitary $f$ satisfying $f^{-1}=f^\dag$. This issue, perhaps pathological for the ordinary category theorist\footnote{People even wonder whether dagger categories are ``evil'' in a technical sense. See \eg the MO-discussion~\cite{mo:dagcatsevil} or the nLab-entry~\cite[section on $\dag$-categories]{nlab:evil}}, is at the heart of why dagger category theory differs from ordinary category theory. Indeed, homotopy type theory treats dagger categories on a different footing than ordinary categories~\cite[9.7]{hott}.

What is the upshot of trying to unify such a disparate collection of categorical structures under dagger category theory? Does this unification not resemble the ``mustard-watch'' of Girard~\cite{ringard:mustardwatches}? First and foremost, most of the basic notions such as limits and monads turn out to sit naturally at the level of unenriched dagger categories -- when one moves to enriched dagger categories new weights are of course available, but having a clear understanding of the unenriched notion is essential. Understanding clearly the basic categorical notions at the level of dagger categories illuminates what is to be added when one works with dagger categories with additional structure. 

Moreover, the results and viewpoints developed can be applied to various domains at hand. In particular, the last two chapters of the thesis discuss topics of interest in ordinary theoretical computer science in the context of dagger categories. Perhaps surprisingly, dagger category theory provides applications also to ordinary category theory: various limit-colimit coincidences can be understood in a new light once one understands dagger limits clearly. A more detailed description of the results is provided in the next section.

\section{Outline, results and content available elsewhere}

\begin{description}
	\item[Chapter \ref{chp:background}] develops the background, defining the concepts needed throughout the rest. Apart from a few minor observations, it contains no original results.
	\item[Chapter \ref{chp:equivalences}] discusses the concepts of equivalences and adjunctions between dagger categories. The main results about equivalences limit the extent to which dagger categories fail to satisfy the ``principle of equivalence''\cite{mo:dagcatsevil,nlab:evil}.
	\item[Chapter \ref{chp:limits}] is a core theoretical chapter of the thesis. We develop a notion of limit for dagger categories, that we show is suitable in the following ways: it subsumes special cases known from the literature; dagger limits are unique up to unitary isomorphism; a wide class of dagger limits can be built from a small selection of them; dagger limits of a fixed shape can be phrased as dagger adjoints to a diagonal functor; dagger limits can be built from ordinary limits in the presence of polar decomposition; dagger limits commute with dagger colimits in many cases. The chapter is based on

		\bibentry{heunenkarvonen:daggerlimits} 

	which is under peer review.
	\item[Chapter \ref{chp:ambilims}] applies the viewpoint of the previous chapter to ordinary category theory, providing a way of understanding various cases of limits coinciding with colimits without requiring enrichment. In particular, biproducts can be defined without assuming the existence of zero morphisms. This is studied in Section~\ref{sec:biprods} which is based on

		\bibentry{karvonen:biproducts}

	which is under peer review. Moreover, the limit-colimit coincidence from domain theory can also be understood without enrichment, as shown in Section~\ref{sec:limcolimcoincidence}. We study this unenriched notion further in Section~\ref{sec:unenrichedfixedpoints}, showing that, under suitable assumptions, polynomial functors over a rig category have canonical fixed points. This provides a starting point for interpreting recursive data types in them, and we show how to interpret a language in such categories. We also discuss a dagger variant of the theory.
	\item[Chapter \ref{chp:monads}] is another core theoretical chapter, discussing the notion of a monad on a dagger category. If the monad and its algebras satisfy the so-called Frobenius law, everything works as expected: dagger adjunctions induce monads and every monad can be resolved as a dagger adjunction, with extremal solutions given by the categories of Kleisli and Frobenius-Eilenberg-Moore algebras. We characterize the Frobenius law as a coherence property between dagger and closure, and characterize strong monads as being induced by Frobenius monoids. Finally, we discuss limits in categories of algebras and prove a monadicity theorem. Apart from the last two topics, which are new, the chapter is based on the conference proceedings paper

		\bibentry{heunenkarvonen:reversiblemonads}

		and the extended journal version
		
		\bibentry{heunenkarvonen:daggermonads}

	\item[Chapter \ref{chp:arrows}] discusses arrows, a generalization of monads used to model side-effects in functional programming. We adapt the theory for dagger and inverse categories and discuss several examples, and characterize arrows categorically. This chapter is based on the conference paper 

		\bibentry{heunenkarvonenkaarsgaard:reversiblearrows}
	
	This same paper has been used by the coauthor Robin Kaarsgaard in his thesis~\cite[Section C.2]{robin:thesis}.

\end{description}

\section{Prerequisites}

To understand this thesis, an understanding of basic category theory ((monoidal) categories, functors, natural transformations, limits and colimits and adjunctions) is essential and can be obtained from any standard textbook~\cite{maclane:categories,awodey:categorytheory,borceux:vol1,borceux:vol2}. To appreciate section~\ref{sec:formal}, one needs to know 2-category theory at the level of formal theory of monads~\cite{street1972formal}. 

Outside of ordinary category theory, our main sources of intuition come from quantum mechanics, quantum computing and reversible computing. However, the reader does not strictly speaking need to be familiar with any of these fields. On occasion some basic knowledge of Hilbert spaces, say at the level of early chapters of one of~\cite{buschetal:quantummeasurement,heinosaariziman:quantum,prugovecki:quantum} will be useful. Similarly, in chapter~\ref{chp:ambilims} some rudimentary facts about Banach spaces are used. Early chapters of any book on functional analysis should suffice for this. As an example we suggest the textbook~\cite{helemskii:functionalanalysis} for the categorically minded reader.

\section{On notation}

We use a bolded font for naming categories. For example, the category of sets is denoted by \cat{Set} and a generic category by \cat{C}\index[symb]{$\cat{C},\cat{D},\dots$,  (dagger) categories}. Objects in a category are usually written with upper case letters $A,B,C,\dots$\index[symb]{$A,B,C,\dots$, objects of a category} or $X,Y,Z,\dots$\index[symb]{$X,Y,Z,\dots$, objects of a category}, and morphisms are usually written by lower case letters $f,g,h,\dots$\index[symb]{$f,g,h,\dots$, morphisms}. Functors are denoted $F,G,H\dots$\index[symb]{$F,G,H,\dots$, functors} and natural transformations are denoted by Greek letters. The main exceptions to this are cones and cocones of diagrams $D\colon\cat{J}\to\cat{C}$: we often denote a cone $\Delta L\to D$ by a generic morphism $l_A\colon L\to D(A)$\index[symb]{$l_A\colon L\to D(A)$, one leg of a cone $\Delta L\to D$ or a generic morphism denoting the whole cone} instead of naming the cone itself as $l$ or using the more cumbersome (but formal) notation $(l_A)_{A\in\cat{J}}$ for indexed families. Whether $l_A$ refers to the whole cone or a particular leg of it should be clear from context. The notation $[\cat{C},\cat{D}]$ refers to the category of dagger functors $\cat{C}\to\cat{D}$, which is itself a dagger category -- ordinary functor categories are referred to using $\cat{Cat}(\cat{C},\cat{D}$

The set of natural numbers includes $0$ and is denoted by $\N$. The one-object category induced by the monoid $(\N,+)$ is also denoted by $\N$,\index[symb]{$\N$, natural numbers, often as a one-object category} whereas when we view the ordered set $(\N,\leq)$ as a category we denote it by $\omega$\index[symb]{$\omega$, the category induced by the order on $\N$}. Other notational conventions are either explained as they appear or clear given the context. We refer a reader confused by notation to the~\hyperref[index:symb]{Index of symbols}. If instead of notation the reader is confused by a word, we suggest starting from the~\hyperref[index:word]{Index of concepts}.




\chapter{Background}\label{chp:background}

\section{Dagger categories}\label{sec:dagcats}

This section sets the scene by discussing dagger categories themselves. Dagger categories can behave rather differently than ordinary categories, see \eg~\cite[9.7]{hott}. We start by establishing terminology and setting conventions.

\begin{definition}\index[word]{dagger category},\index[symb]{$f^\dag$, the dagger of $f$}
  A \emph{dagger} is a contravariant involutive identity-on-objects functor. Explicitly a dagger is a functor $\dag\colon\cat{C}\op\to\cat{C}$ satisfying $A^\dag=A$ on objects and $f^{\dag\dag}=f$ on morphisms. A \emph{dagger category} is a category equipped with a dagger. 
\end{definition}

\begin{example}
  Examples of dagger categories abound. We first give some concrete examples, and will later add more abstract constructions.
  \begin{itemize}
    \item 
    Any monoid $M$ equipped with an involutive homomorphism $f\colon M\op\to M$ may be regarded as a one-object dagger category with $x^\dag=f(x)$. For example: the complex numbers with conjugation, either under addition or multiplication. Or: the algebra of (complex) $n$-by-$n$ matrices with conjugate transpose.

    \item \index[symb]{\cat{Hilb}, Hilbert spaces and bounded linear maps}
    The category \cat{Hilb} of (complex) Hilbert spaces and bounded linear maps is a dagger category, taking the dagger of $f \colon A \to B$ to be its adjoint, \ie the unique morphism satisfying $\inprod{f(a)}{b} = \inprod{a}{f^\dag(b)}$ for all $a \in A$ and $b \in B$. The full subcategory of finite-dimensional Hilbert spaces \cat{FHilb} is a dagger category too. \index[symb]{\cat{FHilb}, f.d. Hilbert spaces and linear maps}

    \item  \cat{Hilb} and \cat{FHilb} have wide subcategories \Hilb\ \index[symb]{\Hilb, Hilbert spaces and non-expansive maps}and \cat{FHilb_{\leq 1}} \index[symb]{\cat{FHilb_{\leq 1}}, f.d. Hilbert spaces and non-expansive maps}respectively, where one restricts morphisms to non-expansive linear maps, \ie maps $f$ satisfying $\|fx\|\leq \|x\|$ for all $x$. The dagger is inherited from \cat{Hilb}.

    \item \index[symb]{\cat{DStoch}, doubly stochastic matrices}
    The category \cat{DStoch} has finite sets as objects. A morphism $A \to B$ is a matrix $f \colon A \times B \to [0,1]$ that is doubly stochastic, \ie satisfies $\sum_{a \in A} f(a,b)=1$ for all $b \in B$ and $\sum_{b \in B} f(a,b)=1$ for all $a \in A$. This becomes a dagger category with $f^\dag(b,a) = f(a,b)$.

    \item 
    In the category \cat{Rel}\index[symb]{\cat{Rel}, sets and relations} with sets as objects and relations $R \subseteq A \times B$ as morphisms $A \to B$, composition is given by $S \circ R = \{ (a,c) \mid \exists b \in B \colon (a,b) \in R, (b,c) \in S \}$. This becomes a dagger category with $R^\dag=\{(b,a)\mid (a,b)\in R\}$. 
    The full subcategory $\cat{FinRel}$\index[symb]{\cat{FinRel}, finite sets and relations} of finite sets is also a dagger category.

    \item \cat{Rel} and \cat{FinRel} have wide subcategories \PInj\index[symb]{\PInj, sets and partial injections} and \cat{FPInj}\index[symb]{\cat{FPinj}, finite sets and partial injections} respectively, where one restricts morphisms to partial injections. The dagger is inherited from \cat{Rel}. 

    \item 
    If $\cat{C}$ is a category with pullbacks, the category $\cat{Span}(\cat{C})$\index[symb]{$\cat{Span}(\cat{C})$, spans of \cat{C}} of spans is defined as follows. Objects are the same as those of $\cat{C}$. A morphism $A \to B$ in $\cat{Span}(\cat{C})$ is a span $A \leftarrow S \to B$ of morphisms in $\cat{C}$, where two spans $A \leftarrow S \to B$ and $A \leftarrow S' \to B$ are identified when there is an isomorphism $S \simeq S'$ making both triangles commute. Composition is given by pullback. The category $\cat{Span}(\cat{C})$ has a dagger, given by $(A \leftarrow S \to B)^\dag = (B \leftarrow S \to A)$.

    \item Any groupoid has a canonical dagger with $f^\dag = f^{-1}$. Of course, a given groupoid might have other daggers as well. For example, the core of \cat{Hilb} inherits a dagger from \cat{Hilb} that is different from the canonical one it has as a groupoid.
  \end{itemize}
\end{example}

Dagger categories are \emph{self-dual} in a strong sense. Consequently dagger categories behave differently than categories on a fundamental level. Intuitively speaking, categories are built from objects $\bullet$ and morphisms $\bullet \rightarrow \bullet$ whereas dagger categories are built from $\bullet$ and $\bullet\leftrightarrows\bullet$. 
 The result is a theory that is essentially directionless. The guiding principle when working with dagger categories, also known as `the way of the dagger', is that all structure in sight should cooperate with the dagger. This amounts to defining dagger versions of various ordinary notions in category theory. For basic notions this is clear or at least has been already achieved, whereas for more complicated notions it takes some effort to ensure that the dagger notions are truly well-behaved. A lot of the terminology for the dagger counterparts of ordinary notions comes from the category \cat{Hilb}.

\begin{definition} 
  A morphism $f \colon A \to B$ in a dagger category is:
  \begin{itemize}
  	\item \emph{unitary}\index[word]{unitary morphism} (or a \emph{dagger isomorphism}) if $f^\dag f = \id[A]$ and $ff^\dag=\id[B]$;

  	\item an \emph{isometry}\index[word]{isometry} (or a \emph{dagger monomorphism})\index[word]{dagger monomorphism} if $f^\dag f = \id[A]$;

  	\item a \emph{coisometry}\index[word]{cosiometry} (or a \emph{dagger epimorphism})\index[word]{dagger epimorphism}) if $ff^\dag = \id[B]$;

  	\item a \emph{partial isometry}\index[word]{partial isometry} if $f=ff^\dag f$;
  \end{itemize}
  Moreover, an endomorphism $f\colon A\to A$ is
  \begin{itemize}
  	\item \emph{self-adjoint}\index[word]{self-adjoint morphism} if  $f=f^\dag$;

  	\item a \emph{projection}\index[word]{projection} (or a \emph{dagger idempotent})\index[word]{dagger idempotent}) if $f=f^\dag=ff$;

  	\item \emph{positive}\index[word]{positive morphism} if $f=g^\dag g$ for some $g \colon A \to B$. 
  \end{itemize}\index[word]{dagger subobject}
  A \emph{dagger subobject} is a subobject that can be represented by a dagger monomorphism.
  A dagger category is \emph{unitary}\index[word]{dagger category!unitary} when objects that are isomorphic are also unitarily isomorphic.
\end{definition}

A morphism $f$ is dagger monic iff $f^\dag$ is dagger epic. The other concepts introduced above are self-adjoint, \ie a morphism $f$ is unitary/a partial isometry/self-adjoint/a projection/positive iff $f^\dag$ is.
Notice that dagger monomorphisms are split monomorphisms, and dagger epimorphisms are split epimorphisms.
A morphism is dagger monic (epic) if and only if it is both a partial isometry and monic (epic).
In diagrams, we will depict 
partial isometries as \index[symb]{$\begin{tikzpicture} \draw[dagger,->] (0,0) to (1,0); \end{tikzpicture}$, partial isometry}$\begin{tikzpicture} \draw[dagger,->] (0,0) to (1,0); \end{tikzpicture}$, 
dagger monomorphisms as \index[symb]{$\begin{tikzpicture} \draw[dagger,>->] (0,0) to (1,0); \end{tikzpicture}$, dagger mono}$\begin{tikzpicture} \draw[dagger,>->] (0,0) to (1,0); \end{tikzpicture}$, 
dagger epimorphisms as \index[symb]{$\begin{tikzpicture} \draw[dagger,->>] (0,0) to (1,0); \end{tikzpicture}$, dagger epi}$\begin{tikzpicture} \draw[dagger,->>] (0,0) to (1,0); \end{tikzpicture}$, 
and dagger isomorphisms as \index[symb]{$\begin{tikzpicture} \draw[dagger,>->>] (0,0) to (1,0); \end{tikzpicture}$, unitary}$\begin{tikzpicture} \draw[dagger,>->>] (0,0) to (1,0); \end{tikzpicture}$, 
If $f$ is a partial isometry, then both $f^\dag f$ and $ff^\dag$ are projections.

The definition of a dagger subobject might seem odd: two monomorphisms $m:M\hookrightarrow A$ and $n\colon N\hookrightarrow A$ are considered to represent the same subobject when there is an isomorphism $f\colon M\to N$ such that $m=nf$, whereas one might expect that dagger monics $m$ and $n$ represent the same \emph{dagger} subobject if $f$ can be chosen to be unitary. Hence there seems to be a possibility that two dagger monics representing the same subobject might nevertheless represent different dagger subobjects. However, this can not happen: if $m,n$ are dagger monic and $f$ is an isomorphism satisfying $m=nf$, then $f$ is in fact unitary. This is because $m=nf$ implies  $n^\dag m=n^\dag n f=f$ and $mf^{-1}=n$ implies $f^{-1}=m^\dag n$ so that $f^{-1}=f^\dag$.

\begin{example}
  An \emph{inverse category}\index[word]{inverse category}\index[symb]{\PInj, sets and partial injections} is a dagger category in which every morphism is a partial isometry, and positive morphisms commute~\cite{kastl:inversecategories}: $ff^\dag f=f$ and $f^\dag f g^\dag g = g^\dag g f^\dag f$ for all morphisms $f \colon A \to B$ and $g \colon A \to C$. A prototypical inverse category is \cat{PInj}, the category of sets and partial injections, which is a subcategory of $\cat{Rel}$. 	
\end{example}

\begin{remark}\label{rem:pisetc} 
  Partial isometries need not be closed under composition~\cite[5.4]{heunen:ltwo}. The same holds for self-adjoint morphisms, projections and positive morphisms. 
  However, given partial isometries $p\colon A\to B$ and $q\colon B\to C$, their composition $qp$ is a partial isometry whenever the projections $q^\dag q$ and $pp^\dag$ commute. Similarly, a morphism $f$ factors through a partial isometry $p$ with the same codomain if and only if $pp^\dag f=f$.
\end{remark}

\begin{definition}\index[word]{dagger functor}
  A \emph{dagger functor} is a functor $F\colon\cat{C}\to \cat{D}$ between dagger categories satisfying $F(f^\dag)=F(f)^\dag$. Denote the category of small dagger categories and dagger functors by \cat{DagCat}. 
\end{definition}

For a dagger functor $F$ we can and will write $Ff^\dag$ with no ambiguity. There is no need to go further and define `dagger natural transformations': if $\sigma\colon F\to G$ is a natural transformation between dagger functors, then taking daggers componentwise defines a natural transformation $\sigma^\dag\colon G\to F$\footnote{This no longer holds for natural transformations between arbitrary functors into a dagger category, though, see Definition~\ref{def:adjointability}}.

\begin{example}\label{ex:functorcat}
  If $\cat{C}$ and $\cat{D}$ are dagger categories, then the category $[\cat{C},\cat{D}]$\index[symb]{$[\cat{C},\cat{D}]$, dagger functors $\cat{C}\to\cat{D}$} of dagger functors $\cat{C} \to \cat{D}$ and natural transformation is again a dagger category.
  In particular, taking $\cat{C}$ to be a group $G$ and $\cat{D}=\cat{Hilb}$, this shows that the category of unitary representations of $G$ and intertwiners is a dagger category.
\end{example}

Example~\ref{ex:functorcat} in fact makes \cat{DagCat} into a \emph{dagger 2-category}, as in the following definition. Notice that products of (ordinary) categories actually provide the category $\cat{DagCat}$ with products, so that the following definition makes sense.

\begin{definition}\label{def:dagtwocat}
  A \emph{dagger 2-category}\index[word]{dagger 2-category} is a category enriched in \cat{DagCat}, and a \emph{dagger 2-functor} is a $\cat{DagCat}$-enriched functor.
\end{definition}

Concretely, a dagger 2-category is a 2-category where the 2-cells have a dagger that cooperates with vertical and horizontal composition, in the sense that the equations
    \begin{align*}
      (\sigma^{\dag})^\dag&=\sigma \\
      (\tau\circ\sigma)^\dagger&=\sigma^\dagger\circ\tau^\dagger \\
      (\sigma*\tau)^\dag&=(\sigma^\dag) * (\tau^\dag)
    \end{align*}
hold,where $\circ$ denotes vertical composition and $*$ denotes horizontal composition.
Dagger 2-functors are 2-functors that preserve this dagger. Strictly speaking, this defines \emph{strict} dagger 2-categories. Defining dagger bicategories is straightforward but not necessary for our purposes.


The forgetful functor $\cat{DagCat}\to\cat{Cat}$, has both adjoints. We recall the definitions of free and cofree dagger categories~\cite[3.1.17,3.1.19]{heunen:thesis}, which will be used in Examples~\ref{ex:daggershaped} and~\ref{ex:cofree} and in the proof of Theorem~\ref{thm:ambilimitsunique}.

\begin{proposition}\label{prop:free}\index[word]{dagger category!free}\index[symb]{\cat{ZigZag}(\cat{C}), the free dagger category on \cat{C}}
  The forgetful functor $\cat{DagCat} \to \cat{Cat}$ has a left adjoint \cat{ZigZag(-)}. The objects of \cat{Zigzag}(\cat{C}) are the same as in \cat{C}, and a morphism $A \to B$ is an alternating sequence of morphisms $A \rightarrow C_1 \leftarrow \cdots \rightarrow C_n \leftarrow B$ from $\cat{C}$, subject to the identifications $\big(A \sxto{f} C \sxfrom{\id} C \sxto{g} D \sxfrom{h} B \big) = \big( A \sxto{g \circ f} D \sxfrom{h} B \big)$ and $\big(A \sxto{f} C \sxfrom{g} D \sxto{\id} D \sxfrom{h} B\big) = \big( A \sxto{f} C \sxfrom{g \circ h} B\big )$; composition is given by juxtaposition. The category $\cat{Zigzag}(\cat{C})$ has a dagger $(f_1,\ldots,f_n)^\dag = (f_n,\ldots,f_1)$.
\end{proposition}

This adjunction induces a monad on \cat{Cat}, and it is easy to see that \cat{DagCat} is the category of algebras for this monad.

\begin{proposition}\label{prop:cofree}\index[word]{dagger category!cofree}\index[symb]{\cat{C_\leftrightarrows}, the cofree dagger category on \cat{C}}
  The forgetful functor $\cat{DagCat} \to \cat{Cat}$ has a right adjoint $(-)_\leftrightarrows$,
  which sends a category $\cat{C}$ to the full subcategory of $\cat{C}\op \times \cat{C}$ with objects of the form $(A,A)$, 
  and sends a functor $F$ to the restriction of $F\op \times F$. The dagger on $\cat{C_\leftrightarrows}$ is given by $(f,g)^\dag=(g,f)$.
\end{proposition}

The first adjunction induces a monad and the second a comonad on \cat{Cat}. Hence one might wonder what are the algebras of the monad and coalgebras of the comonad. Both of these categories in fact coincide with \cat{DagCat}. Recall that a functor $U\colon\cat{D}\to\cat{C}$ is monadic if there is a monad $T$ on \cat{D} and an equivalence of categories $\cat{D}\to{C}^T$ such that $U$ is isomorphic to the composite $\cat{D}\to{C}^T\to\cat{C}$.  Beck's monadicity theorem~\cite{beck:triples} characterizes monadic functors.

\begin{theorem}\label{thm:dagcattocatis(co)monadic}
  The forgetful functor $U\colon\cat{DagCat}\to\cat{Cat}$ is both monadic and comonadic.
\end{theorem}

One could show this directly as follows:
  \begin{enumerate}[(i)]
    \item  By the universal property of \cat{ZigZag(C)}, any dagger on \cat{C} induces a map $a\colon \cat{ZigZag(C)}\to\cat{C}$ such that $a\circ\eta\colon \cat{C}\to \cat{ZigZag(C)}\to\cat{C}$ equals the identity on \cat{C}. The map $a$ can be shown to make \cat{C} into an algebra for the monad, and dagger functors induce algebra homomorphisms. 
    \item Conversely, an algebra structure $a\colon \cat{ZigZag(C)}\to\cat{C}$ induces a dagger on \cat{C}: the dagger on \cat{C} can be defined as the composite $a\circ\dag\circ\eta$, where the $\dag$ in the middle refers to the dagger on \cat{ZigZag(C)}. Algebra homomorphisms induce dagger functors. 
    \item Similarly, by the universal property of \cat{C_\leftrightarrows} any dagger on \cat{C} gives rise to a map $a\colon\cat{C}\to \cat{C_\leftrightarrows}$ such that $\epsilon\circ a=\id[C]$, where $\epsilon$ is the counit. This gives a coalgebra structure on \cat{C} and dagger functors are homomorphisms of the corresponding coalgebras.
    \item If $a\colon\cat{C}\to \cat{C_\leftrightarrows}$ is a coalgebra sructure on \cat{C}, then the composite map $\epsilon\circ\dag\circ a$ is a dagger on \cat{C}, where the $\dag$ in the middle refers to the dagger on dagger on \cat{C_\leftrightarrows}. Homomorphisms of coalgebras are dagger functors between the corresponding dagger categories.
    \item Constructions (i) and (ii)  ((iii) and (iv)) are mutually inverse.
  \end{enumerate}

However, we avoid some of the tedious calculations involved by using Beck's monadicity theorem instead. In fact, the following sufficient condition is enough for us:
\begin{theorem}\label{thm:sufficientformonadicity} A functor $U\colon\cat{D}\to\cat{C}$ is monadic if
  \begin{enumerate}[(i)]
      \item $U$ has a left adjoint
      \item $U$ reflects isomorphisms
      \item $D$ has and $U$ preserves coequalizers.
  \end{enumerate}
\end{theorem}

The usual monadicity theorem gives necessary and sufficient conditions for a functor to be monadic (see \eg~\cite[Theorem 4.4.4.]{borceux:vol2}). To deduce Theorem~\ref{thm:sufficientformonadicity} from it, one just notes that conditions (i) and (ii) are as in the usual statement and condition (iii) is stronger than the usual requirement that concerns so-called $U$-split coequalizers.

\begin{proof}[Proof of Theorem~\ref{thm:dagcattocatis(co)monadic}]
  To see that $U$ is comonadic, note first that it has a right adjoint by Proposition~\ref{prop:cofree}. Moreover, $U$ clearly reflects isomorphisms: if $F\colon \cat{C}\to \cat{D}$ is a dagger functor that is an isomorphism of the underlying categories, then $F^{-1}$ is also a dagger functor. Hence to use (the dual of) Theorem~\ref{thm:sufficientformonadicity}, it remains to show that \cat{DagCat} has all equalizers and $U$ preserves them. If $F,G\colon \cat{C}\to\cat{D}$ are parallel dagger functors, then the equalizer $E$ of $UF,UG$ is the subcategory of $\cat{C}$ containing those objects $A$ and those morphisms $f$ satisfying $FA=GA$ and $Ff=Gf$. Since $F$ and $G$ are dagger functors, if $F(f)=G(f)$, then also $F(f^\dag)=G(f^\dag)$ holds. Hence the subcategory $E$ is closed under the dagger of \cat{C}, and this induces a dagger on $E$. It is trivial to show that this gives an equalizer in \cat{DagCat}, and $U$ preserves it since it is a right adjoint by Proposition~\ref{prop:free}.

  Next we prove that $U$ is monadic. $U$ has a left adjoint by Proposition~\ref{prop:free} and we know that $U$ reflects isomorphisms. Hence it remains to show that \cat{DagCat} has coequalizers and that $U$ preserves them. Preservation is clear since $U$ is a left adjoint. The existence of coequalizers could be shown directly by \eg using the concrete description of them in \cat{Cat} from~\cite{bednarczyk1999generalized}, but follows readily from $U$ being comonadic: categories of algebras are known to have any limits present in the underlying category~\cite[Proposition 4.4.1]{borceux:vol2} so by duality \cat{DagCat} has all colimits that \cat{Cat} has, and hence in particular it has coequalizers.
\end{proof}

\begin{corollary} \cat{DagCat} is complete and cocomplete.
\end{corollary}

\begin{proof} Categories of algebras have all limits present in the underlying category~\cite[Proposition 4.4.1]{borceux:vol2}, so monadicity of $U\colon\cat{DagCat}\to\cat{Cat}$ implies that \cat{DagCat} has all limits that \cat{Cat} has. Since $U$ is also comonadic, \cat{DagCat} has all colimits that \cat{Cat} has. Hence the result follows from \cat{Cat} being complete and cocomplete.
\end{proof}

Neither of the adjoints of $U$ can be defined to operate on 2-cells, and indeed the forgetful 2-functor $\colon\cat{DagCat}\to\cat{Cat}$ has no 2-adjoints. To see this, note that a right 2-adjoint $R$ to $U$ would, depending on the definition, result in a natural family of isomorphisms/equivalences $\cat{DagCat}(C,RD)\cong\cat{Cat}(UC,D)$ --  but the left hand side is always a dagger category, whereas the right hand side might fail to be self-dual: for an easy example, pick $C=\bullet$ and $D=\bullet\to\bullet$. The non-existence of a left 2-adjoint is proved similarly.

\begin{definition}\label{def:adjointability}\index[word]{adjointable natural transformation}
  Consider a natural transformation $\sigma\colon F\Rightarrow G$ where $F,G$ are arbitrary functors with codomain a dagger category \cat{C}. Taking daggers componentwise generally does not yield a natural transformation $G\Rightarrow F$. When it does, we call $\sigma$ \emph{adjointable}. For small \cat{J}, we denote the dagger category of functors $\cat{J}\to\cat{C}$ and adjointable natural transformations by $[\cat{J},\cat{C}]_\dag$\index[symb]{$[\cat{J},\cat{C}]_\dag$ functors and \emph{adjointable} transformations}.
\end{definition}

A natural transformation $\sigma\colon F\Rightarrow G$ is adjointable if and only if $\sigma$ defines a natural transformation $\dag\circ F\Rightarrow \dag\circ G$

\begin{lemma}\label{lem:daggersubfunctorsaredagger} 
  If $G$ is a dagger functor and $\sigma\colon F\Rightarrow G$ is a natural transformation that is pointwise dagger monic, then $F$ is a dagger functor as well.
\end{lemma}
\begin{proof}
  As $\sigma$ is pointwise dagger monic and natural, the following diagram commutes for any $f$:
  \[\begin{tikzpicture}
   	\matrix (m) [matrix of math nodes,row sep=2em,column sep=4em,minimum width=2em]
   	{
    	FA & FB & \\
    	GA & GB \\};
   	\path[->]
   	(m-1-1) edge[dagger,>->] node [left] {$\sigma$} (m-2-1)
      	     edge node [above] {$F(f)$} (m-1-2)
   	(m-2-1) edge node [below] {$G(f)$} (m-2-2)
   	(m-2-2) edge[dagger,->>] node [right] {$\sigma^\dag$} (m-1-2);
  \end{tikzpicture}\]
  Taking the dagger of this diagram, and replacing $f$ with $f^\dag$, respectively, results in two commuting diagrams:
  \[\begin{tikzpicture}
  	\matrix (m) [matrix of math nodes,row sep=2em,column sep=4em,minimum width=2em]
   	{
    	FB & FA & \\
    	GB & GA \\};
   	\path[->]
   	(m-1-1) edge[dagger,>->] node [left] {$\sigma$} (m-2-1)
           edge node [above] {$F(f)^\dag$} (m-1-2)
   	(m-2-1) edge node [below] {$G(f)^\dag$} (m-2-2)
   	(m-2-2) edge[dagger,->>] node [right] {$\sigma^\dag$} (m-1-2);
  	\end{tikzpicture}
  	\quad
   	\begin{tikzpicture}
  	\matrix (m) [matrix of math nodes,row sep=2em,column sep=4em,minimum width=2em]
  	{
   	FB & FA & \\
   	GB & GA \\};
  	\path[->]
  	(m-1-1) edge[dagger,>->] node [left] {$\sigma$} (m-2-1)
     	     edge node [above] {$F(f^\dag)$} (m-1-2)
  	(m-2-1) edge node [below] {$G(f^\dag)$} (m-2-2)
  	(m-2-2) edge[dagger,->>] node [right] {$\sigma^\dag$} (m-1-2);
  \end{tikzpicture}\]
  Because $G$ is a dagger functor, the bottom arrows in both diagrams are equal. Therefore top arrows are equal too, making $F$ a dagger functor.
\end{proof}

We end with a useful folklore result\footnote{Peter Selinger mentions this fact at~\cite{nlab:evil}}: 
\begin{lemma}\label{lem:induced daggers}
  Let $F\colon\cat{C}\to\cat{D}$ be a full and faithful functor. Then any dagger on \cat{D} induces a unique dagger on \cat{C} such that $F$ is a dagger functor.
\end{lemma}
\begin{proof}
    Let $\dag$ be a dagger on \cat{D}. Given a morphism $f\colon A\to B$ in \cat{C}, we define its dagger to be the unique map $B\to A$ that $F$ maps to $F(f)^\dag$. It is easy to check that this assignment gives a dagger on \cat{C} and as $F$ is faithful the uniqueness condition is satisfied.
\end{proof}

\section{Graphical calculus}\label{sec:pictures}

Many proofs in chapter~\ref{chp:monads} are most easily presented in graphical form. This section briefly overviews the graphical calculus that governs monoidal (dagger) categories, such as the category $[\cat C, \cat C]$ where our monads will live. For more information, see~\cite{selinger:graphicallanguages}.

\begin{definition}\index[word]{dagger category!(symmetric) monoidal}
  A \emph{(symmetric) monoidal dagger category} is a dagger category that is also a (symmetric) monoidal category, satisfying $(f\otimes g)^\dag=f^\dag\otimes g^\dag$ for all morphisms $f$ and $g$, whose coherence maps $\lambda \colon I \otimes A \to A$, $\rho \colon A \otimes I \to A$, and $\alpha \colon (A \otimes B) \otimes C \to A \otimes (B \otimes C)$ (and $\sigma \colon A \otimes B \to B \otimes A$) are unitary. 
\end{definition}

\begin{example}\index[symb]{\cat{Rel}, sets and relations}\index[symb]{\cat{Hilb}, Hilbert spaces and bounded linear maps}\index[symb]{\cat{FHilb}, f.d. Hilbert spaces and linear maps}\index[symb]{$[\cat{C},\cat{D}]$, dagger functors $\cat{C}\to\cat{D}$}
  Many monoidal structures on dagger categories make them monoidal dagger categories:
  \begin{itemize}
	\item The dagger category $\cat{Rel}$ is a monoidal dagger category under cartesian product.
	\item The dagger category $\cat{(F)Hilb}$ is a monoidal dagger category under tensor product.
	\item For any dagger category \cat{C}, the dagger category $[\cat C, \cat C]$ of dagger functors $\cat C \to \cat C$ is a monoidal dagger category under composition of functors.
	\item Any monoidal groupoid is a monoidal dagger category under $f^\dag = f^{-1}$.
  \end{itemize}
\end{example}

There is a sound and complete graphical calculus for such categories, that represents a morphism $f \colon A \to B$ as 
$\setlength\morphismheight{3mm}\begin{pic}
  \node[morphism,font=\tiny] (f) at (0,0) {$f$};
  \draw (f.south) to +(0,-.1);
  \draw (f.north) to +(0,.1);
\end{pic}$,\index[symb]{$\setlength\morphismheight{3mm}\begin{pic}
  \node[morphism,font=\tiny] (f) at (0,0) {$f$};
  \draw (f.south) to +(0,-.1);
  \draw (f.north) to +(0,.1);
\end{pic}$ a morphism $f$ in the graphical calculus}
and composition, tensor product, and dagger as follows.\index[word]{graphical calculus}
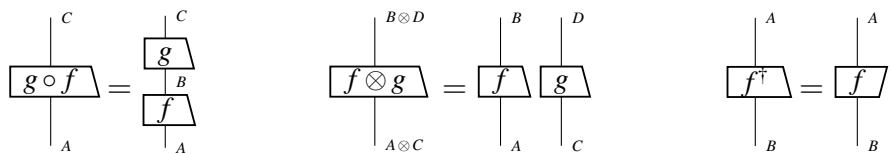
\[
  \begin{pic}
    \node[morphism] (f) {$g \circ f$};
    \draw (f.south) to +(0,-.65) node[right] {$A$};
    \draw (f.north) to +(0,.65) node[right] {$C$};
  \end{pic}
  = 
  \begin{pic}
    \node[morphism] (g) at (0,.75) {$g\vphantom{f}$};
    \node[morphism] (f) at (0,0) {$f$};
    \draw (f.south) to +(0,-.3) node[right] {$A$};
    \draw (g.south) to node[right] {$B$} (f.north);
    \draw (g.north) to +(0,.3) node[right] {$C$};
  \end{pic}
  \qquad\qquad
  \begin{pic}
    \node[morphism] (f) {$f \otimes g$};
    \draw (f.south) to +(0,-.65) node[right] {$A \otimes C$};
    \draw (f.north) to +(0,.65) node[right] {$B \otimes D$};
  \end{pic}
  = 
  \begin{pic}
    \node[morphism] (f) at (-.4,0) {$f$};
    \node[morphism] (g) at (.4,0) {$g\vphantom{f}$};
    \draw (f.south) to +(0,-.65) node[right] {$A$};
    \draw (f.north) to +(0,.65) node[right] {$B$};
    \draw (g.south) to +(0,-.65) node[right] {$C$};
    \draw (g.north) to +(0,.65) node[right] {$D$};
  \end{pic}
  \qquad\qquad
  \begin{pic}
    \node[morphism] (f) {$f^\dag$};
    \draw (f.south) to +(0,-.65) node[right] {$B$};
    \draw (f.north) to +(0,.65) node[right] {$A$};
  \end{pic}
  =
  \begin{pic}
    \node[morphism,hflip] (f) {$f$};
    \draw (f.south) to +(0,-.65) node[right] {$B$};
    \draw (f.north) to +(0,.65) node[right] {$A$};
  \end{pic}
\]

Using these building blocks one can then draw more complicated diagrams such as 
  
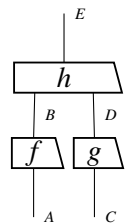
\[
    \begin{pic}
    \node[morphism] (h) at (0,1) {$\quad h\quad $};
    \node[morphism] (f) at (-.4,0) {$f$};
    \node[morphism] (g) at (.4,0) {$g\vphantom{f}$};
    \draw (f.south) to +(0,-.65) node[right] {$A$};
    \draw (f.north) to node[right] {$B$} (h.south west);
    \draw (g.south) to +(0,-.65) node[right] {$C$};
    \draw (g.north) to node[right] {$D$} (h.south east);
    \draw (h.north) to +(0,.65)node[right] {$E$}; 
  \end{pic}
  \]
denoting a morphism $A\otimes B\to E$. Tensor product, composition and the dagger extend to larger diagrams straightforwardly: tensor product of diagrams is depicted by juxtaposition, composition by wiring the outputs of one diagram to the inputs of the other\footnote{for this to be defined the outputs must match the inputs}, and the dagger of a diagram amounts to flipping the diagram upside down.

Distinguished morphisms are often depicted with special diagrams instead of generic boxes. For example, 
the identity $A \to A$ and the swap map of symmetric monoidal dagger categories are drawn as:
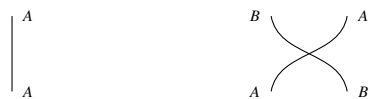
\[\begin{pic}
    \draw (0,0) node[right] {$A$} to (0,1) node[right] {$A$};
  \end{pic}
  \qquad\qquad\qquad
  \begin{pic}
    \draw (0,0) node[left] {$A$} to[out=80,in=-100] (1,1) node[right] {$A$};
    \draw (1,0) node[right] {$B$} to[out=100,in=-80] (0,1) node[left] {$B$};
  \end{pic}
\] 
whereas the (identity on) the monoidal unit object $I$ is drawn as the empty picture: 
\[ 
  \;
\]
  The following definition gives another example: the unit and multiplication of a monoid get a special diagram.

\begin{definition}\index[word]{monoid}\index[symb]{$\tinymult$, multiplication of a monoid}
  A \emph{monoid} in a monoidal category is an object $A$ with morphisms $\tinymult \colon A \otimes A \to A$ and $\tinyunit \colon I \to A$, satisfying the following equations.
  
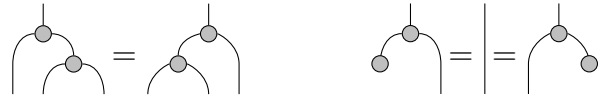
\[
    \begin{pic}[scale=.4]
      \node[dot] (t) at (0,1) {};
      \node[dot] (b) at (1,0) {};
      \draw (t) to +(0,1);
      \draw (t) to[out=0,in=90] (b);
      \draw (t) to[out=180,in=90] (-1,0) to (-1,-1);
      \draw (b) to[out=180,in=90] (0,-1);
      \draw (b) to[out=0,in=90] (2,-1);
    \end{pic}
    =
    \begin{pic}[yscale=.4,xscale=-.4]
      \node[dot] (t) at (0,1) {};
      \node[dot] (b) at (1,0) {};
      \draw (t) to +(0,1);
      \draw (t) to[out=0,in=90] (b);
      \draw (t) to[out=180,in=90] (-1,0) to (-1,-1);
      \draw (b) to[out=180,in=90] (0,-1);
      \draw (b) to[out=0,in=90] (2,-1);
    \end{pic}
  \qquad\qquad
  \begin{pic}[scale=.4]
    \node[dot] (d) {};
    \draw (d) to +(0,1);
    \draw (d) to[out=0,in=90] +(1,-1) to +(0,-1);
    \draw (d) to[out=180,in=90] +(-1,-1) node[dot] {};
  \end{pic}
  =
  \begin{pic}[scale=.4]
    \draw (0,0) to (0,3);
  \end{pic}
  =
  \begin{pic}[yscale=.4,xscale=-.4]
    \node[dot] (d) {};
    \draw (d) to +(0,1);
    \draw (d) to[out=0,in=90] +(1,-1) to +(0,-1);
    \draw (d) to[out=180,in=90] +(-1,-1) node[dot] {};
  \end{pic}
  \]
  A \emph{comonoid}\index[word]{comonoid}\index[symb]{$\tinycomult$, comultiplication of a comonoid}
   in a monoidal category is a monoid in the opposite category; an object $A$ with morphisms $\tinycomult \colon A \to A \otimes A$ and $\tinycounit \colon A \to I$ satisfying the duals of the above equations.
  A monoid in a symmetric monoidal category is \emph{commutative} if it satisfies:
  \[\begin{pic}[scale=.4]
      \node[dot] (d) at (0,0) {};
      \draw (d.north) to +(0,1);
      \draw (d.west) to[out=180,in=90] +(-.75,-1) to +(0,-1);
      \draw (d.east) to[out=0,in=90] +(.75,-1) to +(0,-1);
    \end{pic}
    \quad = \quad
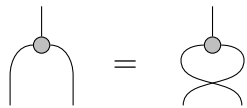
    \begin{pic}[scale=.4]
      \node[dot] (d) at (0,0) {};
      \draw (d.north) to +(0,1);
      \draw (d.west) to[out=180,in=90] +(-.75,-.5) to[out=-90,in=90] +(2,-1.5);
      \draw (d.east) to[out=0,in=90] +(.75,-.5) to[out=-90,in=90] +(-2,-1.5);
    \end{pic}
  \]
  A monoid in a monoidal dagger category is a \emph{dagger Frobenius monoid}\index[word]{monoid!dagger Frobenius} if it satisfies the following \emph{Frobenius law}.
  \begin{equation}\index[word]{Frobenius law!of a monoid}
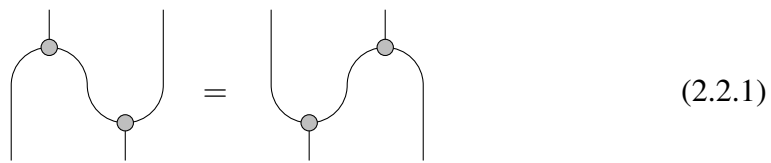
\label{eq:frobeniuslaw}
  \begin{pic}
   \draw (0,0) to (0,1) to[out=90,in=180] (.5,1.5) to (.5,2);
   \draw (.5,1.5) to[out=0,in=90] (1,1) to[out=-90,in=180] (1.5,.5) to (1.5,0);
   \draw (1.5,.5) to[out=0,in=-90] (2,1) to (2,2);
   \node[dot] at (.5,1.5) {};
   \node[dot] at (1.5,.5) {};
  \end{pic}
  \quad = \quad
  \begin{pic}[xscale=-1]
   \draw (0,0) to (0,1) to[out=90,in=180] (.5,1.5) to (.5,2);
   \draw (.5,1.5) to[out=0,in=90] (1,1) to[out=-90,in=180] (1.5,.5) to (1.5,0);
   \draw (1.5,.5) to[out=0,in=-90] (2,1) to (2,2);
   \node[dot] at (.5,1.5) {};
   \node[dot] at (1.5,.5) {};
  \end{pic}
 \end{equation}
\end{definition}

The Frobenius law might look mysterious, but will turn out to be precisely the right property to make monads respect daggers. 
Section~\ref{sec:coherence} below will formally justify it as a coherence property between closure and the dagger. 
For now we illustrate that the Frobenius law corresponds to natural mathematical structures in example categories.

\begin{example}\label{ex:groupoidFrob} 
  See~\cite{heunencontrerascattaneo:groupoids,heunentull:groupoids,vicary:quantumalgebras} for more information on the following examples.
  \begin{itemize} 
    \item Let $\cat{G}$ be a small groupoid, and $G$ its set of objects. The assignments 
    \begin{align*}
      \{*\}&\mapsto \{\id[A]\mid A\in G\} 
      &
      (f,g)&\mapsto 
      \begin{cases} 
        \{f\circ g\} &\text{ if }f\circ g\text{ is defined} \\ 
        \emptyset &\text{ otherwise} 
      \end{cases} 
      \\
    \intertext{define a dagger Frobenius monoid in $\cat{Rel}$ on the set of morphisms of $\cat{G}$.
    Conversely, any dagger Frobenius monoid in $\cat{Rel}$ is of this form.
    \item Let $\cat{G}$ be a finite groupoid, and $G$ its set of objects. The assignments}
      1 &\mapsto \sum_{A\in G}\id[A] 
      &
      f\otimes g & \mapsto 
      \begin{cases} 
        f\circ g &\text{ if }f\circ g\text{ is defined} \\ 
        0 &\text{ otherwise} 
      \end{cases} 
    \end{align*}
    define a dagger Frobenius monoid in \cat{(F)Hilb} on the Hilbert space of which the morphisms of $\cat{G}$ form an orthonormal basis.
    Conversely, any dagger Frobenius monoid in $\cat{(F)Hilb}$ is of this form.
  \end{itemize}
\end{example}

The following lemma exemplifies graphical reasoning. Recall that a (co)monoid homomorphism is a morphism $f$ between (co)monoids satisfying $f \circ \tinyunit = \tinyunit$ and $f \circ \tinymult = \tinymult \circ (f \otimes f)$ ($\tinycounit \circ f = \tinycounit$ and $\tinycomult \circ f = (f \otimes f) \circ \tinycomult$).

\begin{lemma}\label{lem:strictmorphismsareiso}
  A monoid homomorphism between dagger Frobenius monoids in a monoidal dagger category, that is also a comonoid homomorphism, is an isomorphism.
\end{lemma}
\begin{proof}
  Construct an inverse to $A \sxto{f} B$ as follows:
  \[\begin{pic}[xscale=.75,yscale=.5]
      \node (f) [morphism] at (0,0) {$f$};
      \draw (f.north)
      to [out=up, in=right] +(-0.5,0.5) node (d) [dot] {}
      to [out=left, in=up] +(-0.5,-0.5)
      to [out=down, in=up] +(0,-1.8) node [below] {$B$};
      \draw (d.north) to +(0,0.4) node [dot] {};
      \draw (f.south)
      to [out=down, in=left] +(0.5,-0.5) node (d) [dot] {}
      to [out=right, in=down] +(0.5,0.5)
      to [out=up, in=down] +(0,1.8) node [above] {$A$};
      \draw (d.south) to +(0,-0.4) node [dot] {};
    \end{pic}
  \]
  The composite with $f$ gives the identity in one direction:
  
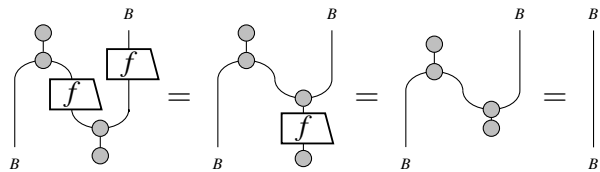
\[
    \begin{pic}[xscale=.75,yscale=.5]
      \node (f) [morphism] at (0,0) {$f$};
      \draw (f.north)
      to [out=up, in=right] +(-0.5,0.5) node (d) [dot] {}
      to [out=left, in=up] +(-0.5,-0.5)
      to [out=down, in=up] +(0,-1.8) node [below] {$B$};
      \draw (d.north) to +(0,0.5) node [dot] {};
      \draw (f.south)
      to [out=down, in=left] +(0.5,-0.5) node (d) [dot] {}
      to [out=right, in=down] +(0.5,0.5)
      to [out=down, in=down] +(0,0.8)
      node (f2) [morphism, anchor=south, width=0.2cm]
      {$f$};
      \draw (d.south) to +(0,-0.5) node [dot] {};
      \draw (f2.north) to +(0,0.5) node [above] {$B$};
    \end{pic}
    =
    \begin{pic}[xscale=.75,yscale=.5]
      \draw (0,0)
      to [out=up, in=right] +(-0.5,0.5) node (d) [dot] {}
      to [out=left, in=up] +(-0.5,-0.5)
      to [out=down, in=up] +(0,-1.8) node [below] {$B$};
      \draw (d.north) to +(0,0.5) node [dot] {};
      \node (f) [morphism] at (0.5,-1.3) {$f$};
      \draw (0,0)
      to [out=down, in=left] +(0.5,-0.5) node (e) [dot] {}
      to [out=right, in=down] +(0.5,0.5)
      to [out=up, in=down] +(0,1.3) node [above] {$B$};
      \draw (f.north) to (e.south);
      \draw (f.south) to +(0,-0.4) node [dot] {};
    \end{pic}
    =
    \begin{pic}[xscale=.75,yscale=.5]
      \draw (0,0)
      to [out=up, in=right] +(-0.5,0.5) node (d) [dot] {}
      to [out=left, in=up] +(-0.5,-0.5)
      to [out=down, in=up] +(0,-1.5) node [below] {$B$};
      \draw (d.north) to +(0,0.5) node [dot] {};
      \draw (0.5,-0.5) to +(0,-0.5) node [dot] {};
      \draw (0,0)
      to [out=down, in=left] +(0.5,-0.5) node (d) [dot] {}
      to [out=right, in=down] +(0.5,0.5)
      to [out=up, in=down] +(0,1.6) node [above] {$B$};
    \end{pic}
    =
    \begin{pic}[yscale=.5]
      \draw (0,0) node [below] {$B$} to +(0,3.1) node [above] {$B$};
    \end{pic}
  \]
  The third equality uses the Frobenius law~\eqref{eq:frobeniuslaw} and the unit law.
  The other composite is the identity by a similar argument.
\end{proof}

\chapter{Equivalences and adjunctions}\label{chp:equivalences}

\section{Dagger equivalences}

Given an equivalence of categories $F\colon \cat{C}\rightleftarrows\cat{D}\colon G$ and a dagger on \cat{C}, one would like to be able to make \cat{D} into a dagger category in such a manner that $F$ and $G$ would be an equivalence of dagger categories. Unfortunately, this isn't possible in general: while Lemma~\ref{lem:induced daggers} gives a dagger on \cat{D} so that $G$ is a dagger functor, $F$ can fail to be a dagger functor. Selinger~\cite{selinger:evil} gave the following argument to this effect. Consider the forgetful functor $F\colon\cat{FHilb}\to\cat{FVect}$. Since $F$ is full, faithful, and essentially surjective on objects, it is part of an equivalence of categories. Equip the same vector space $V$ with two different inner products, and consider the map $f$ between them defined by $v\mapsto v$. Then $f$ is not unitary, but $F(f)=\id[V]$ is unitary with respect to any dagger structure on \cat{FVect}. Thus there is no dagger on \cat{FVect} such that $F$ would be a dagger functor.

While this argument shows that arbitrary equivalences cannot be lifted to equivalences in the 2-category \cat{DagCat}, one can argue that this endeavour is misguided in the first place: if one takes ``the way of the dagger'' seriously, we see that \cat{DagCat} isn't just a 2-category, but it is a $\emph{dagger}$ 2-category, so the notion of equivalence should be different -- one should require the relevant natural isomorphisms to be unitary as in the definition below.  

\begin{definition}\index[word]{dagger equivalence} 
	A dagger equivalence in a dagger 2-category consists of 1-cells 
    \[F\colon \cat{C}\rightleftarrows\cat{D}\colon G\]
   and \emph{unitary} isomorphisms $\eta\colon \id[\cat{C}]\to GF$ and $\epsilon\colon FG\to \id[\cat{D}]$, \ie $\eta$ and $\epsilon$ are invertible 2-cells satisfying $\eta^\dag=\eta^{-1}$ and $\epsilon^\dag=\epsilon^{-1}$.

	Specifically, a dagger equivalence between dagger categories \cat{C} and \cat{D} consists of dagger functors $F\colon \cat{C}\rightleftarrows\cat{D}\colon G$ and unitary natural transformations $\eta\colon \id[\cat{C}]\to GF$ and $\epsilon\colon FG\to \id[\cat{D}]$.
\end{definition}

Now, if $(\cat{C},\dag)$ is a dagger category and $F\colon \cat{C}\rightleftarrows\cat{D}\colon G$, $\eta\colon \id[\cat{C}]\to GF$, $\epsilon\colon FG\to \id[\cat{D}]$ is an arbitrary equivalence in \cat{Cat}, one would not expect it to lift to a dagger equivalence -- at the very least, $\eta$ and $G\epsilon$ should be unitary. Failing to lift to a equivalence in the underlying 2-category of \cat{DagCat}, while a common occurrence, is less interesting as that is not the right notion of equivalence for dagger categories. However, these two necessary conditions turns out to be sufficient. Before the proof, we give the two lemmas.

	\begin{lemma}\label{lem:unitarilyisomorphicisdagger}
		If $F,G\colon\cat{C}\to\cat{D}$ are functors that are unitarily isomorphic, then $F$ preserves the dagger iff $G$ does.
	\end{lemma}

\begin{proof}
	Assume that $F$ preserves the dagger and let $\sigma\colon F\to G$ be a unitary isomorphism. As $\sigma$ is unitary, the following diagrams commute:
	  \[
    	\begin{tikzpicture}
     	\matrix (m) [matrix of math nodes,row sep=2em,column sep=4em,minimum width=2em]
     	{
      	GB & GA & \\
      	FB & FA \\};
     	\path[->]
     	(m-1-1) edge node [left] {$\sigma^\dag$} (m-2-1)
             edge node [above] {$(Gf)^\dag$} (m-1-2)
     	(m-2-1) edge node [below] {$(Ff)^\dag$} (m-2-2)
     	(m-2-2) edge node [right] {$\sigma$} (m-1-2);
    	\end{tikzpicture}
  	\quad
    	\begin{tikzpicture}
     	\matrix (m) [matrix of math nodes,row sep=2em,column sep=4em,minimum width=2em]
     	{
      	GB & GA & \\
      	FB & FA \\};
     	\path[->]
     	(m-1-1) edge node [left] {$\sigma^\dag$} (m-2-1)
        	     edge node [above] {$G(f^\dag)$} (m-1-2)
     	(m-2-1) edge node [below] {$F(f^\dag)$} (m-2-2)
     	(m-2-2) edge node [right] {$\sigma$} (m-1-2);
    	\end{tikzpicture}
  	\]
  As $F$ is a dagger functor, the bottom paths in both diagrams are equal, and thus the top paths are as wells, showing that $G$ is a dagger functor. The other direction follows by symmetry.
\end{proof}

	\begin{theorem}\label{thm:liftdagequiv}
		Let $(\cat{C},\dag)$ be a dagger category and $F\colon \cat{C}\rightleftarrows\cat{D}\colon G$, $\eta\colon \id[\cat{C}]\to GF$, $\epsilon\colon FG\to \id[\cat{D}]$ be an equivalence in \cat{Cat} such that $\eta$ and $G\epsilon$ are unitary. Then there is a unique dagger on \cat{D} such that $F,G,\eta,\epsilon$ forms a dagger equivalence.
	\end{theorem}

\begin{proof}
		There is a unique dagger on $D$ such that $G$ preserves the dagger. For this dagger, $\epsilon$ is clearly unitary. It remains to show that $F$ preserves the dagger. By Lemma~\ref{lem:unitarilyisomorphicisdagger} both $FG$ and $GF$ preserve the dagger. Furthermore $\epsilon F\colon FGF\to F$ is unitary, so that it suffices to prove that $FGF$ preserves the dagger. We calculate: 

  		\begin{equation*}
  			FGF(f^\dag)=F((GFf)^\dag)=FG((Ff)^\dag)=(FGFf)^\dag
  		\end{equation*}
\end{proof} 

For adjoint equivalences the situation is even cleaner since the condition that $G\epsilon$ is unitary is redundant: in other words exactly one half of the equivalence needs to cooperate with the dagger. 

\begin{corollary}\label{cor:liftdagequiv}
    Let $(\cat{C},\dag)$ be a dagger category and $F\colon \cat{C}\rightleftarrows\cat{D}\colon G$, $\eta\colon \id[\cat{C}]\to GF$, $\epsilon\colon FG\to \id[\cat{D}]$ be an adjoint equivalence in \cat{Cat} such that $\eta$ is unitary. Then there is a unique dagger on \cat{D} such that $(F,G,\eta,\epsilon)$ forms a dagger equivalence.
\end{corollary}

\begin{proof} By Theorem~\ref{thm:liftdagequiv} it suffices to show that $G\epsilon$ is unitary. Since the equivalence in question is an adjoint equivalence we have $G\epsilon\circ \eta G=\id$ whence $G\epsilon$ is the inverse of a unitary isomorphism and hence unitary itself.
\end{proof}

As with standard equivalences, one can only work with one half of a dagger equivalence. This is made by precise by the following Lemma which we cite from~\cite{vicary2011completeness}.

  \begin{lemma}[Vicary]\label{lem:halfequiv}
  If $F\colon \cat{C}\to\cat{D}$ is a dagger functor that is full, faithful and unitarily essentially surjective, then it forms part of a dagger equivalence.
  \end{lemma}

  \begin{proof} See \cite[Lemma 5.1]{vicary2011completeness}.
  \end{proof}

  \begin{corollary}
  If $F\colon\cat{C}\to (\cat{D},\dag)$ is full, faithful and unitarily essentially surjective, then there is a unique dagger on $\cat{C}$ such that $F$ forms part of a dagger equivalence.
  \end{corollary}

  \begin{proof}
  Follows immediately from Lemmas~\eqref{lem:induced daggers} and \eqref{lem:halfequiv}.
  \end{proof}

A way to construct the other half $G\colon \cat{FVect}\to \cat{FHilb}$ of the equivalence is to choose an inner product for each vector space. For such $G$, $FG=\id[\cat{FVect}]$, and thus $G\epsilon$ is unitary. As there is no way to make this into a dagger equivalence, the Theorem~\ref{thm:liftdagequiv} implies that for such $G$ the isomorphism $\eta\colon\id[\cat{FHilb}]\to GF$ is never unitary at every component.  

The problem comes from the forgetful functor, but we can replace it with a naturally isomorphic one to get a dagger equivalence. Thus, if one cares about functors only up to natural isomorphism, there is nothing wrong with the equivalence $\cat{FHilb}\to\cat{FVect}$. Before proving this, we first state some useful definitions and facts about them. 

	\begin{definition}\index[word]{dagger category!dagger skeletal}\index[word]{dagger skeleton}
	Write $\cong_\dag$ for the equivalence relation on objects that is generated by unitaries. 
	A dagger category is \emph{dagger skeletal} if each $\cong_\dag$-equivalence class is a singleton. A \emph{dagger skeleton} of a dagger category \cat{C} is a dagger skeletal subcategory that is dagger equivalent to \cat{C}.
	\end{definition}

Assuming enough choice it is not difficult to prove the following. 

	\begin{theorem}
	Every dagger category has a dagger skeleton, which is also unique up to dagger isomorphism. 
	Two dagger categories are dagger equivalent iff they have dagger isomorphic skeletons.
	\end{theorem}

	\begin{theorem}
	Let $(\cat{C},\dag)$ be a dagger category and $F\colon \cat{C}\rightleftarrows\cat{D}\colon G$ an equivalence in \cat{Cat} such that $A\cong_\dag GF(A)$  for every $A$. Then one can replace G and F with naturally isomorphic functors and lift the resulting equivalence into a dagger equivalence.
	\end{theorem}

	\begin{proof}
	We begin by reducing to the case in which \cat{C} is dagger skeletal. Let \cat{C'} be the dagger skeleton of \cat{C} and let $I\colon \cat{C'}\rightleftarrows\cat{C}\colon J$ be a dagger equivalence. Now  $FI\colon \cat{C'}\rightleftarrows\cat{D}\colon GJ$ is an equivalence satisfying the assumption. If one can replace $FI$ and $JG$ with isomorphic functors $F'$ and $G'$ lifting to a dagger equivalence, then $F'J\colon\cat{C} \rightleftarrows \cat{D}\colon IG'$ is also a dagger equivalence, and furthermore $F'J\cong FIJ\cong F$ and $IG'\cong IJG\cong G$, proving the claim.
	
	Thus we may assume that \cat{C} is dagger skeletal. As $F$ defines $G$ up to natural isomorphism, it suffices to produce a new functor $G'$ so that the unit and counit satisfy the requirements of Theorem~\ref{thm:liftdagequiv}. As \cat{C} is dagger skeletal, $GFA=A$ for each $A$. 
	Define $G'(A)=G(A)$ for objects. Now, choose for each $A$ an isomorphism  $\epsilon_A\colon FG'(A)\to A$ so that $\epsilon_{FG'(A)}\colon FG'FG'(A)=FG'(A)\to FG'(A)$ equals $\id$ for each $A$. Define $G'(f)$ to be the unique map $G'(A)\to G'(B)$ that $F$ maps to $\epsilon^{-1}_B \circ f\circ\epsilon_A\colon FG'(A)\to A\to B\to FG'(B)$. It is easy to check that $G'$ is functorial, $G'F=\id[\cat{C}]$ and that $\epsilon\colon FG'A\to A$ is a natural transformation satisfying $G'\epsilon=\id$. Thus by Theorem~\ref{thm:liftdagequiv} there is a dagger on $\cat{D}$ such that $F$ and $G'$ form a dagger equivalence.
	\end{proof}


	\begin{corollary}
	Let $(\cat{C},\dag)$ be a unitary dagger category and $F\colon \cat{C}\rightleftarrows\cat{D}\colon G$ an equivalence in \cat{Cat}. Then one can replace G and F with naturally isomorphic functors and lift the resulting equivalence into a dagger equivalence.
	\end{corollary}

Note that the requirement $A\cong_\dag GF(A)$ for every A is strictly weaker than the existence of a unitary natural transformation $\id\to GF$, as witnessed by the forgetful functor  $\cat{FHilb}\to \cat{FVect}$. Naively, one might think that even this assumption can be dropped: $A\cong_\dag B$ implies $A\cong B$, so that by modifying $G$ and $F$ up to isomorphism, it might be possible to tweak any equivalence to one satisfying this condition. However, this can fail in situations when $A\cong B$ doesn't imply $A\cong_\dag B$. For example, let \cat{C} be the category with two uniquely isomorphic objects, and consider the dagger category \cat{ZigZag(C)} from Proposition~\ref{prop:free}. 
The two objects of this category are isomorphic, but not unitarily isomorphic. Thus \cat{ZigZag(C)} is equivalent in \cat{Cat} to some one-object category \cat{D}, but there is no dagger equivalence between \cat{ZigZag(C)} and a dagger category with a single object.

\section{Dagger adjunctions}\label{sec:adjunctions}

This section considers adjunctions that respect daggers.

\begin{definition}\index[word]{dagger adjunction}
  A \emph{dagger adjunction} is an adjunction between dagger categories where both functors are dagger functors.
\end{definition}

Note that the previous definition did not need to specify left and right adjoints, because the dagger makes the adjunction go both ways. If $F\colon \cat C \to \cat D$ and $G \colon \cat D \to \cat C$ are dagger adjoints, say $F \dashv G$ with natural bijection $\theta \colon \cat{D}(FA,B) \to \cat{C}(A,GB)$, then $f \mapsto \theta(f^\dag)^\dag$ is a natural bijection $\cat{D}(A,FB) \to \cat{C}(GA,B)$, 
whence also $G \dashv F$. Hence if $F$ and $G$ form a dagger adjunction, we will call $G$ the dagger adjoint of $F$ and vice versa\index[word]{dagger adjoint}.

For example, a dagger category $\cat{C}$ has a zero object if and only if the unique dagger functor $\cat{C}\to\cat{1}$ has a dagger adjoint. Here, a \emph{zero object} is one that is both initial and terminal, and hence induces zero maps between any two objects. This is the nullary version of the following example: a product $\smash{A \stackrel{p_A}{\longleftarrow} A \times B \stackrel{p_B}{\longrightarrow} B}$ is a \emph{dagger biproduct} when $p_A \circ p_A^\dag = \id$, $p_A \circ p_B^\dag = 0$, $p_B \circ p_A^\dag = 0$, $p_B \circ p_B^\dag = \id$.

\begin{example}\label{ex:dagbiprods}
  A dagger category $\cat{C}$ with zero object has binary dagger biproducts if and only if the diagonal functor $\cat{C} \to \cat{C} \times \cat{C}$ has a dagger adjoint $(-)\oplus (-)$ with dagger epic counit $(p_A,p_B)\colon (A\oplus B,A\oplus B)\to (A,B)$, \ie $(p_A \circ p_A^\dag,p_B\circ p_B^\dag)=(\id[A],\id[B])$ for $A,B \in \cat{C}$.
\end{example}
\begin{proof}
  The implication from left to right is routine. For the other direction, a right adjoint to the diagonal is well-known to fix binary products~\cite[V.5]{maclane:categories}. 
  If it additionally preserves daggers the product is also a coproduct, so it remains to check that the required equations governing $p_A$ and $p_B$ are satisfied. By naturality, the diagram
  \[\begin{tikzpicture}[xscale=3.5,yscale=1.5]
    \node (tl) at (0,1) {$A$};
    \node (t) at (1,1) {$A \oplus B$};
    \node (tr) at (2,1) {$B$};
    \node (bl) at (0,0) {$A$};
    \node (b) at (1,0) {$A \oplus B$};
    \node (br) at (2,0) {$B$};
    \draw[->] (tl) to node[above] {$p_A^\dag$} (t);
    \draw[->] (t) to node[above] {$p_B$} (tr);
    \draw[->] (bl) to node[below] {$p_A^\dag$} (b);
    \draw[->] (b) to node[below] {$p_B$} (br);
    \draw[->] (tl) to node[left] {$\id$} (bl);
    \draw[->] (tr) to node[right] {$0$} (br);
    \draw[->] (t) to node[right] {$\id\oplus 0$} (b);
  \end{tikzpicture}\]
  commutes, so that $p_B \circ p_A^\dag = 0$. By symmetry $p_A \circ p_B^\dag = 0$, and the remaining equations hold by assumption.
\end{proof}

Here is a more involved example of a dagger adjunction.

\begin{example}\label{ex:imdagadj}
  The monoids $(\mathbb{N},+)$ and $(\mathbb{Z},+)$ become one-object dagger categories under the trivial dagger $k\mapsto k$. The inclusion $\mathbb{N}\hookrightarrow \mathbb{Z}$ is a dagger functor. It induces a dagger functor $F\colon [\mathbb{Z},\cat{FHilb}]\to [\mathbb{N},\cat{FHilb}]$, which has a dagger adjoint $G$. 
\end{example}
\begin{proof}
  An object of $[\mathbb{Z},\cat{FHilb}]$ is a self-adjoint isomorphism $T\colon A\to A$ on a finite-dimensional Hilbert space $A$, whereas an object of $[\mathbb{N},\cat{FHilb}]$ is a just a self-adjoint morphism $T \colon A \to A$ in $\cat{FHilb}$.
  To define $G$ on objects, notice that a self-adjoint morphism $T \colon A \to A$ restricts to a self-adjoint surjection from $\ker(T)^\perp = \overline{\Ima T}$ to itself, and by finite-dimensionality of $A$ hence to a self-adjoint isomorphism $G(T)$ on $\Ima T$.

  On a morphism $f \colon T \to S$ in $[\mathbb{N},\cat{FHilb}]$, define $Gf$ to be the restriction of $f$ to $\Ima T$. To see this is well-defined, \ie the right diagram below commutes if the left one does,
  \[
    \begin{aligned}\begin{tikzpicture}
     \matrix (m) [matrix of math nodes,row sep=2em,column sep=4em,minimum width=2em]
     {A & B \\
      A & B \\};
     \path[->]
     (m-1-1) edge node [left] {$T$} (m-2-1)
             edge node [above] {$f$} (m-1-2)
     (m-2-1) edge node [below] {$f$} (m-2-2)
     (m-1-2) edge node [right] {$S$} (m-2-2);
    \end{tikzpicture}\end{aligned}
    \qquad \implies \qquad
    \begin{aligned}\begin{tikzpicture}
     \matrix (m) [matrix of math nodes,row sep=2em,column sep=4em,minimum width=2em]
     {\Ima  T & \Ima S \\
      \Ima T & \Ima S \\};
     \path[->]
     (m-1-1) edge node [left] {$T$} (m-2-1)
             edge node [above] {$G(f)$} (m-1-2)
     (m-2-1) edge node [below] {$G(f)$} (m-2-2)
     (m-1-2) edge node [right] {$S$} (m-2-2);
    \end{tikzpicture}\end{aligned}
  \] 
  observe that if $b\in\Ima T$ then $b=T(a)$ for some $b\in H$, so that $f(b)=fT(a)=Sf(a)$, and hence $f(b)\in \Ima S$. 
  This definition of $G$ is easily seen to be dagger functorial.

  To prove that $F$ and $G$ are dagger adjoint, it suffices to define a natural transformation $\eta \colon \id \to F \circ G$, because $G \circ F$ is just the identity. Define $\eta_T$ to be the projection $A \to \Ima T$, which is a well-defined morphism in $[\mathbb{N},\cat{FHilb}]$:
  \[
    \begin{tikzpicture}
     \matrix (m) [matrix of math nodes,row sep=2em,column sep=4em,minimum width=2em]
     {A & \Ima T \\
      A & \Ima T\rlap{.} \\};
     \path[->]
     (m-1-1) edge node [left] {$T$} (m-2-1)
             edge node [above] {$\eta_T$} (m-1-2)
     (m-2-1) edge node [below] {$\eta_T$} (m-2-2)
     (m-1-2) edge node [right] {$T$} (m-2-2);
     \end{tikzpicture}
  \] 
  Naturality of $\eta$ boils down to commutativity of 
  \[
  \begin{tikzpicture}
     \matrix (m) [matrix of math nodes,row sep=2em,column sep=4em,minimum width=2em]
     {
      A & \Ima T & \\
      B & \Ima S \\};
     \path[->]
     (m-1-1) edge node [left] {$f$} (m-2-1)
             edge node [above] {$\eta_T$} (m-1-2)
     (m-2-1) edge node [below] {$\eta_S$} (m-2-2)
     (m-1-2) edge node [right] {$f$} (m-2-2);
   \end{tikzpicture}
   \] 
   which is easy to verify.
\end{proof}

There are variations on the previous example. For example, $(n,m)\mapsto (m,n)$ induces daggers on $\mathbb{N}\times\mathbb{N}$ and $\mathbb{Z}\times\mathbb{Z}$. A dagger functor $\mathbb{N}\times\mathbb{N}\to \cat{FHilb}$ corresponds to a choice of a normal map, which again restricts to a normal isomorphism on its image. This defines a dagger adjoint to the inclusion $[\mathbb{N}\times\mathbb{N},\cat{FHilb}]\to [\mathbb{Z}\times\mathbb{Z},\cat{FHilb}]$. 

Recall that $F \colon \cat{C} \to \cat{D}$ is a \emph{Frobenius functor} when it has a left adjoint $G$ that is simultaneously right adjoint. This is also called an \emph{ambidextrous adjunction}~\cite{lauda:ambidextrous}.

\begin{proposition} 
  If $F$ is a Frobenius functor with adjoint $G$, then $F_\leftrightarrows$ and $G_\leftrightarrows$ as in Proposition~\ref{prop:cofree} are dagger adjoint.
\end{proposition}
\begin{proof}
  If $F \colon \cat C \to \cat D$ there is a natural bijection
  \begin{align*}
    \cat{C}_\leftrightarrows\big( (A,A) , G_\leftrightarrows(B,B) \big)
    & = \cat{C}(GB,A) \times \cat{C}(A,GB) \\
    & \cong \cat{D}(B,FA) \times \cat{C}(FA,B) \\
    & = \cat{D}_\leftrightarrows\big( F_\leftrightarrows(A,A), (B,B) \big)
  \end{align*}
  because $G \dashv F \dashv G$.
\end{proof}

Since ordinary adjoints are defined up to unique isomorphism, so are dagger adjoints. One might hope for more, \ie that for dagger adjoints the unique isomorphism is always unitary. However, this turns out to be false, essentially because universal arrows are not unique up to unitary ismorphism. Recall that if $F$ and $F'$ are left adjoints of $G$ with units $\eta$ and $\eta'$ respectively, the unique isomorphism $\sigma\colon F\to F'$ compatible with the adjunction is the one making the triangle
  \[\begin{tikzpicture}
    \matrix (m) [matrix of math nodes,row sep=2em,column sep=4em,minimum width=2em]
    {
     A & GFA \\
      & GF'A \\};
    \path[->]
    (m-1-1) edge node [below] {$\eta_A'$} (m-2-2)
           edge node [above] {$\eta_A$} (m-1-2)
    (m-1-2) edge node [right] {$G(\sigma_A)$} (m-2-2);
  \end{tikzpicture}\]

Now it is easy to produce a dagger adjunction where $\sigma$ is not unitary: set $F=G=\id[\cat{Hilb}]$, and $\eta=\id$ and $\eta'=2\id$, then there is a unique natural isomorphism $\sigma\colon\id[\cat{Hilb}]\to\id[\cat{Hilb}]$ making the triangle 
    \[\begin{tikzpicture}
    \matrix (m) [matrix of math nodes,row sep=2em,column sep=4em,minimum width=2em]
    {
     H & H \\
     & H \\};
    \path[->]
    (m-1-1) edge node [below] {$2\id$} (m-2-2)
           edge node [above] {$\id$} (m-1-2)
    (m-1-2) edge node [right] {$\sigma_H$} (m-2-2);
  \end{tikzpicture}\]
commute, but this isomorphism is not unitary. Note that for dagger equivalences similar issues cannot arise: if $\eta$ and $\eta'$ are unitary, then so is $G(\sigma)$ and hence $\sigma$.

\chapter{Limits}\label{chp:limits}



\section{Introduction}
This chapter studies limits in dagger categories. If $l_A \colon L \to D(A)$ is a limit for a diagram $D\colon\cat{J}\to\cat{C}$, then $l_A^\dag \colon D(A) \to L$ is a colimit for $\dag\circ D\colon\cat{J}\op\to\cat{C}$, so $L$ has two universal properties; they should be compatible with each other. Moreover, a \emph{dagger limit} should be unique not just up to mere isomorphism but up to unitary isomorphism.

In Section~\ref{sec:daglims} we define the notion of dagger limit that subsumes all known examples. Section~\ref{sec:unitarity} shows how dagger limits are unique up to unitary isomorphism. 
Section~\ref{sec:completeness} deals with completeness. If a dagger category has `too many' dagger limits, it degenerates (showcasing how dagger category theory can be quite different than ordinary category theory). A more useful notion of `dagger completeness' is defined, and shown to be equivalent to having dagger equalizers, dagger products, and dagger intersections. 
Section~\ref{sec:globaldaglims} formulates dagger limits in terms of an adjoint to a diagonal functor, and Section~\ref{sec:daft} attempts a dagger version of an adjoint functor theorem. Section~\ref{sec:polar} makes precise the idea that polar decomposition turns ordinary limits into dagger limits. Finally, Section~\ref{sec:commutativity} proves that dagger limits commute with dagger colimits in a wide range of situations.

To end this introduction, let us discuss earlier attempts at defining dagger limits~\cite{vicary2011completeness}. That work defines a notion of a dagger limit for diagrams $D\colon\cat{J}\to\cat{C}$ where \cat{J} has finitely many objects and \cat{C} is a dagger category enriched in commutative monoids. We dispense with both requirements. Proposition~\ref{prop:limswithsums} shows that when these requirements are satisfied the two notions agree for a wide class of diagrams. However, they do not always agree, as discussed in Example~\ref{ex:pullback}.

\section{Dagger limits}\label{sec:daglims}

Before we define dagger limits formally, let us look at few examples to motivate the definition. If \cat{1} is the terminal category, then a functor $D\colon \cat{1}\to \cat{Hilb}$ is given by a choice of an object $H$. Now, a limit of $D$ is is given by a an object $H'$ and an isomorphism $f\colon H'\to H$. However, for a limit of $D$ to qualify as the dagger limit of $D$, one would expect the isomorphism $f$ to be unitary. However, for some diagrams it seems as if one needs to make choices: a limit of $2\times -\colon \mathbb{C}\leftrightarrows\mathbb{C}\colon-/2$ is given by an isomorphism to one (and hence to both) copies of $\mathbb{C}$. For a dagger limit of the same diagram, we can require one of the isomorphisms to be unitary but not both -- which copy of $\mathbb{C}$ one should prefer? 

We will discuss more instructive examples below, but already it seems like dagger limits have something to do with ``normalization'', and one can conceivably choose to normalize at different locations. To keep track of these choices, we build them in the definition below, which is the basic object of study in this chapter. The rest of this section illustrates it.

\begin{definition}\label{def:daglim}\index[word]{dagger limit}
  Let \cat{C} be a dagger category and \cat{J} a category. A class $\Omega$ of objects of \cat{J} is \emph{weakly initial} if for every object $B$ of \cat{J} there is a morphism $f\colon A\to B$ with $A\in\Omega$, \ie if every object of \cat{J} can be reached from $\Omega$. 
  Let $D\colon\cat{J}\to\cat{C}$ be a diagram and let $\Omega\subseteq \cat{J}$ be weakly initial. \emph{A dagger limit of $(D,\Omega)$} is a limit $L$ of $D$ whose cone $l_A \colon L \to D(A)$\index[symb]{$l_A\colon L\to D(A)$, one leg of a cone $\Delta L\to D$ or a generic morphism denoting the whole cone} satisfies the following two properties:
  \begin{description}
    \item[normalization] $l_A$ is a partial isometry for every $A \in \Omega$;
    \item[independence] the projections on $L$ induced by these partial isometries commute, \ie $l_A^\dag l_A l_B^\dag l_B=l_B^\dag l_Bl_A^\dag l_A$ for all $A,B\in \Omega$.
  \end{description}
  A dagger limit of $D$ is a dagger limit of $(D,\Omega)$, for some weakly initial $\Omega$.
  If $L$ is a dagger limit of $(D,\Omega)$, we will also write $L=\dlim^\Omega D$\index[symb]{$\dlim^\Omega D$, the dagger limit of $(D,\Omega)$}. For a fixed $\cat{J}$ and $\Omega$, if $(D,\Omega)$ has a dagger limit in \cat{C} for every $D$, we will say that \cat{C} has all $(\cat{J},\Omega)$-shaped limits.
\end{definition}

Note that if $L$ is a dagger limit of $(D,\Omega)$ and $\Psi\subset\Omega$ is weakly initial, $L$ is also a dagger limit of $(D,\Psi)$. Moreover, if $L$ is a dagger limit of $(D,\Psi)$ and $(D,\Omega)$ also has a dagger limit, Theorem~\ref{thm:daglimsunique} will imply that $L$ is also a dagger limit of $(D,\Omega)$.

\begin{example}\label{ex:concretedaggerlimits}
  Definition~\ref{def:daglim} subsumes various concrete dagger limits from the literature:
  \begin{itemize}
  	\item A terminal object is a limit of the unique functor $\emptyset\to\cat{C}$. As the empty category has no objects, being a dagger limit of $\emptyset\to\cat{C}$ says nothing more than being terminal. In a dagger category, any terminal object is automatically a zero object.

  	\item A \emph{dagger product}\index[word]{dagger product}\index[word]{dagger biproduct} of objects $A$ and $B$ in a dagger category with a zero object is traditionally defined~\cite{selinger:completelypositive} to be a product $A \times B$ with projections $p_A \colon A \times B \to A$ and $p_B \colon A \times B \to B$ satisfying $p_A p_A^\dag = \id[A]$, $p_B p_B^\dag = \id[B]$, $p_A p_B^\dag = 0$, and $p_B p_A^\dag = 0$.
  	This is precisely a dagger limit of $(D,\Omega)$, where $\cat{J}$ is the discrete category on two objects that $D$ sends to $A$ and $B$, and $\Omega$ necessarily consists of both objects of $\cat{J}$; see also Example~\ref{ex:biproducts} below.

  	\item A \emph{dagger equalizer}\index[word]{dagger equalizer} of morphisms $f,g \colon A \to B$ in a dagger category is traditionally defined~\cite{vicary2011completeness} to be an equalizer $e \colon E \to A$ that is dagger monic.
  	This is precisely a dagger limit, where $\cat{J} = \bullet \rightrightarrows \bullet$ which $D$ sends to $f$ and $g$, and $\Omega$ consists of only the first object, which gets sent to $A$. 

  	This example justifies why Definition~\ref{def:daglim} cannot require $\Omega$ to be all of $\cat{J}$ in general, otherwise there would be many pairs $f,g$ that have a dagger equalizer in the traditional sense but not in the sense of Definition~\ref{def:daglim}.

  	\item A \emph{dagger kernel}\index[word]{dagger kernel} of a morphism $f \colon A \to B$ in a dagger category with a zero object is traditionally defined~\cite{heunenjacobs:daggerkernels} to be a kernel $k \colon K \to A$ that is dagger monic. As a special case of a dagger equalizer it is a dagger limit.

  	\item A \emph{dagger intersection}\index[word]{dagger intersection} of dagger monomorphisms $f_i \colon A_i \to B$ in a dagger category is traditionally defined~\cite{vicary2011completeness} to be a (wide) pullback $P$ such that each leg $p_i \colon P \to A_i$ of the cone is dagger monic.
  	This is precisely a dagger limit, where $\Omega$ consists of all the objects of $\cat{J}$ getting mapped to $A_i$. Since pullback of monics are monic, each $p_i$ is not only a partial isometry but also a monomorphism, and hence a dagger monomorphism.

    \item If $p\colon A\to A$ is a projection, a \emph{dagger splitting of $p$}\index[word]{projection!splitting of} is a dagger monic $i\colon I\to A$ such that $p=ii^\dag$~\cite{selinger2008idempotents}. A dagger splitting of $p$ can be seen as the dagger limit of the diagram generated by $p$. 
    More precisely, we can take $\Omega=\{A\}$, by definition $i$ is a partial isometry, and if $l \colon L \to A$ is another limit, then $m=i^\dag l$ is the unique map satisfying $l=im$.
    Conversely, suppose that $l \colon L \to A$ is a dagger limit. Then $l$ is a partial isometry, and so the cone $l$ factors through itself via both $l^\dag l$ and $\id[L]$; but since mediating maps are unique these must be equal, and so $l$ is dagger monic. Similarly, because $p$ is idempotent, $p$ gives a cone, which factors through $l$. This implies $l l^\dag p=p$. Taking daggers we see that $p=pll^\dag=ll^\dag$ since $pl=l$. 
    We say that \cat{C} has dagger splittings of projections if every projection has a dagger splitting. 
  \end{itemize}
\end{example}

\begin{example}\label{ex:daggershaped}\index[word]{dagger limit!dagger shaped}
  Let \cat{J} be a dagger category, and $D\colon \cat{J}\to\cat{C}$ a dagger functor. Any leg $l_A \colon L \to D(A)$ of a cone is a partial isometry if and only if any other leg $l_B \colon L \to D(B)$ in the same connected component is. To see this, fix a morphism $f \colon A \to B$ in $\cat{J}$, and assume that $l_{A}$ is a partial isometry. 
  \begin{align*} 
    l_Bl_B^\dag l_B
    &=D(f) l_A l_A^\dag  D(f)^\dag D(f)l_A && \text{ as }L\text{ is a cone}\\
    &=D(f) l_A l_A^\dag  D(f^\dag f)l_A && \text{ as }D\text{ is a dagger functor}\\
    &=D(f) l_A l_A^\dag l_A && \text{ as }L\text{ is a cone}\\
    &=D(f) l_A && \text{ as }l_A\text{ is a partial isometry} \\
    &=l_B  && \text{ as }L\text{ is a cone}
  \end{align*}
  Similarly $l_A^\dag l_A=l_B^\dag l_B$ for any objects $A$ and $B$ in the same connected component of \cat{J}. These two facts imply that whenever $D$ is a dagger functor, the choice of the parameter $\Omega$ doesn't matter when speaking about dagger limits, as the resulting equations are equivalent. Hence whenever $D$ is a dagger functor we omit $\Omega$, and call a dagger limit of $D$ a \emph{dagger-shaped limit}. In particular, whenever every dagger functor $\cat{J}\to\cat{C}$ has a dagger limit, we will say that \cat{C} has \cat{J}-shaped limits. If \cat{C} has \cat{J}-shaped limits for every small dagger category \cat{J}, we will say that \cat{C} has dagger-shaped limits. If we wish to say something about dagger limits of  functors $\cat{J}\to\cat{C}$ that don't necessarily preserve the dagger, we will make it clear by not omitting $\Omega$.
  \begin{itemize}
    \item Any discrete category has a unique dagger, which is always preserved by maps into dagger categories. Thus dagger products can be seen as dagger-shaped limits. 

    \item Any dagger splitting of a projection is a dagger-shaped limit, as in Example~\ref{ex:concretedaggerlimits}.

    \item We say that a dagger category has \emph{dagger split infima of projections} if, whenever $\mathcal{P}$ is a family of projections on a single object $A$, it has an infimum admitting a dagger splitting, \ie a dagger subobject $K\rightarrowtail A$ such that the induced projection is the infimum of $\mathcal{P}$. Limits of projections can be defined as dagger-shaped limits: consider the monoid  freely generated by a set of idempotents; the dagger on the monoid fixes those idempotents, and reverses words in them. However, we prefer to think of them instead in terms of the partial order on projections. It is not hard to show that the dagger intersection of a family of dagger monics $m_i\colon A_i\to A$ coincides with the dagger limit of the projections $m_im_i^\dag\colon A\to A_i\to A$.

    \item We say that a dagger category has \emph{dagger stabilizers} when it has all $\cat{ZigZag(E)}$-shaped limits, where \cat{E} is the equalizer shape and $\cat{ZigZag}(E)$ is the free dagger category on \cat{E} from Proposition~\ref{prop:free}. Concretely, a dagger functor with domain $\cat{ZigZag(E)}$ is uniquely determined by where it sends $E$, \ie by a choice of a parallel morphisms $f,g\colon A\rightrightarrows B$ in the target category \cat{C}. A cone for such a functor consists of an object $X$ with maps $p_A\colon X\to A$ and $p_B$ satisfying $fp_A=gp_A=p_B$ \emph{and} $f^\dag p_B=g^\dag p_B=p_A$. A dagger stabilizer of $f$ and $g$ is then a terminal such cone that also satisfies normalization and independence. Hence the dagger stabilizer of $f$ and $g$ is \emph{not} in general a dagger equalizer. For example, the (dagger) kernel of a linear map $f\colon A\to B$ in \cat{FHilb} can be computed as the equalizer of $f$ and $0$, whereas the dagger stabilizer of $f$ and $0$ is always $0$.
  \end{itemize}
\end{example}

\begin{example}\label{ex:symmetrictensor}\index[word]{symmetrized tensor power}\index[word]{antisymmetrized tensor power}\index[symb]{$\sym^n X$ the $n$-th symmetric tensor power of $X$}
    One can also understand (anti)symmetrized tensor powers as dagger limits in a monoidal dagger category.  The \emph{$n$-th symmetric tensor power of $X$}, denoted by $\sym^n X$, if it exists, is the dagger limit of the functor $D_X\colon S_n\to \cat{C}$, where $S_n$ is the symmetric group on $n$ elements, $D_X$ sends the unique object to $X^{\otimes n}$ and a permutation on $n$ elements to the corresponding permutation on $X^{\otimes n}$ built from the symmetry. Symmetric tensor powers are in fact dagger-shaped limits.

    If $f\colon X^{\otimes n}\to X^{\otimes n}$ is a symmetry, \ie self-adjoint and unitary, then one can define the $f$-symmetrized tensor power as the dagger limit of the following diagram. Define \cat{J} by setting its objects to correspond elements of the group $S_n$. For any $g\in S_n\setminus\{1\}$, \cat{J} has two morphisms $1\to g$ and no other non-identity morphisms exist. Define $D\colon\cat{J}\to \cat{C}$ by setting $D(-)=X^{\otimes n}$ on objects, and let one of the two morphisms $1\to g$ be mapped to the permutation $X^{\otimes n}\to X^{\otimes n}$ corresponding to $g$ and the other one mapped to $f^{\text{sgn}(g)}$ where $\text{sgn}(g)\in{0,1}$ is the sign of the permutation $g$. Then the $f$-symmetrized tensor product, denoted by $\sym^n_f X$, is the dagger limit of $(D,\cat{J})$. When $\cat{C}=\cat{FHilb}$ and $f=-1$, this corresponds to antisymmetrized tensor powers, and when $f=1$ one obtains an equivalent description of symmetric tensor powers. More generally, one might let $f$ vary with $g$, obtaining a tensor power where some transpositions are symmetric and some are asymmetric. Note that each morphism in the image of $D$ is unitary.

\end{example}

Recall that a dagger category is \emph{connected} if every hom-set is inhabited.\index[word]{dagger category!connected} 
\begin{proposition}\label{prop:variouscatshavedagshapedlims}\index[symb]{\cat{Rel}, sets and relations}\index[symb]{\cat{FinRel}, finite sets and relations},\index[symb]{\PInj, sets and partial injections}\index[symb]{$\cat{Span}(\cat{FinSet})$, spans of finite sets}
  The dagger categories $\cat{Rel}$, $\cat{FinRel}$, $\cat{PInj}$, and $\cat{Span(FinSet)}$ have \cat{J}-shaped limits for any small connected dagger category \cat{J}.
\end{proposition}
\begin{proof} 
  Let $D\colon \cat{J}\to \cat{Rel}$ be a dagger functor; we will construct a dagger limit. Write $G$ for the (undirected multi-)graph with vertices $V=\coprod_{A\in\cat{J}} D(A)$ and edges $E=\coprod_{f\in\cat{J}}D(f)$. Call a vertex $a\in D(A)\subseteq V$ a \emph{$D$-endpoint} if $D(f)a=\emptyset$ for some $f\colon A\to B$ in \cat{J}. Set
	\[
    L=\{X\in \mathcal{P}(V)\mid X\text{ is a path component of }\mathcal{G}\text{ with no }D\text{-endpoints}\}\text,
  \]
	and define $l_A\colon L\to D(A)$ by $l_A(X)=X\cap D(A)$. We will prove that this is a dagger limit of $D$, starting with normalization and independence. First we show that $l_A(X)\neq \emptyset$ for any $A\in\cat{J}$ and $X\in L$. Pick an element $b\in X$, say $b\in D(B)$, and choose some $f\colon B\to A$. Since $x$ is not a $D$-endpoint, $D(f)x$ is nonempty and contained in $X\cap A=l_A(X)$. Because path components of a graph are disjoint, $l_A^\dag l_A (X)=X$ for all $X$. Hence $l_A^\dag l_A=\id[L]$ for all $A$, establishing normalization and independence.

	Next we verify that $l_A$ forms a cone. Path components are closed under taking neighbours, so $D(f)l_A(X)\subseteq D_B(X)$. To see the other inclusion, let $b\in D(B)(X)=X\cap B$. Again $b$ isn't a $D$-endpoint, so $D(f)^\dag (b)=D(f^\dag) (b)\subseteq L_A(X)$ is nonempty. Any element of $D(f)^\dag (b)$ is related to $b$ by $D(f)$. Hence $b\in D(f)l_A(X)$, so $D(f)l_A=l_B$ as desired.

	Finally, we verify that the cone $l_A$ is limiting. Let $R_A\colon Y\to D(A)$ be any cone. If $x\in D(A)$ is a $D$-endpoint, say $D(f)x=\emptyset$, then 
  \[
    x\notin D(f^\dag)D(f)D(A)=D(f^\dag f)D(A)\supseteq D(f^\dag f)R_A(y)=R_A(y)\text.
  \]
  Hence no $R_A(y)$ contains $D$-endpoints. Moreover, if $D(f)R_A(y)=R_B(y)$ for all $f$, then the set $\coprod_{A\in\cat{J}}(R_A(y))\subseteq V$ is closed under taking neighbours in $G$. As it contains no $D$-endpoints, it is a union of a set of connected components of $\mathcal{G}$ without $D$-endpoints. Mapping $y$ to this set of connected components, \ie\ to a subset of $L$, defines the unique relation $R\colon Y\to L$ satisfying $R_A=l_A\circ R$ for all $A\in\cat{J}$.

  The same construction works for \cat{FinRel}:  one merely needs to check that $L$ is finite whenever each $D(A)$ is. In fact $L$ is finite if at least one $D(A)$ is, since $l_A^\dag l_A=\id[L]$ implies that the function $L\to \mathcal{P}(D(A))$ corresponding to $l_A$ is injective.

  Now consider a dagger functor $D\colon\cat{J}\to\cat{PInj}$ and set
  \begin{align*}
    L&=\{(x_A)_{A\in\cat{J}}\in \prod_{A\in J}D(A)\mid D(f)x_A=x_B\text{ for every }f\colon A\to B\}\text, \\
    l_A((x_A)_{A\in\cat{J}})&=x_A\text.
  \end{align*}
  It is easy to verify that this forms a dagger limit. 

  Dagger limits of a dagger functor $D\colon\cat{J}\to\cat{Span(FinSet)}$ resemble the case of \cat{PInj} more than \cat{(Fin)Rel}.  Think of (the isomorphism class of) a span $A\leftarrow \bullet\to B$ of finite sets as a matrix $R\colon A\times B\to\mathbb{N}$ with natural number entries, so that a morphism $f\colon A\to B$ in \cat{J} maps to $D(f)\colon D(A)\times D(B)\to\mathbb{N}$. Set 
    \begin{align*}
    L&=\{(x_A)_{A\in\cat{J}}\in \prod_{A\in J}D(A)\mid D(f)(x_A,z)=\delta_{z,x_B}
      \text{ for every }f\colon A\to B\}\text, \\
    l_A((x_A)_{A\in\cat{J}},z)&=\delta_{z,x_A}\text.
  \end{align*}
  It is easy to see $l_A$ forms a cone. To see that it is limiting, let $R_A\colon Y\to D(A)$ be any cone, and pick $y\in Y$.  Consider $x_A\in R_A$ such that $R_A(y,x_A)=n\neq 0$. We will show that if $B\in \cat{J}$ and $f\colon A\to B$, then $D(f)(x_A,z)=\delta_{z,x_B}$ for some unique $x_B$, so that $x_A$ extends to a unique family $(x_A)_{A\in\cat{J}}\in L$. Consider an arbitrary $f\colon A\to B$. Now $D(f^\dag f) R_A=R_A$, so there has to be some $x_B\in B$ with $D(f)(x_A,x_B)>0.$ If there were several such $x_B$ or if $D(f)(x_A,x_B)>1$, then $D(f^\dag f)R_A(y,x_A)>n$, which is a contradiction. Hence $x_A$ extends uniquely to a family $(x_A)_{A\in\cat{J}}\in L$. Moreover, if $R_B=D(f)R_A$, then $R_B(y,x_B)=n$ as well. Hence we can define $R\colon Y\to L$ by setting $R(y,(x_A)_{A\in\cat{J}})=R_A(y,x_A)$. Now $R$ satisfies $l_BR=R_B$ for each $B$ and it is clearly unique as such.
\end{proof}

Note that this theorem fails for \cat{Span(Set)}, since idempotents do not always split. For instance, the idempotent $1\leftarrow \mathbb{N}\to 1$ does not admit a splitting. We leave open the question of characterizing exactly which categories of spans or relations admit connected dagger-shaped limits.

The following example illustrates the name `independence axiom' in Definition~\ref{def:daglim}.

\begin{example} 
  When working in \cat{FHilb}, consider $\mathbb{C}^2$ as the sum of two non-orthogonal lines, \eg the ones spanned by $\lvert 0 \rangle$ and $\lvert + \rangle$. Projections to these two lines will give rise to two maps $p_1,p_2\colon \mathbb{C}^2\to \mathbb{C}$ making $(\mathbb{C}^2,p_1,p_2)$ into a categorical product. Moreover, $p_1$ and $p_2$ are partial isometries so that the normalization axiom is satisfied. However, the independence axiom fails, and indeed, $(\mathbb{C}^2,p_1,p_2)$ fails to be a dagger product. In other words, the limit structure $(\mathbb{C}^2,p_1,p_2)$ and the colimit structure $(\mathbb{C}^2,p_1^\dag,p_2^\dag)$ are not compatible.
\end{example}

\begin{example}\label{ex:cofree}\index[word]{dagger category!cofree}\index[symb]{\cat{C_\leftrightarrows}, the cofree dagger category on \cat{C}}
  Sometimes in ordinary category theory an object is both a limit and a colimit ``in a compatible way'' to a pair of related diagrams in \cat{C}. Usually this is formulated in terms of a canonical morphism from the colimit to the limit being an isomorphism, but in some cases one can instead formulate them as dagger limits in $\cat{C_\leftrightarrows}$, the cofree dagger category from proposition \ref{prop:cofree}. Moreover, if \cat{C} has zero morphisms\footnote{This is so that one can give every cone a trivial cocone structure and vice versa.}, then dagger limits in $\cat{C_\leftrightarrows}$  give rise to such ``ambilimits '' in \cat{C}. 
  \begin{itemize}
    \item If \cat{C} has zero morphisms, then a biproduct in \cat{C} is the same thing as a dagger product in \cat{C_\leftrightarrows}. 

    \item Idempotents in \cat{C} split if and only if (dagger) projections in $\cat{C_\leftrightarrows}$ have dagger splittings.

    \item A more interesting example comes from domain theory, where there is an important limit-colimit coincidence. Regard the partially ordered set $(\mathbb{N},\leq)$ of natural numbers as a category $\omega$, and let $D\colon \omega\to \cat{DCPO}$ be a chain of embeddings, \ie each map in $\omega$ is mapped to an embedding by $D$. Then the embeddings define unique projections, resulting in a chain of projections $D^*\colon\omega\op\to \cat{DCPO}$. A fundamental fact in domain theory~\cite[3.3.2]{abramsky1994domain} is that the colimit of $D$ coincides with the limit of $D^*$. One can go through the construction and show that this ``ambilimit'' can equivalently be described as the dagger limit of $(D,D^*,\omega)$, where $(D^*,D)\colon\omega\op\to\cat{DCPO}_\leftrightarrows$.
  \end{itemize}
This viewpoint is developed in more detail in chapter~\ref{chp:ambilims}.
\end{example}

\begin{example}
  In an inverse category the normalization and independence axioms of Definition~\ref{def:daglim} are automatically satisfied, and hence dagger limits are simply limits.
\end{example}

\begin{example} 
  Fix $\cat{J}$ and a weakly initial $\Omega$. If $\cat{C}$ has $(\cat{J},\Omega)$-shaped limits, so does $[\cat{D},\cat{C}]$.
  Let $D \colon \cat{J} \to [\cat{D},\cat{C}]$ be a diagram.
  For each $X \in \cat{D}$, there is a dagger limit $L(X) = \dlim^\Omega D(-)(X)$ with cone $l_A^X \colon L(X) \to D(A)(X)$. 
  For $f \colon X \to Y$ in $\cat{D}$, there is a cone $D(A)(f) \circ l_A^X \colon L(X) \to D(A)(Y)$, and hence a unique map $L(f) \colon L(X) \to L(Y)$ satisfying $l^Y_A \circ L(f) = D(A)(f) \circ l^X_A$. 
  The resulting functor $L \colon \cat{D} \to \cat{C}$ is a limit of $D$, with cone $l_A^X \colon L(X) \to D(A)(X)$.
  This limit is in fact a dagger limit $(D,\Omega)$, because the normalization and independence axioms hold for each component $l_A^X$, and the dagger in $[\cat{D},\cat{C}]$ is computed componentwise.
\end{example}

We end this section by recording how dagger functors interact with dagger limits.
Any dagger functor preserves dagger limits as soon as it preserves limits.
The same holds for reflection and creation of (dagger) limits when the functor is faithful.

\begin{lemma}
  Let $\cat{F} \colon \cat{C} \to \cat{D}$ be a dagger functor. If $F$ preserves limits of type $\cat{J}$, then it preserves $(\cat{J},\Omega)$-shaped dagger limits. If $F$ reflects (creates) limits of type \cat{J}, then it reflects (creates) $(\cat{J},\Omega)$-shaped dagger limits.
\end{lemma}
\begin{proof}
  All dagger functors preserve partial isometries and commutativity of projections, and faitful dagger functors also reflect these. 
\end{proof}



\section{Uniqueness up to unitary isomorphism}\label{sec:unitarity}



\begin{theorem}\label{thm:daglimsunique}
  Let $D\colon\cat{J}\to \cat{C}$ be a diagram and $\Omega\subseteq\cat{J}$ be weakly initial. Let $L$ be a dagger limit of $(D,\Omega)$ and let $M$ be a limit of $D$. The canonical isomorphism of cones $L\to M$ is unitary iff $M$ is a dagger limit of $(D,\Omega)$. In particular, the dagger limit of $(D,\Omega)$ is defined up to unitary isomorphism.
\end{theorem}

\begin{example}\label{ex:biproducts}
  Let \cat{J} be a discrete category of arbitrary cardinality. As \cat{J} has only one weakly initial class (the one consisting of all objects), the dagger limit of any diagram $D\colon \cat{J}\to\cat{C}$ is unique up to unitary isomorphism. These are exactly the dagger products. However, note that Definition~\ref{def:daglim} does not require enrichment in commutative monoids nor the equation
  \begin{equation}\label{eq:biproduct}
    \id[A\oplus B]=l_A^\dag l_A+l_B^\dag l_B
  \end{equation} 
  (and thus works for infinite \cat{J} as well). Moreover, it doesn't require \cat{C} to have a zero object or zero morphisms in order to be defined up to unitary iso. On the other hand, if \cat{C} has zero morphisms, it is not hard to show directly that a dagger product in the sense of Definition~\ref{def:daglim} satisfies the traditional equations involving zero. Moreover, if \cat{C} is enriched in commutative monoids, equation~\eqref{eq:biproduct} also follows. These facts are proven for more general \cat{J} in Propositions~\ref{prop:limswithzeros} and~\ref{prop:limswithsums}.

  In fact, from Definition~\ref{def:daglim} one can glean a definition of (ordinary) biproduct of $A$ and $B$ that is unique up to isomorphism in an ordinary category \cat{C}, but doesn't require the existence of zero morphisms, and so generalizes the usual definition. For example, in \cat{Set} the biproduct $\emptyset\oplus\emptyset$ exists and is the empty set. Slightly more interestingly, if \cat{C} has all binary biproducts (in the traditional sense) and \cat{D} is any non-empty category, then $\cat{C}\sqcup \cat{D}$ has binary biproducts of pairs of objects from \cat{C} (in the generalized sense), but doesn't have zero morphisms. Of course, if a category has all binary biproducts in this generalized sense, one can show that it also has zero morphisms, so this definition is more general only in categories with some but not all biproducts. 
  For more details, see Section~\ref{sec:biprods}.
\end{example}

\begin{example}
  Let \cat{J} be the indiscrete category on $n$ objects. When considering functors $\cat{J}\to\cat{C}$ that don't preserve the dagger on $\cat{J}$, the parameter $\Omega$ matters.
  For example, take $n=2$ in the diagram $D \colon \cat{J} \to \cat{Hilb}$ defined by $D(1)=D(2)=\mathbb{C}$ where $D(1 \to 2)$ multiplies by $2$ but $D(2 \to 1)$ divides by $2$.
  Now $\Omega$ cannot be all of $\{1,2\}$, because no limiting cone can consist of partial isometries.
  If $\Omega = \{1\}$, there is a dagger limit $L=\mathbb{C}$ with $l_1=1$ and $l_2=2$.
  If $\Omega = \{2\}$, there is a dagger limit $L=\mathbb{C}$ with $l_1=\tfrac{1}{2}$ and $l_2=1$.
  These two dagger limits are clearly not unitarily isomorphic.

  The same can happen with \emph{chains}, where the preorder of integers is regarded as a category $\cat{J}$.
  Consider the diagram $D \colon \cat{J} \to \cat{Hilb}$ defined by $D(n) = \mathbb{C}$, and $D(n \to n+1)$ is multiplication by $-1^{n}$. 
  Now $\Omega$ can either consist of even numbers or the odd numbers. Hence $D$ has two dagger limits and they are not unitarily isomorphic.

  One might hope to get rid of this dependence on $\Omega$ by strengthening the definition to select exactly one of a diagram's several dagger limits. However, this is impossible in general. Write $\cat{J}(X)$ for the indiscrete category on a nonempty set $X\subset \mathbb{R}\setminus\{0\}$ of objects. Define $D(X)\colon\cat{J}(X)\to\cat{FHilb}$ by mapping the unique arrow $x\to y$ to the morphism $\mathbb{C}\to\mathbb{C}$ that multiplies by $\tfrac{x}{y}$. A choice of a dagger limit for each $D(X)$ amounts to a choice function on $\mathbb{R}\setminus\{0\}$. Thus there is no way to strengthen Definition~\ref{def:daglim} to make dagger limits unique in a way that doesn't depend on a choice of a weakly initial class.
\end{example}

\begin{example}
  In the domain theory part of Example~\ref{ex:cofree}, in fact $\Omega=\cat{J}$. Hence the bilimit is unique up to unique unitary isomorphism in \cat{DCPO_\leftrightarrows}. In  \cat{DCPO}, this means that any isomorphism of the limit half of such bilimits, is also an isomorphism of the colimit half.
\end{example}

\begin{example} 
  Consider dagger equalizers of $f,g\colon A\to B$ in the sense of Definition~\ref{def:daglim}. Any weakly initial class must contain $A$, and thus the dagger equalizer is unique up to unitary iso if it exists. Moreover, it is readily seen to coincide with the traditional definition. For if $e\colon E\to A$ is the dagger equalizer in the above sense, then $e$ is monic and a partial isometry, and thus dagger monic, so that $e$ is a dagger equalizer in the traditional sense.
\end{example}

\begin{theorem}\label{thm:equivalenttounitary} 
  Let $l_A \colon L \to D(A)$ and $m_A \colon M \to D(A)$ be limits of the same diagram $D \colon \cat{J} \to \cat{C}$ where \cat{C} is a dagger category, and let $f\colon L\to M$ be the unique isomorphism of limits. Then the following conditions are equivalent:
  \begin{enumerate}[(i)]
    \item $f$ is unitary;
    \item $f$ is also a morphism of colimits;
  	\item the following diagram commutes for any $A$ and $B$ in $\cat{J}$:
    \[\begin{tikzpicture}
     \matrix (m) [matrix of math nodes,row sep=2em,column sep=4em,minimum width=2em]
     {
      D(A) & L \\
      M & D(B) \\};
     \path[->]
     (m-1-1) edge node [left] {$m_A^\dag$} (m-2-1)
            edge node [above] {$l_A^\dag$} (m-1-2)
     (m-1-2) edge node [right] {$l_B$} (m-2-2)
     (m-2-1) edge node [below] {$m_B$} (m-2-2);
    \end{tikzpicture}\]
  \end{enumerate}
\end{theorem}
\begin{proof}
    $(i)\Rightarrow (ii)$ By definition $f^{-1}$ is the unique map $M\to L$ that is compatible with the limit structure, whereas $f^{\dag}$ is the unique map $M\to L$ that is compatible with the \emph{colimit} structure. As $f$ is unitary, these coincide, whence $f^{-1}$ is simultaneously a map of limits and colimits. Therefore so too is $f$.

    $(ii)\Rightarrow (iii)$  By (ii), both of the triangles in the following diagram commute.
    \[\begin{tikzpicture}
         \matrix (m) [matrix of math nodes,row sep=2em,column sep=4em,minimum width=2em]
         {
          D(A) & L \\
          M & D(B) \\};
         \path[->]
         (m-1-1) edge node [left] {$m_A^\dag$} (m-2-1)
                edge node [above] {$l_A^\dag$} (m-1-2)
         (m-1-2) edge node [right] {$l_B$} (m-2-2)
                 edge node [above] {$f$} (m-2-1)
         (m-2-1) edge node [below] {$m_B$} (m-2-2);
    \end{tikzpicture}\]

    $(iii)\Rightarrow (i)$ To prove that $f\colon L\to M$ is unitary, it suffices to establish $f\circ f^\dag=\id[M]$. By the two universal properties of $M$, we may further reduce to pre- and postcomposing with structure maps to and from the diagram $D$. The following diagram commutes.
        \[
        \begin{tikzpicture}
         \matrix (m) [matrix of math nodes,row sep=2em,column sep=4em,minimum width=2em]
         {
          M & L & M \\
          D(A) & M & D(B) \\};
         \path[->]
         (m-1-1) edge node [above] {$f^\dag$} (m-1-2)
         (m-1-2) edge node [above] {$f$} (m-1-3)
                edge node [left=1mm] {$l_B$} (m-2-3)
         (m-1-3) edge node [right] {$m_B$} (m-2-3)
         (m-2-1) edge node [below] {$m^\dag_A$} (m-2-2)
                edge node [left] {$m^\dag_A$} (m-1-1)
                edge node [right=3mm] {$l^\dag_A$} (m-1-2)
         (m-2-2) edge node [below] {$m_B$} (m-2-3);
        \end{tikzpicture}
        \]
    Hence $f$ is unitary.
\end{proof}

The previous theorem highlights why one would want limits in dagger categories to be defined up to unitary isomorphism: if $(L,\{l_A \}_{A\in \cat{J}})$ is a limit of $D$, then $(L,\{l_A^\dag \}_{A\in \cat{J}})$ is a colimit of $\dag\circ D$, and being defined up to unitary iso ensures that the limit and the colimit structures are compatible with each other.

We now set out to prove Theorem~\ref{thm:daglimsunique}. The following lemma will be crucial.

\begin{lemma}\label{lem:connectingmaps}\index[word]{adjointable natural transformation} 
  Let $l_A \colon L \to D(A)$ and $m_A \colon M \to E(A)$ be dagger limits of $(D,\Omega)$ and $(E,\Omega)$, respectively, where $D,E\colon \cat{J}\rightrightarrows \cat{C}$. Let $\sigma\colon D\to E$ be an adjointable natural transformation. The following diagram commutes for every $A$ and $B$ in $\cat{J}$.
  \begin{equation}\label{diag:connectingmaps}\begin{tikzpicture}
    \matrix (m) [matrix of math nodes,row sep=2em,column sep=4em,minimum width=2em]
    {
     D(A) & L & D(B) \\
     E(A) & M & E(B) \\};
    \path[->]
    (m-1-1) edge node [left] {$\sigma_A$} (m-2-1)
           edge node [above] {$l_A^\dag$} (m-1-2)
    (m-1-2) edge node [above] {$l_B$} (m-1-3)
    (m-1-3) edge node [right] {$\sigma_B$} (m-2-3)
    (m-2-1) edge node [below] {$m_A^\dag$} (m-2-2)
    (m-2-2) edge node [below] {$m_B$} (m-2-3);
  \end{tikzpicture}\end{equation}
\end{lemma}
\begin{proof}
  First we show that it is enough to check that~\eqref{diag:connectingmaps} commutes whenever $A,B\in\Omega$. So assume that it does and let $X,Y\in\cat{J}$ be arbitrary. By weak initiality of $\Omega$ we can find maps $f\colon A\to X$ and $g\colon B\to Y$ for some $A,B\in\Omega$. Now we prove the claim for $X$ and $Y$ assuming the claim for $A$ and $B$ by showing that the diagram 
    \[\begin{tikzpicture}
    \matrix (m) [matrix of math nodes,row sep=2em,column sep=4em,minimum width=2em]
    {
     D(X) & E(X) &  \\
    &  D(A) & E(A) & M \\
    & D(B) & E(B)\\
    L & &D(Y) & E(Y)\\};
    \path[->]
    (m-1-1) edge node [below] {$Df^\dag\ \ $} (m-2-2)
           edge node [above] {$\sigma_X$} (m-1-2)
           edge node [left] {$l_X^\dag$} (m-4-1)
    (m-1-2) edge node [above] {\quad$Ef^\dag$} (m-2-3)
             edge[out=0,in=135] node [above] {$m_X^\dag$} (m-2-4)
    (m-2-2) edge node [below] {$\sigma_A$} (m-2-3)
            edge node [above] {$l_A^\dag\ \ $} (m-4-1)
    (m-2-3) edge node [above] {$m_A^\dag$} (m-2-4)
    (m-2-4) edge node [right] {$m_Y$} (m-4-4)
            edge node [below] {$\ \ m_B$} (m-3-3)
    (m-3-2) edge node [above] {$\sigma_B$} (m-3-3)
            edge node [below] {$Dg$} (m-4-3) 
    (m-3-3) edge node [above] {$Eg$} (m-4-4)
    (m-4-1) edge node [below] {$l_B$} (m-3-2)
            edge node [below] {$l_Y$} (m-4-3)
    (m-4-3) edge node [below] {$\sigma_Y$} (m-4-4);
  \end{tikzpicture}\]
  commutes. The parallelogram on the bottom (top) commutes because $\sigma$ is natural (and adjointable). The triangles on the left of these parallelograms commute because $L$ is a cone and the triangles on the right commute because $M$ is. Finally, the remaining shape is just diagram~\eqref{diag:connectingmaps} with $A,B\in\Omega$.

  For the rest of the proof, write $l_{A,B}:=l_Bl_A^\dag\colon D(A)\to L\to D(B)$ and similarly $m_{A,B}:= m_B m_A^\dag$. Note that $l_{B,A}^\dag=l_{A,B}$ and $m_{B,A}^\dag=m_{A,B}$. Moreover, $l_{A,B}$ and $m_{A,B}$ are partial isometries whenever $A,B\in\Omega$ by Remark~\ref{rem:pisetc}. Our goal is to show that
  \begin{equation}\label{eq:connectingmaps}m_{A,B}\sigma_A=\sigma_B l_{A,B}\end{equation}
  for all $A$ and $B$ in $\cat{J}$. 

  Let $f\colon L\to M$ be the unique map making this square commute for all $A$ in $\cat{J}$:
  \[\begin{tikzpicture}
         \matrix (m) [matrix of math nodes,row sep=2em,column sep=4em,minimum width=2em]
         {
          L & D(A) \\
          M & E(A) \\};
         \path[->]
         (m-1-1) edge[dashed] node [left] {$f$} (m-2-1)
                edge node [above] {$l_A$} (m-1-2)
         (m-1-2) edge node [right] {$\sigma_A$} (m-2-2)
         (m-2-1) edge node [below] {$m_A$} (m-2-2);
  \end{tikzpicture}\]
  If $A \in \Omega$, then $m_A$ is a partial isometry, so the following diagram commutes:
  \[\begin{tikzpicture}[yscale=2.5,xscale=3]
    \node (a1) at (0,1) {$D(A)$};
    \node (a2) at (1,1) {$L$};
    \node (a3) at (2,1) {$D(A)$};
    \node (a4) at (3,1) {$E(A)$};
    \node (b1) at (1,0) {$D(A)$};
    \node (b2) at (2,.5) {$M$};
    \node (b3) at (2,0) {$E(A)$};
    \node (b4) at (3,0) {$M$};
    \draw[dagger,->] (a1) to node[above]{$l_A^\dag$} (a2);
    \draw[dagger,->] (a2) to node[left]{$l_A$} (b1);
    \draw[->] (a2) to node[below]{$f$} (b2);
    \draw[dagger,->] (a2) to node[above]{$l_A$} (a3);
    \draw[->] (a3) to node[above]{$\sigma_A$} (a4);
    \draw[dagger,->] (b2) to node[below=1mm]{$m_A$} (a4);
    \draw[dagger,->] (b2) to node[right]{$m_A$} (b3);
    \draw[dagger,->] (b3) to node[below]{$m_A^\dag$} (b4);
    \draw[dagger,->] (b4) to node[right]{$m_A$} (a4);
    \draw[->] (b1) to node[below]{$\sigma_A$} (b3);
  \end{tikzpicture}\]
  This shows that 
  $
    \sigma_Al_{A,A}=m_{A,A}\sigma_Al_{A,A}\text.
  $
  Repeating the argument with $\sigma$ replaced by $\sigma^\dag$ gives 
  $
    \sigma_A^\dag m_{A,A}=l_{A,A}\sigma_A^\dag m_{A,A}\text.
  $
  Combining the first equation with the dagger of the second shows that for every $A \in \Omega$:
  \begin{equation}\label{eq:proj lemma} 
    m_{A,A}\sigma_A=\sigma_A l_{A,A}\text.
  \end{equation}  

  Next, we check that the following diagram commutes:
  \[\begin{tikzpicture}[yscale=1.75,xscale=3]
    \node (a1) at (0,4) {$D(A)$};
    \node (a2) at (1,4) {$E(A)$};
    \node (a3) at (2,4) {$M$};
    \node (a5) at (4,4) {$E(B)$};
    \node (b2) at (1.5,3.5) {$M$};
    \node (b3) at (2,3) {$E(A)$};
    \node (b4) at (3,3.5) {$E(B)$};
    \node (b5) at (4,3) {$M$};
    \node (c1) at (0,2) {$L$};
    \node (c2) at (1,2.5) {$D(A)$};
    \node (c3) at (2,2) {$M$};
    \node (c5) at (4,2) {$E(A)$};
    \node (d1) at (0,1) {$D(B)$};
    \node (d3) at (2,1) {$E(B)$};
    \node (d5) at (4,1) {$M$};
    \draw[->] (a1) to node[above] {$\sigma_A$} (a2);
    \draw[dagger,->] (a2) to node[above] {$m_A^\dag$} (a3);
    \draw[dagger,->] (a3) to node[above] {$m_B$} (a5);
    \draw[dagger,->] (a1) to node[left] {$l_A^\dag$} (c1);
    \draw[dagger,->] (a2) to node[below left=-1mm] {$m_A^\dag$} (b2);
    \draw[dagger,->] (b2) to node[below left=-1mm] {$m_A$} (b3);
    \draw[dagger,->] (b3) to node[right] {$m_A^\dag$} (a3);
    \draw[dagger,->] (a3) to node[below] {$m_B$} (b4);
    \draw[dagger,->] (b4) to node[below]{$m_B^\dag$} (b5);
    \draw[dagger,->] (b5) to node[right]{$m_B$} (a5);
    \draw[dagger,->] (c1) to node[above]{$l_A$} (c2);
    \draw[->] (c2) to node[below]{$\sigma_A$} (b3);
    \draw[dagger,->] (c3) to node[right]{$m_A$} (b3);
    \draw[dagger,->] (c5) to node[right]{$m_A^\dag$} (b5);
    \draw[dagger,->] (c1) to node[left]{$l_B$} (d1);
    \draw[->] (c1) to node[below]{$f$} (c3);
    \draw[dagger,->] (c3) to node[right]{$m_B$} (d3);
    \draw[->] (d1) to node[below]{$\sigma_B$} (d3);
    \draw[dagger,->] (d3) to node[below]{$m_B^\dag$} (d5);
    \draw[dagger,->] (d5) to node[right]{$m_A$} (c5);
    \draw[gray,draw=none] (a1) to node{(iii)} (c2);
    \draw[gray,draw=none] (b2) to node{(i)} (a3);
    \draw[gray,draw=none] (b4) to node{(ii)} (a5);
    \draw[gray,draw=none] (d3) to node{(vi)} (b5);
    \draw[gray,draw=none] (c2) to node{(iv)} (c3);
    \draw[gray,draw=none] (c1) to node{(v)} (d3);
  \end{tikzpicture}\]
  Region (i) commutes since $m_A$ is a partial isometry and region (ii) since $m_B$ is. Region (iii) is equation~\eqref{eq:proj lemma}. Regions (iv) and (v) commute by definition of  $f$.  Finally, the commutativity of projections $m_A^\dag m_A$ and $m_B^\dag m_B$ shows that region (vi) commutes. 

  Commutativity of the outermost rectangle says that
  \begin{equation}\label{eq:lAB1}
        m_{A,B}\sigma_A=m_{A,B}m_{B,A}\sigma_B l_{A,B}
  \end{equation}
  for every $A,B\in\Omega$.   

  Exchanging $M$ and $L$ and replacing $\sigma$ with $\sigma^\dag$ in the above diagram now gives
  $
        l_{A,B}\sigma_A^\dag=l_{A,B}l_{B,A}\sigma_B^\dag m_{A,B}
  $, and applying the dagger on both sides shows
  \begin{equation}\label{eq:lAB2}
        \sigma_A l_{B,A}=m_{B,A}\sigma_B l_{A,B}l_{B,A} 
  \end{equation}
  for every $A,B\in\Omega$. Exchanging $A$ and $B$ gives 
  \begin{equation}\label{eq:lAB3}
        \sigma_B l_{A,B}=m_{A,B}\sigma_A l_{B,A}l_{A,B} 
  \end{equation}
  for every $A,B\in\Omega$. Finally, combine these equations:
    \begin{align*}
    m_{A,B}\sigma_A&=m_{A,B}m_{B,A}\sigma_B l_{A,B}&&\text{ by }\eqref{eq:lAB1} \\ 
    &=m_{A,B}m_{B,A}\sigma_B l_{A,B}l_{B,A} l_{A,B} &&\text{ since }l_{A,B}\text{ is a partial isometry}\\ 
    &=m_{A,B}\sigma_A l_{B,A}l_{A,B}&&\text{ by }\eqref{eq:lAB2} \\ 
    &=\sigma_B l_{A,B} &&\text{ by }\eqref{eq:lAB3}
    \end{align*}
  This completes the proof.
\end{proof}

\begin{corollary}\label{cor:mapoflimsismapofcolims} 
  In the situation of Lemma~\ref{lem:connectingmaps}, the unique morphism $f \colon L\to M$ satisfying
  \[\begin{tikzpicture}
         \matrix (m) [matrix of math nodes,row sep=2em,column sep=4em,minimum width=2em]
         {
         L & M \\
          D(A) & E(A) \\};
         \path[->]
         (m-1-1) edge node [left] {$l_A$} (m-2-1)
                edge[dashed] node [above] {$f$} (m-1-2)
         (m-1-2) edge node [right] {$m_A$} (m-2-2)
         (m-2-1) edge node [below] {$\sigma_A$} (m-2-2);       
  \end{tikzpicture}\]
  also makes the following diagram commute.
  \begin{equation}\label{diag:colims}
  \begin{aligned}\begin{tikzpicture}
         \matrix (m) [matrix of math nodes,row sep=2em,column sep=4em,minimum width=2em]
         {
         L& M \\
          D(A) & E(A) \\};
         \path[->]
         (m-1-1) edge node [above] {$f$} (m-1-2)
         (m-2-2) edge node [right] {$m_A^\dag $} (m-1-2)
         (m-2-1) edge node [below] {$\sigma_A$} (m-2-2)
                 edge node [left] {$l_A^\dag$} (m-1-1);  
  \end{tikzpicture}\end{aligned}\end{equation}
  In other words, the map of limits $L \to M$ induced by $\sigma$ coincides with the map of colimits $L \to M$ induced by $\sigma\colon \dag\circ D\to\dag \circ E$.
\end{corollary}
\begin{proof} 
  To show that $f$ makes diagram~\eqref{diag:colims} commute, it suffices to postcompose with an arbitrary $m_B$ and show that the following diagram commutes.
        \[
        \begin{tikzpicture}
         \matrix (m) [matrix of math nodes,row sep=2em,column sep=4em,minimum width=2em]
         {
          D(A) & E(A) & M \\
          L & D(B) & E(B )\\ 
           & M \\};
         \path[->]
         (m-1-1) edge node [above] {$\sigma_A$} (m-1-2)
                 edge node [left] {$l_A^\dag$} (m-2-1)
         (m-1-2) edge node [above] {$m_A^\dag$} (m-1-3)
         (m-3-2) edge node [below] {$m_B$} (m-2-3)
         (m-1-3) edge node [right] {$m_B$} (m-2-3)
         (m-2-1) edge node [above] {$l_B$} (m-2-2)
                edge node [below] {$f$} (m-3-2)
          (m-2-2) edge node [above] {$\sigma_B$} (m-2-3);
        \end{tikzpicture}
        \]
  The bottom part commutes by definition of $f$, and the top rectangle by Lemma~\ref{lem:connectingmaps}.
\end{proof}

We can now prove Theorem~\ref{thm:daglimsunique}.

\begin{proof}[Proof of Theorem~\ref{thm:daglimsunique}] If the map of cones $L\to M$ is unitary, it is straightforward to verify that $M$ is a dagger limit of $(D,\Omega)$. For the converse, by Theorem~\ref{thm:equivalenttounitary} it suffices to prove that the diagram
    \[\begin{tikzpicture}
     \matrix (m) [matrix of math nodes,row sep=2em,column sep=4em,minimum width=2em]
     {
      D(A) & L \\
      M & D(B) \\};
     \path[->]
     (m-1-1) edge node [left] {$m_A^\dag$} (m-2-1)
            edge node [above] {$l_A^\dag$} (m-1-2)
     (m-1-2) edge node [right] {$l_B$} (m-2-2)
     (m-2-1) edge node [below] {$m_B$} (m-2-2);
    \end{tikzpicture}\]
commutes for any $A$ and $B$ in $\cat{J}$. This follows from Lemma~\ref{lem:connectingmaps} by choosing $\sigma=\id[D]$, which is clearly adjointable.
\end{proof}

Note that the preceding proof only used Lemma~\ref{lem:connectingmaps} in the special case of $\sigma=\id$. This special case is not sufficient for the sequel and indeed the general versions of Lemma~\ref{lem:connectingmaps} and of Corollary~\ref{cor:mapoflimsismapofcolims} are used later on, \eg in Sections~\ref{sec:globaldaglims} and~\ref{sec:commutativity}.

\section{Completeness}\label{sec:completeness}

Perhaps the most obvious definitions of dagger completeness would be ``every small diagram has a dagger limit for some/all possible weakly initial $\Omega$''. However, here dagger category theory deviates from ordinary category theory. Such a definition would be too strong to allow interesting models, as the following theorems show.

\begin{theorem}\label{thm:finitelydagcompleteimpliesindiscrete} 
  If a dagger category has dagger equalizers, dagger pullbacks and finite dagger products, then it must be indiscrete.
\end{theorem} 

\begin{theorem}\label{thm:dagcompleteimpliesindiscrete} 
  If a dagger category has dagger equalizers and infinite dagger products, then it must be indiscrete.
\end{theorem} 

In proving these degeneration theorems, it is useful to isolate some lemmas that derive key properties implied by having dagger equalizers, dagger pullbacks and finite dagger products.

\begin{lemma}\label{lem:dagpbs->everythingispi} 
  The dagger pullback of a morphism $f\colon A\to B$ along $\id[B]$ exists if and only if $f$ is a partial isometry. In particular, if a dagger category has all dagger pullbacks, then every morphism is a partial isometry. 
\end{lemma}
\begin{proof} 
  If $f$ is a partial isometry, then $A$ is a dagger pullback of $f$ along $\id[B]$, with cone $\id[A]$ and $f$.
  For the converse, let $P$ be a dagger pullback of $f$ along $\id[B]$, with cone $p_A \colon P \to A$ and $p_B \colon P \to B$. Since $A$ is also a pullback, there exists a unique isomorphism $g\colon A\to P$ making the diagram 
  \[\begin{tikzpicture}
         \matrix (m) [matrix of math nodes,row sep=2em,column sep=4em,minimum width=2em]
         {
          A \\
          &P & B \\
          &A & B \\};
         \path[->]
         (m-1-1) edge[out=-90,in=180] node [left] {$\id$} (m-3-2)
                 edge node [right=2mm] {$g$} (m-2-2)
                 edge[dagger,out=0] node [above] {$f$} (m-2-3)
         (m-2-2) edge[dagger] node [left] {$p_A$} (m-3-2)
                edge[dagger] node [above] {$p_B$} (m-2-3)
         (m-2-3) edge node [right] {$\id$} (m-3-3)
         (m-3-2) edge[dagger] node [below] {$f$} (m-3-3);
  \end{tikzpicture}\]
  commute. Since $g$ is an isomorphism, this implies that $p_A=g^{-1}$. Because $p_A$ is a partial isometry, both $p_A$ and $g$ are therefore unitary. Now $f=p_Bg$ means that $f$ factors as the composite of a unitary morphism and a partial isometry and hence is a partial isometry itself.
\end{proof}

It follows from the previous lemma that in a dagger category with dagger pullbacks, every subobject is a dagger subobject since every monomorphism is dagger monic.

Recall that a monoid $(M,+,0)$ is \emph{cancellative} when $x+z=y+z$ implies $y=z$.

\begin{lemma}\label{lem:matrixcalculus}
  If a dagger category has finite dagger products, then it is uniquely enriched in commutative monoids and admits a matrix calculus. If it furthermore has dagger equalizers, then addition is cancellative.
\end{lemma}
\begin{proof} 
  Dagger products are in particular biproducts, and it is well known that biproducts make a category uniquely semiadditive; see \eg~\cite[18.4]{mitchell}, \cite[2.4]{vicary2011completeness}, or~\cite[1.1]{heunen:semimodules}.
  For the claim that dagger equalizers make addition cancellative, we refer to~\cite[2.6]{vicary2011completeness}.
\end{proof}

We can now prove the theorems stated in the beginning of this section.

\begin{proof}[Proof of Theorem~\ref{thm:finitelydagcompleteimpliesindiscrete}] 
  Let $f \colon A \to B$; we will prove that $f=0_{A,B}$. By Lemma~\ref{lem:dagpbs->everythingispi}, the tuple $\langle \id, f\rangle \colon A\to A\oplus B$ is a partial isometry. Expanding this fact using the matrix calculus results in the equations $\id+f^\dag f=\id$ and $f=f+f$. Applying cancellativity of addition to the latter equation now gives $f=0$.
\end{proof}

\begin{proof}[Proof of Theorem~\ref{thm:dagcompleteimpliesindiscrete}] 
  Let $f \colon A \to B$; we will prove that $f=0_{A,B}$. 
  Observe that in Lemma~\ref{lem:matrixcalculus}, infinite dagger products induce the ability to add infinitely many parallel morphisms in the same way as binary dagger products enable binary addition: if $f_i\colon A\to B$ is an $I$-indexed family of morphisms and \cat{C} has $\abs{I}$-ary dagger products, the sum $\sum f_i$ can be defined as the composite $\Delta^\dag_B(\bigoplus_i f_i)\Delta_A\colon A\to\bigoplus_i A\to \bigoplus_i B\to B$, where $\Delta_A\colon A\to\bigoplus_{i} A$ is the diagonal map. This implies that the category is enriched in $\Sigma$-monoids~\cite{haghverdiscott:goi}. 
  Hence the following variant of the Eilenberg swindle makes sense.
  \begin{align*} 
    0+(f+f+\cdots)
    &=f+f+\cdots \\
	&=f+(f+f+\cdots)
  \end{align*}
  It now follows from cancellativity (of binary addition) that $f=0$.
\end{proof}

Note that both theorems rest on an interplay between various dagger limits. Indeed, none of the dagger limits in question are problematic on their own: \cat{Hilb} has dagger equalizers and finite dagger products, \cat{Rel} has arbitrary dagger products, and \cat{PInj} has dagger pullbacks. 

Given Theorems~\ref{thm:finitelydagcompleteimpliesindiscrete} and~\ref{thm:dagcompleteimpliesindiscrete}, dagger completeness is not an interesting concept. However, below we define a (finite) compleness condition for dagger categories, and show that it is equivalent to the existence of dagger products, dagger equalizers, and dagger intersections.

\begin{definition}\label{def:based}\index[word]{dagger limit!(finitely) based}
  A class $\Omega$ of objects of a category $\cat{J}$ is called a \emph{basis} when every object $B$ allows a unique $A \in \Omega$ making $\cat{J}(A,B)$ non-empty.
  The category $\cat{J}$ is called \emph{based} when there exists a basis, and \emph{finitely based} when there exists a finite basis.
  We say a dagger category \cat{C} has \emph{(finitely) based dagger limits}, or that it is \emph{(finitely) based dagger complete} if for every category \cat{J} with a (finite) basis $\Omega$, and any diagram $D\colon\cat{J}\to\cat{C}$ the dagger limit of $(D,\Omega)$ exists.
\end{definition} 


\begin{example}
  The shape $\bullet \rightrightarrows \bullet$, giving rise to equalizers, is finitely based. Any (finite) discrete category, the shape giving rise to (finite) products, is (finitely) based. Any  indiscrete category is also finitely based. A category is (finitely) based if and only if it is the (finite) disjoint union of categories, each having a weakly initial object. In particular, any dagger category is based and is finitely based iff it has finitely many connected components.
\end{example}

Since both dagger equalizers and arbitrary dagger products exist in a based dagger complete category, all such categories are indiscrete by Theorem~\ref{thm:dagcompleteimpliesindiscrete} and hence the notion is uninteresting. However, being finitely (based) complete need not trivialize the category as seen below in Corollary~\ref{cor:hilbiscomplete}.

\begin{example}
  Finitely based diagrams need not be finite. For example, given any family $f_i\colon A\to B$ of parallel arrows, their joint dagger equalizer is the dagger limit of a finitely based diagram. Similarly, a dagger intersection is the dagger limit of a diagram of dagger monomorphisms $m_i \colon A_i \to B$, \ie a wide (dagger) pullback. Whenever there are at least two monomorphisms, this diagram is not a based category. However, we may also obtain the dagger intersection as the dagger limit of the diagram consisting of $m_im_i^\dag \colon B \to B$, \ie as the limit of the induced projections on $B$. Then we may always take $\Omega$ to be a singleton. Hence being finitely based dagger complete is a stronger requirement than being finitely complete.
\end{example}

\begin{theorem} \label{thm:finitecompleteness}
  A dagger category is finitely based dagger complete if and only if it has dagger equalizers, dagger intersections and finite dagger products.
\end{theorem} 
\begin{proof}
  One direction is obvious, because the shapes of dagger equalizers, dagger intersections, and finite dagger products are all  finitely based. For the converse, assume that a dagger category $\cat{C}$ has dagger equalizers, dagger intersections and finite dagger products. Let $\cat{J}$ have a finite basis $\Omega$. For a given diagram $D\colon\cat{J}\to\cat{C}$, we will construct a dagger limit of $(D,\Omega)$. Fix $A\in\Omega$. For every parallel pair $f,g\colon A\to B$ of morphisms in \cat{J}, pick a dagger equalizer $e_{f,g}\colon E_{f,g}\to D(A)$ of $D(f)$ and $D(g)$. Pick a dagger intersection $m_A\colon L_A\to D(A)$ of all $e_{f,g}$. Pick a dagger product $L$ of $L_A$ for all $A \in \Omega$, and write $p_A \colon L \to L_A$ for the product cone. Given a morphism  $f\colon A\to B$ in $\cat{J}$ with $A\in \Omega$, define $l_A$ as the composite $D(f) m_Ap_A\colon L\to L_A\to D(A)\to D(B)$. Since $\Omega$ is a basis, the choice of $A$ is forced on us. By construction $l_A$ is independent of the choice of $f\colon A\to B$. It is easy to see that $l_A \colon L \to D(A)$ is a limiting cone, so it remains to check the normalization and independence axioms for $\Omega$. Given $A\in\Omega$, 
  \begin{align*}
    l_Al_A^\dag l_A&=m_Ap_A p_A^\dag m_A^\dag m_Ap_A \\
    &= m_Ap_A p_A^\dag p_A &&\text{ since }m_A\text{ is an isometry}\\
    &= m_Ap_A &&\text{ since }p_A\text{ is a partial isometry}\\
     &=l_A\text.
  \end{align*}
  Moreover, since $p_A \colon L \to D(A)$ forms a product, $l_A^\dag l_A l_B^\dag l_B=0_L=l_B^\dag l_Bl_A^\dag l_A$ if $A,B\in \Omega$.
\end{proof} 

\begin{corollary}\label{cor:hilbiscomplete}\index[symb]{\cat{Hilb}, Hilbert spaces and bounded linear maps}\index[symb]{\cat{FHilb}, f.d. Hilbert spaces and linear maps} The categories \cat{FHilb} and \cat{Hilb} are finitely based dagger complete.
\end{corollary}

\begin{corollary} Both categories have dagger equalizers and dagger products. To see that dagger intersections exist, note that a dagger monic corresponds to an inclusion of a closed subspace, and hence the intersection can be computed as the set-theoretic intersection of the corresponding closed subspaces, which is still closed. Hence the result follows from Theorem~\ref{thm:finitecompleteness}.
\end{corollary}

The following theorem characterizes dagger categories having dagger-shaped limits similarly.

\begin{theorem}\label{thm:dagshapedcompleteness}
  Let $\kappa$ be a cardinal number. A dagger category has dagger-shaped limits of shapes with at most $\kappa$ many connected components if and only if it has dagger split infima of projections, dagger stabilizers, and dagger products of at most $\kappa$ many objects. 
\end{theorem}
\begin{proof}
  One implication is trivial. For the other, we first reduce to the case of a single connected component. Given $D\colon\cat{J}\to\cat{C}$, with connected components $\cat{J}_i$, assume that we can use dagger limits of projections and dagger stabilizers to build a dagger limit $L_i$ of each restriction $D_i$. Then we can define the dagger limit of $D$ by taking the dagger product of all the $L_i$. For $L$ is a limit of $D$, as each $L_i$ is a limit of $D_i$. Moreover, since $L_i\to D_i(A)$ is a partial isometry, $L_i\to D_i(A)\to L_i$ is the identity. Therefore each $L\to L_i\to D(A)$ is a partial isometry. If $A$ and $B$ are in the same connected component of $\cat{J}$, then $l_A^\dag l_A=l_B^\dag l_B$. If $A$ and $B$ are in different connected components, then $l_Bl_A=0_{DA,DB}$. Hence the independence equations are satisfied, and $L$ indeed is the dagger limit of $D$. 

  Now focus on the case when $\cat{J}$ is connected. Pick any object $A$ of \cat{J}. For every $f,g\colon A\to B$, take the dagger stabilizer $E_{f,g}\to D(A)$ of $D(f)$ and $D(g)$, and let $p_{f,g}$ be the induced projection $D(A)\to E_{f,g}\to D(A)$ on $D(A)$. Let $\mathcal{P}$ be the set of all such projections, and let $l\colon L\to D(A)$ be the limit. For each $B$ in \cat{J}, choose a morphism $f\colon A\to B$, and define $l_B \colon L\to D(B)$ as $D(f)\circ l$. By construction, $l_B$ is independent of the choice of $f$. We claim that this makes $L$ into a dagger limit of $D$. First, to see it is a cone, take an arbitrary $g\colon B\to C$. Then $D(g)l_B=D(g)D(f)l_A=D(gf)l_A=l_C$, as desired. Because $L\to  D(A)$ is a partial isometry, so is each $L\to D(B)$, as in Example~\ref{ex:daggershaped}. Moreover, $l_A^\dag l_A=l_B^\dag l_B$ when $A,B\in\cat{J}$. Finally, to show that $L$ is indeed a limit, take an arbitrary cone $m_B \colon M \to D(B)$. As it is a cone, it is uniquely determined by $m_A$. But $m_A$ must factor through each stabilizer, so $m_A$ is also a cone for $\mathcal{P}$. The universal property of $L$ ensures that $m_A$ factors through $l_A$. 
\end{proof}

\begin{corollary} 
  The dagger category \cat{Rel} has all dagger-shaped limits, and the dagger categories \cat{FinRel} and $\cat{Span}(\cat{FinSet})$ have all dagger-shaped limits with finitely many connected components.
\end{corollary}
\begin{proof} 
  Observe that \cat{Rel} has all (small) dagger products, and that similarly \cat{FinRel} and $\cat{Span}(\cat{FinSet})$ have finite dagger products. Combine Proposition~\ref{prop:variouscatshavedagshapedlims} and Theorem~\ref{thm:dagshapedcompleteness}.
\end{proof}

We end this section by comparing our notion of finite dagger completeness to that of~\cite{vicary2011completeness}. Thus we must restrict to categories enriched in commutative monoids. This simplifies things, as in the following lemma. 

\begin{proposition}\label{prop:limswithzeros} 
  Let $\Omega$ be a basis for \cat{J} and $D\colon\cat{J}\to\cat{C}$ a diagram in a category with zero morphisms. A limit $l_A \colon L \to D(A)$ of $D$ is a dagger limit of $(D,\Omega)$ if and only if:
    \begin{itemize}
      \item $l_A$ is a partial isometry whenever $A\in\Omega$; and
      \item $l_Bl_A^\dag=0_{A,B}$ whenever $A,B\in\Omega$ and $A\neq B$.
    \end{itemize}
\end{proposition}
\begin{proof} 
  Any limit satisfying the above conditions also satisfies normalization. For independence, the second condition gives $l^\dag_A l_A l_B^\dag l_B=0_{L,L}=l^\dag_Bl_Bl^\dag_Al_A$. 

  Conversely, given a dagger limit $l_A \colon L \to D(A)$ of $(D,\Omega)$, we wish to show that $l_Bl_A^\dag=0_{A,B}$ whenever $A,B\in\Omega$ are distinct. Fix distinct $A,B\in\Omega$ and define a natural transformation $\sigma\colon D\Rightarrow D$ by
  \[
    \sigma_C = \begin{cases} \id[D(C)] & \text{if }\cat{J}(A,C)\neq \emptyset \\
             0_{D(C),D(C)} &\text{otherwise.} \end{cases}
  \]
  By construction $\sigma$ is adjointable. Hence Lemma~\ref{lem:connectingmaps} guarantees $l_Bl_A^\dag=l_Bl_A^\dag\id=l_Bl_A^\dag\sigma_A=\sigma_B l_Bl_A^\dag=0l_Bl_A^\dag=0$.
\end{proof}

\begin{proposition}\label{prop:limswithsums} 
  Let $\Omega$ be a finite basis for \cat{J} and $D\colon\cat{J}\to\cat{C}$ a diagram in a dagger category $\cat{C}$ enriched in commutative monoids. A limit $l_A \colon L \to D(A)$ is a dagger limit of $(D,\Omega)$ if and only if $\id[L]=\sum_{A\in \Omega} l_A^\dag l_A$.
\end{proposition}
\begin{proof}
  Assume first that $L$ is a dagger limit of $(D,\Omega)$. Since \cat{C} is enriched in commutative monoids, the composition $\cat{C}(A,B) \otimes \cat{C}(B,C) \to \cat{C}(A,C)$ is bilinear by definition of tensor product of commutative monoids, and so the unit morphisms of addition are zero morphisms.
  Therefore $l_Bl_A^\dag=0$ by Proposition~\ref{prop:limswithzeros} when $A,B\in\Omega$ are distinct. Hence 
  \[
    l_B\sum_{A\in \Omega} l_A^\dag l_A=l_Bl_B^\dag l_B+ l_B\sum_{B\neq A\in\Omega} l_Al_A  
    =l_B+0=l_B
  \]
  for every $B\in \Omega$, which implies $\id[L]=\sum_{A\in \Omega} l_A^\dag l_A$.

  For the converse, assume that $L$ is a limit of $D$ satisfying $\id[L]=\sum_{A\in \Omega} l_A^\dag l_A$. Fix $A\in\Omega$ and define $\sigma\colon D\Rightarrow D$ as in the proof of Proposition~\ref{prop:limswithzeros}.
  Now $\sigma$ induces a unique morphism $f\colon L\to L$ making the following square commute for any $C$ in $\cat{J}$:
  \[\begin{tikzpicture}
    \matrix (m) [matrix of math nodes,row sep=2em,column sep=4em,minimum width=2em]
    {
    L & L \\
     D(C) & D(C) \\};
    \path[->]
    (m-1-1) edge node [left] {$l_C$} (m-2-1)
           edge node [above] {$f$} (m-1-2)
    (m-1-2) edge node [right] {$l_C$} (m-2-2)
    (m-2-1) edge node [below] {$\sigma_C$} (m-2-2);       
  \end{tikzpicture}\]
  Now
  $
    f
    = \id[L]f 
    = (\sum_{B\in \Omega} l_B^\dag l_B)f 
    = \sum_{B\in \Omega} l_B^\dag l_B f 
    = \sum_{B\in\Omega} l_B^\dag \sigma_B l_B 
    =l_A^\dag l_A +0=l_A^\dag l_A
  $.
  It follows that 
  \begin{equation}\label{eq:limswithsums}
    l_Bl_A^\dag l_A=0 
  \end{equation}
  for distinct $A,B\in\Omega$, so that 
  $
    l_A
    =l_A\id[L] 
    =l_A\sum_{A\in \Omega} l_A^\dag l_A
    =l_A l_A^\dag l_A+0
    =l_A l_A^\dag l_A
  $
  and each $l_A$ is a partial isometry. Finally~\eqref{eq:limswithsums} implies $l_Bl_A^\dag=0_{A,B}$ for distinct $A,B$, making $L$ into a dagger limit of $(D,\Omega)$ by Proposition~\ref{prop:limswithzeros}.
\end{proof}

Thus whenever \cat{J} has a finite basis and \cat{C} is appropriately enriched, Definition~\ref{def:based} coincides with the notion of completeness in~\cite{vicary2011completeness}. But our notion is not more general: when the diagram does not admit a finite basis, the two notions are different. For instance, \cat{FHilb} does not have all dagger pullbacks, because not all morphisms are partial isomorphisms. But for every diagram $A\to C\leftarrow B$ there exists a pullback $A\leftarrow L\to B$ with $\id[L]=l_A^\dag l_A+l_B^\dag l_B$. Conversely, the following example below exhibits a dagger pullback in \cat{Rel} that does not satisfy such a summation.

\begin{example}\label{ex:pullback} 
  Consider the objects $A=\{1,2\}$ and $B=\{1\}$ in $\cat{Rel}$, and morphisms $R=A \times A \colon A \to A$ and $S=B \times A \colon B \to A$.
  Then the object $A$ with the cone morphisms $l_A=\id[A]\colon A\to A$ and $l_B=A\times B\colon A\to B$ form a pullback of $R$ and $S$. Now both $l_A$ and $l_B$ are partial isometries, and $l_A^\dag l_A l_B^\dag l_B = l_B^\dag l_B l_A^\dag l_A$, so this is a dagger pullback.
  But $l_A^\dag l_A+l_B^\dag l_B=A\times A\neq\id[A]$.
\end{example}

\section{Global dagger limits}\label{sec:globaldaglims}

If limits happen to exist for all diagrams of a fixed shape, it is well known that they can also be formulated as an adjoint to the constant functor. This section explores this phenomenon of global limits in the dagger setting.

More precisely, fix a shape \cat{J} and a weakly initial $\Omega$. Assume that a dagger category \cat{C} has $(\cat{J},\Omega)$-shaped dagger limits. Then there is an induced right adjoint $L$ to the diagonal functor $\Delta\colon \cat{C}\to \cat{Cat}(\cat{J},\cat{C})$. 
It restricts to a dagger functor $\hat{L}\colon[\cat{J},\cat{C}]_\dag\to\cat{C}$: $L(\sigma^\dag) = L(\sigma)^\dag$ for any $D,E\colon\cat{J}\to\cat{C}$ and adjointable $\sigma \colon D \Rightarrow E$. This follows from Corollary~\ref{cor:mapoflimsismapofcolims}, since $L(\sigma^\dag)$ is the map of limits induced by $\sigma^\dag\colon E\to D$, and  $L(\sigma)^\dag$ is the map of colimits induced by $\sigma^\dag\colon \dag \circ E\to \dag \circ D$. 
Normalization and independence endow the adjunction $\Delta \dashv L$ with further properties. 

\begin{description}
  \item[Normalization] 
  Consider the counit $\varepsilon\colon \Delta L\to \id[{\cat{Cat}(\cat{J},\cat{C})}]$. By construction, for any diagram $D\colon\cat{J}\to\cat{C}$, each $(\varepsilon_D)_A$ is a partial isometry whenever $A\in\Omega$. ``The counit is a partial isometry when restricted to $\Omega$''. 

  \item[Independence]
  Write $d_A$ for the component of $\varepsilon_D\colon \Delta L(D)\to D$ at $A$, and $e_A$ for the component of $\varepsilon_E$ at $A$. For a fixed $A\in \Omega$, define a natural transformation $\rho_A\colon \hat{L}\to\hat{L}$ by setting $(\rho_A)_D=d_A^\dag d_A\colon L(D)\to D(A)\to L(D)$. To see that this forms a natural transformation $\hat{L}\Rightarrow \hat{L}$, let $\sigma\colon D\Rightarrow E$ be adjointable and consider the following diagram.
	\[\begin{tikzpicture}
	         \matrix (m) [matrix of math nodes,row sep=2em,column sep=4em,minimum width=2em]
	         {
	          L(D) & D(A) & L(D) \\
	          L(E) & E(A) & L(E) \\};
	         \path[->]
	         (m-1-1) edge node [above] {$d_B$} (m-1-2)
	                 edge node [left] {$L(\sigma)$} (m-2-1)
	         (m-1-2) edge node [above] {$d_A^\dag$} (m-1-3)
	                edge node [left] {$\sigma_A$} (m-2-2)
	         (m-1-3) edge node [right] {$L(\sigma)$} (m-2-3)
	         (m-2-1) edge node [below] {$e_A$} (m-2-2)
	          (m-2-2) edge node [below] {$e_A^\dag$} (m-2-3);
	\end{tikzpicture}\]
The left square commutes by definition of $L(\sigma)$, and the right one by definition of $L(\sigma^\dag)^\dag=L(\sigma)$, so that the rectangle expressing naturality of $\rho_A$ commutes. Clearly, $\rho_A\rho_B=\rho_B\rho_A$ whenever $A,B\in\Omega$.
\end{description}

These properties in fact characterize dagger categories $\cat{C}$ having all $(\cat{J},\Omega)$-shaped limits.

\begin{theorem}\label{thm:globallimits}\index[symb]{$[\cat{J},\cat{C}]_\dag$ functors and \emph{adjointable} transformations}
  A dagger category \cat{C} has all $(\cat{J},\Omega)$-shaped limits if and only if the diagonal functor $\Delta\colon \cat{C}\to \cat{Cat}(\cat{J},\cat{C})$ has a right adjoint $L$ such that:
    \begin{itemize}
        \item the counit is a partial isometry when restricted to $\Omega$;
        \item $L$ restricts to a dagger functor $\hat{L}\colon[\cat{J},\cat{C}]_\dag\to\cat{C}$;
        \item the family $\rho(A)_D=d_A^\dag d_A$ is natural $\hat{L}\to \hat{L}$ for any $A\in\Omega$;
        \item if $A,B \in \Omega$, then $\rho_A\rho_B=\rho_B\rho_A$.
    \end{itemize}
\end{theorem}
\begin{proof}
  We already proved the implication from left to right above. The other implication is straightforward: it is well-known that if $L$ is the right adjoint to the diagonal then $LD$ is a limit of $D$ for each $D$. Hence we only need to check normalization and independence on $\Omega$. The counit being a partial isometry when restricted to $\Omega$ means that for each $A\in\Omega$ the structure map $LD\to DA$ of the limit is a partial isometry, giving us normalization. The condition $\rho_A\rho_B=\rho_B\rho_A$ for $A,B\in\Omega$ then amounts to independence.
\end{proof}

We leave open the question whether the fourth condition of Theorem~\ref{thm:globallimits} is necessary in general.

Further restricting the shape \cat{J} yields a cleaner special cases of the previous theorem.

\begin{theorem} 
  Let $\Omega$ be a basis for \cat{J}. A dagger category \cat{C} has all $(\cat{J},\Omega)$-shaped limits if and only if the diagonal functor $\Delta\colon \cat{C}\to \cat{Cat}(\cat{J},\cat{C})$ has a right adjoint $L$ such that:
    \begin{itemize}
        \item the counit is a partial isometry when restricted to $\Omega$;
        \item $L$ restricts to a dagger functor $\hat{L}\colon[\cat{J},\cat{C}]_\dag\to\cat{C}$.
    \end{itemize}
\end{theorem} 
\begin{proof} 
    The implication from left to right follows from the previous theorem.
    For the other direction, the conditions imply that a diagram $D$ has a limit given by $d_A \colon L(D) \to D(A)$, and furthermore $d_A$ is a partial isometry if $A\in \Omega$. It remains to verify that $d_A^\dag d_A$ commutes with $d_B^\dag d_B$ when $A,B\in\Omega$. We may assume that $\Omega$ has at least two distinct objects. We will show that \cat{C} has zero morphisms and $d_Bd_A^\dag=0_{D(A),D(B)}$ whenever $A,B\in\Omega$ are distinct. 

    Fix distinct $A,B\in\Omega$ and $X_0\in\cat{C}$. For objects $X,Y$ of $\cat{C}$, define $D_{X,Y}\colon\cat{J}\to\cat{C}$ by 
        \begin{equation*}
            D_{X,Y}(C)= \begin{cases} X & \text{if }\cat{J} (A,C)\neq \emptyset\text, \\
                                       Y &  \text{if }\cat{J} (B,C)\neq \emptyset\text, \\
                                       X_0 & \text{otherwise,}
                          \end{cases}
        \end{equation*}
    and mapping morphisms to \id[X], \id[Y] or \id[X_0] as appropriate. Define $0_{X,Y}$ as the composite $X=D_{X,Y}(A)\to L(D_{X,Y})\to D_{X,Y}(B)=Y$. Let $f\colon Y\to Z$ be arbitrary. This induces a natural transformation $\sigma\colon D_{X,Y}\Rightarrow D_{X,Z}$ such that $\sigma_C$ is $f$ if $\cat{J} (B,C)\neq \emptyset$ and $\sigma_C=\id$ otherwise. It is clearly adjointable, so the following diagram commutes.
    \[\begin{tikzpicture}
         \matrix (m) [matrix of math nodes,row sep=2em,column sep=4em,minimum width=2em]
         {
          D_{X,Y}(A)=X & LD_{X,Y} & D_{X,Y}(B)=Y \\
         D_{X,Z}(A)=X & LD_{X,X} & D_{X,Z}(B)=Z \\};
         \path[->]
         (m-1-1) edge node [left] {$\sigma_A=\id$} (m-2-1)
                edge node [above] {$$} (m-1-2)
         (m-1-2) edge node [above] {$$} (m-1-3)
                edge node [right] {$L\sigma=(L\sigma^\dag)^\dag$} (m-2-2)
         (m-1-3) edge node [right] {$\sigma_B=f$} (m-2-3)
         (m-2-1) edge node [below] {$$} (m-2-2)
         (m-2-2) edge node [below] {$$} (m-2-3);
    \end{tikzpicture}\]
    Hence $f\circ 0_{X,Y}=0_{X,Z}$. Taking daggers shows that $0_{Y,X}f^\dag=0_{Z,Y}$ for all $X,Y$ and $f\colon Y\to Z$. Thus \cat{C} indeed has zero morphisms.

    Finally, given an arbitrary $D\colon\cat{J}\to\cat{D}$, we will show that $d_{A,B}=0_{D(A),D(B)}$ for distinct $A,B$. Define $\sigma\colon D\Rightarrow D$ by 
    \begin{equation*}
            \sigma_C = \begin{cases} 0_{D(C),D(C)} & \text{if }\cat{J} (A,C)\neq \emptyset\text, \\
                                       \id & \text{otherwise.}
                          \end{cases}
    \end{equation*}
    As $\sigma$ is adjointable, the following diagram commutes.
    \[\begin{tikzpicture}
         \matrix (m) [matrix of math nodes,row sep=2em,column sep=6em,minimum width=2em]
         {
          D(A) & LD & D(B) \\
         D(A)=X & LD & D(B) \\};
         \path[->]
         (m-1-1) edge node [left] {$\sigma_A=0$} (m-2-1)
                edge node [above] {$d_A^\dag $} (m-1-2)
         (m-1-2) edge node [above] {$d_B$} (m-1-3)
                edge node [right] {$L\sigma=(L(\sigma)^\dag)^\dag$} (m-2-2)
         (m-1-3) edge node [right] {$\sigma_B=\id $} (m-2-3)
         (m-2-1) edge node [below] {$d_A^\dag$} (m-2-2)
         (m-2-2) edge node [below] {$d_B$} (m-2-3);
    \end{tikzpicture}\]
    Thus $d_{A,B}=0_{D(A),D(B)}$.
\end{proof}

\begin{theorem}\index[symb]{$[\cat{C},\cat{D}]$, dagger functors $\cat{C}\to\cat{D}$}
  Let \cat{J} be a dagger category. A dagger category \cat{C} has a dagger limit for every dagger functor $\cat{J}\to \cat{C}$ if and only if the diagonal functor $\Delta\colon \cat{C}\to \cat{Dagcat}(\cat{J},\cat{C})$ has a dagger adjoint such that the counit is a partial isometry.
\end{theorem}
\begin{proof} 
  The implication from left to right is straightforward once one remembers from Example~\ref{ex:daggershaped} that for any dagger limit $(L,\{l_A\}_{A\in\cat{J}})$ of a dagger-shaped diagram, every $l_A$ is a partial isometry.

  For the other direction, the dagger adjoint to the diagonal clearly gives a limit for each dagger functor $\cat{J}\to \cat{C}$. It remains to verify that they are all dagger limits. The counit being a partial isometry implies the normalization condition, so it suffices to check independence. If \cat{J} is connected, this is trivial. If \cat{J} is not connected, then, as in the previous proof, \cat{C} has zero morphisms and for any diagram $D$ we have $d_{A,B}=0_{D(A),D(B)}$ whenever $A$ and $B$ are in different components of \cat{J}.
\end{proof}

\section{Dagger adjoint functors}\label{sec:daft}

When dealing with dagger limits and dagger functors preserving them, the obvious question arises when there exists a dagger adjoint. That is, is there a dagger version of the adjoint functor theorem? 

Simply replacing limits with dagger limits in any standard proof of the adjoint functor theorem~\cite{mitchell} doesn't quite get there. If $\cat{C}$ has all dagger products and dagger equalizers, it must be indiscrete by Theorem~\ref{thm:dagcompleteimpliesindiscrete}. Hence any continuous functor $\cat{C} \to \cat{D}$ satisfying the solution set condition vacuously has an adjoint. A more interesting dagger adjoint functor theorem must therefore work with a finitely based complete category $\cat{C}$ and solution sets of a finite character.

There is a further obstacle. Ordinarily, the adjoint functor theorem shows that its assumptions imply a universal arrow $\eta_A \colon A \to GF(A)$ for each object $A$, so that the desired adjoint is given by $A \mapsto F(A)$. This will not do for dagger categories, as the resulting functor $F$ need not preserve the dagger. This is essentially because universal arrows need not be defined up to unitary isomorphism.
For an example, consider the identity functor $G \colon \cat{FHilb}\to\cat{FHilb}$, and define $\eta_A\colon A\to A$ to be multiplication by $1+\dim A$. Then each $\eta_A$ is an universal arrow for $G$, defining an adjoint $F\colon \cat{FHilb}\to\cat{FHilb}$ that sends $f\colon A\to B$ to $f(1+\dim B)/(1+\dim A)$, which is not a dagger functor. Of course, this example evaporates by choosing the `correct' universal arrows. But there are more involved examples of dagger functors admitting adjoints but no dagger adjoints. 

The moral is that a dagger adjoint to $G\colon \cat{C}\to\cat{D}$ requires more than $G$-universal arrows $A\to GF(A)$ for each object. The universal arrows must fit together, in the sense that they form an adjointable natural transformation. 
Unfortunately, we do not know of conditions on solution sets guaranteeing this. 
As a first step towards a dagger adjoint functor theorem proper we provide the following theorem.

\begin{theorem}\label{thm:adjunction}
  Suppose a dagger category \cat{C} has, and a dagger functor $G \colon \cat{C} \to \cat{D}$ preserves, dagger intersections and dagger equalizers. Then $G$ has a dagger adjoint if and only if there is a dagger functor $H\colon \cat{D}\to\cat{C}$ and a natural transformation $\tau \colon \id[\cat{D}]\to GH$ such that each component of $\tau$ is weakly G-universal.
\end{theorem}
\begin{proof} 
  One implication is trivial. For the other, define $F(A)$ to be a dagger intersection of all dagger monomorphisms $m\colon M\to H(A)$ for which $\tau_A$ factorizes through $G(m)$. As $G$ preserves dagger intersections, $\tau_A$ factorizes via $GF(A)\to GH(A)$, say $\tau_A=G(\sigma_A)\eta_A$. It suffices to prove that (i) $\eta_A$ is $G$-universal, so that $A\mapsto FA$ extends uniquely to a functor, and then that (ii) $F$ is a dagger functor. 

  First of all, $\eta_A$ is weakly $G$-universal since $\tau_{A}$ is so: given $f\colon A\to G(X)$, pick $h$ such that $f=G(h)\tau_A=G(h\sigma_A)\eta_A$. Moreover, if $h$ and $h'$ satisfied $G(h)\eta_A=G(h')\eta_A$, consider the equalizer $e\colon E\to F(A)$ of $h$ and $h'$. By assumption, $\eta_A$ factors through $G(e)$ and hence $\tau_A$ factors through $G(\sigma_A e)$, so that $\sigma_A e$ is already in the dagger intersection defining $F(A)$. In other words, $e$ is unitary, and hence $h=h'$. Thus $\eta_A$ is $G$-universal, and we can extend $A\mapsto F(A)$ to a functor by defining $F(f)$ as the unique map making the following square commute.
  \[\begin{tikzpicture}
    \matrix (m) [matrix of math nodes,row sep=2em,column sep=4em,minimum width=2em]
    { A & GFA \\
      B & GFB \\};
    \path[->]
    (m-1-1) edge node [left] {$f$} (m-2-1)
            edge node [above] {$\eta_A$} (m-1-2)
    (m-1-2) edge node [right] {$GF(f)$} (m-2-2)
    (m-2-1) edge node [below] {$\eta_B$} (m-2-2);	
  \end{tikzpicture}\]

  Next we show that $F$ is a dagger functor. By Lemma~\ref{lem:daggersubfunctorsaredagger} it suffices to show that $\sigma$ is a natural transformation $F\to H$. By naturality of $\tau$, the top part of
  \[\begin{tikzpicture}
    \matrix (m) [matrix of math nodes,row sep=2em,column sep=4em,minimum width=2em]
    { GFA &&& GHA \\
      A & B & GFB & GHB \\
      & GFA \\};
    \path[->]
    (m-1-1) edge node [above] {$G(\sigma_A)$} (m-1-4)
    (m-1-4) edge node [right] {$GH(f)$} (m-2-4)
    (m-2-1) edge node [above] {$f$} (m-2-2)
     		edge node [left] {$\eta_A$} (m-1-1)
         	edge node [below] {$\eta_A$} (m-3-2)
    (m-3-2) edge node [below] {$\quad GFf$} (m-2-3)
    (m-2-3) edge node [above] {$G(\sigma_B)$} (m-2-4)
    (m-2-2) edge node [above] {$\eta_B$} (m-2-3);
  \end{tikzpicture}\]
  commutes, whereas the bottom part commutes by naturality of $\eta$. As $\eta_A$ is $G$-universal, we conclude that the square 
  \[\begin{tikzpicture}
    \matrix (m) [matrix of math nodes,row sep=2em,column sep=4em,minimum width=2em]
    { FA & HA \\
	  FB & HB \\};
    \path[->]
    (m-1-1) edge node [left] {$F(f)$} (m-2-1)
            edge node [above] {$\sigma_A$} (m-1-2)
    (m-1-2) edge node [right] {$H(f)$} (m-2-2)
    (m-2-1) edge node [below] {$\sigma_B$} (m-2-2);	
  \end{tikzpicture}\]
  commutes, making $\sigma \colon F \Rightarrow H$ natural.
\end{proof}

\section{Polar decomposition}\label{sec:polar}

Polar decomposition, as standardly understood, provides a way to factor any bounded linear map between Hilbert spaces into a partial isometry and a positive morphism~\cite[Chapter 16]{halmos:problembook}. As dagger limits are defined in terms of partial isometries, one might hope that polar decomposition connects dagger limits and ordinary limits. This section explores this connection. We start by defining polar decomposition abstractly, and prove that this modified property holds in the category of Hilbert spaces. Recall that a bimorphism in a category is a morphism that is both monic and epic.

\begin{definition}\label{def:polar}\index[word]{polar decomposition}
  Let $f\colon A\to B$ be a morphism in a dagger category. A \emph{polar decomposition} of $f$ consists of two factorizations of $f$ as $f=pi=jp$,
  \[\begin{tikzpicture}
     \matrix (m) [matrix of math nodes,row sep=2em,column sep=4em,minimum width=2em]
     {
      A & A \\
      B & B \\};
     \path[->]
     (m-1-1) edge[dagger] node [left] {$p$} (m-2-1)
            edge node [above] {$i$} (m-1-2)
            edge node [right=3mm] {$f$} (m-2-2)
     (m-1-2) edge[dagger] node [right] {$p$} (m-2-2)
     (m-2-1) edge node [below] {$j$} (m-2-2);
  \end{tikzpicture}\]
  where $p$ is a partial isometry and $i$ and $j$ are self-adjoint bimorphisms. 
  A category \emph{admits polar decomposition} when every morphism has a polar decomposition.
\end{definition}

Note that if $f$ is monic then so is $p$: indeed, if $pg=ph$ then $fg=jpg=jph=fh$, whence $g=h$. Hence if $f$ is monic then $p$ is dagger monic since it is both monic and a partial isometry.

Unlike usual expositions of polar decomposition, we require $i$ and $j$ to be bimorphisms. On the other hand, we don't require them to be positive since mere self-adjointness suffices for our purposes -- most notably Theorems~\ref{thm:polarlimitsaredaggerlimits} and~\ref{thm:daglimsofisomorphicdiagrams}. However, for other purposes this definition might need to be modified and hence it should be only seen as a starting point for an abstract notion of polar decomposition.

We begin by proving that \cat{Hilb} admits polar decomposition in the above sense. Given that our notion of polar decomposition is slightly different from the usual one (merely, a single factorization as $f=pi$ with $p$ a partial isometry and $i$ positive), there is some work to do. First, recall that if $f$ is a bounded linear map, i.e. a morphism of \cat{Hilb}, then $\abs{f}$ is the (unique) positive square root of $f^\dag f$. 


\begin{lemma}\label{lem:absolutevaluecommutes} For any morphism $f$ in \cat{Hilb} we have $\abs{f^\dag}f=f\abs{f}$.
\end{lemma}
\begin{proof} First we show that it suffices to prove the claim for non-expansive $f$, \ie  we may assume that $\norm{f}\leq 1$. Assuming the claim holds for non-expansive maps, take $f$ with $\norm{f}>1$. Define $g:=f/\norm{f}$. Now $\abs{g}=\abs{f}/\norm{f}$ since positive square roots are unique and both sides square to $f^\dag f/\norm{f}^2$. Hence $\abs{g^\dag}g=g\abs{g}$ amounts to saying that 
    \[\abs{f^\dag}f/\norm{f}^2=f\abs{f}/\norm{f}^2\]
  so that multiplying by $\norm{f}^2$ gives the result for $f$. Hence we may assume that $f$ is non-expansive.

  By (the proof of) \cite[Theorem 23.2]{bachman2012functional}, the square root of a positive non-expansive operator $p$ is the strong limit of the sequence $(p_n)_{n\in\N}$, where $p_0=0$ and $p_{n+1}:=p_n-(p-p_n^2)/2$. Applying this result to $f^\dag f$ we see that $\abs{f}$ is the strong limit of the sequence $(q_n)_{n\in\N}$ where $q_0=0$ and $q_{n+1}:=q_n-(f^\dag f-q_n^2)/2$. Similarly, $\abs{f^\dag}$ is the strong limit of the sequence $(r_n)_{n\in\N}$ defined by $r_0=0$ and $r_{n+1}:=r_n-(ff^\dag-r_n^2)/2$. Since precomposing (and postcomposing) with a non-expansive map is continuous (in the strong operator topology), to prove that $\abs{f^\dag}f=f\abs{f}$ it suffices to prove that $r_nf=fq_n$, which we do by induction. For $n=0$, both sides evaluate to $0$. Assuming the claim holds for $n$ we see that it holds for $n+1$, since
  \begin{align*}
    r_{n+1} f&=(r_n-(ff^\dag-r_n^2)/2)f \\
            &=fq_n-(ff^\dag f-fq_n^2)/2\text{ by the induction hypothesis} \\
            &=f(q_n-(f^\dag f-q_n^2)/2)=fq_{n+1}
  \end{align*}
  completing the proof.
\end{proof}

\begin{theorem}
  The category $\cat{Hilb}$ admits polar decomposition.	
\end{theorem}
\begin{proof}
  We modify the standard construction of a polar decomposition~\cite[Problem 134]{halmos:problembook} to satisfy Definition~\ref{def:polar}: the standard construction gives a factorization of $f\colon H\to K$ into $f=p\abs{f}$, where $p$ is a partial isometry satisfying $\ker p=\ker f$. This in fact fixes $p$ uniquely: $H$ decomposes into a direct sum $H\cong \ker f\oplus\ker f^\perp$ and since (i) $\ker f=\ker\abs{f}=\ker p$  and (ii) $\ker \abs{f}^\perp=\overline{\im \abs{f}}$ (see \eg\cite[Lemma 2.1]{buschetal:quantummeasurement}), $H$ in fact decomposes as $H\cong \overline{\im \abs{f}}\oplus \ker p$. The action of $p$ on its kernel is clear and on its orthocomplement it has to be given by continuous extension of $\abs{f}x\mapsto fx$ for $p\abs{f}=f$ to hold. Now, let $r$ be the projection onto $\ker p=\ker\abs{f}$, and set $i:=\abs{f}+r$. On $\ker \abs{f}$ $i$ acts as the identity and on its orthocompletent it acts as $\abs{f}$. Hence $i$ is positive. By construction, $\ker i=0$ so that $i$ is monic. Being self-adjoint it is also epic. Moreover, since $\ker p=\ker\abs{f}$ the factorization $f=pi$ follows from $f=p\abs{f}$.

  Similarly $f^\dag$ factors as $f^\dag=q\abs{f^\dag}$ where $q$ is defined similarly, and this factorization can be modified to obtain $f^\dag=q j$, where $j$ is a self-adjoint (in fact positive) bimorphism. Taking daggers, we have $f=jq$. 

  Hence it remains to show that $q^\dag=p$. Since the dagger in \cat{Hilb} is given by adjoints of bounded linear maps, this boils down to showing that
    \begin{equation}\label{eq:inprod}\inprod{px}{y}=\inprod{x}{qy}\end{equation}
  holds for all $x\in H$ and $y\in K$. We can now use our orthogonal decompositions $H\cong \overline{\im \abs{f}}\oplus \ker p$ and $K\cong\overline{\im \abs{f^\dag}}\oplus \ker q $ and consider separately the cases where $x$ and $y$ are in each of the summands. If $x\in \ker p$ then the left hand side of \eqref{eq:inprod} equals zero. But then $x$ is ortohogonal to $\im {f}=\im f=\im q$ so the right hand side is zero as well. Similarly, both sides equal zero when $y\in\ker q$. Thus we're left to consider the case when $x\in \overline{\im \abs{f}}$ and $y\in \overline{\im \abs{f^\dag}}$, and by continuity we may further assume $x\in \im \abs{f}$ and $y\in\im\abs{f^\dag}$. Pick $z\in H$ and $w\in K$ with $\abs{f}z=x$ and $\abs{f^\dag} w=y$. Now $px=p\abs{f}z=fz$ and $qy=q\abs{f^\dag}w=f^\dag w$, so \eqref{eq:inprod} boils down to whether 
    \[\inprod{fz}{\abs{f^\dag}w}=\inprod{\abs{f}z}{f^\dag w}\]
  Now the left hand side equals $\inprod{\abs{f^\dag}fz}{w}$  and the right hand side equals $\inprod{f\abs{f}z}{w}$ so they are equal by Lemma~\ref{lem:absolutevaluecommutes}, completing the proof.
\end{proof}

In \cat{Hilb}, we can not guarantee that $i$ and $j$ are isomorphisms in general. A good example is when $f\colon \ell^2(\N)\to\ell^2(\N)$ is defined on the $n$-th basis element $e_n$ by $f(e_n)=e_n/(n+1)$. Then the factorization above gives $\id$ as the partial isometry and $f$ as the positive bimorphism. However, $f$ does not have an inverse in \cat{Hilb} -- indeed, the ``inverse'' defined by $e_n\mapsto (n+1)e_n$ is not bounded.


Other dagger categories admitting polar decomposition include inverse categories, such as $\cat{PInj}$, and any groupoid, in which every morphism itself is already a partial isometry.


One might think that polar decomposition is an orthogonal factorization system, 
but there are several differences. First, the composition of partial isometries need not be a partial isometry, and the composition of self-adjoint bimorphisms need not be self-adjoint. Second, an isomorphism need not be a partial isometry nor self-adjoint. Third, $p$, $i$, and $j$ are not required to be unique to $f$. Fourth, the factorization $f=pi$ respects the dagger: even though one may also factor $f^\dag=qj$ and hence $f=j^\dag q$, we are additionally requiring that $p=q$.

Recall that a dagger category is unitary whenever two objects being isomorphic implies that they are also unitarily isomorphic.

\begin{proposition} 
  Dagger categories that have polar decomposition are unitary. 
\end{proposition}
\begin{proof}
  Factor an isomorphism $f\colon A\to B$ as $f=pi=jp$ with $p$ a partial isometry and $i,j$ self-adjoint bimorphisms.  Now $pi=f$ implies that $p$ has a right inverse and $jp=f$ that it has a left inverse. Hence $p$ must be an isomorphism. Also being a partial isometry, $p \colon A \to B$ is therefore unitary.
\end{proof}



The theme of the rest of this section will be that dagger limits may be viewed as the partial isometry part of a polar decomposition of ordinary limits. 

\begin{example}\label{ex:polarlimitsaredaggerlimits}
  The theme of the rest of this section will be that dagger limits may be viewed as the partial isometry part of a polar decomposition of ordinary limits. Here are some examples to warm up to this theme.
  \begin{itemize}
    \item  	
    Let $e \colon E \to A$ be an equalizer of morphisms $f,g \colon A \to B$ in a dagger category. If $e=pi=jp$ is a polar decomposition, then $p \colon A \to B$ is a dagger equalizer of $f$ and $g$.
    Indeed, since $i$ is a bimorphism we see that $fp=gp$, so that $p$ factors through $e$ as $p=ek$ for some $k\colon A\to A$. Precomposing with $i$ we see that $e=eki$, whence $ki=\id$. Since $i$ is a bimorphism and has a left inverse, it is an isomorphism and $k=i^{-1}$. Hence $p=ei^{-1}$ is an equalizer and hence monic, and by definition $p$ is a partial isometry and hence dagger monic.%

   \item 
    A cone over a family of dagger monomorphisms $f_i \colon A_i \to B$ consists formally of a map $l_i\colon A\to A_i$ for each $i$ and of a map $f\colon A\to B$. However, the whole cone is determined by the map $f$: this is because $f_il_i=f$ implies $l_i=f_i^\dag f_i l_i=f_i^\dag f$ by dagger monicness of $f_i$. A map $f\colon A\to B$ defines a cone in this manner iff $f_i f_i^\dag f=f$ for each $i$. Now, if $f$ defines a limiting cone and $f=pi=jp$ is a polar decomposition, then $p \colon A \to B$ is a dagger intersection of $\{f_i\}$: since $i$ is a bimorphism we see that $p$ defines a cone for $f_i$ and hence factors through $f$. As in the equalizer case, this implies that $i$ is an isomorphism whence  $p=fi^{-1}$ is a pullback of monics and hence monic. Since $p$ is a partial isometry by definition it is also dagger monic.%

     \item 
     Let a projection $e=e^\dag=e^2 \colon A \to A$ be split by $f \colon B \to A$ and $g \colon A \to B$, so $e=fg$ and $gf=\id[B]$. If $f=pi=jp$ is a polar decomposition, then $p \colon B \to A$ is a dagger splitting of $e$: indeed $f$ is the equalizer of $\id$ and $e$ and hence so is $p$ by the above, so that $p$ is monic and hence dagger monic. It remains to check that $e=pp^\dag$. Now $ep=p$ so that $pp^\dag=epp^\dag$ as well. Applying the dagger to both sides of this results in $pp^\dag=pp^\dag e$.  On the other hand, $e=e^2=fge=pige=pp^\dag pige=pp^\dag fge=pp^\dag e^2=pp^\dag e$. Hence $e=pp^\dag e=pp^\dag$.
  \end{itemize}
\end{example}%

These examples are no accident, and the theorems proven below will bear out this theme. Before studying polar decompositions of more general diagrams we need the following lemma.

\begin{lemma}\label{lem:PiOfPolarCommutes}
  Let $f\colon A\to B$  be a morphism in a dagger category with polar decomposition $f=pi=jp$, and let $g\colon A\to A$ be self-adjoint. If $g$ commutes with $f^\dag f$, then it also commutes with $p^\dag p$.
\end{lemma}
\begin{proof} 
  Observe that $f^\dag f = p^\dag p i i$.
  \[\begin{tikzpicture}
  	\matrix (m) [matrix of math nodes,row sep=2em,column sep=4em,minimum width=2em]
  	{
  	 A & B & A \\
  	 A & B \\
  	 A \\};
  	\path[->]
  	(m-1-1) edge node [left] {$i$} (m-2-1)
  	       edge node [above] {$f$} (m-1-2)
  	(m-1-2) edge node [above] {$f^\dag$} (m-1-3)
  	        edge node [left] {$j$} (m-2-2)
  	(m-2-2) edge[dagger] node [below] {$p^\dag$} (m-1-3)
  	(m-2-1) edge[dagger] node [above] {$p$} (m-1-2)
  	         edge node [below] {$f$} (m-2-2)
  	         edge node [left] {$i$} (m-3-1)
  	(m-3-1) edge[dagger] node [below] {$p$} (m-2-2);
  \end{tikzpicture}\]
  It follows that 
  \[
    p^\dag p g p^\dag p i i
    = p^\dag p g f^\dag f 
    =p^\dag p f^\dag f g 
    =f^\dag f g 
    =g f^\dag f 
    =g p^\dag p i i\text.
  \]   
  Because $i$ is a bimorphism, $p^\dag p g p^\dag p =g p^\dag p$.
  Since $g$ is self-adjoint, the left hand side of this equation is self-adjoint. Therefore also the right-hand side is self-adjoint. Thus $p^\dag p g  =g p^\dag p$.
\end{proof}

The next theorem roughly shows that ``polar decomposition turns based limits into dagger limits''.

\begin{theorem}\label{thm:polarlimitsaredaggerlimits}
  Let $\Omega$ be a basis of \cat{J}. Assume that $D\colon \cat{J}\to\cat{C}$ has a limit $l_A \colon L \to D(A)$ satisfying $l_A^\dag l_A l_B^\dag l_B=l_B^\dag l_Bl_A^\dag l_A$ for all $A,B\in \Omega$. If \cat{C} admits polar decomposition, $(D,\Omega)$ has a dagger limit.
\end{theorem}
\begin{proof} 
  Pick a polar decomposition $l_A = p_Ai_A=j_Ap_A$ for each $A \in \Omega$. For $B\in \cat{J}\setminus{\Omega}$, set $p_B=D(f)p_A$, where $A$ is the unique object in $\Omega$ with $\cat{J}(A,B)\neq\emptyset$. If $f,g\colon A\to B$, then 
  \[
    D(f)p_Ai_A=D(f)l_A=D(g) l_A=D(g)p_Ai_A
  \]
  and hence $D(f)p_A=D(g)p_A$. So $p_B$ is independent of the choice of $f\colon A\to B$, and $p_A \colon L \to D(A)$ forms a cone. By construction, each $p_A$ is a partial isometry whenever $A\in \Omega$. Moreover, by assumption and Lemma~\ref{lem:PiOfPolarCommutes}, $p_A^\dag p_A$ commutes with $l_B^\dag l_B$ when $A,B \in \Omega$. Then, by another application of the lemma $p_A^\dag p_A$ commutes with $p_B^\dag p_B$. 

  It remains to show that $p_A \colon L \to D(A)$ forms a limiting cone. We will establish this by proving that the unique map $f \colon L \to L$ of cones from $p_A$ to $l_A$ is an isomorphism. Thus we need to find a map $g \colon L \to L$ from $l_A$ to $p_A$ that is the inverse of $f$. That $g$ is a map of cones means that the triangle 
  \[\begin{tikzpicture}
    \matrix (m) [matrix of math nodes,row sep=2em,column sep=4em,minimum width=2em]
    {
     L & D(A) \\
     L  \\};
    \path[->]
    (m-1-1) edge node [left] {$g$} (m-2-1)
           edge node [above] {$l_A$} (m-1-2)
    (m-2-1) edge[dagger] node [below] {$p_A$} (m-1-2);
  \end{tikzpicture}\]
  commutes for each $A$ in $\cat{J}$, or equivalently for each $A\in \Omega$. Postcomposing with the bimorphisms $j_A$ we see that this is equivalent to finding $g\colon L\to L$ such that 
  \[\begin{tikzpicture}
    \matrix (m) [matrix of math nodes,row sep=2em,column sep=4em,minimum width=2em]
    {
     L& D(A) \\
     L & D(A) \\};
    \path[->]
    (m-1-1) edge node [left] {$g$} (m-2-1)
           edge node [above] {$l_A$} (m-1-2)
    (m-1-2) edge node [right] {$j_A$} (m-2-2)
    (m-2-1) edge node [below] {$l_A$} (m-2-2);       
  \end{tikzpicture}\]
  for each $A\in \Omega$. As $l_A \colon L \to D(A)$ is a limit cone, the existence of such a $g$ follows as soon as $j_A l_A \colon L \to D(A)$ with $A \in \Omega$ generates a cone. But
  \[
    j_A l_A=j_Ap_Ai_A= l_Ai_A
  \]
  and $l_Ai_A \colon L \to D(A)$ with $A \in \Omega$ obviously generates a cone. Thus we have found a cone map $g \colon L \to L$ from $l_A$ to $p_A$. It suffices to show that it is the inverse of $f$. On the one hand $fg=\id[L]$ by the universal property of the cone $l_A$. On the other hand, $p_A g f = p_A$ 
  for each $A\in \Omega$, and by postcomposing with $j_A$ we see that $gf$ is also a cone map from $l_A$ to $l_A$, and thus equal to the identity.
\end{proof}

In particular, if \cat{C} admits polar decomposition and $A_1$ and $A_2$ have a product $(A_1\times A_2,p_1,p_2)$ satisfying $p_1^\dag p_1 p_2^\dag p_2=p_2^\dag p_2 p_1^\dag p_1$ then the dagger product of $A_1$ and $A_2$ exists as well. Using polar decomposition and splittings of projections, one can also construct the dagger product of $A$ and $B$ from their biproduct, without any further conditions required from the biproduct. 

Recall that, in a category with a zero object (or more generally, zero morphisms), a biproduct of $A_1$ and $A_2$ consists of $(A_1\oplus A_2,p_1,p_2,i_1,i_2)$ where $(A_1\oplus A_2,p_1,p_1)$ is a product of $A_1$ and $A_2$, $(A_1\oplus A_2,i_1,i_2)$ is their coproduct, and moreover
\begin{align*}
  p_1i_1&=\id[A_1] \quad &p_2i_2=\id[2] \\
  p_2i_1&=0_{A_1,A_2} \quad &p_1i_2=0_{A_2,A_1}
\end{align*}
or more succinctly, 
  \[p_ni_k=\delta_{k,n}:=\begin{cases}\id[A_n]\text{ if }n=k \\
                                       0_{A_k,A_n}\text{ otherwise.} \end{cases}\]

In Section~\ref{sec:biprods} we will see how to generalize this definition, allowing us to work without assuming zero morphisms. The theorem below remains true for such generalized biproducts, as explained in Remark~\ref{rem:polarofbiprod}.

\begin{theorem}\label{thm:polarofbiprod} Let \cat{C} be a dagger category that admits polar decomposition and has dagger splittings of projections. If two objects $A_1$ and $A_2$ of \cat{C} have a biproduct, they have a dagger product as well.
\end{theorem}

\begin{proof}
  Given two objects $A_1$ and $A_2$, take their biproduct $(A_1\oplus A_2,p_1,p_2,i_1,i_2)$.
  Assume first that we've produced a cone $q_i\colon P\to A_i$ such that $q_nq^\dag_k=\delta_{n,k}$.  As $(A_1\oplus A_2, i_1,i_2)$ is a coproduct and $(A_1\oplus A_2,p_1,p_2)$ is a product we can find unique maps $f\colon A_1\oplus A_2\to P$ and $g\colon P\to A_1\oplus A_2$ making the diagram 
  \[\begin{tikzpicture}
     \matrix (m) [matrix of math nodes,row sep=2em,column sep=4em,minimum width=2em]
     {
      A_k & & A_n \\
      A_1\oplus A_2 & P & A_1\oplus A_2 \\};
     \path[->]
     (m-1-1) edge node [left] {$i_k$} (m-2-1)
             edge node [above] {$q_k^\dag$} (m-2-2)
            edge node [above] {$\delta_{n,k}$} (m-1-3)
      (m-2-1) edge node [below] {$f$} (m-2-2) 
      (m-2-2) edge node [above] {$q_n$} (m-1-3)
              edge node [below] {$g$} (m-2-3)
      (m-2-3) edge node [right] {$p_n$} (m-1-3);
  \end{tikzpicture}\]
commute. By the universal properties of $A_1\oplus A_2$ this implies that $gf=\id$. Now, $f$ is a map of cocones $(A_1\oplus A_2,i_1,i_2)\to (P,q_1^\dag,q_2^\dag)$. We will prove that it is also a map of cones $(A_1\oplus A_2,p_1,p_2)\to (P,q_1,q_2)$. Consider the diagram 
        \[\begin{tikzpicture}
     \matrix (m) [matrix of math nodes,row sep=2em,column sep=4em,minimum width=2em]
     {
      A_1\oplus A_2 &  & P \\
      A_k &  A_1\oplus A_2  & A_n \\};
     \path[->]
     (m-2-1) edge node [left] {$i_k$} (m-1-1)
            edge node [below] {$i_k$} (m-2-2)
            edge node [above] {$q_k^\dag$} (m-1-3)
     (m-1-1) edge node [above] {$f$} (m-1-3)
     (m-1-3) edge node [right] {$q_n$} (m-2-3)
     (m-2-2) edge node [below] {$p_n$} (m-2-3);
    \end{tikzpicture}\]
The bottom part commutes since both paths equal $\delta_{n,k}$ and the top part commutes by definition of $f$. Hence the rectangle commutes, and since $i_1$ and $i_2$ are jointly epic this shows that $f$ is a map of cones. Now, $g^\dag$ is a map of cocones $(A_1,A_2,p_1^\dag,p_2^\dag)\to (P,q_1^\dag,q_2)$ so the same argument shows that $g^\dag$ is also a map of cones. Taking the dagger again, this means that $g$ is not only a map of cones $(P,q_1,q_2)\to (A_1\oplus A_2,p_1,p_2)$ but also a map of cocones $(P,q_1^\dag,q_2^\dag)\to (A_1\oplus A_2,i_1,i_2)$. These observations are true for any cone $P$ satisfying $q_nq^\dag_k=\delta_{n,k}$. As any such cone satisfies normalization and independence, it is sufficient to find a cone $P$ for which we can prove that $fg=\id[P]$ since then $(P,q_1,q_2)$ will also be a product. Finding such a $P$ is what we'll do in the remainder.

  Factorize $p_1$ as $p_1=pi=jp$ and $i_2$ as $i_2=kr=rl$. Set $d_1:=p$ and $d_2:=r^\dag$. We claim that
    \[d_nd_k^\dag=\delta_{n,k}\]
  Indeed, since $p_1$ and $i_2^\dag$ are epimorphisms the maps $d_1$ and $d_2$ are dagger epic. Moreover, $0=p_2i_1=j d_1d_2^\dag l=j0l$ whence $d_1d_2^\dag=0$ and thus $d_2d_1^\dag=0$. Hence $(A_1\oplus A_2,d_1,d_2)$ is a cone satisfying (i), so as remarked at the end of the first half of the proof, we get maps 
    \[f\colon (A_1\oplus A_2,p_1,p_2,i_1,i_2)\leftrightarrows (A_1\oplus A_2,d_1,d_2,d_1^\dag,d_2^\dag)\colon g\] 
  that are maps of cones and cocones. Taking daggers, we have maps 
    \[g^\dag\colon (A_1\oplus A_2,i_1^\dag,i_2^\dag,p_1^\dag,p_2^\dag)\leftrightarrows (A_1,A_2,d_1,d_2,d_1^\dag,d_2^\dag)\colon f^\dag\]
  that are also compatible with the (co)cone structures. Moreover, since $(A_1\oplus A_2,p_1,p_2,i_1,i_2)$ and $(A_1\oplus A_2,i_1^\dag,i_2^\dag,p_1^\dag,p_2^\dag)$ are biproducts, there is exactly one (co)cone map in both directions, and these are isomorphisms. Since 
    \[f^\dag f\colon (A_1\oplus A_2,p_1,p_2,i_1,i_2)\leftrightarrows(A_1\oplus A_2,i_1^\dag,i_2^\dag,p_1^\dag,p_2^\dag) \colon gg^\dag\]
  are (co)cone maps, these maps have to be inverses to each other. 
  Hence $h:=fg g^\dag f^\dag$ is a self-adjoint cone map $(A_1\oplus A_2,d_1,d_2)\to (A_1\oplus A_2,d_1,d_2)$ satisfying
    \[h^2=(fg g^\dag f^\dag)(fg g^\dag f^\dag)=f(g g^\dag f^\dag f)g g^\dag f^\dag=f g g^\dag f^\dag=h\]
  Thus $h$ is a projection; let $e\colon P\to (A_1\oplus A_2)$ split it and set $q_i:=d_ie$. Now  \[q_nq_k^\dag=d_n ee^\dag d_k^\dag=d_n h d_k^\dag=d_nd_k=\delta_{n,k}\]
  so we have canonical maps 
      \[\tilde{f}\colon (A_1\oplus A_2,p_1,p_2,i_1,i_2)\leftrightarrows (P,q_1,q_2,q_1^\dag,q_2^\dag)\colon \tilde{g}\] 
  and $\tilde{g}\circ\tilde{f}=\id[A_1\oplus A_2]$ holds automatically, whereas whether $\tilde{f}\circ \tilde{g}=\id[P]$ is at stake. By construction $e$ is a map of cones and hence $e^\dag$ is a map of cocones. Hence the bottom path in region (i) of
    
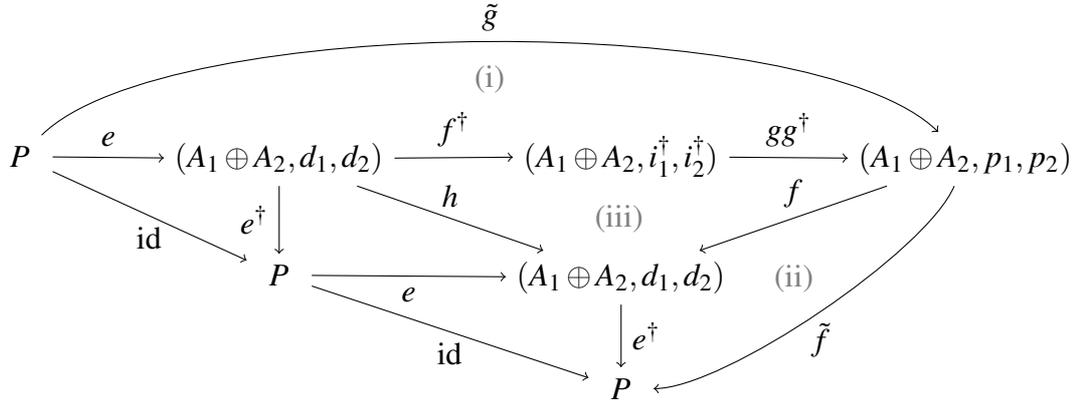
\[\begin{tikzpicture}
     \matrix (m) [matrix of math nodes,row sep=2em,column sep=3.5em,minimum width=2em]
     {
      P & (A_1\oplus A_2,d_1,d_2) & (A_1\oplus A_2,i_1^\dag,i_2^\dag) & (A_1\oplus A_2,p_1,p_2) \\
       &  P  & (A_1\oplus A_2,d_1,d_2) \\
       && P \\};
     \path[->]
     (m-1-1) edge node [above] {$e$} (m-1-2)
             edge node [below] {$\id$} (m-2-2)
             edge[out=45,in=135,looseness=.5]  node [above] {$\tilde{g}$} (m-1-4)
     (m-1-2) edge node [above] {$f^\dag$} (m-1-3)
             edge node [left] {$e^\dag$} (m-2-2)
             edge node [above] {$h$} (m-2-3)
     (m-1-3) edge node [above] {$gg^\dag$} (m-1-4)
     (m-1-4) edge[out=-112.5,in=0,looseness=.5]  node[below] {$\tilde{f}$} (m-3-3) 
             edge node [above] {$f$} (m-2-3)
     (m-2-2) edge node [below] {$e$} (m-2-3)
             edge node [below] {$\id$} (m-3-3)
      (m-2-3) edge node [right] {$e^\dag$} (m-3-3);
      \node[gray] at (-0.6,2.5) {(i)};
      \node[gray] at (3.4,-0.15) {(ii)};
      \node[gray] at (1.1,.65) {(iii)};
    \end{tikzpicture}\]
 is a map of cones and thus equal to $\tilde{g}$, establishing commutativity of region (i). Similarly, $e^\dag f$ is a composite of two maps of cocones and hence equal to $\tilde{f}$, so that region (ii) commutes. Region (iii) commutes by definition of $h$ and the remaining triangles on the left commute by definition of $e$. Hence the whole diagram commutes, showing that $\tilde{f}\circ \tilde{g}=\id[P]$, as desired.
\end{proof}

\begin{corollary} 
  A dagger category admitting polar decomposition has finitely based dagger limits if and only if its underlying category has equalizers, intersections and 
  either (i) finite biproducts or (ii) finite products $(A_1\dots A_,p_1,\dots p_n)$ satisfying $p^\dag_i p_ip^\dag_j p_j=p^\dag_j p_jp^\dag_i p_i$. 
\end{corollary}
\begin{proof}
  Case (i) follows from Example~\ref{ex:polarlimitsaredaggerlimits} and Theorems~\ref{thm:finitecompleteness} and \ref{thm:polarofbiprod}, whereas case (ii) follows directly from Theorems~\ref{thm:finitecompleteness} and~\ref{thm:polarlimitsaredaggerlimits}.
\end{proof}

What made the previous results work for based diagrams is that we did not need to worry about the path taken to an object. However, if $A,B\in\Omega$ are distinct and admit maps $f\colon A\to C\leftarrow B\colon g$, the equation $D(f)l_A=l_C=D(g)l_B$ need not imply that $D(f)p_A=D(g)p_B$. Hence forgetting about the bimorphisms is not possible for arbitrary diagrams, and to make it work one has to change the diagram so that joint polar decomposition becomes available. The following theorem makes precise this idea that ``polar decomposition turns limits into dagger limits of isomorphic diagrams''. Recall that a category is balanced if all bimorphisms in it are isomorphisms.

\begin{theorem}\label{thm:daglimsofisomorphicdiagrams} 
  Consider a diagram  $D\colon \cat{J}\to\cat{C}$ in a balanced dagger category $\cat{C}$ that admits polar decomposition.
  Suppose that the diagram has a limit $l_A \colon L \to D(A)$ such that $l_A^\dag l_A l_B^\dag l_B=l_B^\dag l_Bl_A^\dag l_A$ for all $A$ and $B$ in some weakly initial $\Omega$.
  There is a diagram $E$ naturally isomorphic to $D$ such that $(E,\Omega)$ has a dagger limit.
\end{theorem}
\begin{proof}
  Pick a polar decomposition $l_A=p_Ai_A=j_Ap_A$ for each $A \in \cat{J}$. As \cat{C} is balanced, each $i_A$ and $j_A$ is an isomorphism, so we can define a new diagram $E\colon \cat{J}\to\cat{C}$ by $E(A)=D(A)$ on objects, and by $E(f)=j_B^{-1}D(f)j_A$ on morphisms $f\colon A\to B$. By construction $j_A \colon E(A) \to D(A)$ forms a natural isomorphism $E\Rightarrow D$. Moreover, since $l_A=j_Ap_A$, in fact $L$ is a limit of $E$ with a limiting cone $p_A \colon L \to E(A)$ of partial isometries.
  It remains to check independence. This is done as in the proof of Theorem~\ref{thm:polarlimitsaredaggerlimits}. By assumption and Lemma \ref{lem:PiOfPolarCommutes}, $p_A^\dag p_A$ commutes with $l_B^\dag l_B$ when $A,B \in \Omega$. Then, by another application of the lemma $p_A^\dag p_A$ commutes with $p_B^\dag p_B$. 
\end{proof}

\section{Commutativity of limits and colimits}\label{sec:commutativity}

In this section we investigate to what extent dagger limits commute with dagger (co)limits.


Let us start by looking at whether dagger limits commute with dagger limits.
Assume that \cat{C} has all $(\cat{J},\Omega)$-shaped dagger limits and all $(\cat{K},\Psi)$-shaped dagger limits. Then, a bifunctor $D\colon\cat{J}\times\cat{K}\to \cat{C}$ induces functors $\cat{J} \to \cat{C}$ and $\cat{K} \to \cat{C}$ defined by $j\mapsto \dlim^\Psi_k D(j,k)$ and $k\mapsto \dlim^\Omega_j D(j,k)$. Since limits commute with limits, there exists a canonical isomorphism $\dlim^\Omega_j \dlim^\Psi_k D(j,k) \simeq \dlim^\Psi_k\dlim^\Omega_j D(j,k)$ between the two limits of $D$. In keeping with dagger category theory, we would like this canonical isomorphism to be unitary. Moreover, we would like both sides to be dagger limits of $(D,\Omega\times\Psi)$. We now prove that this holds whenever \cat{J} and \cat{K} are based and \cat{C} has zero morphisms. Later we will relax this to arbitrary \cat{J} and \cat{K} under a technical condition on the bifunctor $D$, that we conjecture is in fact not necessary.

\begin{theorem}\label{thm:limscommutewithlims} 
  Let $\Omega$ and $\Psi$ be bases for \cat{J} and \cat{K}. Assume that \cat{C} has zero morphisms, all $(\cat{J},\Omega)$-shaped dagger limits, all $(\cat{K},\Psi)$-shaped dagger limits, and let $D\colon \cat{J}\times\cat{K}\to\cat{C}$ be a bifunctor. Then $\dlim^\Psi_k\dlim^\Omega_j D(j,k)$ and $\dlim^\Omega_j\dlim^\Psi_k D(j,k)$ are both dagger limits of $(D,\Omega\times\Psi)$ and hence unitarily isomorphic.
\end{theorem}
\begin{proof} 
  By symmetry, it suffices to prove the claim for $\dlim^\Psi_k\dlim^\Omega_j D(j,k)$. Since ordinary limits commute with limits, $\dlim^\Psi_k\dlim^\Omega_j D(j,k)$ is a limit of $D$. Hence we need only check normalization and independence for $(j,k)\in\Omega\times\Psi$. In the presence of zero morphisms, we can use the simpler description from Proposition~\ref{prop:limswithzeros}.

  For normalization, we start by showing that, for fixed $(j,k)\in\Omega\times \Psi$, the morphism $p_k$ defined as the composition of canonical morphisms
  \[
    \dlim^\Omega_j D(j,k)\to\dlim^\Psi_k\dlim^\Omega_j D(j,k)\to\dlim^\Omega_j D(j,k)\to D(j,k)\to \dlim^\Omega_j D(j,k)
  \]
  factors through the canonical morphism $\dlim^\Psi_k\dlim^\Omega_j D(j,k)\to \dlim^\Omega_j D(j,k)$. This is done by extending it to a cone on $\dlim^\Omega_j D(j,k)$ for the functor $\dlim^\Omega_j D(j,-)\colon\cat{K}\to\cat{C}$. For $h$ in $\cat{K}$, define $p_h\colon \dlim^\Omega_j D(j,k)\to \dlim^\Omega_j D(j,h)$ by 
  \[
    p_h = \begin{cases}
      \dlim^\Omega_j D(j,f)p_k & \text{if }\cat{K}(k,h)\neq\emptyset\text, \\
      0 & \text{otherwise.}
    \end{cases}
  \]
  To see that this defines a cone, it suffices to check that $p_h$ is independent of the choice of $f\colon k\to h$. By the universal property of $\dlim^\Omega_j D(j,h)$, we may postcompose with a projection to an arbitrary $D(i,h)$ with $i\in\Omega$ and show that the resulting morphism does not depend on the choice of $f$. This splits into two cases depending on whether $i\neq j$ or not. If $i\neq j$, then the end result is always zero, since the following diagram commutes for every $f$:
  \[\begin{tikzpicture}
    \matrix (m) [matrix of math nodes,row sep=2em,column sep=6em,minimum width=2em]
    {
     &\dlim^\Omega_j D(j,k)& \dlim^\Omega_j D(j,h) \\
    D(j,k) & D(i,k) & D(i,h) \\
    & &  D(i,h) \\};
    \path[->]
    (m-1-2) edge node [left] {$$} (m-2-2)
           edge node [above] {$\dlim^\Omega_j D(j,f)$} (m-1-3)
    (m-1-3) edge node [right] {$$} (m-2-3)
    (m-2-2) edge node [below] {$D(i,f)$} (m-2-3)
    (m-2-3) edge node [right] {$\id$} (m-3-3)
    (m-2-1) edge[out=90,in=180,looseness=.5] node [above] {$$} (m-1-2)
            edge node [above] {$0$} (m-2-2)
            edge[out=-90,in=180,looseness=.4] node [below] {$0$} (m-3-3);       
  \end{tikzpicture}\]
  In case $i=j$, we will prove that the following diagram commutes:
  \[\begin{tikzpicture}
    \matrix (m) [matrix of math nodes,row sep=2em,column sep=2em,minimum width=2em]
    {
      \dlim^\Omega_j D(j,k) & \dlim^\Psi_k \dlim^\Omega_j D(j,k) & \dlim^\Omega_j D(j,k) & D(j,k)\\
     && D(j,k) & \dlim^\Omega_j D(j,k) \\
     & \dlim^\Omega_j D(j,h) &D(j,h)&  \dlim^\Omega_j D(j,h)\\};
    \path[->]
    (m-1-1) edge node [above] {$$} (m-1-2)
    (m-1-2) edge node [above] {$$} (m-1-3)
    (m-1-3) edge node [above] {$$} (m-1-4)
            edge[out=-160,in=90] node [left=2mm] {$\dlim^\Omega_j D(j,f)$} (m-3-2)
            edge node [left] {$$} (m-2-3)
    (m-1-4) edge node [above] {$$} (m-2-4)
    (m-3-2) edge node [below] {$$} (m-3-3)
    (m-2-3) edge node [right] {$D(j,f)$} (m-3-3)
    (m-2-4) edge node [below] {$$} (m-2-3)
             edge node [right] {$\dlim^\Omega_j D(j,f)$} (m-3-4)
    (m-3-4) edge node [below] {$$} (m-3-3);
  \end{tikzpicture}\]
  The top right square commutes because $j\in\Omega$, and the rest of the diagram commutes by definition of $ \dlim^\Omega_j D(j,f)$. The path along the top is $p_h$ followed by a projection to $D(j,h)$, whereas the other path is independent of the choice of $f\colon k\to h$ since $\dlim^\Psi_k \dlim^\Omega_j D(j,k)$ is a cone for $ \dlim^\Omega_j D(j,-)$. Thus \[p_h \colon \dlim^\Omega_j D(j,k) \to \dlim^\Omega_j D(j,h)\] forms a cone. Hence it factors through $\dlim^\Psi_k\dlim^\Omega_j D(j,k)$. In particular, $p_k$ factors through $\dlim^\Psi_k\dlim^\Omega_j D(j,k)\to \dlim^\Omega_j D(j,k)$. By Remark~\ref{rem:pisetc} this implies that the following diagram of canonical morphisms commutes:
    \[\begin{tikzpicture}[font=\footnotesize]
       \matrix (m) [matrix of math nodes,row sep=2em,column sep=1.25em,minimum width=2em]
       {
        \dlim^\Omega_j D(j,k) & \dlim^\Psi_k\dlim^\Omega_j D(j,k) & \dlim^\Omega_j D(j,k) & D(j,k) & \dlim^\Omega_j D(j,k)\\
        \dlim^\Psi_k\dlim^\Omega_j D(j,k)& \dlim^\Omega_j D(j,k) & D(j,k) & \dlim^\Omega_j D(j,k) & \dlim^\Psi_k\dlim^\Omega_j D(j,k) \\};
       \path[->]
       (m-1-1) edge node [above] {$$} (m-1-2)
               edge node [left] {$$} (m-2-1)
       (m-1-2) edge node [above] {$$} (m-1-3)
       (m-1-3) edge node [above] {$$} (m-1-4)
       (m-1-4) edge node [above] {$$} (m-1-5)
       (m-2-1) edge node [below] {$$} (m-2-2)
       (m-2-2) edge node [below] {$$} (m-2-3)
       (m-2-3) edge node [below] {$$} (m-2-4)
      (m-2-4) edge node [below] {$$} (m-2-5)
      (m-2-5) edge node [right] {$$} (m-1-5);
    \end{tikzpicture}\]
    As the path along the bottom is self-adjoint, so is the top path. So the projections \[\dlim^\Omega_j D(j,k)\to \dlim^\Psi_k\dlim^\Omega_j D(j,k)\to \dlim^\Omega_j D(j,k)\] and \[\dlim^\Omega_j D(j,k)\to D(j,k)\to \dlim^\Omega_j D(j,k)\] commute. Finally, Remark~\ref{rem:pisetc} guarantees that the composition of canonical morphisms $\dlim^\Psi_k\dlim^\Omega_j D(j,k)\to \dlim^\Omega_j D(j,k)\to D(j,k)$ is a partial isometry. 

    For independence, it suffices to show for $(j,k),(i,h)\in\Omega\times\Psi$ that the composition $D(j,k)\to \dlim^\Omega_j D(j,k)\to \dlim^\Psi_k\dlim^\Omega_j D(j,k)\to \dlim^\Omega_j D(j,h)\to D(i,h)$ of canonical morphisms is zero when $(j,k)\neq (i,h)$. If $k\neq h$ this follows from the fact that $\dlim^\Omega_j D(j,k)\to \dlim^\Psi_k\dlim^\Omega_j D(j,k)\to \dlim^\Omega_j D(j,h)$ is zero. Hence we may assume that $h=k$ and consider the case $i\neq j$. As above, the projections  \[\dlim^\Omega_j D(j,k)\to \dlim^\Psi_k\dlim^\Omega_j D(j,k)\to \dlim^\Omega_j D(j,k)\] and \[\dlim^\Omega_j D(j,k)\to D(j,k)\to \dlim^\Omega_j D(j,k)\] commute. Hence the following diagram of canonical morphisms commutes:
    \[\begin{tikzpicture}
      \matrix (m) [matrix of math nodes,row sep=2em,column sep=3em,minimum width=2em]
      {
       D(j,k)& \dlim^\Omega_j D(j,k) & \dlim^\Psi_k\dlim^\Omega_j D(j,k) & \dlim^\Omega_j D(j,k) \\
       \dlim^\Omega_j D(j,k) & D(j,k)& D(j,k)& D(i,k) \\
       \dlim^\Psi_k\dlim^\Omega_j D(j,k) & \dlim^\Omega_j D(j,k) \\};
      \path[->]
      (m-1-1) edge node [left] {$$} (m-2-1)
             edge node [above] {$$} (m-1-2)
      (m-1-2) edge node [above] {$$} (m-1-3)
      (m-1-3) edge node [above] {$$} (m-1-4)
      (m-1-4) edge node [right] {$$} (m-2-4)
      (m-2-1) edge node [below] {$$} (m-2-2)
              edge node [left] {$$} (m-3-1)
      (m-2-2) edge node [right] {$$} (m-1-2)
      (m-2-3) edge node [below] {$0$} (m-2-4)
              edge node [above] {$$} (m-1-4)
      (m-3-1) edge node [below] {$$} (m-3-2)
      (m-3-2) edge node [below] {$$} (m-2-3);        
    \end{tikzpicture}\]
    This concludes the proof.
\end{proof}

Next, consider whether dagger limits commute with dagger colimits.
If \cat{C} has all $(\cat{J},\Omega)$-shaped dagger limits and all $(\cat{K},\Psi)$-shaped dagger \emph{co}limits (\ie $(\cat{K}\op,\Psi)$-shaped dagger limits), it is natural to ask when the canonical morphism 
\[\dcolim^\Psi_k\dlim^\Omega_j D(j,k)\to\dlim^\Omega_j\dcolim^\Psi_k D(j,k)\]
 is unitary. This canonical morphism $\tau$ and morphisms $\alpha_k$ are defined by making the following diagram commute for each $k$: 
\[\begin{tikzpicture}
  \matrix (m) [matrix of math nodes,row sep=2em,column sep=4em,minimum width=2em]
  {
   \dcolim^\Psi_k\dlim^\Omega_j D(j,k) &  \dlim^\Omega_j D(j,k)  & D(j,k) \\
   & \dlim^\Omega_j\dcolim^\Psi_k D(j,k)  & \dcolim^\Psi_k D(j,k) \\};
  \path[->]
  (m-1-1) edge[dashed] node [below] {$\tau$} (m-2-2)
  (m-1-2) edge node [above] {$$} (m-1-1)
  (m-1-2) edge node [above] {$$} (m-1-3)
  (m-1-3) edge node [right] {$$} (m-2-3)
  (m-1-2) edge[dashed] node [right] {$\alpha_k$} (m-2-2)
  (m-2-2) edge node [below] {$$} (m-2-3);
\end{tikzpicture}\]
A priori one might hope $\tau$ to be unitary very generally. After all, $(\cat{K},\Psi)$-shaped dagger colimits are just $(\cat{K}\op,\Psi)$-shaped dagger limits, so one might expect that commutativity of limits with colimits boils down to commutativity of limits with limits. To be slightly more precise, given $D\colon \cat{J}\times\cat{K}\to\cat{C}$, one would like to define $\hat{D}\colon \cat{J}\times\cat{K}\op\to\cat{C}$ by ``applying the dagger to the second variable'' and then calculating as follows:
\begin{align*}
  &\dcolim^\Psi_k\dlim^\Omega_j D(j,k)=\dlim^\Psi_k\dlim^\Omega_j \hat{D}(j,k) \\
  &\simeq_\dag \dlim^\Omega_j\dlim^\Psi_k \hat{D}(j,k)=\dlim^\Omega_j\dcolim^\Psi_k D(j,k)
\end{align*}
This, however, is a trap: $\hat{D}$ is not guaranteed to be a bifunctor. Indeed, the formula 
$
  \hat{D}(f,g)=D(f,\id)D(g,\id)^\dag
$
defines a bifunctor $\cat{J}\times\cat{K}\op\to\cat{C}$ if and only if every morphism $(f,g) \colon (j,k)\to (i,h)$ in $\cat{J}\times\cat{K}$ makes the following diagram commute:
\[\begin{tikzpicture}
  \matrix (m) [matrix of math nodes,row sep=2em,column sep=7em]
  {
   D(j,h)& D(i,k) \\
   D(j,k) & D(i,k) \\};
  \path[->]
  (m-1-1) edge node [left] {$D(\id,g)^\dag$} (m-2-1)
         edge node [above] {$D(f,\id)$} (m-1-2)
  (m-1-2) edge node [right] {$D(\id,g)^\dag$} (m-2-2)
  (m-2-1) edge node [below] {$D(f,\id) $} (m-2-2);       
\end{tikzpicture}\]

\begin{definition}
  A bifunctor $D \colon \cat{J} \times \cat{K} \to \cat{C}$ into a dagger category $\cat{C}$ is \emph{adjointable}\index[word]{adjointable bifunctor} when $D(-,g) \colon D(-,k) \Rightarrow D(-,h)$ is an adjointable natural transformation for each $g \colon k \to h$ in $\cat{K}$.
  Equivalently, $D$ is adjointable when $D(f,-)$ is an adjointable natural transformation for each morphism $f$ in $\cat{J}$.
\end{definition}

Let us temporarily go back to considering whether dagger limits commute with dagger limits. The extra condition of adjointability of $D$ lets us prove that this is true for arbitrary shapes $\cat{J}$ and $\cat{K}$. We conjecture that the extra condition is in fact not needed.

\begin{theorem}\label{thm:limscommutewithlims2} 
  Assume that \cat{C} has all $(\cat{J},\Omega)$-shaped dagger limits, all $(\cat{K},\Psi)$-shaped dagger limits, and let $D\colon \cat{J}\times\cat{K}\to\cat{C}$ be an adjointable bifunctor. Then $\dlim^\Psi_k\dlim^\Omega_j D(j,k)$ and $\dlim^\Omega_j\dlim^\Psi_k D(j,k)$ are both dagger limits of $(D,\Omega\times\Psi)$ and hence unitarily isomorphic.
\end{theorem}
\begin{proof} 
  By symmetry, it suffices to prove the claim for $\dlim^\Psi_k\dlim^\Omega_j D(j,k)$. Since ordinary limits commute with limits, $\dlim^\Psi_k\dlim^\Omega_j D(j,k)$ is the limit of $D$. Hence we only need to check normalization and independence for $(j,k)\in\Omega\times\Psi$. Consider the following diagram of canonical morphisms for some morphism $f \colon k \to h$ in $\cat{K}$:
  \[\begin{tikzpicture}
    \matrix (m) [matrix of math nodes,row sep=1em,column sep=6em,minimum width=2em]
    {
    \dlim^\Omega_j D(j,k)& \dlim^\Omega_j D(j,h) \\
     D(j,k) & D(j,h) \\
     \dlim^\Omega_j D(j,k)& \dlim^\Omega_j D(j,h)\\};
    \path[->]
    (m-1-1) edge node [left] {$$} (m-2-1)
           edge node [above] {$\dlim^\Omega_j D(j,f)$} (m-1-2)
    (m-1-2) edge node [right] {$$} (m-2-2)
    (m-2-1) edge node [below] {$D(\id,f)$} (m-2-2)
            edge node [left] {$$} (m-3-1)
     (m-2-2) edge node [right] {$$} (m-3-2)
     (m-3-1) edge node [below] {$\dlim^\Omega_j D(j,f)$} (m-3-2);       
  \end{tikzpicture}\]
  The top square commutes by definition of $\dlim^\Omega_j D(j,f)$. Since $D(-,f)$ is adjointable, the bottom square commutes too by Corollary~\ref{cor:mapoflimsismapofcolims}. Hence the whole diagram commutes, and the family $\sigma_k\colon \dlim^\Omega_j D(j,k)\to D(j,k)\to \dlim^\Omega_j D(j,k)$ is natural in $k$. Since $\sigma\colon \dlim^\Omega_j D(j,-)\to \dlim^\Omega_j D(j,-)$ is natural and pointwise self-adjoint, it is an adjointable natural transformation. Lemma~\ref{lem:connectingmaps} now makes the following diagram of canonical morphisms commute for each $j \in \Omega$ and $h,k \in \Psi$:
  \begin{equation*}\label{diag:commutinglimswithlims}\tag{$\star\star$}
    \begin{aligned}\begin{tikzpicture}
         \matrix (m) [matrix of math nodes,row sep=2em,column sep=1.5em,minimum width=2em]
         {
          \dlim^\Omega_j D(j,k) & D(j,k) & \dlim^\Omega_j D(j,k) & \dlim^\Psi_k\dlim^\Omega_j D(j,k)\\
          \dlim^\Psi_k\dlim^\Omega_j D(j,k)& \dlim^\Omega_j D(j,h) & D(j,h) & \dlim^\Omega_j D(j,h) \\};
         \path[->]
         (m-1-1) edge node [above] {$$} (m-1-2)
                 edge node [left] {$$} (m-2-1)
         (m-1-2) edge node [above] {$$} (m-1-3)
         (m-1-3) edge node [above] {$$} (m-1-4)
         (m-1-4) edge node [above] {$$} (m-2-4)
         (m-2-1) edge node [below] {$$} (m-2-2)
         (m-2-2) edge node [below] {$$} (m-2-3)
         (m-2-3) edge node [below] {$$} (m-2-4);
    \end{tikzpicture}\end{aligned}
  \end{equation*}
  Choosing $k=h$ shows that the projections $\dlim^\Omega_j D(j,k)\to D(j,k)\to \dlim^\Omega_j D(j,k)$ and $\dlim^\Omega_j D(j,k)\to \dlim^\Psi_k\dlim^\Omega_j D(j,k)\to \dlim^\Omega_j D(j,k)$ commute. Remark~\ref{rem:pisetc} then guarantees that the composite $\dlim^\Psi_k\dlim^\Omega_j D(j,k)\to \dlim^\Omega_j D(j,k)\to D(j,k)$ of canonical morphisms is a partial isometry for each $(j,k)\in\Omega\times\Psi$, establishing normalization.

  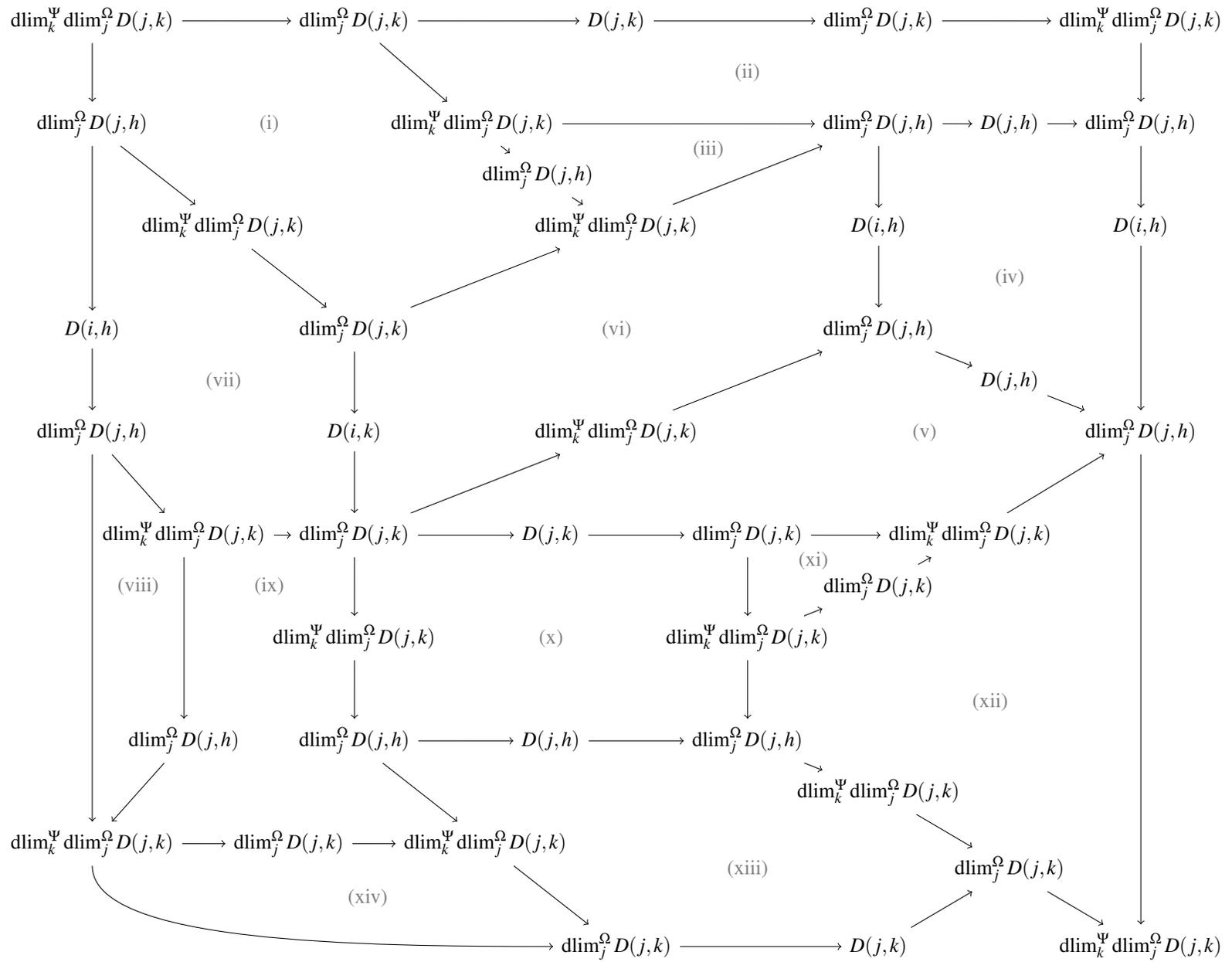
\begin{figure}\centering\begin{sideways}
      \begin{tikzpicture}[xscale=4.5, yscale=1.75, font=\footnotesize]
        \node (a1) at (1,9) {$\dlim^\Psi_k \dlim^\Omega_j D(j,k)$};
        \node (a2) at (2,9) {$\dlim^\Omega_j D(j,k)$};
        \node (a3) at (3,9) {$D(j,k)$};
        \node (a4) at (4,9) {$\dlim^\Omega_j D(j,k)$};
        \node (a5) at (5,9) {$\dlim^\Psi_k \dlim^\Omega_j D(j,k)$};
        \node (b1) at (1,8) {$\dlim^\Omega_j D(j,h)$};
        \node (b2) at (2.45,8) {$\dlim^\Psi_k \dlim^\Omega_j D(j,k)$};
        \node (b3) at (4,8) {$\dlim^\Omega_j D(j,h)$};
        \node (b4) at (4.5,8) {$D(j,h)$};
        \node (b5) at (5,8) {$\dlim^\Omega_j D(j,h)$};
        \node (c1) at (1.5,7) {$\dlim^\Psi_k \dlim^\Omega_j D(j,k)$};
        \node (c2) at (2.7,7.5) {$\dlim^\Omega_j D(j,h)$};
        \node (c3) at (3,7) {$\dlim^\Psi_k \dlim^\Omega_j D(j,k)$};
        \node (c4) at (4,7) {$D(i,h)$};
        \node (c5) at (5,7) {$D(i,h)$};
        \node (d1) at (1,6) {$D(i,h)$};
        \node (d2) at (2,6) {$\dlim^\Omega_j D(j,k)$};
        \node (d3) at (3,5) {$\dlim^\Psi_k \dlim^\Omega_j D(j,k)$};
        \node (d4) at (4,6) {$\dlim^\Omega_j D(j,h)$};
        \node (e1) at (1,5) {$\dlim^\Omega_j D(j,h)$};
        \node (e2) at (2,5) {$D(i,k)$};
        \node (e3) at (2.75,4) {$D(j,k)$};
        \node (e4) at (4.5,5.5) {$D(j,h)$};
        \node (e5) at (5,5) {$\dlim^\Omega_j D(j,h)$};
        \node (f2) at (2,4) {$\dlim^\Omega_j D(j,k)$};
        \node (f3) at (3.5,4) {$\dlim^\Omega_j D(j,k)$};
        \node (f4) at (4.35,4) {$\dlim^\Psi_k\dlim^\Omega_j D(j,k)$};
        \node (g1) at (1.35,4) {$\dlim^\Psi_k \dlim^\Omega_j D(j,k)$};
        \node (g2) at (2,3) {$\dlim^\Psi_k \dlim^\Omega_j D(j,k)$};
        \node (g3) at (3.5,3) {$\dlim^\Psi_k \dlim^\Omega_j D(j,k)$};
        \node (g4) at (4,3.5) {$\dlim^\Omega_j D(j,k)$};
        \node (h1) at (1.35,2) {$\dlim^\Omega_j D(j,h)$};
        \node (h2) at (2,2) {$\dlim^\Omega_j D(j,h)$};
        \node (h3) at (3.5,2) {$\dlim^\Omega_j D(j,h)$};
        \node (h4) at (4,1.5) {$\dlim^\Psi_k \dlim^\Omega_j D(j,k)$};
        \node (i1) at (1,1) {$\dlim^\Psi_k \dlim^\Omega_j D(j,k)$};
        \node (i2) at (2.5,1) {$\dlim^\Psi_k \dlim^\Omega_j D(j,k)$};
        \node (i3) at (2.75,2) {$D(j,h)$};
        \node (i4) at (4.5,0.75) {$\dlim^\Omega_j D(j,k)$};
        \node (j2) at (1.75,1) {$\dlim^\Omega_j D(j,k)$};
        \node (j3) at (3,0) {$\dlim^\Omega_j D(j,k)$};
        \node (j4) at (4,0) {$D(j,k)$};
        \node (j5) at (5,0) {$\dlim^\Psi_k \dlim^\Omega_j D(j,k)$};
        \draw[->] (a1) to (a2);
        \draw[->] (a2) to (a3);
        \draw[->] (a3) to (a4);
        \draw[->] (a4) to (a5);
        \draw[->] (a1) to (b1);
        \draw[->] (a2) to (b2);
        \draw[->] (a5) to (b5);
        \draw[->] (b1) to (c1);
        \draw[->] (b2) to (c2);
        \draw[->] (b2) to (b3);
        \draw[->] (b3) to (b4);
        \draw[->] (b3) to (c4);
        \draw[->] (b4) to (b5);
        \draw[->] (b5) to (c5);
        \draw[->] (b1) to (d1);
        \draw[->] (c1) to (d2);
        \draw[->] (c2) to (c3);
        \draw[->] (c3) to (b3);
        \draw[->] (c4) to (d4);
        \draw[->] (c5) to (e5);
        \draw[->] (d1) to (e1);
        \draw[->] (d2) to (e2);
        \draw[->] (d2) to (c3);
        \draw[->] (d3) to (d4);
        \draw[->] (d4) to (e4);
        \draw[->] (e4) to (e5);
        \draw[->] (e1) to (g1);
        \draw[->] (e2) to (f2);
        \draw[->] (e3) to (f3);
        \draw[->] (e1) to (i1);
        \draw[->] (f2) to (d3);
        \draw[->] (f2) to (e3);
        \draw[->] (f2) to (g2);
        \draw[->] (f3) to (f4);
        \draw[->] (f3) to (g3);
        \draw[->] (f4) to (e5);
        \draw[->] (e5) to (j5);
        \draw[->] (g1) to (f2);
        \draw[->] (g3) to (g4);
        \draw[->] (g4) to (f4);
        \draw[->] (g1) to (h1);
        \draw[->] (g2) to (h2);
        \draw[->] (g3) to (h3);
        \draw[->] (h3) to (h4);
        \draw[->] (h1) to (i1);
        \draw[->] (h2) to (i2);
        \draw[->] (h2) to (i3);
        \draw[->] (i3) to (h3);
        \draw[->] (h4) to (i4);
        \draw[->] (i1) to (j2);
        \draw[->] (j2) to (i2);
        \draw[->] (i2) to (j3);
        \draw[->] (j3) to (j4);
        \draw[->] (j4) to (i4);
        \draw[->] (i4) to (j5);
        \draw[->] (i1) to[out=-90,in=180] (j3);
        \draw[gray,draw=none] (b1) to node{(i)} (b2);
        \draw[gray,draw=none] (a2) to node{(ii)} (b5);
        \draw[gray,draw=none] (c2) to node{(iii)} (b3);
        \draw[gray,draw=none] (b3) to node{(iv)} (e5);
        \draw[gray,draw=none] (d4) to node{(v)} (f4);
        \draw[gray,draw=none] (c3) to node{(vi)} (d3);
        \draw[gray,draw=none] (d1) to node{(vii)} (e2);
        \draw[gray,draw=none] (e1) to node{(viii)} (h1);
        \draw[gray,draw=none] (g1) to node{(ix)} (g2);
        \draw[gray,draw=none] (g2) to node{(x)} (g3);
        \draw[gray,draw=none] (f3) to node{(xi)} (g4);
        \draw[gray,draw=none] (f4) to node{(xii)} (i4);
        \draw[gray,draw=none] (h4) to node{(xiii)} (j3);
        \draw[gray,draw=none] (j3) to node{(xiv)} (i1);
      \end{tikzpicture}\end{sideways}
    \caption{Proof of independence in Theorem~\ref{thm:limscommutewithlims2}.}
    \label{fig:independence}
  \end{figure}

  For independence, pick $(j,k)$ and $(i,h)$ in  $\Omega\times\Psi$. We will show that the diagram in Figure~\ref{fig:independence} commutes. The fact that $k,h\in\Psi$ and $\dlim^\Psi_k\dlim^\Omega_j D(j,k)$ are normalized and independent at $\Psi$ ensures the commutativity of regions (i), (iii), (viii). (xi), (xi),(xii) and (xiv). Similarly, $i,j\in\Omega$ implies the commutativity of region (iv). The remaining regions, namely (ii), (v), (vi), (vii), (x) and (xiii), are all instances of diagram \eqref{diag:commutinglimswithlims}.
\end{proof}

We return to considering whether dagger limits commute with dagger colimits.
If $D$ is not adjointable, the colimit of limits need not be (unitarily) isomorphic to the limit of colimits.

\begin{example}\label{ex:limsneednotcommutewithcolims} 
  Dagger kernels need not commute with dagger cokernels in \cat{FHilb} if the bifunctor $D$ is not adjointable. For a counterexample, let $\cat{J}$ and $\cat{K}$ both be the shape $f,g \colon A \rightrightarrows B$ giving rise to equalizers.
  If $D\colon\cat{J}\times \cat{K}\to \cat{FHilb}$ maps each $D(f,-)$ and $D(-,f)$ to zero, then fixing the rest of $D$ corresponds to a choice of a commuting square in \cat{FHilb}. Let $D$ be thus defined by the square:
  \[\begin{tikzpicture}
    \matrix (m) [matrix of math nodes,row sep=2em,column sep=4em,minimum width=2em]
    {
     0 &\mathbb{C} &  \\
     \mathbb{C} &  \mathbb{C} \\};
    \path[->]
    (m-1-1) edge node [left] {$0$} (m-2-1)
           edge node [above] {$0$} (m-1-2)
    (m-1-2) edge node [right] {$\id $} (m-2-2)
    (m-2-1) edge node [below] {$\id $} (m-2-2);       
  \end{tikzpicture}\]
  Taking daggers of the horizontal arrows gives a square that does not commute, so $D$ is not adjointable. 
  Now, on the one hand, first taking cokernels horizontally and then taking kernels vertically gives $\mathbb{C}$: 
  \[\begin{tikzpicture}
    \matrix (m) [matrix of math nodes,row sep=2em,column sep=4em,minimum width=2em]
    { & & \ker 0=\mathbb{C} \\
     0 &\mathbb{C} & \coker 0=\mathbb{C} \\
     \mathbb{C} &  \mathbb{C} &\coker \id[\mathbb{C}]= 0\\};
    \path[->]
    (m-1-3) edge node [right] {$\id$} (m-2-3)
    (m-2-1) edge node [left] {$0$} (m-3-1)
           edge node [above] {$0$} (m-2-2)
    (m-2-2) edge node [right] {$\id $} (m-3-2)
            edge node [above] {$\id $} (m-2-3)
    (m-2-3) edge node [right] {$0$} (m-3-3)
    (m-3-1) edge node [below] {$\id $} (m-3-2)
    (m-3-2) edge node [below] {$0$} (m-3-3);       
  \end{tikzpicture}\]
  On the other hand, first taking kernels vertically and then taking cokernels horizontally gives $0$:
  \[\begin{tikzpicture}
         \matrix (m) [matrix of math nodes,row sep=2em,column sep=4em,minimum width=2em]
         {\ker 0=0 & \ker \id[\mathbb{C}]=0 &\coker 0=0 \\
          0 &\mathbb{C} \\
          \mathbb{C} &  \mathbb{C} \\};
         \path[->]
         (m-1-1) edge node [left] {$0$} (m-2-1)
                 edge node [above] {$0$} (m-1-2)
         (m-1-2) edge node [right] {$0$} (m-2-2)
                 edge node [above] {$0$} (m-1-3)
         (m-2-1) edge node [left] {$0$} (m-3-1)
                edge node [above] {$0$} (m-2-2)
         (m-2-2) edge node [right] {$\id $} (m-3-2)
         (m-3-1) edge node [below] {$\id $} (m-3-2);       
  \end{tikzpicture}\]
  Thus $\dlim_j\dcolim_k D(j,k)= \mathbb{C}$ is not isomorphic to $\dcolim_k \dlim_j D(j,k) = 0$, let alone unitarily so.
\end{example}

However, $D$ not being adjointable is the only obstruction to dagger limits commuting with dagger colimits.

\begin{theorem}\label{thm:limsandcolimscommute} 
  If \cat{C} has all $(\cat{J},\Omega)$-shaped dagger limits, all $(\cat{K},\Psi)$-shaped dagger colimits, and $D\colon \cat{J}\times\cat{K}\to\cat{C}$ is an adjointable bifunctor, then the canonical morphism 
  $\dcolim^\Psi_k\dlim^\Omega_j D(j,k)\to\dlim^\Omega_j\dcolim^\Psi_k D(j,k)$ is unitary.
\end{theorem}
\begin{proof} 
  Because $D$ is adjointable, $\hat{D}(f,g)=D(f,\id)D(g,\id)^\dag$ defines an adjointable bifunctor $\hat{D} \colon \cat{J}\times\cat{K}\op\to\cat{C}$. It follows from Theorem~\ref{thm:limscommutewithlims2} that
  \begin{align*}
    \dcolim^\Psi_k\dlim^\Omega_j D(j,k)=\dlim^\Psi_k\dlim^\Omega_j \hat{D}(j,k) \\
    \simeq_\dag \dlim^\Omega_j\dlim^\Psi_k \hat{D}(j,k)=\dlim^\Omega_j\dcolim^\Psi_k D(j,k)
  \end{align*}
  Thus there exists a unitary morphism $u$ making the following diagram commute:
  \[\begin{tikzpicture}
     \matrix (m) [matrix of math nodes,row sep=2em,column sep=4em,minimum width=2em]
     {
      \dcolim^\Psi_k\dlim^\Omega_j D(j,k) &&  \dlim^\Omega_j D(j,k) \\
      \dlim^\Omega_j\dcolim^\Psi_k D(j,k) & \dcolim^\Psi_k D(j,k) & \dcolim^\Psi_k D(j,k) \\};
     \path[->]
     (m-1-1) edge[dashed] node [left] {$u$} (m-2-1)
            edge node [above] {$$} (m-1-3)
     (m-1-3) edge node [above] {$$} (m-2-3)
     (m-2-1) edge node [below] {$$} (m-2-2)
     (m-2-2) edge node [below] {$$} (m-2-3);
  \end{tikzpicture}\]
  We will show that $u$ satisfies the commuting diagram that defines the canonical morphism $\tau \colon \dcolim^\Psi_k\dlim^\Omega_j D(j,k)\to\dlim^\Omega_j\dcolim^\Psi_k D(j,k)$, whence $u=\tau$. Postcomposing with the canonical morphism $\dlim^\Omega_j\dcolim^\Psi_k D(j,k)\to \dcolim^\Psi_k D(j,k)\to D(j,k)$, this follows from commutativity of the following diagram:
  \[\begin{tikzpicture}
    \matrix (m) [matrix of math nodes,row sep=2.5em,column sep=4em,minimum width=2em]
    {
     \dlim^\Omega_j\dcolim^\Psi_k D(j,k) && \dcolim^\Psi_k D(j,k) \\
     \dcolim^\Psi_k\dlim^\Omega_j D(j,k) & \dlim^\Omega_j D(j,k) & D(j,k)\\
     \dlim^\Omega_j D(j,k) &D(j,k) & \dcolim^\Psi_k D(j,k) \\
     & \dlim^\Omega_j\dcolim^\Psi_k D(j,k) \\};
    \path[->]
    (m-1-1) edge node [above] {$$} (m-1-3)
    (m-1-3) edge node [right] {$$} (m-2-3)
     (m-2-1) edge node [left] {$u$} (m-1-1)
            edge node [below] {$$} (m-2-2)
     (m-2-2) edge node [below] {$$} (m-2-3)
     (m-3-1) edge node [above] {$$} (m-3-2)
           edge node [left] {$$} (m-2-1)
           edge node [below] {$\alpha_k$} (m-4-2)
      (m-3-2) edge node [above] {$$} (m-3-3)
      (m-3-3) edge node [right] {$$} (m-2-3)
      (m-4-2) edge node [below] {$$} (m-3-3);
  \end{tikzpicture}\]
  After all, the top part and bottom parts commute by definition of $u$ and $\alpha_k$. The commutativity of the middle rectangle is guaranteed by~Lemma \ref{lem:connectingmaps} once we show that the limiting cone morphisms $\dlim^\Omega_j D(j,k)\to D(j,k)$ are the components of an adjointable natural transformation $\dlim^\Omega_j D(j,-)\Rightarrow D(j,-)$. The naturality square commutes by definition of $\dlim^\Omega_j D(j,f)$:
  \[\begin{tikzpicture}
    \matrix (m) [matrix of math nodes,row sep=2em,column sep=6em,minimum width=2em]
    {
    \dlim^\Omega_j D(j,k)& \dlim^\Omega_j D(j,h) \\
     D(j,k) & D(j,h) \\};
    \path[->]
    (m-1-1) edge node [left] {$$} (m-2-1)
           edge node [above] {$\dlim^\Omega_j D(j,f)$} (m-1-2)
    (m-1-2) edge node [right] {$$} (m-2-2)
    (m-2-1) edge node [below] {$D(\id,f)$} (m-2-2);       
  \end{tikzpicture}\]
  Adjointability follows from Corollary~\ref{cor:mapoflimsismapofcolims} since $D(-,f)$ is adjointable:
  \[\begin{tikzpicture}
     \matrix (m) [matrix of math nodes,row sep=2em,column sep=6em,minimum width=2em]
     {
     \dlim^\Omega_j D(j,k)& \dlim^\Omega_j D(j,h) \\
      D(j,k) & D(j,h) \\};
     \path[->]
     (m-1-1) edge node [above] {$\dlim^\Omega_j D(j,f)$} (m-1-2)
     (m-2-2) edge node [right] {$$} (m-1-2)
     (m-2-1) edge node [below] {$D(\id,f)$} (m-2-2)
             edge node [left] {$$} (m-1-1);       
  \end{tikzpicture}\]
  This concludes the proof.
\end{proof}

\chapter{Ordinary limit-colimit coincidences via dagger limits}\label{chp:ambilims}





\section{Ambilimits}

Sometimes in ordinary (=non-dagger) category theory certain limits coincide with certain colimits, biproducts in abelian categories being a paradigmatic example. The usual way of understanding this is via absolute (co)limits, \ie those that are preserved by every (enriched) functor. For instance, biproducts can be defined purely equationally in terms of $+$ and $0$, and hence are preserved by any functor preserving $+$ and $0$, or in other words, by any functor respecting the enrichment in commutative monoids. In fact, in the enriched setting one can show that limits coinciding with colimits is equivalent to absolutenes~\cite{street:absolutecolims,garner:absolutecolims}.

Here we would like to offer an alternative perspective, not requiring enrichment. For instance, it turns out to be perfectly feasible to axiomatize biproducts without requiring \emph{any} enrichment, and arguably this answers the question ``What does it mean to be a product and a coproduct in a compatible way?''. From this point of view, it is then a theorem that the more general notion of biproduct coincides with the \cat{CMon}-enriched notion whenever the category in question is suitably enriched. 

Instead of talking at length of ``limit-colimit coincidences'' or risking confusion with bicategory theory by using the word ``bilimit'', we will talk about ambilimits. The idea is to capture an object being simultaneously a limit and a colimit in a compatible manner. 

\begin{definition}\label{def:ambilimit}\index[word]{ambilimit} Let $D_1\colon\cat{J}\to\cat{C}$ and $D_2\colon\cat{J}\op\to\cat{C}$ be a pair of diagrams satisfying $D_1(A)=D_2(A)$ for all $A\in\cat{J}$. Then an \emph{ambilimit} of $(D_1,D_2)$ is a tuple $(L,(p_A)_{A\in \cat{J}},(i_A)_{A\in \cat{J}})$ such that $(L,(p_A)_{A\in \cat{J}})$ is a limit of $D_1$, $(L,(i_A)_{A\in \cat{J}})$ is a colimit of $D_2$ and moreover the following equations hold: 
  \begin{description}
    \item[normalization] $p_A$ and $i_A$ are regular inverses to each other for every $A\in\cat{J}$, \ie $i_Ap_Ai_A=i_A$ and $p_Ai_Ap_A=p_A$;
    \item[independence] the idempotents on $L$ induced by these commute, \ie $ i_A p_A i_B p_B=i_B p_B i_A p_A$ for all $A,B\in \cat{J}$.
  \end{description}
\end{definition}

In concrete examples of ambilimits, one often just specifies one $D_i$, the other one being derivable from it, see \eg the limit-colimit coincidence in domain theory below. Before justifying the definition by the examples it captures, let us first satisfy a theoretical concern of canonicity: whenever one has two ambilimits of the same diagrams, a priori there are two isomorphisms between the ambilimits -- one induced by the limit structure and one by the colimit structure. It turns out that these two isomorphisms always coincide.

\begin{theorem}\label{thm:ambilimitsunique} If $(L,(p_A)_{A\in \cat{J}},(i_A)_{A\in \cat{J}})$ and $(M,(q_A)_{A\in \cat{J}},(j_A)_{A\in \cat{J}})$ are two ambilimits of $(D_1,D_2)$, there is a unique isomorphism $M\to L$ that is both a morphism of cones and cocones.
\end{theorem}

\begin{proof}\index[word]{dagger category!cofree}\index[symb]{\cat{C_\leftrightarrows}, the cofree dagger category on \cat{C}}
	The axioms of ambilimits imply that $L$ and $M$ are dagger limits of $(D_1,D_2,\cat{J})$, where $(D_1,D_2)\colon\cat{J}\to\cat{C}_\leftrightarrows$ and $\cat{C}_\leftrightarrows$ is the cofree dagger category on \cat{C} from Proposition~\ref{prop:cofree}. Hence there is a unitary isomorphism $M\to L$ in $\cat{C}_\leftrightarrows$ by Theorem~\ref{thm:daglimsunique}. Unpacking the meaning of this back to \cat{C} we conclude the existence and uniqueness of the desired isomorphism $M\to L$.
\end{proof}

Of course, the preceding proof could be unwound to provide a ``dagger-free'' proof of the theorem above. It is unclear if doing so would give additional insight or if a different, direct proof exists. Similarly, one could have worked with the more general notion of ``ambilimit of $(D_1,D_2,\Omega)$'' where $\Omega\subset\cat{J}$ is weakly initial. However, apart from dagger equalizers and other dagger limits, we do not know of interesting examples that would warrant the added generality, so we assume $\Omega=\cat{J}$ throughout for simplicity and hence suppress it from the notation. We list some examples below, and the rest of the chapter develops two additional examples in more detail.

\begin{example} 
  \begin{itemize}
    \item A zero object in \cat{C} can be characterized as an ambilimit of the unique functor $\emptyset\to \cat{C}_\leftrightarrows$.
    \item Splittings of idempotents are ambilimits: let $p=p^2$ be an idempotent in \cat{C} and let $M$ be the one-object category generated by a single non-identity arrow $e$ satisfying $e=e^2$. Then $e\mapsto p$ induces a functor $M\to\cat{C}_\leftrightarrows$ and a splitting of $p$ is exactly the same thing as an ambilimit of this functor.
    \item If \cat{G} is a group, considered as a one-object category, then a functor $F\colon\cat{G}\to\cat{C}$ induces a functor $F'\colon\cat{G}\to\cat{C}\op$ by $F'(g)=F(g^{-1})$, \ie $F'=F\circ (-)^{-1}$, where $(-)^1$ is the canonical dagger on \cat{G}. Now an ambilimit $(F,F')$ corresponds to a split subobject of $F(*)$ that is simultanously $G$-invariant and $G$-coinvariant. These exist for finite-dimensional $k$-representations of finite group \cat{G} when $k$ is characteristic zero, and they exist for any \cat{G} in \cat{Rel}, \cat{Span(FinSet)} and \cat{PInj} by Proposition~\ref{prop:variouscatshavedagshapedlims}.
    \item Both of the previous examples are special cases of the following: let \cat{M} be an involutive monoid, \ie a one object dagger category. Then any functor $\cat{M}\to \cat{C}$ factors uniquely via a dagger functor $\cat{M}\to \cat{C}_\leftrightarrows$, and one may consider ambilimits of this functor. Again, by Proposition~\ref{prop:variouscatshavedagshapedlims} \cat{Rel}, \cat{Span(FinSet)} and \cat{PInj} have such ambilimits.
  \end{itemize}
\end{example}

\section{Biproducts without zero morphisms}\label{sec:biprods}\index[word]{biproduct}

Given two objects $A$ and $B$ living in some category \cat{C}, their biproduct -- according to a standard definition \cite{maclane:categories} -- consists of an object $A\oplus B$ together with maps 

\begin{equation*} A\xrightarrows[$p_A$][$i_A$] A\oplus B\xrightarrows[$i_B$][$p_B$] B \end{equation*}
such that 

\begin{align}
  p_Ai_A&=\id[A] \quad &p_Bi_B=\id[B] \label{def:biprodeqs}\\
  p_Bi_A&=0_{A,B} \quad &p_Ai_B=0_{B,A} \label{def:biprodswithZeros}
\end{align}
and
\begin{equation}\label{def:biprodswithsums}
  \id[A\oplus B]=i_Ap_A+i_Bp_B\text.
\end{equation}

For us to be able to make sense of the equations, we must assume that \cat{C} is enriched in commutative monoids. One can get a slightly more general definition that only requires zero morphisms but no addition -- that is, enrichment in pointed sets -- by replacing the last equation with the condition that  $(A\oplus B,p_A,p_B)$ is a product of $A$ and $B$ and that $(A\oplus B,i_A,i_B)$ is their coproduct.
This definition is not vastly more general, for one can show that if \cat{C} has all binary biproducts in this latter sense, then \cat{C} is uniquely enriched in commutative monoids and the biproducts in \cat{C} are biproducts in the earlier sense as well. 

But what if we do not assume that \cat{C} is enriched over pointed sets? One alternative would be to postulate, just for the objects $A$ and $B$ in question, maps 
$0_{A,B} \colon A \to B$ and $0_{B,A} \colon B \to A$
that act like zero maps in that it does not matter what you pre- and postcompose them with, but this is not very satisfactory -- it is as if one tried to generalize the first definition by assuming that just the particular homset $\hom(A\oplus B,A\oplus B)$ happens to be a commutative monoid. 

However, by Theorem~\ref{thm:ambilimitsunique}, one can get a well-behaved notion of a biproduct in any category \cat{C}, with no assumptions about enrichment, by replacing the equations referring to zero with the single equation
\begin{equation}\label{eq:projsCommute}
i_Ap_Ai_Bp_B=i_Bp_Bi_Ap_A\text,
\end{equation}
which states that the two canonical idempotents on $A\oplus B$ commute with one another. In this section we show that the notion agrees with the other definitions whenever \cat{C} is appropriately enriched. We also show how to characterize them in terms of ambidextrous adjunctions and discuss some examples recognized by the more general definition.

We start by spelling out the new, enrichment-free definition of a biproduct.

\begin{definition}\label{def:newbiprods} A biproduct of $A$ and $B$ in \cat{C} is a tuple $(A\oplus B,p_A,p_B,i_A,i_B)$ such that $(A\oplus B,p_A,p_B)$ is a product of $A$ and $B$, $(A\oplus B,i_A,i_B)$ is their coproduct, and the following equations hold:
\begin{align*}
p_Ai_A&=\id[A] \\
p_Bi_B&=\id[B] \\
i_Ap_Ai_Bp_B&=i_Bp_Bi_Ap_A
\end{align*}
\end{definition}

Even though the definition above does not refer to zero morphisms, the map $p_Bi_A\colon A\to A\oplus B\to B$ behaves a lot like one. 

\begin{definition} A morphism $a:A\to B$ is constant if $af=ag$ for all $f,g\colon C\to A$. Coconstant morphisms are  defined dually and a morphism is called a zero morphism if it is both constant and coconstant. A category has zero morphisms if for every pair of objects $A$ and $B$ there is a zero morphism $A\to B$. 
\end{definition}

\begin{remark}\label{remark:absorbing} If there is a zero morphisms $A\to B$ and a morphism $B\to A$, then the zero morphism $A\to B$ is unique. Furthermore, a category has zero morphisms iff for every $A$ and $B$ there is a morphism $0_{A,B}$ such that for every $g\colon B\to D$ and $f\colon C\to A$ we have $g0_{A,B}f=0_{C,D}$, and this is equivalent to the category being enriched in pointed sets. The collection of all zero morphisms in a category forms a partial zero structure in the sense of \cite{goswami2017structure}.
\end{remark}

\begin{proof}
  Assume first that $0$ and $0'$ are zero morphisms $A\to B$ and let $f\colon B\to A$ be an arbitrary morphism. Since $0$ and $0'$ are zero morphisms, we see that $0=0'f0=0'$, as desired. The other claims follow readily. 
\end{proof}

\begin{lemma}\label{lem:absorbing}
Let $(A\oplus B,p_A,p_B,i_A,i_B)$ be the biproduct of $A$ and $B$. Then $p_Bi_A$ is a zero morphism.
\end{lemma}

\begin{proof}
As the definition of biproducts is self-dual, it suffices to prove that $p_Bi_A$ is coconstant. First we observe that the diagram 
  \[
  \begin{tikzpicture} 
     \matrix (m) [matrix of math nodes,row sep=3em,column sep=4em,minimum width=2em]
     {
     A  &A & A\oplus B & B & C \\
        &  &           &  A\oplus B  & B\\
      A\oplus B & B & A\oplus B &A & A\oplus B\\
       & A & A\oplus B & B\\};
     \path[->]
     (m-3-1) edge node [left] {$p_A$} (m-1-2)
            edge node [below] {$p_B$} (m-3-2)
            edge node [below] {$p_A$} (m-4-2)
     (m-4-2) edge node [below] {$i_A$} (m-4-3)
     (m-4-3) edge node [below] {$p_B$} (m-4-4)
     (m-4-4) edge node [below] {$i_B$} (m-3-5)
     (m-1-1) edge node [above] {$\id$} (m-1-2)
            edge node [left] {$i_A$} (m-3-1)
     (m-3-2) edge node [below] {$i_B$} (m-3-3)
     (m-3-3) edge node [below] {$p_A$} (m-3-4)
     (m-3-4) edge node [below] {$i_A$} (m-3-5)
            edge node [right] {$i_A$} (m-2-4)
     (m-1-4) edge node [left] {$i_B$} (m-2-4)
     (m-2-4) edge node [below] {$h$} (m-1-5)
     (m-1-2) edge node [above] {$i_A$} (m-1-3)
     (m-1-3) edge node [above] {$p_B$} (m-1-4)
     (m-1-4) edge node [above] {$f$} (m-1-5)
     (m-3-5) edge node [right] {$p_B$} (m-2-5)
     (m-2-5) edge node [right] {$g$} (m-1-5); 
  \end{tikzpicture}
  \]
commutes, where $h$ is the cotuple $[g p_Bi_A,f]$. Now the top path in the outernmost shape is $fp_Bi_A$ and by the biproduct equations the bottom path is equal to $gp_Bi_A$.
\end{proof}

\begin{remark}\label{rem:polarofbiprod}
  We can now see that Theorem~\ref{thm:polarofbiprod} is still true even if the word biproduct in it is interpreted in the sense of Definition~\ref{def:newbiprods}. This is because the only zero morphisms needed for the proof are between $A$ and $B$, and hence the proof is valid verbatim.
\end{remark}

\begin{corollary}\label{thm:zeros} If \cat{C} has all binary biproducts, then it has zero morphisms.
\end{corollary}
\begin{proof} Combine Lemma~\ref{lem:absorbing} and Remark~\ref{remark:absorbing}.
\end{proof}

Given Lemma~\ref{lem:absorbing}, it is easy to check that whenever \cat{C} has zero morphisms: equation~\eqref{eq:projsCommute} is equivalent to equation~\eqref{def:biprodswithZeros}. Given a biproduct in the sense of Definition~\ref{def:biprodswithZeros}, it is a biproduct in the sense of Definition~\ref{def:newbiprods}, since $i_Ap_Ai_Bp_B=0=i_Bp_Bi_Ap_A$. Conversely, let $(A\oplus B,p_A,p_B,i_A,i_B)$ be a biproduct in the sense of Definition~\ref{def:newbiprods} in a category with zero morphisms. Now by Lemma~\ref{lem:absorbing} and Remark~\ref{remark:absorbing} we have $p_Bi_A=0_{A,B}$, as desired, and similarly $p_Ai_B=0_{B,A}$. If \cat{C} is enriched in commutative monoids, then our definition is equivalent to~\eqref{def:biprodswithsums} just because the other definition in terms of \eqref{def:biprodswithZeros} and universal properties is.

Besides being equivalent to the usual definitions when \cat{C} is appropriately enriched, it behaves well even when \cat{C} is not assumed to be enriched. 

Using Lemma~\ref{lem:absorbing}, one can then proceed to check that biproducts in our sense work just like one would expect. For example, Definition~\ref{def:newbiprods} and other results of this section generalize from the binary case to the biproduct of an arbitrary-sized collection of objects, and one can easily show that if $(A\oplus B)$ and $(A\oplus B)\oplus C$ exist, then $(A\oplus B)\oplus C$ satisfies the axioms for the ternary biproduct of $A,B,C$. Similarly, using Lemma \ref{lem:absorbing} one can show that for $f\colon A\to C$ and $g\colon B\to D$ we have $f+g=f\times g$ whenever the biproducts $A\oplus B$ and $C\oplus D$ exist. 

Recall that an ambiadjoint to a functor $F\colon \cat{C}\to\cat{D}$ is a functor $D\colon \cat{D}\to\cat{C}$ that is simultaneously both left and right adjoint to $F$.

\begin{theorem} \cat{C} has biproducts iff the diagonal $\Delta\colon\cat{C}\to \cat{C}\times \cat C$ has an ambiadjoint $(-)\oplus (-)$ such that the unit $(i_A,i_B)\colon (A,B)\to (A\oplus B,A\oplus B)$ of the adjunction $(-)\oplus (-)\dashv \Delta$, is a section of the counit $(p_A,p_B)\colon (A\oplus B,A\oplus B)\to (A,B)$ of the adjunction $\Delta\dashv (-)\oplus (-)$,\ie $(p_A \circ i_A,p_B\circ i_B)=(\id[A],\id[B])$ for $A,B \in \cat{C}$.
\end{theorem}

\begin{proof}
  The implication from left to right is routine. For the other direction, a right adjoint to the diagonal is well-known to fix binary products, and dually, a left adjoint fixes binary coproducts. Thus it remains to check that the required equations governing $p_A,p_B,i_A$ and $i_B$ are satisfied. By naturality, the diagram
  \[\begin{tikzpicture}[xscale=3.5,yscale=1.5]
    \node (tl) at (0,1) {$A$};
    \node (t) at (1,1) {$A \oplus B$};
    \node (tr) at (2,1) {$B$};
    \node (bl) at (0,0) {$A$};
    \node (b) at (1,0) {$A \oplus C$};
    \node (br) at (2,0) {$C$};
    \draw[->] (tl) to node[above] {$i_A$} (t);
    \draw[->] (t) to node[above] {$p_B$} (tr);
    \draw[->] (bl) to node[below] {$i_A$} (b);
    \draw[->] (b) to node[below] {$p_C$} (br);
    \draw[->] (tl) to node[left] {$\id$} (bl);
    \draw[->] (tr) to node[right] {$f$} (br);
    \draw[->] (t) to node[right] {$\id\oplus f$} (b);
  \end{tikzpicture}\]
  commutes for any $f$. Thus $i_Ap_B$ is coconstant and by duality it is constant, so it is zero. Hence \cat{C} has zero morphisms and $i_Ap_B=0$. Thus $i_Ap_Ai_Bp_B=0=i_Bp_Bi_Ap_A$. As $(p_A \circ i_A,p_B\circ i_B)=(\id[A],\id[B])$ by assumption, this concludes the proof.
\end{proof}

Given Theorem~\ref{thm:zeros}, genuinely new examples recognized under this must be in categories that have neither all binary biproducts nor zero morphisms. 

\begin{itemize}
  \item In \cat{Set} the biproduct $\emptyset\oplus\emptyset$ exists and is the empty set.
  \item Let \cat{C} be any category with biproducts, and let \cat{D} be any non-empty category. Then in the coproduct category $\cat{C}\sqcup \cat{D}$, the biproduct $A\oplus B$ exists whenever $A,B\in\cat{C}$. More concretely, in $\cat{Ab}\sqcup \cat{Set}$ the binary biproduct of any two abelian groups exists and is computed just as in \cat{Ab}, even though $\cat{Ab}\sqcup \cat{Set}$ lacks zero morphisms.
  \item In any preorder $A\oplus B$ exists if and only if $A\cong B$.
  \item Commutative inverse semigroups are sets equipped with a binary operation that is commutative and associative and in which for every element $x$ there is an unique $y$ such that $xyx=x$ and $yxy=x$~\cite{lawson:inversesemigroups}. The obvious choice of morphism is a function that preserves the binary operation. Not every such semigroup has a neutral element, so we call a homomorphism $f\colon S\to T$ unital if it preserves neutral elements. Let \cat{C} be the category of commutative inverse semigroups and unital homomorphisms. Then $S\oplus T$ exists if and only if $S$ and $T$ both have a neutral element, in which case $S\oplus T$ is constructed just as in \cat{Ab}.
  \item A function $f\colon (X,d_X)\to (Y,d_Y)$ between metric spaces is non-expansive if $d_X(x,y)\geq d_Y(f(x),f(y))$ for all $x,y\in X$. It is contractive if there is some $c\in [0,1)$ such that $c d_X(x,y)\geq d_Y(f(x),f(y))$ for all $x,y\in X$. Let \cat{Met} be the category of metric spaces and non-expansive maps and let \cat{Con} be the category of contractive endomorphisms. To be more precise, let $\mathbb{N}$ denote the monoid of natural numbers. Then we define \cat{Con} as the full subcategory of $[\mathbb{N},\cat{Met}]$ with objects given by contractive endomorphisms. In \cat{Con}, the terminal object is $! \colon \{*\}\to \{*\}$, and for any $s$ in \cat{Con}, the biproduct $s\oplus !$ exists if and only if $s$ has a (necessarily unique) fixed point. To see this, assume first that the biproduct $s\oplus !$ exists. Then in particular there is a morphism $!\to s$, whence $s$ has a fixed point. Conversely, if $s$ has a fixed point, $s$ itself can be given the structure of a biproduct $s\oplus !$ in a relatively obvious manner.
\end{itemize}

\section{Limit-colimit coincidence from domain theory}\label{sec:limcolimcoincidence}


In this section we will show how to understand the limit-colimit coincidence from domain theory (originally remarked in~\cite{scott:continuouslattices} and further generalized in~\cite{smythplotkin:datatypes}) as an ambilimit. Before doing so, we briefly recall the setup of~\cite{smythplotkin:datatypes}.

\begin{definition}\index[word]{\cat{O}-category} An \cat{O}-category is a category in which (i) every homset is a partial order that has least upper bounds of $\omega$-chains\index[symb]{$\omega$, the category induced by the order on $\N$|(} and (ii) composition is $\omega$-continuous.
\end{definition}

\cat{O}-categories were originally introduced in~\cite{wand:fixedpoints} and are enriched categories in the sense of~\cite{kelly:enriched}: the enrichment is over the category \cat{O} where objects are posets with l.u.b.s of $\omega$-chains and morphisms are $\omega$-continuous functions. Since we do not need the full generality of enriched category theory, we work with the concrete definition above. 

\begin{definition} A pair $(f\colon A\to B,g\colon B\to A)$ of morphisms in an \cat{O}-category is an \emph{embedding-projection pair}\index[word]{embedding-projection pair} if  
  \[gf=\id[A] \text{ and } fg\leq \id[B],\]
in which case we call $f$ an embedding and $g$ a projection. 
\end{definition}

It turns out that either half of such a pair determines the other: hence for an \cat{O}-category \cat{C} one can define $\cat{C}^E$, the category of embeddings in \cat{C}, as the wide subcategory of \cat{C} with embeddings as its morphisms. The category of projections $\cat{C}^P$ is defined similarly. In general neither of these is an \cat{O}-category. Since one half of a projection-embedding pair determines the other, there are in fact contravariant identity-on-objects isomorphisms $(-)^P\colon\cat{C}^{E}\leftrightarrows \cat{C}^P\colon (-)^E$ taking one half of a projection-embedding pair to the other one.

\begin{definition}  Let \cat{C} be an \cat{O}-category and $D\colon \omega\to \cat{C}^E$ be a diagram. A cocone  $(L,(e_n)_{n\in \omega})$ for $\omega\xrightarrow{D} \cat{C}^E\hookrightarrow \cat{C}$ is an \emph{\cat{O}-colimit}\index[word]{\cat{O}-colimit} if each $e_n$ is an embedding, and $(e_n\circ e_n^P)_{n\in\omega}$ is an increasing sequence having \id[L] as its least upper bound. 
\end{definition}

\begin{theorem}\label{thm:limitcolimitcoincidence} Let \cat{C} be an \cat{O}-category and $D\colon \omega\to \cat{C}^E$ be a diagram. The following conditions are equivalent for any cocone $(L,(e_n)_{n\in \omega})$  for $\omega\xrightarrow{D} \cat{C}^E\hookrightarrow \cat{C}$:
  \begin{enumerate}[(i)]
    \item $L$ is a colimit in \cat{C}
    \item $L$ is an \cat{O}-colimit in \cat{C}
    \item $(L,(e_n)^P)$ is a limit in \cat{C}
    \item $(L,(e_n)^P)$ is an \cat{O}-limit in \cat{C}
  \end{enumerate}
\end{theorem}
\begin{proof} See the proof of ~\cite[Theorem 2]{smythplotkin:datatypes}. Note that the relevant part of the theorem states a slightly weaker result, saying that $D$ has a colimit iff it has an \cat{O}-colimit, but the proof shows in Proposition D that in fact that any colimit is necessarily an \cat{O}-colimit.
\end{proof}

Instead of \cat{O}-(co)limits one could just work with the unenriched notion provided by Definition~\ref{def:ambilimit}. First we note that the independence axiom is superfluous for ambilimits of chains. For that purpose, we fix some notation. For a chain $(D_1,D_2)\colon \omega\op\to \cat{C}_\leftrightarrows$, we write $D(n)$ for $D_1(n)=D_2(n)$ and $q_n$ for $D_1(n+1\to n)$ and $e_n$ for $D_2(n\to n+1)$. Note that the chain is uniquely defined by the objects $D(n)$ and the morphisms $q_n$ and $e_n$. More generally, we will write $q_{m,n}$ for $D_1(m\to n)$ and $e_{n,m}$ for $D_2(n\to m)$.

\begin{proposition}\label{prop:chainsdontneedindependence} Let $(D_1,D_2)\colon \omega\op\to \cat{C}_\leftrightarrows$ be a chain and $(L,(p_n,i_n))$ be a cone for it. If $L$ satisfies normalization, then it satisfies independence.
\end{proposition}

\begin{proof} 
  We wish to prove that \[i_np_ni_mp_m=i_mp_mi_np_n\] follows from the assumptions. We will prove a stronger claim, namely that whenever $n<m$, both sides of this equation are equal to $i_np_n$. For the left hand side this is done as follows:
    \begin{align*}
        i_np_ni_mp_m&=i_n q_{m,n}p_mi_mp_m \text{ because $L$ is a cone} \\
                    &=i_n q_{m,n}p_m \text{ by normalization} \\
                    &=i_np_n \text{ because $L$ is a cone.}
    \end{align*}
  The right hand side is similar: 
    \begin{align*}
        i_mp_mi_np_n&=i_mp_mi_me_{n,m}p_n \text{ because $L$ is a cone} \\
                    &=i_m e_{m,n}p_n \text{ by normalization} \\
                    &=i_np_n \text{ because $L$ is a cone.}
    \end{align*}
\end{proof}

\begin{corollary} Given the assumptions of Theorem~\ref{thm:limitcolimitcoincidence}, the equivalent conditions (i)-(iv) are equivalent to $(L,(e_n^P,e_n))$ being an ambilimit of $(D^P,D)\colon \omega\op\to \cat{C}_\leftrightarrows$.
\end{corollary}

Such ambilimits of chains might exist in categories that are not \cat{O}-categories. For instance, they are known to exist in categories enriched over complete metric spaces~(\cite{america:metric,adamek:banach,birkedal:metric}). The next section studies another case not falling under either of these.

 \section{Fixed points}\label{sec:unenrichedfixedpoints}

The main use of Theorem~\ref{thm:limitcolimitcoincidence} is as a building block for models of algebraic datatypes. Algebraic datatypes are usually modelled as initial algebras for certain endofunctors~\cite{lehmannsmyth:datatypes}. Using Ad\'amek's Theorem~\cite{adamek1979least}, one can show that every functor preserving \cat{O}-colimits has an initial algebra. Moreover, \cat{O}-colimits are absolute in the sense that any locally continuous functor preserves them. Hence initial algebras are numerous in \cat{O}-categories. An added benefit is that the initial algebras thus built turn out to be canonical fixed points, \ie also terminal coalgebras. This enables one to both program with these datatypes (via the initial algebra structure) and to reason about them (via the terminal coalgebra structure). 

In our case we are working without enrichment, which confers a possible benefit: now one can hope to find suitable models for programming semantics in categories that are not suitably enriched. However, this comes at a price: it is less clear that many functors preserve such ambilimits and hence it is unclear if one can build canonical fixed points for a wide class of functors. In this section we show that, under suitable assumptions, rig-polynomial functors have canonical fixed points. We interleave the discussion with a dagger analogue of the same story. The goal of the dagger version is to develop a theory suitable for languages taking semantics in dagger categories. For instance, the language defined in~\cite{valironetal:quantumcontrol} takes its semantics in \Hilb, the category of Hilbert spaces and non-expansive linear maps. Moreover, the semantics implicitly use various initial algebras. 

We start by covering the usual theory very briefly.

\begin{definition}\index[word]{$F$-algebra} Let $F$ be an endofunctor on a category \cat{C}. An $F$-algebra is a map $a\colon F(A)\to A$, and a morphism of $F$-algebras $a\to b$ is a map $f\colon A\to B$ in \cat{C} making the square 
    \[\begin{tikzpicture}
         \matrix (m) [matrix of math nodes,row sep=2em,column sep=4em,minimum width=2em]
         {
          F(A) & F(B) \\
          A & B \\};
         \path[->]
         (m-1-1) edge node [left] {$a$} (m-2-1)
                edge node [above] {$Ff$} (m-1-2)
         (m-1-2) edge node [right] {$b$} (m-2-2)
         (m-2-1) edge node [below] {$f$} (m-2-2);
  \end{tikzpicture}\]
commute. This results in a category of $F$-algebras, resulting in the notion of an \emph{initial algebra} for a endofunctor. $F$-coalgebras and terminal coalgebras are defined dually.
\end{definition}%

Lambek's lemma, a basic result about initial algebras, asserts that whenever $a\colon F(A)\to A$ is an initial $F$-algebra, the morphism $a$ is in fact an isomorphism in \cat{C}. This makes the following definition reasonable.%

\begin{definition}\index[word]{canonical fixed point}\index[word]{dagger fixed point} A canonical fixed point of an endofunctor $F$ is an initial $F$-algebra $a$ such that $a^{-1}$ is a terminal coalgebra. A \emph{dagger fixed point} for a dagger endofunctor $F\colon\cat{C}\to\cat{C}$ is a canonical fixed point $a\colon FA\to A$ such that $a$ is unitary.
\end{definition} 

Note that the self-duality given by the dagger implies that $a^\dag$ is a terminal coalgebra. Since Lambek's lemma says that any initial algebra $a$ is an isomorphism, one could equivalently define a dagger fixed point as an initial algebra $a\colon FA\to A$ for which $a$ is unitary. Hence we are merely requiring something that is invertible to be dagger invertible. 

We refer the reader to~\cite{freyd:algebraically,barr:algebraically} for more on canonical fixed points and to the review article~\cite{adamek2018fixed} for more on fixed points of functors in general.  As mentioned before, Ad\'amek's Theorem~\cite{adamek1979least} is used to construct many initial algebras. 
As is usual in theoretical computer science, we focus on the special case of chains indexed by natural numbers, the generalization to larger ordinals being straightforward. 

\begin{theorem}\label{thm:adamek} Let \cat{C} be a category that has an initial object $0$ and let $F$ be an endofunctor on it. If the chain 
  \[0\to F(0)\to FF(0)\to\dots\]
has a colimit $L$ and $F$ preserves it, then $a\colon F(L)\to L$ is an initial $F$-algebra, where $a\colon F(L)\to L$ is the unique morphism making the triangle
      \[\begin{tikzpicture}
         \matrix (m) [matrix of math nodes,row sep=2em,column sep=4em,minimum width=2em]
         {
          F^n(0) & F(L) \\
           & L \\};
         \path[->]
         (m-1-1) edge node [left] {$i_n$} (m-2-2)
                edge node [above] {$F(i_{n-1})$} (m-1-2)
         (m-1-2) edge node [right] {$a$} (m-2-2);
  \end{tikzpicture}\]
commute.
\end{theorem}

\begin{proof} See \eg~\cite{adamek1979least} or~\cite{smythplotkin:datatypes}.
\end{proof}

Next we prove variants of Ad\'amek's Theorem for ambilimits and for dagger categories. By a chain of isometries we mean a sequence $A_0\sxto{f_0} A_1\sxto{f_1}\dots$ where each $f_i$ is a dagger isometry. A dagger colimit of a chain of isometries is a dagger colimit of such a diagram with support all objects in the diagram. One can easily adapt Ad\'amek's Theorem\cite{adamek1979least} to the dagger setting. Below is a special case for $\omega$-chains, generalizing to ordinals $\alpha >\omega$ is straightforward.

\begin{theorem}\label{thm:daggeradamek}\index[word]{dagger fixed point} Assume that \cat{C} has dagger colimits of chains of isometries and a zero object. If $F\colon\cat{C}\to\cat{C}$ is a dagger functor preserving such colimits, then it has a dagger fixed point.
\end{theorem}

\begin{proof} Let $(L,l_n)$ be the dagger colimit of the chain $0\to F(0)\to FF(0)\to\dots$, and let $a\colon F(L)\to L$ be the unique map making the triangle 
      \[\begin{tikzpicture}
         \matrix (m) [matrix of math nodes,row sep=2em,column sep=4em,minimum width=2em]
         {
          F^n(1) & F(L) \\
           & L \\};
         \path[->]
         (m-1-1) edge node [left] {$i_n$} (m-2-2)
                edge node [above] {$F(i_{n-1})$} (m-1-2)
         (m-1-2) edge node [right] {$a$} (m-2-2);
  \end{tikzpicture}\]
commute. Then $a$ is an initial algebra by Theorem~\ref{thm:adamek} and hence $a^\dag$ is a terminal coalgebra. Moreover, since $a$ is an isomorphism of dagger colimits, it is unitary by Theorem~\ref{thm:daglimsunique}, and hence $a$ is a dagger fixed point.
\end{proof}

\begin{theorem}\label{thm:ambiadamek}\index[word]{canonical fixed point} Let \cat{C} be a category that has a zero object and $F$ an endofunctor on it.
If the chain 
 \[0\xrightarrows[$0$][$0$] F(0) \xrightarrows[$F(0)$][$F(0)$] FF(0)\xrightarrows\dots\]
has an ambilimit preserved by $F$, then $F$ has a canonical fixed point.
\end{theorem}

\begin{proof} 
 Let $(L,(p_n,i_n))$ be an ambilimit of the chain above. By Theorem~\ref{thm:adamek} $a\colon F(L)\to L$ is an initial algebra and by the dual of the same theorem $b\colon L\to F(L)$ is a terminal coalgebra, where $a$ and $b$ are the unique maps making the triangles 
          \[\begin{tikzpicture}
         \matrix (m) [matrix of math nodes,row sep=2em,column sep=4em,minimum width=2em]
         {
          F^n(1) & F(L) \\
           & L \\};
         \path[->]
         (m-1-1) edge node [below] {$i_n$} (m-2-2)
                edge node [above] {$F(i_{n-1})$} (m-1-2)
         (m-1-2) edge node [right] {$a$} (m-2-2);
  \end{tikzpicture} \qquad 
  \begin{tikzpicture}
         \matrix (m) [matrix of math nodes,row sep=2em,column sep=4em,minimum width=2em]
         {
          F^n(1) & F(L) \\
           & L \\};
         \path[<-]
         (m-1-1) edge node [below] {$p_n$} (m-2-2)
                edge node [above] {$F(p_{n-1})$} (m-1-2)
         (m-1-2) edge node [right] {$b$} (m-2-2);
  \end{tikzpicture}\]
  commute. It suffices to prove that $b=a^{-1}$. This is easy: $a$ is in fact an isomorphism of colimits and $b$ is an isomorphism of limits, and by Theorem~\ref{thm:ambilimitsunique} these coincide.
\end{proof}

The goal of this section is to leverage Theorems~\ref{thm:daggeradamek} and \ref{thm:ambiadamek} to build canonical fixed points for a wide class of endofunctors. Concrete examples of ambilimits of chains seem to suggest that they tend to exist when $D_1$ and $D_2$ are compatible in some sense. For instance, in the case of \cat{O}-categories, either half of the chain determines the other half uniquely. We first investigate the case of chains of split monics \ie chains
 \[D(0)\xrightarrows[$q_0$][$e_0$] D(1) \xrightarrows[$q_1$][$e_1$] D(2)\xrightarrows\dots\]
 that satisfy $q_ne_n=\id$ for every $n$. Recall the notation defined before Proposition~\ref{prop:chainsdontneedindependence}: For a chain $(D_1,D_2)\colon \omega\op\to \cat{C}_\leftrightarrows$, we write $D(n)$ for $D_1(n)=D_2(n)$ and $q_n$ for $D_1(n+1\to n)$ and $e_n$ for $D_2(n\to n+1)$, and more generally $q_{m,n}$ denotes $D_1(m\to n)$ and $e_{n,m}$ means $D_2(n\to m)$.

\begin{theorem}\label{thm:ambilimitiffiso} Given a chain  \[D(0)\xrightarrows[$q_0$][$e_0$] D(1) \xrightarrows[$q_1$][$e_1$] D(2)\xrightarrows\dots\] 
satisfying $q_ne_n=\id[D(n)]$, the following are equivalent:
  \begin{enumerate}[(i)]
    \item It has an ambilimit $(L,(p_n,i_n))$
    \item The diagram 
      \[D(0)\sxto{e_0} D(1) \sxto{e_1} D(2)\sxto\dots\] 
    has a colimit $(L,(i_n))$ and the diagram 
      \[D(0)\sxfrom{q_0} D(1) \sxfrom{q_1} D(2)\sxfrom\dots\]
     a limit $(M,(q_n))$ such that the canonical morphism $f\colon L\to M$ is an isomorphism. This canonical map is the unique map $f\colon L\to M$ making the diagram
          \[\begin{tikzpicture}
         \matrix (m) [matrix of math nodes,row sep=2em,column sep=4em,minimum width=2em]
         {
          L & M \\
          D(n) & D(m)  \\};
         \path[->]
         (m-1-1)  edge node [above] {$f$} (m-1-2)
         (m-1-2) edge node [right] {$q_m$} (m-2-2)
         (m-2-1) edge node [below] {$f_{n,m}$} (m-2-2)
                 edge node [left] {$i_n$} (m-1-1);
  \end{tikzpicture}\]
    commute, where $f_{n,m}$ is the cone structure on $D(n)$ given by
      \[f_{n,m}:=\begin{cases} e_{n,m}&\text{if }n\leq m  \\
                              q_{n,m}&\text{otherwise.}\end{cases}\]
  \end{enumerate}
\end{theorem}

\begin{remark}
The assumption $q_ne_n=\id[D(n)]$ is needed to ensure that $f_{n,m}$ defines a cone structure on $D(n)$, which in turn is needed to guarantee the existence of $f$. Given the theorem, both biproducts (in the presence of zero morphisms) and ambilimits of chains of split monics can be phrased in terms of a canonical map from the colimit to the limit being an isomorphism. Hence one might be tempted to guess that this holds for all ambilimits. However, it is unclear if this holds for arbitrary diagrams: indeed, given an arbitrary diagram $(D_1,D_2)\colon\cat{J}\to\cat{C}_\leftrightarrows$ it is unclear why there should be a canonical morphism from the colimit of $D_2$ to the limit of $D_1$. 
\end{remark}

\begin{proof} 
  To prove that (i) implies (ii) it suffices to show that \id[L] makes the diagram
       \[\begin{tikzpicture}
         \matrix (m) [matrix of math nodes,row sep=2em,column sep=4em,minimum width=2em]
         {
          L & L \\
          D(n) & D(m)  \\};
         \path[->]
         (m-1-1)  edge node [above] {$\id$} (m-1-2)
         (m-1-2) edge node [right] {$p_m$} (m-2-2)
         (m-2-1) edge node [below] {$f_{n,m}$} (m-2-2)
                 edge node [left] {$i_n$} (m-1-1);
      \end{tikzpicture}\]
  commute. We begin by showing that $p_ni_n=\id[D(n)]$. This follows from normalization once we show that $p_n$ is epic. In fact it is split epic: the maps $(f_{n,m})_{m\in \N }$ define a cone structure on $D(n)$ which hence must factor via some $g\colon D(n)\to L$ (it will follow that $g=i_n$). Hence $\id[D(n)]=f_{n,n}=p_n g$ so that $p_n$ splits, as desired. Now it is easy to show that the diagram above commutes for all $n$ and $m$. We show this holds in case $n\geq m$, when $f_{n,m}=q_{n,m}$, the case $n>m$ being similar. Since $p_ni_n=\id[D(n)]$, the triangle on the left in the diagram 
    \[\begin{tikzpicture}
         \matrix (m) [matrix of math nodes,row sep=2em,column sep=4em,minimum width=2em]
         {
          D(n)&  \\
           & D(n) & D(m)  \\
          L \\};
         \path[->]
         (m-1-1)  edge[out=0,in=135] node [above] {$q_{n,m}$} (m-2-3)
                  edge node [left] {$i_n$} (m-3-1)
                  edge node [below] {$\id$} (m-2-2)
         (m-2-2) edge node [above] {$q_{n,m}$} (m-2-3)
         (m-3-1) edge node [above] {$p_n$} (m-2-2)
                 edge[out=0,in=-135] node [below] {$p_m$} (m-2-3);
      \end{tikzpicture}\]
  commutes. The triangle on the top commutes by definition and the bottom triangle commutes because $(L,(p_k)_{k\in\N})$ is a cone. Hence the whole diagram commutes, as desired.

  To prove that (ii) implies (i), note that defining $p_n:=(q_n\circ f)$ makes $(L,p_n)$ into a limit. Hence by Proposition~\ref{prop:chainsdontneedindependence} it suffices to check that $(L,(p_n,i_n))$ satisfies normalization. But this is immediate from the fact that $p_n i_n=q_n f i_n=f_{n,n}=\id[D(n)]$.
\end{proof}

\begin{corollary}\label{cor:connectingmapsofambilims}
  If $(L,(p_n,i_n))$ is the ambilimit of a chain
   \[[D(0)\xrightarrows[$q_0$][$e_0$] D(1) \xrightarrows[$q_1$][$e_1$] D(2)\xrightarrows\dots\] 
  satisfying $q_ne_n=\id[D(n)]$, then the composite $p_me_n\colon D(n)\to L\to D(m)$ satisfies \[p_me_n=\begin{cases} e_{n,m} &\text{if }n\leq m  \\
                              q_{n,m}&\text{otherwise.}\end{cases}\]
\end{corollary}

In the setting of \cat{O}-categories, for a functor to preserve \cat{O}-colimits it is enough for it to be locally continuous, \ie an enriched functor. In the unenriched setting of this chapter, we don't have as convenient a sufficient condition to ensure preservation of ambilimits of chains. However, it turns out that for chains of split monics, it is enough to preserve the universal property of the limit (or of the colimit), provided that the resulting diagram admits an ambilimit. In particular, being (co)continuous will often ensure preservation of ambilimits of suitable chains, even if in principle an ambilimit is both a limit and a colimit.  

\begin{lemma}\label{lem:preservationofambilims} 
  Let 
  \[D(0)\xrightarrows[$q_0$][$e_0$] D(1) \xrightarrows[$q_1$][$e_1$] D(2)\xrightarrows\dots\] 
  be a chain in \cat{C} satisfying $q_ne_n=\id[D(n)]$ with an ambilimit $(L,(p_n),(i_n))$, and let $F\colon\cat{C}\to\cat{D}$ be a functor. If 
  \begin{itemize}
      \item the chain
        \[FD(0)\xrightarrows[$Fq_0$][$Fe_0$] FD(1) \xrightarrows[$Fq_1$][$Fe_1$] FD(2)\xrightarrows\dots\] 
        has an ambilimit and
        \item $(FL,Fi_n)$ is a colimit of the chain
          \[FD(0)\sxto{Fe_0} FD(1) \sxto{Fe_1} FD(2)\sxto\dots\]  
  \end{itemize}
 then $(FL,(Fp_n),(Fi_n))$ is in fact an ambilimit.
\end{lemma}

\begin{proof}
  Let $(L',(p_n'),(i'_n))$ be the ambilimit of 
    \[FD(0)\xrightarrows[$Fq_0$][$Fe_0$] FD(1) \xrightarrows[$Fq_1$][$Fe_1$] FD(2)\xrightarrows\dots\] 
  Since $(FL,Fi_n)$ is a colimit of the chain
     \[FD(0)\sxto{Fe_0} FD(1) \sxto{Fe_1} FD(2)\sxto\dots\]
  there is a unique isomorphism of colimits $f\colon (FL,Fi_n)\to (L,i_n)$, and this makes $(FL,p'_n\circ f,Fi_n)$ into an ambilimit. Moreover, the the top triangle of
        \[\begin{tikzpicture}
         \matrix (m) [matrix of math nodes,row sep=2em,column sep=4em,minimum width=2em]
         {
          FL & FD(m) \\
          FD(n) & FL \\};
         \path[->]
         (m-1-1)  edge node [above] {$p'_m\circ f$} (m-1-2)
         (m-2-2) edge node [right] {$Fp_m$} (m-1-2)
         (m-2-1) edge node [below] {$Fi_n$} (m-2-2)
                 edge node [left] {$Fi_n$} (m-1-1)
                 edge node [above] {$Ff_{n,m}$} (m-1-2); 
  \end{tikzpicture}\]
commutes by Corollary~\ref{cor:connectingmapsofambilims} (since $(FL,p'_n\circ f,Fi_n)$ is an ambilimit in \cat{D}), as does the bottom triangle (since $L$ is an ambilimit in \cat{C}). Hence the universal property of the colimit gives us $p'_n\circ f=Fp_n$, so that $(FL,(Fp_n),(Fi_n))$ is an ambilimit as desired.
\end{proof}

\begin{definition}\index[word]{rig category} A rig category\footnote{Also known as a bimonoidal category} is a category \cat{C} with a symmetric monoidal structure $(\cat{C},\oplus,0)$ and a monoidal structure $(\cat{C},\otimes, I)$. The coherence isomorphisms of the $\oplus$-part will have components denoted by
  \begin{align*}
    &\alpha^\oplus_{A,B,C}\colon (A\oplus B)\oplus C\to A\oplus (B\oplus C)  &\lambda^\oplus_A\colon 0\oplus A\to A \\  
    &\rho^\oplus_A\colon 0\oplus A\to A  &\sigma^\oplus_{A,B}\colon A\oplus B\to B\oplus A
  \end{align*}
Similarly, the coherence isomorphisms for  $(\cat{C},\otimes, I)$ will be denoted by $\alpha^\otimes,\lambda^\otimes,\rho^\otimes$ and $\sigma^\otimes.$ The two monoidal structures should of course interact. Concretely, this means that we have 
  \begin{itemize}
    \item Left and right distributors $\delta^l$ and $\delta^r$ with components 
      \begin{align*}
      \delta^l_{A,B,C}\colon A\otimes (B\oplus C)\cong (A\otimes B)\oplus (A\otimes C) \\
      \delta^r_{A,B,C}\colon (A\oplus B)\otimes C\cong (A\otimes C)\oplus (B\otimes C)
      \end{align*}
    \item And left and right annihilators $\tau^l$ and $\tau^r$ with components
      \[\tau^l_A\colon A\otimes 0\cong 0 \qquad \tau^r_A\colon 0\otimes A\cong 0\]
  \end{itemize} 
These coherence isomorphisms can be seen as categorifications of axioms for a rig (semi-ring). This data is supposed to satisfy various coherence laws, see~\cite{kelly1974coherence} and~\cite{laplaza1972coherence} for details. 
\end{definition}
In the sequel all rig categories are assumed to be symmetric without further warning. This means that also the monoidal category $(\cat{C},\otimes, I)$ is symmetric in manner coherent with the rest of the structure. A \emph{rig dagger category}\index[word]{dagger category!rig} is a rig category with a dagger such that all the coherence isomorphisms are unitary and moreover both monoidal structures are dagger functors $\cat{C}\times\cat{C}\to \cat{C}$.

\begin{example}\label{ex:rigcats}
  \begin{itemize}
    \item \Pfn,\index[symb]{\Pfn, sets and partial functions} the category of sets and partial functions, is a rig category: $\oplus$ is given by the disjoint union and $\otimes$ by the cartesian product (in \cat{Set}).
    \item \PInj\ \index[symb]{\PInj, sets and partial injections}is a rig dagger category with the same interpretations for $\oplus$ and $\otimes$. 
    \item \cat{Hilb}\index[symb]{\cat{Hilb}, Hilbert spaces and bounded linear maps} is a rig dagger category, when $\oplus$ is given by the dagger biproduct and $\otimes$ by the tensor product. The subcategory \Hilb\ of Hilbert spaces and non-expansive linear maps inherits a rig dagger structure from this, even though $\oplus$ no longer satisfies a universal property.\index[symb]{\Hilb, Hilbert spaces and non-expansive maps}
    \item \Ban,\index[symb]{\Ban, Banach spaces and non-expansive maps} the category of Banach spaces and non-expansive linear maps, is a rig category, when $\oplus$ is interpreted as the $\ell^1$-sum, \ie the coproduct of \Ban\ and $\otimes$ is the projective tensor product. Formally, this is the completion of the algebraic tensor product $A\otimes B$ under the norm 
      \[\norm{x}=\inf \{\sum_{k=1}^n\norm{a_k}\norm{b_k}|x=\sum_{k=1}^n a_k\otimes b_k\}\]. 
  \end{itemize}
\end{example}


We won't define in detail a notion of a morphism of rig categories. Roughly speaking they are functors that preserve the rig structure up to coherent natural isomorphisms (or transformations, depending on whether one wants the stronger or weaker notion). Intuitively, it is clear that the inclusions $\PInj\to \Pfn$ and $\Hilb\to\cat{Hilb}$ satisfy this. 

More interestingly, there are functors \index[symb]{$\ell^2$, a functor $\PInj\to\Hilb$}$\ell^2\colon\PInj\to\Hilb$ and $\ell^1\colon\Pfn\to\Ban$ and these turn out to be compatible with the rig structure. The functor $\ell^2$ is defined on objects by 
  \[A\mapsto \ell^2{A}:=\{\phi\colon A\to \C \mid \sum \abs{\phi(A)}^2<\infty\}\] 
and its action on a partial injection $f\colon A\to B$ is given by 
  \[\ell^2(f)(\phi)(b):=\sum_{a\in f^{-1}b}\phi(a)
                            =\begin{cases} \phi(f^\dag a)\text{ if }f^\dag a\text{ is defined,} \\
                                                                0\text{ otherwise.}\end{cases}\]
The functor $\ell^2$ is known to preserve the dagger and the two monoidal structures $\oplus$ and $\otimes$, see~\cite{heunen:ltwo}. 

The functor \index[symb]{$\ell^1$, a functor $\Pfn\to\Ban$}$\ell^1$ is defined similarly: an object $A$ is sent to the Banach space \[\ell^1{A}:=\{\phi\colon A\to \C \mid \sum \abs{\phi(A)}<\infty\}\] and a partial function $f\colon A\to B$ is sent to the contraction given by 
  \[\ell^1(f)(\phi)(b):=\sum_{a\in f^{-1}b}\phi(a).\]
One usually thinks of $\ell^1$ as a functor $\cat{Set}\to\Ban$, in which case it is the left adjoint to the unit ball functor $U\colon \Ban\to\Set$ defined by $U(B):=\{v\in B\mid \norm{v}\leq 1\}$.

\begin{proposition}\label{prop:ell1preservesrigstructure}\index[symb]{$\ell^1$, a functor $\Pfn\to\Ban$} The functor $\ell^1\colon \Pfn\to\Ban$ preserves the rig structure, \ie  we have natural isomorphisms \[\ell^1(A\oplus B)\cong\ell^1(A)\oplus\ell^1(B)\qquad\text{and}\qquad \ell^1(A\otimes B)\cong\ell^1(A)\otimes\ell^1(B).\] 
\end{proposition}

\begin{proof}
  Write $i\colon\cat{Set}\to\Pfn$ for the usual inclusion. Then $\ell^1\circ i$ has a right adjoint $U$ and hence preserves colimits. Thus $\ell^1(A\oplus B)=\ell^1(i(A+B))\cong\ell^1(A)\oplus\ell^1(B)$ as desired. For $\otimes$, we calculate in terms of Schauder bases: $\ell^1(A)$ has a canonical Schauder basis indexed by $A$: the basis vector corresponding to $a\in A$ is given by the Dirac delta $\delta_a$ defined by
    \[\delta_{a}(b):=\begin{cases}1\text{ if }a=b \\ 
                                  0\text{ otherwise.}\end{cases}\]
  Hence $\ell^1(A\otimes B)$ has a canonical basis indexed by $A\times B$. Consider the function sending the basis element $\delta_{(a,b)}$ to $\delta_{a}\otimes\delta_b$. This induces an isometric isomorphism from the algebraic span of $\{\delta_{(a,b)}\mid (a,b)\in A\times B\}$ to the algebraic span of $\{\delta_a\otimes\delta_b\mid (a,b)\in A\times B\}$. Since both of these are dense subspaces, this induces an isometric isomorphism $\ell^1(A\otimes B)\to\ell^1(A)\otimes\ell^1(B)$, \ie an isomorphism in \Ban.
\end{proof}

The functors $\ell^1$ and $\ell^2$ in fact preserve much more structure. To make this precise, we will formalize a wide class of chains that always have ambilimits in our categories of interest. A naive attempt at formalizing this might be to require the existence of ambilimits of chains of split monics. However, this is asking for too much: given a chain of split monics, there might be many chains of epis splitting them, and only some pairs of chains might admit an ambilimit. 

A concrete example of this is the following. Let $D_2\colon\omega\to \Pfn$ be given by 
	\[(n\leq m)\mapsto \{1\dots,n\}\hookrightarrow \{1,\dots ,m\}.\]
 The ``correct'' way of splitting the inclusion $\{1\dots,n\}\hookrightarrow \{1,\dots ,m\}$ is via the map 
  \[k\mapsto\begin{cases} k\text{ if }k\leq n \\
                          \text{undefined, otherwise} \end{cases}\] 
since this results in a chain $D_1$ such that $(D_1,D_2)$ has an ambilimit given by $\N$. If instead one splits each inclusion by the map $k\mapsto \min \{k,n\}$, the limit is $\N\cup\{\infty\}$ and the canonical map  is the inclusion $\N\to \N\cup\{\infty\}$ which is not an isomorphism.

To avoid pathologies like this, we need restrict to a smaller class of diagrams that we require to have ambilimits. Note that in our categories of interest, the unit for $\oplus$ is in fact a zero object. This gives rise to canonical projections and inclusions 
\begin{equation*} A\xrightarrows[$p_A$][$i_A$] A\oplus B\xrightarrows[$i_B$][$p_B$] B \end{equation*} 
subject to the biproduct equations~\eqref{def:biprodeqs} and~\eqref{def:biprodswithZeros}, even though $A\oplus B$ need not satisfy a universal property (indeed, in \Hilb\ it does not). The projection $p_B\colon A\oplus B\to B$ can be defined as the composite $\lambda_{B}^\oplus \circ (0\oplus\id )\colon A\oplus B\to 0\oplus B\cong B$, and the other morphisms are defined similarly.  

To keep in line with domain theoretic terminology, we call the pair $(i_A,p_A)$ an embedding-projection pair\index[word]{embedding-projection pair}. Similarly, a chain $(D_1,D_2)\colon \omega\op\to\cat{C}_\leftrightarrows$ is called an \emph{embedding-projection chain}\index[word]{embedding-projection chain} if it is of the form
  \[A_0\xrightarrows[$p_0$][$i_0$] A_0\oplus A_1 \xrightarrows[$p_{0,1}$][$i_{0,1}$] A_0\oplus A_1\oplus A_2\xrightarrows\dots\]
Similarly, chains of the form 
    \[A_0\sxto{i_0} A_0\oplus A_1\sxto{i_{0,1}} A_0\oplus A_1\oplus A_2\sxto{}\dots\]
 are called chains of embeddings and chains of the form
    \[A_0\sxfrom{p_0}A_0\oplus A_1\sxfrom{p_{0,1}}A_0\oplus A_1\oplus A_2\sxfrom{}\dots\]
 are called chains of projections. 

\begin{definition}
A (dagger) rig category is $\omega$-continuous\index[word]{rig category!$\omega$-continuous}\index[word]{dagger category!$\omega$-continous rig} if 
  \begin{itemize}
    \item The monoidal unit $0$ is in fact a zero object.
    \item It has ambilimits of embedding-projection chains. (It has dagger colimits of chains of isometries.)
    \item The functors $X\mapsto A\oplus X$  and  $X\mapsto A\otimes X$ preserve them. 
  \end{itemize} 
\end{definition}

It is not immediately obvious from the definition that an $\omega$-continuous rig dagger category is $\omega$-continuous as a rig category. However, it is easy to show that this holds. Because $(\cat{C},\oplus)$ is dagger monoidal, the canonical inclusions and projections are dagger monics and dagger epimorphisms respectively. Hence every chain of embeddings has a dagger colimit, and the normalization and independence equations guarantee that it is in fact an ambilimit for the corresponding embedding-projection chain.

The main reason for defining $\omega$-continuous rig categories is that they provide a natural setting for building canonical fixed points for polynomial functors. We give the general theory before verifying that the categories we have studied so far in this section are $\omega$-continuous.

\begin{definition} Let \cat{C} be a rig category. A \emph{rig-polynomial functor}\index[word]{rig-polynomial functor} on \cat{C} is a functor of the form \[X\mapsto (\bigoplus_{i=0}^{n } A_i\otimes X^{\otimes i}))\]

If \cat{C} is a  a rig dagger category that has (anti)symmetrized tensor powers, then a generalized polynomial endofunctor is defined as a rig-polynomial endofunctor, except that some of the powers $X^{\otimes i}$ can be replaced by (anti-)symmetrized tensor powers from example~\ref{ex:symmetrictensor}.\index[word]{symmetrized tensor power}\index[word]{antisymmetrized tensor power}\index[symb]{$\sym^n X$ the $n$-th symmetric tensor power of $X$}
\end{definition}

Note that when $n=0$ we have $X^{\otimes n}\cong 1$ by convention. 

\begin{lemma}\label{lem:preservationofepchains}\index[word]{embedding-projection chain} The functors $B\oplus -$, $B \otimes -$ and $X\mapsto X^{\otimes n}$ preserve embedding-projection chains up to isomorphism. In a rig dagger category, they preserve chains of embeddings up to unitary isomorphism.
\end{lemma}

\begin{proof}
	It is clear that applying the endofunctor $B\oplus -$ to an embedding-projection chain results in an embedding-projection chain. This is also true  $B \otimes -$ since
  \[B\otimes A_0\xrightarrows[$\id\otimes p_0$][$\id\otimes i_0$] B\otimes (A_0\oplus A_1)\xrightarrows[$\id\otimes p_{0,1}$][$\id\otimes i_{0,1}$] B\otimes (A_0\oplus A_1\oplus A_2)\xrightarrows\dots\]
is isomorphic to 
\[B_0\xrightarrows[$p_0$][$i_0$] B_0\oplus B_1 \xrightarrows[$p_{0,1}$][$i_{0,1}$] B_0\oplus B_1\oplus B_2\xrightarrows\dots\]
where $B_i=B\otimes A_i$. 
	Finally, an easy calculation shows that applying $X\mapsto X^{\otimes n}$ to the chain 
	 \[A_0\xrightarrows[$p_0$][$i_0$] A_0\oplus A_1 \xrightarrows[$p_{0,1}$][$i_{0,1}$] A_0\oplus A_1\oplus A_2\xrightarrows\dots\]
	 results in a chain isomorphic to 
	\[B_0\xrightarrows[$$][$$] B_1 \xrightarrows[$$][$$] B_2\xrightarrows\dots\]
	where 
		\[B_i=\bigoplus_{f\colon\{1,\dots n\}\to\{0,\dots i\}} A_{f(1)}\otimes\dots\otimes A_{f(n)}\]
	and the embedding-projection pair $B_i\to B_{i-1}$ comes from the inclusion $\{0,\dots,i\}\hookrightarrow\{0,\dots,i+1\}$.	
In a rig dagger category, all of these isomorphisms are unitary.
\end{proof}


Note that the functor $X\mapsto X^{\oplus n}$ is naturally isomorphic to $X\mapsto (I^{\oplus n}\otimes X)$ and hence automatically preserves ambilimits of embedding-projection chains in an $\omega$-continuous rig category. More surprisingly, the functor $X\mapsto X^{\otimes n}$ also preserves such ambilimits.

\begin{lemma}\label{lem:tensorpowercontinuous}\index[word]{embedding-projection chain} Let \cat{C} be an $\omega$-continuous rig category. Then the functor $X\mapsto X^{\otimes n}$ preserves ambilimits of embedding-projection chains for any $n\in \N$. If \cat{C} is an $\omega$-continuous rig dagger category, then the functor $X\mapsto X^{\otimes n}$ preserves dagger colimits of chains of isometries for any $n\in \N$.
\end{lemma}

\begin{proof} 
  Since $X\mapsto X^{\otimes n}$ preserves  embedding-projection chains by Lemma~\ref{lem:preservationofepchains} it suffices to show that it preserves colimits of chains of embeddings by Lemma~\ref{lem:preservationofambilims}. 
  So let $D\colon\omega\to\cat{C}$ be such a chain. Now we have canonical isomorphisms 
  \begin{align*}
    (\colim_i D(i))^{\otimes n}&= (\colim_{i_1} D(i_1))\otimes (\colim_i D(i))^{\otimes n-1} \\
      &\cong  \colim_{i_1} (D(i_1)\otimes (\colim_i D(i))^{\otimes n-1} )\text{ by }\omega\text{-continuity of }\cat{C}\\
      &\cong \dots\cong \colim_{i_1,\dots i_n} D(i_1)\otimes\dots\otimes D(i_n) 
  \end{align*}
  so it suffices to prove that \[\colim_{i_1,\dots i_n} D(i_1)\otimes\dots\otimes D(i_n)\cong \colim_i D(i)^{\otimes n}.\]
  This is not hard to show directly, but follows readily from more general principles. It is easy to check that the diagonal functor $\omega\to\omega^n$ is final in the sense of~\cite[IX.3]{maclane:categories} (since its image is cofinal in the order-theoretic sense) and hence the colimits agree by~\cite[Theorem IX.3.1]{maclane:categories}. 

  For the dagger case, note that the calculation above shows that a colimit of isometries is taken to a colimit of isometries by  $X\mapsto X^{\otimes n}$. Since $X\mapsto X^{\otimes n}$ is a dagger functor it automatically preserves cocones satisfying normalization and independence, and hence it also preserves \emph{dagger} colimits of isometries.
\end{proof}

In a sense, the proof above boils down to a special case of the general fact that directed colimits can be built from colimits of chains, see~\cite[Corollary 1.7]{adamek:accessible}. 

\begin{theorem}\label{thm:fixedpointsinrigcats}\index[word]{rig category!$\omega$-continuous}\index[word]{dagger category!$\omega$-continous rig}\index[word]{canonical fixed point}\index[word]{dagger fixed point}\index[word]{rig-polynomial functor} If \cat{C} is an $\omega$-continuous rig category, then every rig-polynomial endofunctor has a canonical fixed point.

If \cat{C} is an $\omega$-continuous rig dagger category, then every generalized rig-polynomial dagger endofunctor has a dagger fixed point.
\end{theorem}

\begin{proof}
  It is easy to check that $\omega$-continuous rig categories with zero objects are closed under countable products and that any tupling of functors preserving ambilimits of embedding-projection chains also preserves such ambilimits. Moreover, projection functors and constant functors preserve such ambilimits. As any rig-polynomial functor can be built as a composite of these and the basic building blocks $X\mapsto A\otimes X$, $X\mapsto X^{\otimes n}$ and $X\mapsto A\oplus X$, it suffices to show that the basic building preserve ambilimits of embedding-projection chains: $X\mapsto A\oplus X$ and $X\mapsto A\otimes X$ do so by assumption and $X\mapsto X^{\otimes n}$ does so by Lemma~\ref{lem:tensorpowercontinuous}. Hence
   any rig-polynomial functor satisfies the assumptions of Theorem~\ref{thm:ambiadamek} and hence has a canonical fixed point. 
  In the dagger case, one uses Theorem~\ref{thm:daggeradamek},  and the argument above can be easily adapted to the dagger case. The only new thing to check is that (anti)symmetrized tensor powers preserve colimits of chains of isometries. This is the case by Theorem~\ref{thm:limsandcolimscommute}, since (anti)symmetrized tensor products are dagger limits of diagrams with only unitary morphisms, and hence the adjointability condition is automatically satisfied. Thus any generalized rig-polynomial endofunctor satisfies the assumptions of Theorem~\ref{thm:daggeradamek} and hence has a dagger fixed point.
\end{proof}

\begin{theorem}\index[symb]{\PInj, sets and partial injections}\index[symb]{\Pfn, sets and partial functions}\index[symb]{\Ban, Banach spaces and non-expansive maps}\index[symb]{\Hilb, Hilbert spaces and non-expansive maps} The categories  \PInj\ and \Hilb\ are $\omega$-continous rig dagger categories. \Pfn\ and \Ban\ are $\omega$-continuous rig categories. Moreover, the functor $\ell^1\colon \Pfn\to\Ban$ preserves ambilimits of chains of embeddings.
\end{theorem}

\begin{proof}
    We start with \PInj. Clearly the empty set is both a zero object and a monoidal unit for disjoint unions. Moreover, \PInj\ is known to have colimits of chains~\cite[Proposition 3.6]{heunen:ltwo}: for a chain of isometries $D\colon \omega\to\PInj$ the colimit $L$ can be calculated as 
    \[ L:=\coprod D(i)/\approx \]
  where the coproduct is taken in \cat{Set} and the equivalence relation $\approx$ is generated by $x\approx D(n\to n+1)x$. Hence it remains to check that $X\mapsto A\oplus X$ and $X\mapsto A\otimes X$ preserve colimits of such chains. This follows easily once one observes that a chain of isometries in \cat{PInj} is also a chain of monomorphisms in \cat{Set} and that the construction of the colimit agrees in both categories. Since the disjoint union is a colimit in \cat{Set} and colimits commute with colimits, $X\mapsto A\oplus X$ preserve colimits of chains of isometries. The functor $X\mapsto A\otimes X$ preserves such colimits since since $\otimes$ is the categorical product in set and filtered colimits commute with finite limits in \cat{Set}~\cite[Theorem IX.2.1]{maclane:categories}. 

  Next we move on to \Hilb. Clearly the zero-dimensional Hilbert space is both a zero object and a monoidal unit for the direct sum. It is known to be $\omega_1$-accessible~\cite[Example 2.3.9]{adamek:accessible}, which implies that it has colimits of isometries. Moreover, the functor $\ell^2\colon \cat{PInj}\to \Hilb$ preserves the rig structure up to (unitary) isomorphism and preserves colimits of chains of isometries, see \eg \cite{heunen:ltwo} and~\cite{barr:algebraically}.

  To prove that constant multiplication and addition preserve these colimits, an explicit description in terms of a (Schauder) basis will be useful. Let $D\colon\omega\to \Hilb$ be a chain of isometries. We may identify $D(n)$ with a closed subspace of $D(n+1)$, and using this identification, choose a Schauder basis $I_n$ for each $D(n)$ such that $I_n\subset I_{+1}$. Then the colimit of $D$ is characterized by having a basis $\bigcup_{n}I_n$. From this description it is easy to check that the colimit is in fact a dagger colimit. Instead of computing directly in \Hilb, we use the functor $\ell^2$: the functor $n\mapsto I_n$ defines a functor $D'\colon \omega\to\cat{PInj}$ and moreover there is an obvious unitary isomorphism between $D$ and $\ell^2\circ D'$. Hence any colimit of a chain of isometries can be first computed in \cat{PInj}, whence $\omega$-continuity of \cat{PInj} implies that of \Hilb.

  In \Pfn\ the colimit
     \[A_0\hookrightarrow A_0\oplus A_1 \hookrightarrow A_0\oplus A_1\oplus A_2\hookrightarrow\dots\]
    is built as in \PInj\ (or \cat{Set}) and is given by the disjoint union  $L:=\coprod_{k} A_k$
    with the inclusions $i_k\colon \bigoplus_{i=1}^k A_i\to L$. Define projections $p_k\colon L\to \bigoplus_{i=1}^k A_i$, using the universal property of the \emph{coproduct} 
    by defining 
      \[A_n\hookrightarrow L\sxto{p_k} \bigoplus_{i=1}^k A_i:=\begin{cases} A_n\hookrightarrow\bigoplus_{i=1}^k A_i\text{ if }n\leq k\\
                                                                            A_n\sxto{0}\bigoplus_{i=1}^k A_i\text{ otherwise.}\\
                                                                \end{cases}\]
  It is easy to check that $(L,(p_n,i_n))$ satisfies normalization, so it suffices to check that $(L,(p_n))$ is a limit of 
     \[A_0\sxfrom{p_0}A_0\oplus A_1 \sxfrom{p_{0,1}} A_0\oplus A_1\oplus A_2\sxfrom\dots\]
  So, consider a cone $(f_k\colon X\to \bigoplus_{i=1}^k A_i)_k$. Define $f\colon X\to L$ by 
   \[f(x)=\begin{cases} i_kf_k(x)\text{, where }k\text{ is least such that }f_k(x)\text{ is defined}\\
              \text{undefined, if no such }k\text{ exists.}\end{cases}\]
  It is now straightforward to check that $f$ satisfies $p_kf=f_k$ and is unique in that respect. Hence \Pfn\ has ambilimits of chains of embeddings. Next we prove that the functors 
   $X\mapsto A\oplus X$  and  $X\mapsto A\otimes X$ preserve them. As both of these functors preserve  embedding-projection chains by Lemma~\ref{lem:preservationofepchains}, it it is enough to show that both of these functors preserve the universal property of the colimit  by Lemma~\ref{lem:preservationofambilims} . This is done as in the case of \PInj: $X\mapsto A\oplus X$ is cocontinuous since colimits commute with colimits, whereas 
   $X\mapsto A\otimes X$ preserves colimits of countable chains since the colimit of a chain \Pfn\ is computed as in \cat{Set} and filtered colimits commute with finite limits in \cat{Set}~\cite[Theorem IX.2.1]{maclane:categories}.

  For \Ban, we first note that \Ban\ is complete and cocomplete. It is locally presentable by \cite[Example 1.48]{adamek:accessible}, hence cocomplete by definition and complete by \cite[Remark 1.56]{adamek:accessible}: alternatively, this has been proven directly in~\cite{semadeni:baniscompleteandcocomplete}.

   The colimit of 
    \[A_0\sxto{i_0} A_0\oplus A_1\sxto{i_{0,1}} A_0\oplus A_1\oplus A_2\sxto{}\dots\]
  is in fact the coproduct $\coprod_{k} A_k$ and the limit of  
    \[A_0\sxfrom{p_0}A_0\oplus A_1\sxfrom{p_{0,1}}A_0\oplus A_1\oplus A_2\sxfrom{}\dots\]
  can be seen via the usual construction as a certain subobject of the product $\prod_k (\oplus_{j=1}^k A_k)$. Since products in \Ban\ are $\ell^\infty$-sums, this is concretely given by the vector space
    \[B:=\{(b_k)_{k\in\N}\mid b_k\in \oplus_{j=1}^k A_k,\sup_k\norm{b_k}<\infty,p_{0,\dots, k}(b_{k+1})=b_k\}\]
  with the $\ell^\infty$-norm $\norm{(b_k)_{k\in\N}}:=\sup_k\norm{b_k}$
  The condition $p_{0,\dots, k}(b_{k+1})=b_k$ guarantees that the sequence can be equivalently viewed as a sequence $(a_k)_{k\in\N}$ with $a_k\in A_k$. Since $\oplus_{j=1}^k A_k$ has the $\ell^1$-norm, the condition $\sup\norm{b_k}<\infty$ then implies that 
    \[\sum_{k}\norm{a_k}=\sup_k\sum_{j=1}^k \norm{a_j}=\sup_k\norm{b_k} <\infty \]
  which amounts to saying that $(a_k)_{k\in\N}\in\coprod_{k} A_k$. This amounts to us having constructed an isometric inverse to the canonical map $\coprod_{k} A_k\to B$, so \Ban\ has ambilimits of chains of embeddings by Theorem~\ref{thm:ambilimitiffiso}. 

  Next we need to show that the functors $X\mapsto A\oplus X$ and  $X\mapsto A\otimes X$ 
  preserve ambilimits of embedding-projection chains. Since colimits commute with colimits, $X\mapsto A\oplus X$ is cocontinuous. Because \Ban\ is closed (see \eg~\cite[Chapter II]{cigler:banach}), the functor $X\mapsto A\otimes X$ is also cocontinous. Since both of these functors preserve embedding-projection chains, they preserve ambilimits by Lemmas~\ref{lem:preservationofepchains} and ~\ref{lem:preservationofambilims} .

  For $\ell^1$, note first that the composite $\cat{Set}\hookrightarrow\Pfn\sxto{\ell^1}\Ban$ is a left adjoint and hence cocontinuous. Since colimits of split monics are computed in \Pfn\ as they are in \cat{Set} and moreover $\ell^1$ cooperates with the rig structure and hence preserves  embedding-projection chains, it preserves ambilimits by Lemma~\ref{lem:preservationofambilims}.

\end{proof}

One might be tempted to prove cocontinuity of $X\mapsto A\oplus X$ in \Hilb\ by referring to Theorem~\ref{thm:limsandcolimscommute}. Unfortunately, this does not quite work as-is: the direct sum $H\oplus K$ of Hilbert spaces is a dagger product in \cat{Hilb} but not in \Hilb, whereas the colimits of isometries are colimits in the latter but not in the former category. However, one can get around this and provide an alternative proof as follows: the colimits of \Hilb\ are \emph{bounded} colimits in \cat{Hilb}, meaning that they satisfy a universal property for every jointly bounded cocone. One could then develop a theory of bounded dagger colimits and show that Theorem~\ref{thm:limsandcolimscommute} still holds. We have chosen not to do so since doing so would take us too far afield. Moreover the proof above highlights an interesting similarity between \PInj\ and \Hilb, where both cases boil down to (ordinary) polynomial functors over \cat{Set} having initial algebras.

One could avoid a direct proof for \Hilb\ by leveraging properties of \PInj\ and $\ell^2$. Similarly, one might hope to simplify the proof above by reducing the case of \Ban\ to properties of \Pfn\ and $\ell^1$. However, there is a crucial difference that makes this impossible: every Hilbert space is unitarily isomorphic to a space of the form $\ell^2(A)$, but not every Banach space is isometrically isomorphic to a a space of the form $\ell^{1}(A)$.

\Ban\ is enriched over itself and hence also over metric spaces. However, ambilimits of embedding-projection chains in \Ban\ do not fall under the theory covered in \cite{america:metric,adamek:banach,birkedal:metric}: they wish the chain $D$ to satisfy $\norm{e_{n}q_{n}-\id}\to 0$ as $n\to\infty$, whereas embedding-projection chains typically satisfy $\norm{e_n q_n-\id}=1$ for all $n$. This corresponds with the fact that, for the ambilimit in question, the sequence $i_n p_n$ converges to $\id[L]$ only strongly but not in the operator norm.

\begin{example} 
  \begin{itemize}
    \item The Hilbert space $\ell^2(\N)$ is a dagger fixed point for the endofunctor  $-\oplus\C$. Indeed, by construction, the dagger fixed point is given as the dagger colimit of the sequence $\C\to\C^2\to \C^3\to\dots$, where $\C^{n}\to \C^{n+1}$ is given by inclusion into the first $n$ coordinates. This colimit is given by $\ell^2(\N)$. Being a dagger colimit of this sequence explains why the inclusions $i_n\colon \C\to\ell^2(\N)$ satisfy infinitary biproduct equations, while not being a biproduct. This also clarifies the universal property of $\ell^2(\N)$, which is weaker than that of a biproduct. This discussion lifts more generally to $\ell^2$-sums of any cardinality (though for a $\kappa$-ary sum one needs a $\kappa$-filtered colimit).
  \item  Quantum lists of type $A$ can be modelled as $\oplus_{i=1}^{\infty} A^{\otimes i}$, \ie as the dagger fixed point for the functor $\mathbb{C} \oplus (A\otimes -)$.
  \item  Quantum trees of type $A$ can be modelled as the dagger initial algebra for ${A\oplus (-)^{\otimes 2}}$
\end{itemize}
\end{example}

Note that all of the polynomial endofunctors extend to \cat{Hilb}, and admit unitary solutions to $F(X)\cong X$. However, these solutions no longer satisfy a universal property. For instance, the codiagonal $\mathbb{C}^2\to\mathbb{C}$ is an algebra of $-\oplus \mathbb{C}$, but there is no algebra homomorphism from $\ell^2(\N)$ to it.



\index[word]{Fock space}\index[word]{symmetrized tensor power}\index[word]{antisymmetrized tensor power}\index[symb]{$\sym^n X$ the $n$-th symmetric tensor power of $X$}The symmetrized Fock space of type $A$ is defined as $\bigoplus_{n\in\N} \sym^n A$  and the antisymmetrized Fock space is defined as $\bigoplus_{n\in\N} \sym^n_{-1} A$. Since the infinitary sum $\bigoplus$ is the $\ell^2$-sum, which can be seen as a dagger colimit of a chain of isometries, one might hope to understand Fock spaces as dagger fixed points. This would result in another categorical understanding of Fock spaces, in contrast with~\cite{bluteetal:fockspace,fiore:fockspace,vicary:fockspace}. Unfortunately, I have not been able to make this work. The issue is the following: while $\bigoplus_{n\in\N} A^{\otimes n}$ can be built as the dagger fixed point of $X\mapsto I\oplus(A\otimes -)$, there is no obvious way to replace the term $(A\otimes -)$ with something that would generate higher (anti)symmetrized powers of $A$. This is a shame since one would also like to use the Fock space of type $A$ to model unordered lists \ie sets of type $A$.


\section{Traces}

In this section we recall some background material on traced symmetric monoidal categories and then go on to show that the ``standard sum formula'' defines a trace on \Ban, with the functor $\ell^1\colon\Pfn\to\Ban$ being traced. We start by recalling the definition of a traced symmetric monoidal category~\cite{joyalstreetverity:traced}.

\begin{definition}\index[word]{trace} A trace on a symmetric monoidal category consists of a family of functions $\Tr^U_{A,B}\colon\cat{C}(A\otimes U,B\otimes U)\to \cat{C}(A,B)$ satisfying the following axioms:
  \begin{itemize}
    \item Naturality in $A$ and $B$: for all $g\colon A\otimes U\to B\otimes U$, $f\colon A'\to A$ and $h\colon B\to B'$,
        \[h\circ \Tr^U_{A,B}(g)\circ f=\Tr^U_{A',B'}((h\otimes\id[U])\circ g\circ (f\otimes \id[U])).\]
    \item Dinaturality in $U$: for all $f\colon A\otimes U\to B\otimes U'$ and $g\colon U'\to U$,
        \[\Tr^U_{A,B}((\id[B]\otimes g)\circ f)=\Tr^{U'}_{A,B}(f\circ(\id[A]\otimes g)).\]
    \item Strength: for all $f\colon A\otimes U\to B\otimes U$ and $g\otimes C\otimes D$, 
      \[g\otimes \Tr^U_{A,B}(f)=\Tr^U_{C\otimes A,C\otimes B}(g\otimes f).\]
    \item Vanishing I: for all $f\colon A\otimes I\to B\otimes I$,
      \[f=\Tr^I_{A,B}(f).\]
    \item Vanishing II: for all $f\colon A\otimes U\otimes V\to B\otimes U\otimes V$,
      \[\Tr^{U\otimes V}_{A,B}(f)=\Tr^U_{A,B}(\Tr^V_{A\otimes U,B\otimes U}(f)).\]
    \item Yanking: for all $A$, 
      \[\Tr^A_{A,A}(\sigma_{A,A})=\id[A].\]
  \end{itemize}
\end{definition}

Graphically the trace is usually depicted as follows:
  \[\Tr^U_{A,B}(\begin{pic} 
    \node[morphism] (f) {$f$};
    \draw ([xshift=-.5mm]f.south west) to +(0,-.5) node [left] {$A$};
    \draw ([xshift=.5mm]f.south east) to +(0,-.5) node [right] {$U$};
    \draw ([xshift=-.5mm]f.north west) to +(0,.5) node [left] {$B$};
    \draw ([xshift=.5mm]f.north east) to +(0,.5) node [right] {$U$};
  \end{pic})=:
  \begin{pic} 
    \node[morphism] (f) {$f$};
    \draw ([xshift=-.5mm]f.south west) to +(0,-.5) node [left] {$A$};
    \draw ([xshift=.5mm]f.south east) to[in=180,out=-90] +(.25,-.25) to[in=-90,out=0] +(.25,.25);
    \draw ([xshift=-.5mm]f.north west) to +(0,.5) node [left] {$B$};
    \draw ([xshift=.5mm]f.north east) to[in=180,out=90] +(.25,.25) to[in=90,out=0] +(.25,-.25) to +(0,-.4);
  \end{pic}\]
With this convention at hand, the axioms for a trace become quite natural. The reader is invited to either draw the axioms on their own or to consult one of~\cite{joyalstreetverity:traced,selinger:graphicallanguages} for more details. 

In a rig category $(\cat{C},\oplus,\otimes)$  there are two monoidal products, either one of which might be traced. In such a setting, one often calls a trace on $(\cat{C},\oplus)$  ``sum-like'' or ``particle-like'', especially when $\oplus$ is an actual coproduct. Similarly, a trace on $(\cat{C},\otimes)$ is called a ``product-like trace'' or a ``wave-like trace'', especially when  $\otimes$ is a categorical product. 
See~\cite{haghverdiscott:goi,abramsky:retracing,abramskyetal:goiandlca} for further discussion. 


In our context, we are primarily interested in particle-like traces. Even when $\oplus$ isn't a biproduct (let alone a coproduct), one might have enough structure to express each morphism $f\colon A\oplus U\to B\oplus U$ as a matrix 
    \[\begin{bmatrix}
     f_{A,B}\colon A\to B & f_{U,B}\colon U\to B \\
     f_{A,U}\colon A\to U & f_{U,U}\colon U\to U
    \end{bmatrix}\]
and then hope to define a trace by the ``standard trace formula'' 
    \begin{equation*}\label{eq:traceformula}\Tr^U_{A,B}(f):=f_{A,B}+\sum_{n\in\mathbb{N}} f_{U,B}\circ f^{n}_{U,U}\circ f_{A,U},\end{equation*}
provided the sum is defined. In so-called unique decomposition categories, if these sums are always defined, then the standard formula defines a trace~\cite[Proposition 4.0.11]{haghverdi:thesis}. Moreover, even when this sum isn't always defined, it always defines a \emph{partial trace},  (introduced~\cite{haghverdiscott:typedgoi}, studied further in~\cite{malherbe:partialtraces}). In fact every strong unique decomposition category can be embedded in a totally traced unique decomposition category~\cite{hoshino:representation}. 

The results on unique decomposition categories give sum-like traces for \Pfn\ and \PInj, and we cite these results below. Before doing so, we recall how sums are interpreted in these categories. In \Pfn, a sum $\sum_{i\in I} f_i$ is defined iff the functions $f_i$ have disjoint domains of definition, in which case the sum is defined to be the union of the morphisms $f_i$. In \PInj, a sum $\sum_{i\in I} f_i$ is defined iff the functions $f_i$ have disjoint domains of definition \emph{and} disjoint images, in which case the sum is defined to be the union of the $f_i$.

\begin{proposition}[\cite{haghverdi:udcs-goi-combinatorylogic}] \Pfn\ and \PInj\ have sum-like traces, given by the standard trace formula~\eqref{eq:traceformula}. 
\end{proposition}

However, even if maps admit a matrix representation as above, and (some) infinite sums make sense, the standard formula need not define a trace. For instance, the standard trace formula does not define a (partial) trace on $(\cat{Vect_\C},\oplus)$, when the sum is interpreted as weak convergence~\cite[Proposition 3.10]{malherbe:partialtraces}. 

Note that in \Ban\ the underlying vector space of $A\oplus B$ is a biproduct in \cat{Vect_\C} and hence morphisms $f\colon A\oplus U\to B\oplus U$  admit a matrix representation. Hence one might wonder if the trace-formula defines a (partial) trace in \Ban. Given that the standard trace formula does not work in $(\cat{Vect_\C},\oplus)$, it might be surprising that it does work in \Ban\ when the sums involved are interpreted as convergent sums in the strong operator topology.

\begin{proposition}\label{prop:ell1traced}\index[symb]{$\ell^1$, a functor $\Pfn\to\Ban$} \Ban\ has a sum-like trace defined by standard trace formula~\eqref{eq:traceformula}. Moreover, the functor $\ell^1\colon\Pfn\to\Ban$ is traced.
\end{proposition}

\begin{proof}
  First we will show that the sum in the formula~\eqref{eq:traceformula} always converges strongly to a weak contraction. Since $f$ is weakly contractive and the monoidal product we are working over is given by the $\ell^1$-sum, for any $v\in A$ and $k\in U$ we have
    \[\norm{v}\geq\norm{f_{A,B}v}+\norm{f_{A,U}v} \qquad \norm{w}\geq\norm{f_{U,B}w}+\norm{f_{U,U}w} \]
  Piecing these together, we see that for any $v$ in $A$ and $k\in \mathbb{N}$ we have 
    \begin{align*}
      \norm{v}\geq\norm{f_{A,B}v}+\norm{f_{A,U}v}\geq \norm{f_{A,B}v}+\norm{f_{U,B}f_{A,U}v}+\norm{f_{U,U}f_{A,U}v} \\
      \geq\dots\geq \norm{f_{A,B}v}+\sum_{n=0}^k \norm{f_{U,B}f_{U,U}^n f_{A,U}v}
    \end{align*}
  This implies that the series $\sum_{n\in\mathbb{N}} f_{U,B}\circ f^{n}_{U,U}\circ f_{A,U}v$ is absolutely convergent and hence convergent. Moreover, the triangle inequality gives us 
     \[\norm{f_{A,B}v}+\sum_{n=0}^k \norm{f_{U,B}f_{U,U}^n f_{A,U}v}\geq\norm{f_{A,B}+\sum_{n=0}^k f_{U,B}\circ f^{n}_{U,U}\circ f_{A,U}}\] 
  whence 
    \[\norm{v}\geq\norm{f_{A,B}v}+\sum_{n\in\mathbb{N}} \norm{f_{U,B}f_{U,U}^n f_{A,U}v}\geq\norm{f_{A,B}+\sum_{n\in\mathbb{N}} f_{U,B}\circ f^{n}_{U,U}\circ f_{A,U}}=\norm{\Tr^U_{A,B}(f)v},\] 
  as desired.

  Proving the axioms for the trace amounts to straightforward calculations. In fact, the proof of~\cite[Proposition 4.0.11]{haghverdi:thesis} works verbatim, even though the proposition itself does not apply since \Ban\ is not a so-called unique decomposition category in an obvious manner.

  We conclude by proving that the functor $\ell^1$ is traced. Since $\ell^1$ preserves coproducts and hence the matrix representation of a morphism, it suffices to prove that it preserves countable sums, \ie that whenever $\sum_{n\in \mathbb{N}} f_n$ is defined in \cat{Pfn}, the sum $\sum_{n\in \mathbb{N}} \ell^1 (f_n)$ converges strongly to $\ell^1(\sum_{n\in \mathbb{N}} f_n)$.

  So, let $f_n\colon A\to B$ be given. For the sum $\sum_{n\in \mathbb{N}} f_n$ to be defined, the partial functions $f_n$ must have disjoint domain. Write $A_n:=\dom f_n$ and $A_{-1}:=A\setminus \bigcup_{n\in\N} A_n$. For $k\geq 1$ define $p_n\colon \ell^1(A)\to \ell^1 (A)$ by 
    \[ 
      p_n(\varphi)(a)=\begin{cases}\varphi(a)\text{ if }a\in A_n \\
      0\text{ otherwise.}
        \end{cases}
     \]
  Since $\{A_k|k\geq -1\}$ is a partition of $A$, we see that $\sum_{k\geq 1} p_k=\id$ in the strong operator topology. Moreover, $\ell^1(\sum_{n\in \N}f_n)p_k=\ell^1(f_k)$ for any $k\geq 0$, and $\ell^1(\sum_{n\in \N}f_n)p_{-1}=0$. Hence 
    \[\ell^1(\sum_{n\in \N}f_n)=\ell^1(\sum_{n\in \N}f_n)\sum_{k\geq -1}p_k=\sum_{k\geq -1}(\ell^1(\sum_{n\in \N}f_n)p_k)=\sum_{k\in \N}\ell_1(f_k),\]
  as desired.
%
%
\end{proof}

\section{A language for traced $\omega$-continuous rig categories}\label{sec:rigcatslanguage}

In this section we show how to interpret the $\Pi^o$-language studied in~\cite{bowman:dagger,carette:computingwithsemirings,jamessabry:infeff} in any traced $\omega$-continuous rig category. 

We start from the syntax, by listing our syntactic values and value types
\begin{align*}
  &(\text{Value Types}) &a:=0\mid  1 \mid (a_1\oplus a_2)  \mid (a_1\otimes a_2) \mid x\mid \mu x.a  \\
  &(\text{Values}) &v:= ()\mid \inl v\mid\inr v \mid (v,v)\mid \langle v\rangle \\
\end{align*}

The values are typed according to the following rules.
\[
  \AXC{} \UIC{$()\colon 1$}\DisplayProof \qquad \AXC{$v\colon a$} \AXC{$w\colon b$} \BIC{$(v,w)\colon a\otimes b$}\DisplayProof \qquad \AXC{$v\colon a$} \UIC{$\inl (v)\colon a\oplus b$} \DisplayProof \qquad \AXC{$v\colon a$} \UIC{$\inr (v)\colon b\oplus a$} \DisplayProof \qquad \AXC{$v\colon a[\mu x.a/x]$} \UIC{$\langle v\rangle\colon \mu x.a$} \DisplayProof
\]

Where the notation $a[b/x]$ means the type obtained by replacing each free occurrence of $x$ by $b$. 
Thinking categorically, our value types corresponds to objects. Besides value types, our language has combinators, which then correspond to morphisms. The collection of combinator types is given by $a\leftrightarrow b$ for value types $a,b$. We write $\leftrightarrow$ instead of $\rightarrow$ because our combinators are intended to be partial isomorphisms. The notation $f\colon a\leftrightarrow b\colon g$ means that $f$ is of type $a\leftrightarrow b$ and $g$ is of type $b\colon \leftrightarrow a$. The basic combinators, along with their types are given below. 
\begin{align*}
  &\id[a]\colon a\leftrightarrow a \\
  &\assoclplus_{a,b,c}\colon &a\oplus (b\oplus c)\leftrightarrow (a\oplus b)\oplus c  &\colon\assocrplus_{a,b,c} \\
  &\unitlplus_a\colon &0\oplus a\leftrightarrow a  &\colon\unitrplus_a\\ 
  &\swapplus_{a,b}\colon &a\oplus b\leftrightarrow b\oplus a \\
  &\assocltimes_{a,b,c}\colon &a\otimes (b\otimes c)\leftrightarrow (a\otimes b)\otimes c  &\colon\assocrtimes_{a,b,c} \\ 
  &\unitltimes_a\colon &  1\otimes a\leftrightarrow a &\colon\unitrtimes_a \\ 
  &\swaptimes_{a,b}\colon &a\otimes b\leftrightarrow b\otimes a \\ 
  &\distr_{a,b,c}\colon & (a\oplus b)\otimes c\leftrightarrow (a\otimes c)\oplus (b\otimes c)  &\colon\factor_{a,b,c} \\
  &\absorb_a\colon & 0\otimes a\leftrightarrow 0  &\colon\unabsorb_a\ \\
  &\fold_{\mu x.a}\colon &a[\mu x.a/x]\leftrightarrow \mu x.a &\colon\unfold_{\mu x.a}
\end{align*}

Basic combinators can be combined, using $\circ$, $\oplus$ and $\otimes$. The typing rules for the combinators built using these are as one might guess:
\[
  \AXC{$f\colon a\leftrightarrow b$} \AXC{$g\colon b\leftrightarrow c$}\BIC{$g\circ f\colon a\leftrightarrow c$}\DisplayProof \qquad \AXC{$f\colon a\leftrightarrow b$} \AXC{$g\colon c\leftrightarrow d$}\BIC{$f\oplus g\colon a\oplus c\leftrightarrow b\oplus d$}\DisplayProof \colon  \AXC{$f\colon a\leftrightarrow b$} \AXC{$g\colon c\leftrightarrow d$}\BIC{$f\otimes g\colon a\otimes c\leftrightarrow b\otimes d$}\DisplayProof
\]

 Moreover, we have a trace operator $\Tr$, together with the typing rule 
 \[\AXC{$f\colon a\oplus c\leftrightarrow b\oplus c$}\UIC{$\Tr(f)\colon a\leftrightarrow c$}\DisplayProof\]

We list the full description of the syntax in a single table to make it easier to refer back to. 
\begin{align*}
  &(\text{Value Types}) &a:=0\mid  1 \mid (a_1\oplus a_2)  \mid (a_1\otimes a_2) \mid x\mid \mu x.a  \\
  &(\text{Values}) &v:= ()\mid \inl v\mid\inr v \mid (v,v)\mid \langle v\rangle \\
  &(\text{Combinator Types}) &a_1\leftrightarrow a_2 \\ 
  &(\text{Basic Combinators}) &g:=\id\mid\assoclplus\mid\assocrplus\mid\unitlplus\mid\unitrplus\mid \\ 
  & &\swapplus\mid\assocltimes\mid\assocrtimes\mid\unitltimes\mid\unitrtimes\mid\swaptimes\mid\\ 
  && \distr\mid\factor\mid \absorb\mid\unabsorb\mid \fold\mid\unfold \\
  &(\text{Combinators}) &f:=g \mid f_1\circ f_2\mid f_1\oplus f_2\mid f_1\otimes f_2\mid \Tr(f) 
\end{align*}

Next we move on to semantics. 
Given a traced $\omega$-continuous rig category \cat{C}, it is now easy to define inductively an interpretation \sem{-} of $\Pi^o$ into \cat{C}. This will then be a functor from the syntactic category to the category \cat{C}, at least as long as we interpret the syntactic category as the category that has \emph{closed types} as its objects, \ie those types with no free occurrences of variables $x$. However, to interpret $\mu x.b$ we need to be able to interpret variables $x$. Formally, we need a finite assignment (function) $s=\{x_1\mapsto A_1,\dots x_n\mapsto A_n\}$. For such a function $s$, we define $s[x\mapsto A]$ to mean the function defined by 
  \[s[x\mapsto A](y)=\begin{cases} A\text{ if }y=x \\ 
   s(y) \text{ otherwise.}
    \end{cases}\]
Then, we define a partial interpretation $\sem{-}_s$ on all value types inductively. 
  \begin{align*}
    &\sem{0}_s=0_\cat{C} \\ 
    &\sem{1}_s=1_\cat{C}\\  
    &\sem{a\oplus b}_s=\sem{a}_s\oplus\sem{b}_s \\
    &\sem{a\otimes b}_s=\sem{a}_s\otimes\sem{b}_s\\
    &\sem{x}_s=\begin{cases} s(x)\text{ if  defined,} \\ 
\text{undefined otherwise.} \end{cases} \\
    &\sem{\mu x.b}_s=\text{the canonical fixed point of the polynomial functor }A\mapsto \sem{b}_{s[x\mapsto a]} 
  \end{align*}

Next we extend $\sem{-}_s$ to basic combinators. To avoid cluttering the notation, we omit $s$ from the indices.
  \begin{align*}
    &\sem{\id[a]}=\id[\sem{a}] \\
    &\sem{\assocrplus_{a,b,c}}=\alpha^\oplus_{\sem{a},\sem{b},\sem{c}} &\sem{\assoclplus_{a,b,c}}=\sem{\assocrplus_{a,b,c}}^{-1} \\
    &\sem{\unitlplus_{a}}=\lambda^\oplus_{\sem{a}}  &\sem{\unitrplus_{a}}=\sem{\unitlplus_{a}}^{-1}\\ 
    &\sem{\swapplus_{a,b}}=\sigma^{\oplus}_{\sem{a}} \\
    &\sem{\assocrtimes_{a,b,c}}=\alpha^\otimes_{\sem{a},\sem{b},\sem{c}} &\sem{\assocltimes_{a,b,c}}=\sem{\assocrtimes_{a,b,c}}^{-1} \\
    &\sem{\unitltimes_{a}}=\lambda^\otimes_{\sem{a}}  &\sem{\unitrtimes_{a}}=\sem{\unitltimes_{a}}^{-1}\\ 
    &\sem{\swaptimes_{a,b}}=\sigma^{\otimes}_{\sem{a}} \\
    &\sem{\distr_{a,b,c}}=\delta^l_{\sem{a},\sem{b},\sem{c}} &\sem{\factor{a,b,c}}=\sem{\distr_{a,b,c}}^{-1} \\
    &\sem{\absorb_{a}}=\tau^{r}_{\sem{a}} &\sem{\unabsorb{a}}=\sem{\absorb_{a}}^{-1} \\
    &\sem\fold_{\mu x.a}=\text{the canonical isomorphism }\sem{a[\mu x.a/x]}\to\sem{\mu x.a} &\unfold_{\mu x.a}=\sem\fold_{\mu x.a}^{-1}
  \end{align*}

The extension to all combinators is as one would guess: 
  \begin{align*} 
    &\sem{(g\oplus f)}=\sem{g}\oplus\sem {f} &\sem{g\circ f}=\sem{g}\circ\sem{f} \\
    &\sem{(g\otimes f)}=\sem{g}\otimes\sem {f} &\sem{\Tr(f)}=\Tr(\sem{f})\\
  \end{align*}
For types without free variables the choice of an assignment does not matter. An easy induction shows the following.
\begin{proposition} If $f\colon a\leftrightarrow b$ can be derived, then $\sem{f}\colon\sem{a}\to\sem{b}$.
\end{proposition}

\begin{proposition} If \cat{C} and \cat{D} are traced $\omega$-continuous rig categories and $F\colon\cat{C}\to\cat{D}$ is a traced $\omega$-continuous rig functor, then $\sem{-}_\cat{D}\cong F\circ\sem{-}_\cat{C}$
\end{proposition}

The last proposition implies that in each of our examples, the semantics \sem{-} factors via \PInj. This is not very surprising, since the language is designed to capture reversible features, and indeed it is logically reversible~\cite[Proposition 4.2]{jamessabry:infeff}. However, this might also be seen as a drawback or as a sign of low expressive power. Certainly if one wishes to use a variant of the language for naming morphisms in  \Hilb\ or \Ban\ one might hope to customize the language to suit better either of these categories. Both of these categories might benefit from having more inbuilt functions (reversible or not) coming with given interpretations and perhaps new ways of combining them. We conclude by discussing how to possibly change the language, treating the changes suggested by \Hilb\ and \Ban\ separately.

From the construction, it is clear that dropping parts of the syntax corresponds to dropping requirements for the models. For instance, if one drops the trace operator from the language, the result is then readily interpretable in \Hilb. One could then add new operations \eg superpositions of values, to get a language better suited for discussing quantum programming in \Hilb. Indeed, this is essentially what is done in~\cite{valironetal:quantumcontrol}. Moreover, one could easily add a primitive operator $(-)^\dag$ with the rule 
  \[\AXC{$f\colon a\leftrightarrow b$}\UIC{$f^\dag \colon b\leftrightarrow a$}\DisplayProof\]
and extend the semantics by $\sem{f^\dag}=\sem{f}^\dag$. Of course, every combinator already has a canonical dagger as noticed in~\cite{jamessabry:infeff}, but this need not be the case for further inbuilt combinators unless one adds a primitive like $(-)^\dag$.

One might even hope to extend the full language to \Hilb by finding an appropriate (partial?) trace on it, so that $\ell^2$ becomes a traced functor. However, I have not been able to do so: $\ell^2$ being traced seems to strongly suggest using the sum formula, but it is unclear if it works in \Hilb. It is easy to find examples showing that it's not always defined in contrast with \Ban\ but I haven't been able to modify the counterexample of~\cite[Proposition 3.10]{malherbe:partialtraces} to \Hilb\ in order to show that it does not define a valid partial trace.  Alternatively, one might hope to show conclusively that no such trace exists, by \eg adapting the results of~\cite{hoshino:partial}. 
However, those results apply to Abelian categories and hence don't work as-is for either of \Hilb\ or \cat{Hilb}. Hence I haven't been able to resolve the question whether there is a partial trace on \Hilb\ making $\ell^2$ traced. Besides being a shortcoming in research, this might also be taken as a further sign of problems with the idea of quantum control (see \eg~\cite{buadescu:quantumalternation} for a discussion).

For \Ban, the most obvious way of extending the language is by adding a richer family of inbuilt functions. However, one might also add function types since \Ban\ is closed. To do this successfully, one should first check that the closure works well with the canonical fixed points, which we leave for further work. 

In both of these cases, while one might get a better syntactic handle of the categories in question, one is unlikely to get an internal language in the sense of an equivalence between models and syntactic theories. This is essentially because the syntactic $\mu x.a$ acts more like a ``chosen fixed point of a functor'' rather than as one satisfying a universal property. With no further inbuilt functions and types, this is probably not a problem but if the user is free to add types and functions, there is no obvious guarantee that the type $\mu x.a$ is an initial algebra in the syntactic category.\index[symb]{$\omega$, the category induced by the order on $\N$|)}

\chapter{Monads}\label{chp:monads}

\section{Introduction}

If $T$ is a monad and a dagger functor, then it is automatically also a comonad. We contend that the theory of a monad on  a dagger category works best when the monad and comonad satisfy the following \emph{Frobenius law}, depicted in a graphical calculus reviewed in Section~\ref{sec:pictures}.
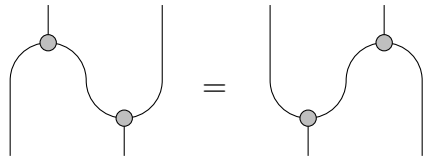
\[\index[word]{Frobenius law!of a monoid}
  \begin{pic}
   \draw (0,0) to (0,1) to[out=90,in=180] (.5,1.5) to (.5,2);
   \draw (.5,1.5) to[out=0,in=90] (1,1) to[out=-90,in=180] (1.5,.5) to (1.5,0);
   \draw (1.5,.5) to[out=0,in=-90] (2,1) to (2,2);
   \node[dot] at (.5,1.5) {};
   \node[dot] at (1.5,.5) {};
  \end{pic}
  \quad = \quad
  \begin{pic}[xscale=-1]
   \draw (0,0) to (0,1) to[out=90,in=180] (.5,1.5) to (.5,2);
   \draw (.5,1.5) to[out=0,in=90] (1,1) to[out=-90,in=180] (1.5,.5) to (1.5,0);
   \draw (1.5,.5) to[out=0,in=-90] (2,1) to (2,2);
   \node[dot] at (.5,1.5) {};
   \node[dot] at (1.5,.5) {};
  \end{pic}
\]
This law has the following satisfactory consequences.
\begin{itemize}
  \item Any dagger adjunction induces a monad satisfying the law; see Section~\ref{sec:monads}.
  \item The Kleisli category of a monad that preserves daggers and satisfies the law inherits a dagger; see Section~\ref{sec:fem}.
  \item For such a monad, the category of those Eilenberg-Moore algebras that satisfy the law inherits a dagger; see Section~\ref{sec:fem}.
  \item In fact, this Kleisli category and Frobenius-Eilenberg-Moore category are the initial and final resolutions of such a monad as dagger adjunctions in the dagger 2-category of dagger categories; see Section~\ref{sec:formal}.
  \item Any monoid in a monoidal dagger category satisfying the law induces a monad satisfying the law; see Section~\ref{sec:monads}.
  \item Moreover, the adjunction between monoids and strong monads becomes an equivalence in the dagger setting; see Section~\ref{sec:strength}.
\end{itemize}
Additionally, Section~\ref{sec:coherence} characterizes the Frobenius law as a natural coherence property between the dagger and closure of a monoidal category. 

Because of these benefits, it is tempting to simply call such monads `dagger monads'. However, many of these results also work without daggers, see~\cite{street:ambidextrous,lauda:ambidextrous,bulacutorrecillas:extensions}. This chapter is related to those works, but not a straightforward extension. 
Daggers and monads have come together before in coalgebra~\cite{jacobs:involutive,jacobs:walks}, quantum programming languages~\cite{greenetal:quipper,selingervaliron:lambda}, and matrix algebra~\cite{devosdebaerdemacker:matrix}. The current chapter differs by taking the dagger into account as a fundamental principle from the beginning.
Finally, Section~\ref{sec:strength} is a noncommutative generalization of~\cite[Theorem~4.5]{pavlovic:abstraction}. It also generalizes the classic Eilenberg-Watts theorem~\cite{watts:eilenberg}, that characterizes certain endofunctors on abelian categories as being of the form $- \otimes B$ for a monoid $B$, to monoidal dagger categories; note that there are monoidal dagger categories that are not abelian~\cite[Appendix~A]{heunen:embedding}.

\section{Frobenius monads}\label{sec:monads}

We now come to our central notion: the \emph{Frobenius law} for monads. It is the dagger version of a similar notion in~\cite{street:ambidextrous}. The monads of~\cite{street:ambidextrous} correspond to ambijuctions, whereas our monads correspond to dagger adjunctions.

\begin{definition}\index[word]{dagger Frobenius monad}
  A \emph{dagger Frobenius monad} on a dagger category $\cat{C}$ is a dagger Frobenius monoid in $[\cat{C},\cat{C}]$;
  explicitly, a monad $(T,\mu,\eta)$\index[symb]{$(T,\mu,\eta)$, a (dagger Frobenius) monad} on $\cat{C}$ with $T(f^\dag)=T(f)^\dag$ and 
  \begin{equation}\label{eq:frobeniusformonads}\index[word]{Frobenius law!of a monad}
    T(\mu_A) \circ \mu^\dag_{T(A)} = \mu_{T(A)} \circ T(\mu_A^\dag).
  \end{equation}
\end{definition}

The following family is our main source of examples of dagger Frobenius monads.
We will see in Example~\ref{ex:measurement} below that it includes quantum measurement.

\begin{example}\label{example:tensor}
   A monoid $(B,\tinymult,\tinyunit)$ in a monoidal dagger category $\cat{C}$ is a dagger Frobenius monoid if and only if the monad $- \otimes B \colon \cat C \to \cat{C}$ is a dagger Frobenius monad.
\end{example}
\begin{proof}
   The monad laws become the monoid laws.
   \[
     \mu_A = \begin{pic}
       \node[dot] (d) at (0,0) {};
       \draw (d.north) to (0,.5) node[right] {$B$};
       \draw (d.east) to[out=0,in=90] +(.3,-.3) to +(0,-.2) node[right] {$B$};
       \draw (d.west) to[out=180,in=90] +(-.3,-.3) to +(0,-.2) node[right] {$B$};
       \draw (-.8,-.5) node[left] {$A$} to (-.8,.5) node[left] {$A$};
     \end{pic} 
     \qquad
     \eta_A = \begin{pic}
       \node[dot] (d) at (0,0) {};
       \draw (d.north) to (0,.5) node[right] {$B$};
       \draw (-.4,-.5) node[left] {$A$} to (-.4,.5) node[left] {$A$};
     \end{pic} 
  \]
  The Frobenius law of the monoid implies the Frobenius law of the monad:
  \[
    T\mu \circ \mu^\dag_T
    =
    \begin{pic}[scale=.75]
      \draw (0,0) node[below] {$B$} to (0,1) to[out=90,in=180] (.5,1.5) to (.5,2) node[above] {$B$};
      \draw (.5,1.5) to[out=0,in=90] (1,1) to[out=-90,in=180] (1.5,.5) to (1.5,0) node[below] {$B$};
      \draw (1.5,.5) to[out=0,in=-90] (2,1) to (2,2) node[above] {$B$};
      \node[dot] at (.5,1.5) {};
      \node[dot] at (1.5,.5) {};
      \draw (-.5,0) node[below] {$A$} to (-.5,2) node[above] {$A$};
    \end{pic}
    =
    \begin{pic}[yscale=.75,xscale=-.75]
      \draw (0,0) node[below] {$B$} to (0,1) to[out=90,in=180] (.5,1.5) to (.5,2) node[above] {$B$};
      \draw (.5,1.5) to[out=0,in=90] (1,1) to[out=-90,in=180] (1.5,.5) to (1.5,0) node[below] {$B$};
      \draw (1.5,.5) to[out=0,in=-90] (2,1) to (2,2) node[above] {$B$};
      \node[dot] at (.5,1.5) {};
      \node[dot] at (1.5,.5) {};
      \draw (2.5,0) node[below] {$A$} to (2.5,2) node[above] {$A$};
    \end{pic}
    =
    \mu_T \circ T\mu^\dag.
  \]
  The converse follows by taking $A=I$.
\end{proof}

For another example: if $T$ is a dagger Frobenius monad on a dagger category $\cat{C}$, and $\cat{D}$ is any other dagger category, then $T \circ -$ is a dagger Frobenius monad on $[\cat D, \cat C]$.

\begin{lemma}\label{lem:extendedfrobeniuslaw}
  If $T$ is a dagger Frobenius monad on a dagger category, $\mu^\dag \circ \mu = \mu T \circ T\mu^\dag$.
\end{lemma}
\begin{proof}
  The following graphical derivation holds for any dagger Frobenius monoid.
  
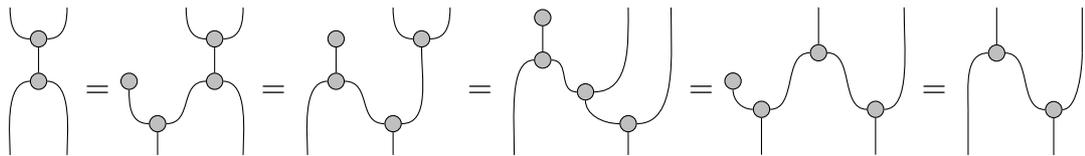
\begin{align*}
    \begin{pic}[yscale=.5625,xscale=.75]
      \node (0a) at (-0.5,-1) {};
      \node (0b) at (0.5,-1) {};
      \node[dot] (1) at (0,1) {};
      \node[dot] (2) at (0,2) {};
      \node (3a) at (-0.5,3) {};
      \node (3b) at (0.5,3) {};
      \draw[out=90,in=180,looseness=.66] (0a) to (1.west);
      \draw[out=90,in=0,looseness=.66] (0b) to (1.east);
      \draw (1.north) to (2.south);
      \draw[out=180,in=270] (2.west) to (3a);
      \draw[out=0,in=270] (2.east) to (3b);
    \end{pic}
    =
    \begin{pic}[yscale=.5625,xscale=.75]
      \node (0d) at (-1,-1) {};
      \node [dot] (0c) at (-1.5,1) {};
      \node [dot] (0a) at (-1,0) {};
      \node (0b) at (0.5,-1) {};
      \node[dot] (1) at (0,1) {};
      \node[dot] (2) at (0,2) {};
      \node (3a) at (-0.5,3) {};
      \node (3b) at (0.5,3) {};
      \draw (0a) to (0d);
      \draw[out=180,in=270] (0a) to (0c);
      \draw[out=0,in=180] (0a) to (1.west);
      \draw[out=90,in=0,looseness=0.66] (0b) to (1.east);
      \draw (1.north) to (2.south);
      \draw[out=180,in=270] (2.west) to (3a);
      \draw[out=0,in=270] (2.east) to (3b);
    \end{pic}
    =
    \begin{pic}[yscale=.5625,xscale=.75]
      \node (0d) at (-2,-1) {};
      \node [dot] (0c) at (-1.5,2) {};
      \node [dot] (0a) at (-.5,0) {};
      \node (0b) at (-.5,-1) {};
      \node[dot] (1) at (-1.5,1) {};
      \node[dot] (2) at (0,2) {};
      \node (3a) at (-0.5,3) {};
      \node (3b) at (0.5,3) {};
      \draw[out=180,in=90,looseness=0.66] (1.west) to (0d);
      \draw (1.north) to (0c.south);
      \draw[out=180,in=0,looseness=1] (0a.west) to (1.east);
      \draw (0b) to (0a.south);
      \draw[out=0,in=270,out looseness=.7] (0a.east) to (2.south);
      \draw[out=180,in=270] (2.west) to (3a);
      \draw[out=0,in=270] (2.east) to (3b);
    \end{pic}
    =
    \begin{pic}[yscale=.5625,xscale=.75]
      \node (0d) at (-2,-1) {};
      \node [dot] (0c) at (-1.5,2.5) {};
      \node [dot] (0a) at (-.75,.75) {};
      \node (0b) at (0,-1) {};
      \node[dot] (1) at (-1.5,1.5) {};
      \node[dot] (2) at (0,0) {};
      \node (3a) at (0,3) {};
      \node (3b) at (0.75,3) {};
      \draw[out=180,in=90,out looseness=.5] (1.west) to (0d);
      \draw (1.north) to (0c.south);
      \draw[out=180,in=0] (0a.west) to (1.east);
      \draw (0b) to (2.south);
      \draw[out=270,in=180] (0a.south) to (2.west);
      \draw[out=0,in=270,out looseness=.66] (2.east) to (3b);
      \draw[out=0,in=270,out looseness=.8] (0a.east) to (3a);
    \end{pic}
    =
    \begin{pic}[yscale=.75,xscale=.75]
      \node [dot] (a) at (-0.5, 0.5) {};
      \node [dot] (b) at (0,0) {};
      \node [dot] (c) at (1,1) {};
      \node [dot] (d) at (2,0) {};
      \node (A) at (0,-1) {};
      \node (B) at (2,-1) {};
      \node (C) at (1,2) {};
      \node (D) at (2.5,2) {};
      \draw (c.north) to (C);
      \draw (B) to (d.south);
      \draw (A) to (b.south);
      \draw[out=270,in=180] (a.south) to (b.west);
      \draw[out=0,in=180] (b.east) to (c.west);
      \draw[out=0,in=180] (c.east) to (d.west);
      \draw[out=0,in=270,out looseness=.66] (d.east) to (D);
    \end{pic}
    =
    \begin{pic}[yscale=.75,xscale=.75]
      \node (0) at (0,0) {};
      \node (0a) at (0,1) {};
      \node[dot] (1) at (0.5,2) {};
      \node[dot] (2) at (1.5,1) {};
      \node (3) at (1.5,0) {};
      \node (4) at (2,3) {};
      \node (4a) at (2,2) {};
      \node (5) at (0.5,3) {};
      \draw (0) to (0a.center);
      \draw[out=90,in=180] (0a.center) to (1.west);
      \draw[out=0,in=180] (1.east) to (2.west);
      \draw[out=0,in=270,looseness=.66] (2.east) to (4a.center);
      \draw (4a.center) to (4);
      \draw (2.south) to (3);
      \draw (1.north) to (5);
    \end{pic}
  \end{align*}
  These equalities use the unit law, the Frobenius law, and associativity.
\end{proof}

The following lemma shows that dagger Frobenius monads have the same relationship to dagger adjunctions as ordinary monads have to ordinary adjunctions.

\begin{lemma}\label{lem:daggeradjunction}
  If $F \dashv G$ is a dagger adjunction, then $G \circ F$ is a dagger Frobenius monad.
\end{lemma} 
\begin{proof}
  It is clear that $T = G \circ F$ is a dagger functor.
  The Frobenius law follows from applying~\cite[Corollary 2.22]{lauda:ambidextrous} to $\cat{DagCat}$.
  We will be able to give a self-contained proof after Theorem~\ref{thm:comparison} below.
\end{proof}

For example, in \cat{Rel} and \cat{Hilb}, the dagger biproduct monad induced by the dagger adjunction of Example~\ref{ex:dagbiprods} is of the form $-\otimes (I\oplus I)$ as in Example~\ref{example:tensor}. 
However, not all dagger Frobenius monads are of this form: the Frobenius monad induced by the dagger adjunction of Example~\ref{ex:imdagadj} in general decreases the dimension of the underlying space, and hence cannot be of the form $-\otimes B$ for a fixed $B$.

\section{Algebras}\label{sec:fem} 

Next we consider algebras for dagger Frobenius monads. We start by showing that Kleisli categories of dagger Frobenius monads inherit a dagger.

\begin{lemma}\label{lem:kleislidagger}\index[symb]{$\Kl(T)$ the Kleisli category of a monad}
  If $T$ is a dagger Frobenius monad on a dagger category $\cat{C}$, then $\Kl(T)$ carries a dagger that commutes with the canonical functors $\Kl(T) \to \cat{C}$ and $\cat{C}\to\Kl(T)$.
\end{lemma}
\begin{proof} 
  A straightforward calculation establishes that
  \[
    \big(A \sxto{f} T(B)\big)
    \;\mapsto\;
    \big(B \sxto{\eta} T(B) \sxto{\mu^\dag} T^2(B) \sxto{T(f^\dag)} T(A)\big)
  \]
  is a dagger on $\Kl(T)$ commuting with the functors $\cat{C}\to \Kl(T)$ and $\Kl(T)\to \cat{C}$.
\end{proof} 

If we want algebras to form a dagger category, it turns out that the category of all Eilenberg-Moore algebras is too large. The crucial law is the following \emph{Frobenius law} for algebras.

\begin{definition} 
   Let $T$ be a monad on a dagger category $\cat{C}$. A \emph{Frobenius-Eilenberg-Moore algebra}, or \emph{FEM-algebra} for short, is an Eilenberg-Moore algebra $a \colon T(A) \to A$ that makes the following diagram commute.
   \begin{equation}\index[word]{Frobenius law!of an algebra}\label{eq:femlaw}\begin{aligned}\begin{tikzpicture}
     \matrix (m) [matrix of math nodes,row sep=2em,column sep=4em,minimum width=2em]
     {
      T(A) & T^2(A) & \\
      T^2(A) & T(A) \\};
     \path[->]
     (m-1-1) edge node [left] {$\mu^\dag$} (m-2-1)
             edge node [above] {$T(a)^\dag$} (m-1-2)
     (m-2-1) edge node [below] {$T(a)$} (m-2-2)
     (m-1-2) edge node [right] {$\mu$} (m-2-2);
   \end{tikzpicture}\end{aligned}\end{equation}   
   Denote the category of FEM-algebras $(A,a)$ and algebra homomorphisms by $\FEM(T)$.
\end{definition}

In effect, equation~\eqref{eq:femlaw} says that the map $T(a)\circ\mu^\dag\colon T(A)\to T^2(A)\to T(A)$ is self-adjoint. When the dagger Frobenius monad is of the form $T(A)=A \otimes B$ as in Example~\ref{example:tensor}, the Frobenius law~\eqref{eq:femlaw} for an algebra $a \colon T(A) \to A$ becomes the following equation, which resembles the Frobenius law~\eqref{eq:frobeniuslaw} for monoids and monads.
\[\index[word]{Frobenius law!of an algebra}
 \begin{pic}[xscale=1.5]
   \node[morphism, minimum width=25mm] (a)  at (0,0.75) {$a$};
   \node[dot] (b) at (.4,0) {};
   \draw ([xshift=-.5mm]a.south west) to +(0,-1);
   \draw (a.north) to +(0,.4);
   \draw (b) to[out=0,in=-90] +(.3,.5) to +(0,.85);
   \draw ([xshift=.5mm]a.south east) to[out=-90,in=180]  (b);
   \draw (b) to +(0,-.45);
 \end{pic}
 \quad = \quad
 \begin{pic}[xscale=1.5]
   \node[morphism,hflip, minimum width=25mm] (a)  at (0,0) {$a$};
   \node[dot] (b) at (.4,0.75) {};
   \draw ([xshift=-.5mm]a.north west) to +(0,1);
   \draw (a.south) to +(0,-.4);
   \draw (b) to[out=0,in=90] +(.3,-.5) to +(0,-.85);
   \draw ([xshift=.5mm]a.north east) to[out=90,in=180]  (b);
   \draw (b) to +(0,.45);
 \end{pic} 
\]

\begin{example}\label{ex:kleislifem}
  Any free algebra $\mu_A \colon T^2(A) \to T(A)$ of a dagger Frobenius monad $T$ on a dagger category $\cat C$ is a FEM-algebra.
  Hence there is an embedding $\Kl(T) \to \FEM(T)$.
\end{example}
\begin{proof}
  The Frobenius law for the free algebra is the Frobenius law of the monad.
\end{proof}

There are many EM-algebras that are not FEM-algebras; a family of examples can be derived from~\cite[Theorem~6.4]{pavlovic:abstraction}. Here is a concrete example.

\begin{example}\label{ex:emnonfem}
  The complex $n$-by-$n$-matrices form a Hilbert space $A$ with inner product $\left\langle a,b\right\rangle=\tfrac{1}{n}\Tr (a^\dag \circ b)$. Matrix multiplication $m\colon A\otimes A\to A$ makes $A$ into a dagger Frobenius monoid in $\cat{FHilb}$, and hence $T=-\otimes A$ into a dagger Frobenius monad on $\cat{FHilb}$. Define a monoid homomorphism $U\colon A\to A$ by conjugation $a\mapsto u^\dag\circ a\circ u$ with a unitary matrix $u \in A$. Then $m\circ (\id[A] \otimes U)$ is an EM-algebra that is a FEM-algebra if and only if $u=u^\dag$.
\end{example}
\begin{proof}
  Because $U^\dag=u \circ - \circ u^\dag$, the Frobenius law~\eqref{eq:femlaw} for $T$ unfolds to the following.
  
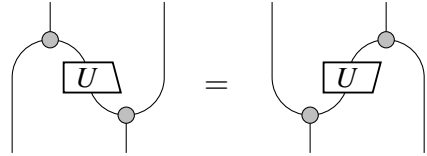
\[
    \begin{pic}
     \draw (0,0) to (0,1) to[out=90,in=180] (.5,1.5) to (.5,2);
     \draw (.5,1.5) to[out=0,in=90] (1,1) to[out=-90,in=180] (1.5,.5) to (1.5,0);
     \draw (1.5,.5) to[out=0,in=-90] (2,1) to (2,2);
     \node[morphism] at (1,1) {$U$};
     \node[dot] at (.5,1.5) {};
     \node[dot] at (1.5,.5) {};
    \end{pic}
    \quad = \quad
    \begin{pic}[xscale=-1]
     \draw (0,0) to (0,1) to[out=90,in=180] (.5,1.5) to (.5,2);
     \draw (.5,1.5) to[out=0,in=90] (1,1) to[out=-90,in=180] (1.5,.5) to (1.5,0);
     \draw (1.5,.5) to[out=0,in=-90] (2,1) to (2,2);
     \node[morphism,hflip] at (1,1) {$U$};
     \node[dot] at (.5,1.5) {};
     \node[dot] at (1.5,.5) {};
    \end{pic}
  \]
  This comes down to $U=(U^*)^\dag$, that is, $u=u^\dag$. 
\end{proof}

Before we calculate $\FEM(-\otimes B)$ for a dagger Frobenius monoid $B$ induced by an arbitrary groupoid, we work out an important special case.

\begin{example}\label{ex:measurement}
  Let $B$ be a dagger Frobenius monoid in $\cat{FHilb}$ induced by a finite discrete groupoid $\cat{G}$ as in Example~\ref{ex:groupoidFrob}. A FEM-algebra structure on a Hilbert space $A$ for $- \otimes B$ consists of \emph{quantum measurements} on $A$: orthogonal projections on $A$ that sum to the identity.
\end{example} 
For more information about quantum measurements, see~\cite[Section~3.2]{heinosaariziman:quantum}.
\begin{proof}
  A FEM-algebra structure on $A$ consists of a map $a\colon A\otimes B\to A$ subject to the FEM-laws. Since $B$ has a basis indexed by objects of \cat{G}, it suffices to understand the maps $P_G\colon A\to A$ defined by $v\mapsto a(v\otimes \id[G])$. The associative law implies that each $P_G$ satisfies $P_G\circ P_G=P_G$, and from the Frobenius law we get that each $P_G$ is self-adjoint, so that each $P_G$ is an orthonormal projection. The unit law says that $\sum_G P_G = \id[A]$.

  There is also another, graphical, way of seeing this. Quantum measurements can also be characterized as `$B$-self-adjoint' coalgebras for the comonad $-\otimes B$, where being $B$-self-adjoint means that the following equation holds~\cite{coeckepavlovic:measurement}.
  \begin{equation}\label{eq:self-adjoint}
  \begin{pic}
   \node[morphism] (a)  at (0,0.75) {};
    \node[dot] (b) at (.5,0) {};
    \node[dot] (c) at (.5,-0.35) {};
   \draw ([xshift=-1mm]a.south west) to +(0,-1);
   \draw (a.north) to +(0,.4);
   \draw (b) to[out=0,in=-90] +(.3,.5) to +(0,.85);
   \draw ([xshift=1mm]a.south east) to[out=-90,in=180]  (b);
   \draw (b) to (c);
  \end{pic}
  \quad = \quad
  \begin{pic}
    \node[morphism,hflip] (f) at (-.4,0) {};
    \draw ([xshift=-1mm]f.north west) to +(0,.65);
    \draw (f.south) to +(0,-.65);
    \draw ([xshift=1mm]f.north east) to +(0,.65);
  \end{pic} 
  \end{equation}
  Such coalgebras correspond precisely to $\FEM$-algebras, as we will now show.
  Because of the dagger, coalgebras of the comonad $-\otimes B$ are just algebras of the monad $-\otimes B$. Thus it suffices to show that an algebra is FEM if and only if it satisfies~\eqref{eq:self-adjoint}. The implication from left to right is easy. 
  \[
  \begin{pic}[yscale=.7]
      \node[morphism] (a)  at (0,0.75) {};
      \node[dot] (b) at (.5,0) {};
      \node[dot] (c) at (.5,-0.35) {};
      \draw ([xshift=-1mm]a.south west) to +(0,-1);
      \draw (a.north) to +(0,.4);
      \draw (b) to[out=0,in=-90] +(.3,.5) to +(0,.85);
      \draw ([xshift=1mm]a.south east) to[out=-90,in=180]  (b);
      \draw (b) to (c);
    \end{pic}
  \quad = \quad 
    \begin{pic}[yscale=.7]
      \node[morphism,hflip] (a)  at (0,0) {};
      \node[dot] (b) at (.5,0.75) {};
      \node[dot] (c) at (.9,-0.5) {};
      \draw ([xshift=-1mm]a.north west) to +(0,1);
      \draw (a.south) to +(0,-.4);
      \draw (b) to[out=0,in=90] +(.4,-.4) to (c);
      \draw ([xshift=1mm]a.north east) to[out=90,in=180]  (b);
      \draw (b) to +(0,.45);
    \end{pic}
    \quad = \quad
    \begin{pic}[yscale=.7]
      \node[morphism,hflip] (f) at (-.4,0) {};
      \draw ([xshift=-1mm]f.north west) to +(0,.65);
      \draw (f.south) to +(0,-.65);
      \draw ([xshift=1mm]f.north east) to +(0,.65);
    \end{pic} 
  \]
  The other implication can be proven as follows.
  \[
    \begin{pic}[yscale=.7]
      \node[morphism,hflip] (a)  at (0,0) {};
      \node[dot] (b) at (.5,0.75) {};
      \draw ([xshift=-1mm]a.north west) to +(0,1.2);
      \draw (a.south) to +(0,-.4);
      \draw (b) to[out=0,in=90] +(.35,-.5) to +(0,-.85);
      \draw ([xshift=1mm]a.north east) to[out=90,in=180]  (b);
      \draw (b) to +(0,.64);
    \end{pic}
    \quad \overset{(\ref{eq:self-adjoint})}{=}  \quad
    \begin{pic}[yscale=.7]
      \node[morphism] (a)  at (0,0.75) {};
      \node[dot] (b) at (.5,0) {};
      \node[dot] (c) at (.5,-0.35) {};
      \node[dot] (d) at (1.1,0.75) {};
      \draw ([xshift=-1mm]a.south west) to +(0,-1);
      \draw (a.north) to +(0,.4);
      \draw (b) to[out=0,in=180] (d);
      \draw ([xshift=1mm]a.south east) to[out=-90,in=180]  (b);
      \draw (d) to[out=0, in=90] +(0.3,-.5) to +(0,-.75);
      \draw (d) to +(0,0.6);
      \draw (b) to (c);
    \end{pic}
    \quad = \quad 
    \begin{pic}[yscale=.7]
      \node[morphism] (a)  at (0,0.65) {};
      \node[dot] (b) at (.3,0) {};
      \node[dot] (c) at (.03,-0.6) {};
      \node[dot] (d) at (.85,-0.6) {};
      \draw ([xshift=-1mm]a.south west) to +(0,-1.5);
      \draw (a.north) to +(0,.4);
      \draw (b) to[out=0,in=180] (d);
      \draw ([xshift=2.5mm]a.south east) to (b);
      \draw (d) to[out=0, in=90] +(0.3,.5) to +(0,1.3);
      \draw (d) to +(0,-0.5);
      \draw (b) to[out=180, in=90] (c);
    \end{pic}
    \quad = \quad
    \begin{pic}[yscale=.7]
      \node[morphism] (a)  at (0,0.75) {};
      \node[dot] (b) at (.5,0) {};
      \draw ([xshift=-1mm]a.south west) to +(0,-1);
      \draw (a.north) to +(0,.4);
      \draw (b) to[out=0,in=-90] +(.3,.5) to +(0,.85);
      \draw ([xshift=1mm]a.south east) to[out=-90,in=180]  (b);
      \draw (b) to +(0,-.45);
    \end{pic}
    \]
    This finishes the alternative proof.
\end{proof}

\begin{example}\label{ex:femtensor}
  Let a groupoid $\cat{G}$ induce a dagger Frobenius monoid $B$ in $\cat{C}=\cat{Rel}$ or $\cat{C}=\cat{FHilb}$ as in Example~\ref{ex:groupoidFrob}. There is an equivalence $\FEM(-\otimes B) \simeq [\cat{G},\cat{C}]$.
\end{example}
\begin{proof}
  Separate the cases $\cat{Rel}$ and $\cat{FHilb}$.
  \begin{itemize}
  \item In $\cat{Rel}$, a FEM-algebra is a set $A$ with a \emph{relation} $g \colon A \to A$ for each $g \in B$ satisfying several equations.
    For each object $G$ of \cat{G}, define $A_G=\{a\in A\mid \id[G] a=\{a\} \}$. 
    The unit law implies that each $a \in A$ is in at least one $A_G$, and the other EM-law implies that no $a \in A$ can be in more than one $A_G$. 
    Now if $g\colon G\to H$ in \cat{G}, then $g$ defines a \emph{function} $A_G\to A_H$ and maps everything outside of $A_G$ to the empty set. 
    Thus the FEM-algebra $A$ defines an \emph{action} of \cat{G} in \cat{Rel}: a functor $F_A\colon\cat{G}\to\cat{Rel}$ making the sets $F_A(G)$ pairwise disjoint for distinct objects of \cat{G}. 

    Conversely, each such functor $F$ defines a FEM-algebra $A_F$ by setting $A_F=\bigcup_G F(G)$. 
    But the category of such functors is equivalent to $[\cat{G},\cat{Rel}]$.

  \item In $\cat{FHilb}$, a FEM-algebra is a Hilbert space $A$ with a morphism $g \colon A \to A$ for each $g \in B$. 
    For each object $G$ of \cat{G}, define $A_G=\{a\in A\mid \id[G] a=a \}$. 
    As above, $A$ is a direct sum of the $A_G$ and $g\colon G\to H$ in \cat{G} defines a morphism $A_G\to A_H$ and annihilates $A_G^\perp$. 
    This defines a representation of \cat{G} in \cat{FHilb}. 
    The Frobenius law implies that this representation is unitary.
  \end{itemize}
\end{proof}

In both of these examples, the same reasoning goes through over all of \cat{Hilb}. Moreover, the same constructions give rise to an equivalence between all EM-algebras for $-\otimes B$ and \emph{all} (=also non-dagger) functors $\cat{G}\to\cat{(F)Hilb}$: hence any non-unitary presentation of a groupoid gives rise to a EM-algebra that is not FEM, providing us with more examples than those in Example~\ref{ex:emnonfem}.

The fact that all of the categories of FEM-algebras from the previous example had daggers is no accident. 

\begin{lemma}\label{lem:femdag}
  Let $T$ be a dagger Frobenius monad on a dagger category $\cat{C}$. The dagger on $\cat{C}$ induces one on $\FEM(T)$.\index[symb]{$\FEM(T)$, the category of $\FEM$-algebras of $T$}
\end{lemma} 
\begin{proof}
  Let $f \colon (A,a) \to (B,b)$ be a morphism of FEM-algebras; we are to show that $f^\dag$ is a morphism $(B,b)\to(A,a)$. 
  It suffices to show that $b \circ T(f) = f \circ a$ implies $a \circ T(f^\dag) = f^\dag \circ b$.
  Consider the following diagram:
  
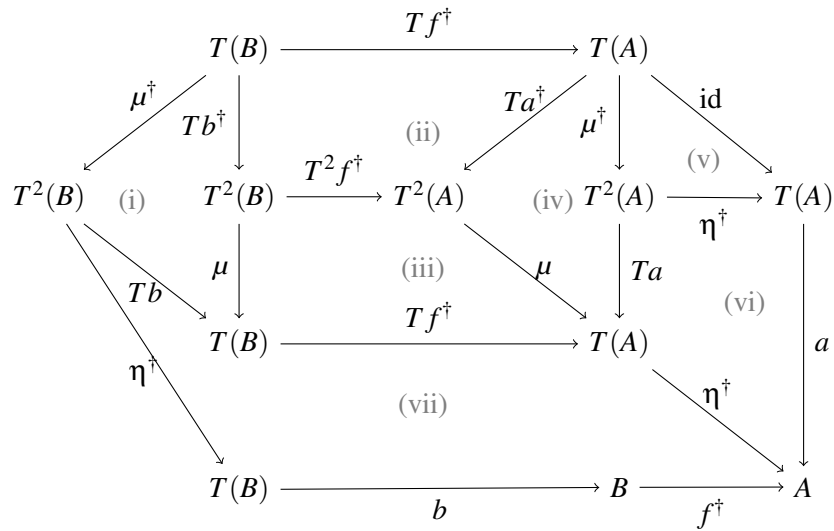
\[\begin{tikzpicture}[font=\small,scale=0.9]
    \matrix (m) [matrix of math nodes,row sep=3em,column sep=3em,minimum width=1em]
    {
      & T(B) & & T(A)  \\
     T^2(B) & T^2(B) & T^2(A) & T^2(A) & T(A)  \\
      &T(B) & &T(A) &       \\      
      &T(B) & &B &A    \\};
    \path[->]
    (m-1-2) edge node [above] {$Tf^\dag$} (m-1-4)
                  edge node [above] {$\mu^\dag$} (m-2-1)
                  edge node [left] {$Tb^\dag$}(m-2-2)
    (m-1-4)  edge node [above] {$\id$} (m-2-5)
                  edge node [left] {$\mu^\dag$}(m-2-4)
                  edge node [above] {$Ta^\dag$} (m-2-3) 
    (m-2-1) edge node [below] {$Tb$} (m-3-2)
                  edge node [below] {$\eta^\dag$} (m-4-2)
    (m-2-2) edge node [left] {$\mu$} (m-3-2) 
                 edge node [above] {$T^2f^\dag$} (m-2-3)
    (m-2-3) edge node [right] {$\mu$} (m-3-4) 
    (m-2-4) edge node [right] {$Ta$} (m-3-4) 
                 edge node [below] {$\eta^\dag$} (m-2-5)
    (m-2-5) edge node [right] {$a$} (m-4-5)
    (m-3-2) edge node [above] {$Tf^\dag$} (m-3-4)
    (m-3-4) edge node [above] {$\eta^\dag$} (m-4-5)
    (m-4-2) edge node [below] {$b$} (m-4-4) 
    (m-4-4) edge node [below] {$f^\dag$} (m-4-5);
    \node[gray] at (-4.25,1) {(i)};
    \node[gray] at (0,2) {(ii)};
    \node[gray] at (0,0) {(iii)};
    \node[gray] at (1.9,1) {(iv)};
    \node[gray] at (4.1,1.6) {(v)};
    \node[gray] at (4.7,-0.5) {(vi)};
    \node[gray] at (0,-2) {(vii)};
  \end{tikzpicture}\]
  Region (i) is the Frobenius law of $(B,b)$; commutativity of (ii) follows from the assumption that $f$ is a morphism $(A,a)\to (B,b)$ by applying $T$ and the dagger; (iii) is naturality of $\mu$; (iv) is the Frobenius law of $(A,a)$; (v) commutes since $T$ is a comonad; (vi) and (vii) commute by naturality of $\eta^\dag$. 
\end{proof}

In fact, the equivalence of Example~\ref{ex:femtensor} is a dagger equivalence.

\begin{lemma}\label{lem:femlaw} 
  Let $T$ be a dagger Frobenius monad. An EM-algebra $(A,a)$ is FEM if and only if $a^\dag$ is a homomorphism $(A,a)\to (TA,\mu_A)$.
\end{lemma}
\begin{proof} 
  If $(A,a)$ is a FEM-algebra, its associativity means that $a$ is a homomorphism $(TA,\mu_A) \to (A,a)$. Here $T(A)$ is a FEM-algebra too because $T$ is a dagger Frobenius monad. Thus $a^\dag$ is a homomorphism $(A,a) \to (TA,\mu_A)$ by Lemma~\ref{lem:femdag}.

  For the converse, assume $a^\dag$ is a homomorphism $(A,a)\to (TA,\mu_A)$, so the diagram
  \[
  \begin{tikzpicture}
     \matrix (m) [matrix of math nodes,row sep=2em,column sep=4em,minimum width=2em]
     {
      T(A) & T^2(A) \\
      A & T(A) \\};
     \path[->]
     (m-1-1) edge node [left] {$a$} (m-2-1)
             edge node [above] {$Ta^\dag$} (m-1-2)
     (m-2-1) edge node [below] {$a^\dag$} (m-2-2)
     (m-1-2) edge node [right] {$\mu $} (m-2-2);
  \end{tikzpicture}
  \]
  commutes. 
  Hence $\mu \circ Ta^\dag$ is self-adjoint as $a^\dag\circ a$ is, giving the Frobenius law~\eqref{eq:femlaw}.
\end{proof}

Interpreting the associative law for algebras as saying that $a\colon TA\to A$ is a homomorphism $(TA,\mu_A)\to (A,a)$, Lemmas~\ref{lem:femdag} and~\ref{lem:femlaw} show that this morphism is universal, in the sense that if its dagger is an algebra homomorphism, then so is the dagger of any other algebra homomorphism to $A$ (whose domain satisfies the Frobenius law).

\begin{theorem}\label{thm:comparison}\index[symb]{$\Kl(T)$ the Kleisli category of a monad}\index[symb]{$\FEM(T)$, the category of $\FEM$-algebras of $T$}
  Let $F$ and $G$ be dagger adjoints, and write $T=G\circ F$ for the induced dagger Frobenius monad. 
  There are unique dagger functors $K$ and $J$ making the following diagram commute.
  \[\begin{tikzpicture}[xscale=4,yscale=2]
    \node (tl) at (-1,1) {$\Kl(T)$};
    \node (t) at (0,1) {$\cat D$};
    \node (tr) at (1,1) {$\FEM(T)$};
    \node (b) at (0,0) {$\cat C$};
    \draw[->, dashed] (tl) to node[above] {$K$} (t);
    \draw[->, dashed] (t) to node[above] {$J$} (tr);
    \draw[->] ([xshift=.2mm]t.south) to node[right] {$G$} ([xshift=.2mm]b.north);
    \draw[->] ([xshift=-.2mm]b.north) to node[left] {$F$} ([xshift=-.2mm]t.south);
    \draw[->] (tl.-45) to (b.135);
    \draw[->] ([xshift=-.5mm,yshift=-.5mm]b.135) to ([yshift=-.5mm,xshift=-.5mm]tl.-45);
    \draw[->] ([xshift=.5mm,yshift=-.5mm]tr.-135) to ([yshift=-.5mm,xshift=.5mm]b.45);
    \draw[->] (b.45) to (tr.-135);
  \end{tikzpicture}\]
  Moreover, $J$ is full, $K$ is full and faithful, and $J\circ K$ is the canonical inclusion.
\end{theorem}
\begin{proof}
  It suffices to show that the comparisons $K \colon \Kl(T) \to \cat{D}$ and $J \colon \cat{D} \to \EM(T)$ (see~\cite[VI.3, IV.5]{maclane:categories}) are dagger functors, and that $J$ factors through $\FEM(T)$.

  Let $(A,a)$ be in the image of $J$. As $J\circ K$ equals the canonical inclusion, $(T(A),\mu_A)$ is also in the image of $J$. Because $J$ is full, the homomorphism $a\colon (TA,\mu_A)\to (A,a)$ is in the image of $J$, say $a=J(f)$. But then $G(f)=a$, and so $G(f^\dag)=a^\dag$. This implies that $a^\dag$ is a homomorphism $(A,a)\to (TA,\mu_A)$. Lemma~\ref{lem:femlaw} guarantees $(A,a)$ is in $\FEM(T)$. 
 
  Clearly $J$ is a dagger functor. It remains to show that $K$ is a dagger functor. As $K$ is full and faithful, Lemma~\ref{lem:induced daggers} gives $\Kl(T)$ a unique dagger making it a dagger functor. This also makes $J\circ K$ a dagger functor, and since $J\circ K$ equals the canonical inclusion, the induced dagger on $\Kl(T)$ must equal the canonical one from Lemma~\ref{lem:kleislidagger}.
\end{proof}

The previous theorem leads to a direct proof of Lemma~\ref{lem:daggeradjunction} above, as follows. The definition of $\FEM(T)$ makes sense for arbitrary monads (that might not satisfy the Frobenius law), and the proof above still shows that the image of $J \colon \cat{D} \to \EM(T)$ lands in $\FEM(T)$. Hence every free algebra is a FEM-algebra. This implies the Frobenius law~\eqref{eq:frobeniusformonads} for $T$.

\section{Formal monads on dagger categories}\label{sec:formal}

Both ordinary monads~\cite{street1972formal} and Frobenius monads~\cite{lauda:ambidextrous} have been treated formally in 2-categories. This section establishes the counterpart for dagger 2-categories. Its main contribution is to show that the category of FEM-algebras satisfies a similar universal property for dagger Frobenius monads as EM-algebras do for ordinary monads. 
Recall from Definition~\ref{def:dagtwocat} that a dagger 2-category is a category enriched in $\cat{DagCat}$.

\begin{definition}\index[word]{dagger adjunction!in a dagger 2-category}\index[word]{dagger 2-adjunction} 
  An adjunction in a dagger 2-category is just an adjunction in the underlying 2-category. 
  A \emph{dagger 2-adjunction} consists of two $\cat{DagCat}$-enriched functors that form a 2-adjunction in the usual sense. 
\end{definition}

Adjunctions in a dagger 2-category need not specify left and right, just like the dagger adjunctions they generalize, but dagger 2-adjunctions need to specify left and right. 

\begin{definition}\index[word]{dagger Frobenius monad!in a dagger 2-category}
  Let \cat{C} be a dagger 2-category. 
  A \emph{dagger Frobenius monad} consists of an object $C$, a morphism $T\colon C\to C$, and 2-cells $\eta\colon \id[C]\to T$ and $\mu\colon T^2\to T$ that form a monad in the underlying 2-category of $\cat{C}$ and satisfy the Frobenius law
  \[
      \mu T\circ T\mu^\dag=T\mu\circ \mu^\dag T.
  \]
  A \emph{morphism of dagger Frobenius monads} $(C,S)\to (D,T)$ is a morphism $F\colon C\to D$ with a 2-cell $\sigma\colon TF\to FS$ making the following diagrams commute.
  \begin{equation}\label{diag:mapofFrobmonads}
    \begin{aligned}\begin{tikzpicture}
     \matrix (m) [matrix of math nodes,row sep=1em,column sep=2em]
     {
       & TF \\
      F \\
       & FS \\};
     \tikzset{font=\small};
     \path[->]
     (m-2-1) edge node [below=1ex,left] {$F\eta^S$} (m-3-2)
             edge node [above=1ex,left] {$\eta^TF$} (m-1-2)
     (m-1-2) edge node [left] {$\sigma$} (m-3-2);
    \end{tikzpicture}\end{aligned}
    \begin{aligned}\begin{tikzpicture}
     \matrix (m) [matrix of math nodes,row sep=1em,column sep=2em]
     {
       & TFS & FSS \\
      TTF & & \\
       & TF & FS \\};
     \tikzset{font=\small};
     \path[->]
     (m-2-1) edge node [below=1ex,left] {$\mu^T F$} (m-3-2)
             edge node [above=1ex,left] {$T\sigma$} (m-1-2)
     (m-1-2) edge node [above] {$\sigma S$} (m-1-3)
     (m-1-3) edge node [left] {$F\mu^S$} (m-3-3)
     (m-3-2) edge node [below] {$\sigma$} (m-3-3);
    \end{tikzpicture}\end{aligned}
    \begin{aligned}\begin{tikzpicture}
     \matrix (m) [matrix of math nodes,row sep=1em,column sep=2em]
     {
       & FSS & FS \\
      TFS & & \\
       & TTF & TF \\};
     \tikzset{font=\small};
     \path[->]
     (m-2-1) edge node [below=1ex,left] {$T\sigma^\dag$} (m-3-2)
             edge node [above=1ex,left] {$\sigma S$} (m-1-2)
     (m-1-2) edge node [above] {$F\mu^S$} (m-1-3)
     (m-1-3) edge node [left] {$\sigma^\dag$} (m-3-3)
     (m-3-2) edge node [below] {$\mu^T F$} (m-3-3);
    \end{tikzpicture}\end{aligned}
  \end{equation}
  A \emph{transformation between morphisms of dagger Frobenius monads} $(F,\sigma) \to (G,\tau)$ is a 2-cell $\phi\colon F\to G$ making the following diagrams commute.
  \[
    \begin{tikzpicture}
     \matrix (m) [matrix of math nodes,row sep=2em,column sep=4em,minimum width=2em]
     {
      TF & TG & \\
      FS & GS \\};
     \path[->]
     (m-1-1) edge node [left] {$\sigma$} (m-2-1)
             edge node [above] {$T\phi$} (m-1-2)
     (m-2-1) edge node [below] {$\phi S$} (m-2-2)
     (m-1-2) edge node [right] {$\tau$} (m-2-2);
    \end{tikzpicture}
  \quad
    \begin{tikzpicture}
     \matrix (m) [matrix of math nodes,row sep=2em,column sep=4em,minimum width=2em]
     {
      TG & TF & \\
      GS & FS \\};
     \path[->]
     (m-1-1) edge node [left] {$\sigma$} (m-2-1)
             edge node [above] {$T\phi^\dag$} (m-1-2)
     (m-2-1) edge node [below] {$\phi ^\dag S$} (m-2-2)
     (m-1-2) edge node [right] {$\tau$} (m-2-2);
    \end{tikzpicture}
  \]
  Define the composition of morphisms to be $(G,\tau) \circ (F,\sigma) = (GF,G\sigma \circ \tau F)$.
  Horizontal and vertical composition of 2-cells in $\cat{C}$ define horizontal and vertical composition of transformations of morphisms of dagger Frobenius monads, and the dagger on 2-cells of $\cat{C}$ gives a dagger on these transformations. 
  This forms a dagger 2-category $\DFMonad(\cat{C})$.
\end{definition}

Omitting the third diagram of~\eqref{diag:mapofFrobmonads} gives the usual definition of a monad morphism.
We require this coherence with the dagger for the following reason: just as the first two diagrams of~\eqref{diag:mapofFrobmonads} ensure that $(C,T)\mapsto \EM(T)$ is a 2-functor $\mathrm{Monad}(\cat{Cat}) \to \cat{Cat}$, the third one ensures that $(C,T) \mapsto \FEM(T)$ is a dagger 2-functor $\DFMonad(\cat{DagCat}) \to \cat{DagCat}$.

There is an inclusion dagger 2-functor $\cat{C}\to\DFMonad(\cat{C})$ given by $C\mapsto (C,\id[C])$, $F\mapsto (F,\id)$, and $\psi\mapsto \psi$. There is also a forgetful dagger 2-functor $\DFMonad(\cat{C})\to \cat{C}$ given by $(C,T)\mapsto C$, $(F,\sigma)\mapsto F$, and $\psi\mapsto\psi$. As with the formal theory of monads~\cite{street1972formal}, the forgetful functor is the left adjoint to the inclusion functor. To see this, it suffices to exhibit a natural isomorphism of dagger categories $[C,D]\cong [(C,T),(D,\id[D])]$, such as  sending $F\colon C\to D$ to $(F,F\eta^T)$ and $\psi$ to $\psi$. 

\begin{definition}
  A dagger 2-category \cat{C} \emph{admits the construction of FEM-algebras} when the inclusion $\cat{C}\to\DFMonad(\cat{C})$ has a right adjoint $\FEM \colon \DFMonad(\cat{C}) \to \cat{C}$. 
\end{definition}

We will abbreviate $\FEM(C,T)$ to $\FEM(T)$ when no confusion can arise.

\begin{theorem} 
  If \cat{C} admits the construction of FEM-algebras, then dagger Frobenius monads factor as dagger adjunctions.
\end{theorem}
\begin{proof} 
  We closely follow the proof of the similar theorem for ordinary monads in~\cite{street1972formal}, but need to verify commutativity of some additional diagrams. 
  To verify that $(T,\mu)$ is a morphism of dagger Frobenius monads $(C,\id)\to (C,T)$, the first two diagrams are exactly as for ordinary monads, and the third diagram commutes by Lemma~\ref{lem:extendedfrobeniuslaw}.
  Denoting the counit of the adjunction of the assumption by $(E,\varepsilon)\colon (\FEM(T),\id)\to (C,T)$, the universal property gives a unique morphism $(J,\id)\colon (C,\id)\to (\FEM(T),\id)$ making the following diagram commute.
  \[
          \begin{tikzpicture}
           \matrix (m) [matrix of math nodes,row sep=2em,column sep=4em,minimum width=2em]
           {
             (C,\id)&  & (C,T) \\
             & (\FEM(T),\id ) \\};
           \path[->]
           (m-1-1) edge node [below=1ex,left] {$(J,\id)$} (m-2-2)
                   edge node [above] {$(T,\mu)$} (m-1-3)
           (m-2-2) edge node [below=1ex,right] {$(E,\varepsilon)$} (m-1-3);
          \end{tikzpicture}
  \]
  Thus $T=EJ$ and $\mu=\varepsilon J$. Next we verify $\varepsilon$ is a transformation of morphisms of dagger Frobenius monads $(EJE,\mu E)\to (E,\varepsilon)$, by showing the following diagrams commute.
  \[
        \begin{tikzpicture}
         \matrix (m) [matrix of math nodes,row sep=2em,column sep=4em,minimum width=2em]
         {
          TEJE & TE \\
          TE & E \\};
         \path[->]
         (m-1-1) edge node [left] {$\mu E$} (m-2-1)
                 edge node [above] {$T\varepsilon$} (m-1-2)
         (m-2-1) edge node [below] {$\varepsilon$} (m-2-2)
         (m-1-2) edge node [right] {$\varepsilon$} (m-2-2);
        \end{tikzpicture}
      \qquad
        \begin{tikzpicture}
         \matrix (m) [matrix of math nodes,row sep=2em,column sep=4em,minimum width=2em]
         {
          TE & TEJE \\
          E & TE \\};
         \path[->]
         (m-1-1) edge node [left] {$\varepsilon$} (m-2-1)
                 edge node [above] {$T\varepsilon^\dag$} (m-1-2)
         (m-2-1) edge node [below] {$\varepsilon^\dag$} (m-2-2)
         (m-1-2) edge node [right] {$\mu E$} (m-2-2);
        \end{tikzpicture}
  \]
  These diagrams are instances of the second and third equations of~\eqref{diag:mapofFrobmonads}.

  The rest of the proof proceeds exactly as in~\cite{street1972formal}. As $\varepsilon$ is a transformation, the adjunction gives a unique 2-cell $\xi\colon JE\to \id[\FEM(T)]$ in \cat{C} with $E\xi=\varepsilon$. Now 
  \begin{align*}
  E(\xi J\circ J\eta)
  =E\xi J\circ EJ\eta 
  =\varepsilon J\circ T\eta 
  =\mu\circ T\eta=\id,
  \end{align*}
  and so $\xi J \circ J\eta =\id$ by the universal property of the counit. 
  Furthermore
  \[
      E\xi \circ \eta E=\varepsilon \circ \eta E=\id,
  \] 
  so $(E,J,\mu,\xi)$ is an adjunction generating $(C,T)$.
\end{proof}

\begin{theorem} 
  The dagger 2-category \cat{DagCat} admits the construction of FEM-algebras.
\end{theorem} 
\begin{proof}
  We first show $(\cat{C},T)\mapsto (\FEM(T))$ extends to a well-defined dagger 2-functor $\FEM \colon \DFMonad(\cat{DagCat})\to\cat{DagCat}$. On a 1-cell $(F,\sigma)\colon(\cat{C},S)\to (\cat{D},T)$, define a dagger functor $\FEM(F,\sigma)\colon \FEM(S)\to \FEM(T)$ as follows: map the algebra $a\colon SA\to A$ to $F(a)\circ \sigma_A\colon TFA\to FSA\to FA$ and the homomorphism $f\colon (A,a)\to (B,b)$ to $F(f)$. The laws for FEM-algebras of $(F,\sigma)$ now show that $\FEM(F,\sigma)$ sends FEM-algebras to FEM-algebras. 
  On a 2-cell $\psi\colon (F,\sigma)\to (G,\tau)$, define $\FEM(\psi)$ by $\FEM(\psi)_A=\psi_A$; each $\psi_A$ is a morphism of FEM-algebras.

  There is a natural isomorphism of dagger categories $[\cat{C},\FEM(T)]\cong [(\cat{C},\id[\cat{C}]),(\cat{D},T)]$ as follows. 
  A dagger functor $\cat{C}\to \FEM(T)$ consists of a dagger functor $F\cat{C}\to \cat{D}$ and a family of maps $\sigma_A\colon TFA\to FA$ making each $(FA,\sigma_A)$ into a FEM-algebra. 
  Then $(F,\sigma)\colon(\cat{C},\id[\cat{C}])\to (\cat{D},T)$ is a well-defined morphism of dagger Frobenius monads. Similarly, any $(F,\sigma)\colon(\cat{C},\id[\cat{C}])\to (\cat{D},T)$ defines a dagger functor $\cat{C}\to \FEM(T)$. On the level of natural transformations both of these operations are obvious. 
\end{proof}

In~\cite{street1972formal} this last result is proved as follows: instead of starting with the definition of $\cat{C}^T$, the construction is recovered from the assumption that the right adjoint exists and considering functors from the categories \cat{1} and \cat{2} to the right adjoint, thus recovering the objects and arrows of $\cat{C}^T$. It is unclear how to write a similar proof in our case: while $[\cat{2},\cat{C}]$ classifies arrows and commutative squares for an ordinary category, it is not obvious how to replace \cat{2} with a dagger category playing an analogous role.

\section{Strength}\label{sec:strength}

As we saw in Example~\ref{example:tensor}, (Frobenius) monoids in a monoidal dagger category induce (dagger Frobenius) monads on the category. This in fact sets up an adjunction between monoids and strong monads~\cite{wolff:monads}. This section shows that the Frobenius law promotes this adjunction to an equivalence. Most of this section generalizes to the non-dagger setting.

A pedestrian approach to the adjunction between monoids in \cat{C} and strong monads on \cat{C} proceeds by definining the relevant functors and verifying that they work and form an adjunction directly. Our official approach to the equivalence between dagger Frobenius monoids and strong dagger Frobenius monads proceeds similarly. However, there is a higher-level approach to this that we'll explain briefly afterwards. Making the higher-level picture as detailed as the rest of the thesis would use more 2-categories than needed elsewhere, which is why it is not our official approach.

The abstract viewpoint will explain why -- in the Frobenius case -- one would expect to get an equivalence of categories rather than an adjunction. This will then justify our definition of a strong Frobenius monoid, which assumes that the strength natural transformation $\st_{A,B} \colon A \otimes F(B) \to F(A\otimes B)$ is unitary, which then directly implies that $F(-)\cong F(-\otimes I)\cong -\otimes F(I)$.

\begin{definition}\label{def:strength}\index[symb]{$\st$, strength natural transformation}
  A dagger functor $F$ between monoidal dagger categories is \emph{strong} if it is equipped with natural unitary morphisms
  $\st_{A,B} \colon A \otimes F(B) \to F(A\otimes B)$ satisfying $\st \circ \alpha = F(\alpha) \circ \st \circ (\id \otimes \st)$ and $F(\lambda) \circ \st = \lambda$.
  A (dagger Frobenius) monad on a monoidal dagger category is \emph{strong} if it is a strong dagger functor with $\st \circ (\id \otimes \mu) = \mu \circ T(\st) \circ \st$ and $\st \circ (\id \otimes \eta) = \eta$.
  A morphism $\beta$ of dagger Frobenius monads is \emph{strong} if $\beta \circ \st = \st \circ (\id \otimes \beta)$.
\end{definition} 

To prove the equivalence between dagger Frobenius monoids and strong dagger Frobenius monads, we need two lemmas.
 
\begin{lemma}\label{lem:froblaw} 
  If $T$ is a strong dagger Frobenius monad on a monoidal dagger category, then $T(I)$ is a dagger Frobenius monoid.
\end{lemma}
\begin{proof}
  Consider the diagram in \figurename~\ref{fig:froblaw}. 
  Region (i) commutes because $T$ is a dagger Frobenius monad, (ii) because $\mu^\dag$ is natural, (iii) because $\rho^{-1}$ is natural, (iv) because $\st^\dag$ is natural, (v) is a consequence of $T$ being a strong monad, (vi) commutes as $\rho$ is natural, (vii) and (viii) because $\st$ is natural, (ix) commutes trivially and (x) because $\st$ is natural. Regions (ii)'-(x)' commute for dual reasons. 
  Hence the outer diagram commutes.
\end{proof}

\begin{figure*}
  \centering
  \begin{sideways}
  \begin{tikzpicture}[font=\footnotesize,xscale=.825,yscale=.825]
  \matrix (m) [matrix of math nodes,row sep=2em,column sep=1.75em]
  {
     T(I)\otimes T(T(I)\otimes I) & T(I)\otimes (T(I)\otimes T(I)) & & (T(I)\otimes T(I))\otimes T(I) & T(T(I)\otimes I)\otimes T(I)) \\
     &T(T(I)\otimes (T(I)\otimes I)) & T((T(I)\otimes T(I))\otimes I) & T(T(I)\otimes I)\otimes T(I) & \\
     T(I)\otimes T^2(I) & T(T(I)\otimes T(I)) & T(T(T(I)\otimes I)\otimes I) & T(T^2(I)\otimes I) & T^2(I)\otimes T(I) \\
     &T^2(T(I)\otimes I) &T^3(I) & T(T(I)\otimes I)& \\ 
     T(I)\otimes T(I) & T^2(I) & &T^2(I) & T(I)\otimes T(I) \\
     & T(T(I)\otimes I) & T^3(I) & T^2(T(I)\otimes I) & \\
     T^2(I)\otimes T(I) & T(T^2(I)\otimes I) & T(T(T(I)\otimes I)\otimes I) & T(T(I)\otimes T(I)) & T(I)\otimes T^2(I) \\
     & T(T(I)\otimes I)\otimes T(I) & T((T(I)\otimes T(I))\otimes I) & T(T(I)\otimes (T(I)\otimes I)) & \\
     T(T(I)\otimes I)\otimes T(I) & (T(I)\otimes T(I))\otimes T(I) & & T(I)\otimes (T(I)\otimes T(I)) & T(I)\otimes T(T(I)\otimes I) \\};
   \path[->]
    (m-1-1) edge node [above] {$\id\otimes \st^\dag$} (m-1-2)
    (m-1-2) edge node [above] {$\alpha$} (m-1-4)
    (m-1-4) edge node [above] {$\st\otimes\id$} (m-1-5)
    (m-1-5) edge node [right] {$T(\rho)\otimes\id$} (m-3-5)
    (m-1-1) edge node [above] {$\st$} (m-2-2)
    (m-2-2) edge node [above] {$T(\alpha)$} (m-2-3)
    (m-2-3) edge node [above] {$\st^\dag$} (m-1-4)
    (m-2-4) edge node [right] {$\st^\dag\otimes\id$} (m-1-4) 
    (m-2-4) edge node [above] {$\quad\quad T(\rho)\otimes \id$} (m-3-5)
    (m-3-1) edge node [above] {$\st$} (m-3-2)
    (m-3-1) edge node [left] {$\id\otimes T(\rho^{-1})$} (m-1-1)
    (m-3-2) edge node [left] {$T(\id\otimes\rho^{-1})$} (m-2-2)
    (m-3-2) edge node [below] {$\quad T(\rho^{-1})$} (m-2-3)
    (m-3-3) edge node [above] {$T(T(\rho)\otimes\id)$} (m-3-4)
    (m-3-3) edge node [right] {$T(\st^\dag\otimes\id)$} (m-2-3)
    (m-3-3) edge node [right] {$\st^\dag$} (m-2-4)
    (m-3-4) edge node [above] {$\st^\dag$} (m-3-5)
    (m-3-4) edge node [right] {$T(\mu\otimes\id)$} (m-4-4)
    (m-3-5) edge node [right] {$\mu\otimes\id$} (m-5-5)
    (m-4-2) edge node [left] {$T(\st^\dag)$} (m-3-2)
    (m-4-2) edge node [right] {$\quad T(\rho^{-1})$} (m-3-3)
    (m-4-2) edge node [above] {$T^2(\rho)$} (m-4-3)
    (m-4-3) edge node [above] {$T(\mu_I)$} (m-5-4)
    (m-4-4) edge node [above] {$\st^\dag$} (m-5-5)
    (m-5-1) edge node [above] {$\st$} (m-6-2)
    (m-5-1) edge node [left] {$\id\otimes\mu^\dag$} (m-3-1)
    (m-5-1) edge node [left] {$\mu^\dag\otimes\id$} (m-7-1)
    (m-5-2) edge node [above] {$\mu^\dag_{T(I)}$} (m-4-3)
    (m-5-2) edge node [above] {$T(\mu^\dag_I)$} (m-6-3)
    (m-5-4) edge node [left] {$T(\rho^{-1})$} (m-4-4)
    (m-6-2) edge node [right] {$T(\rho)$} (m-5-2)
    (m-6-2) edge node [left] {$T(\mu^\dag\otimes\id)$} (m-7-2)
    (m-6-3) edge node [above] {$\mu_{T(I)}$} (m-5-4)
    (m-6-3) edge node [above] {$\quad T^2(\rho^{-1})$} (m-6-4)
    (m-7-1) edge node [above] {$\st$} (m-7-2)
    (m-7-1) edge node [left] {$T(\rho^{-1})\otimes\id$} (m-9-1)
    (m-7-1) edge node [above] {$\ \ \quad T(\rho^{-1})\otimes\id$} (m-8-2)
    (m-7-2) edge node [above] {$T(T(\rho^{-1})\otimes\id)$} (m-7-3)
    (m-7-3) edge node [above] {$T(\rho)$} (m-6-4)
    (m-7-4) edge node [right] {$T(\st)$} (m-6-4)
    (m-7-4) edge node [above] {$\st^\dag$} (m-7-5)
    (m-7-5) edge node [right] {$\id\otimes\mu$} (m-5-5)
    (m-8-2) edge node [above] {$\st$} (m-7-3)
    (m-8-3) edge node [right] {$T(\st\otimes\id)$} (m-7-3)
    (m-8-3) edge node [right] {$\quad T(\rho)$} (m-7-4)
    (m-8-3) edge node [above] {$T(\alpha^{-1})$} (m-8-4)
    (m-8-4) edge node [right] {$T(\id\otimes\rho)$} (m-7-4)
    (m-8-4) edge node [above] {$\ \ \st^\dag$} (m-9-5)
    (m-9-1) edge node [below] {$\st^\dag\otimes\id $} (m-9-2)
    (m-9-2) edge node [left] {$\st\otimes\id $} (m-8-2)
    (m-9-2) edge node [above] {$\st$} (m-8-3)
    (m-9-2) edge node [below] {$\alpha^{-1}$} (m-9-4)
    (m-9-4) edge node [below] {$\id\otimes \st$} (m-9-5)
    (m-9-5) edge node [right] {$\id\otimes T(\rho)$} (m-7-5)
    (m-6-2) edge [bend left=30] (m-4-2)
    (m-6-4) edge [bend right=30] (m-4-4);
    \node at (-6.5,0) {$\mu^\dag$};
    \node at (6.25,0) {$\mu$}; 
    \node[gray] at (0,0) {(i)}; 
    \node[gray] at (-5,1) {(ii)}; 
    \node[gray] at (5,-1) {(ii)'}; 
    \node[gray] at (3,2) {(iii)}; 
    \node[gray] at (-3,-2) {(iii)'};
    \node[gray] at (8,2.5) {(iv)}; 
    \node[gray] at (-8,-2.5) {(iv)'}; 
    \node[gray] at (-8,1.5) {(v)'}; 
    \node[gray] at (8,-1.5) {(v)}; 
    \node[gray] at (2.75,-3.5) {(vi)}; 
    \node[gray] at (-2.75,3.5) {(vi)'};
    \node[gray] at (-2.75,-5.25) {(vii)}; 
    \node[gray] at (2.75,5.25) {(vii)'}; 
    \node[gray] at (-5,-4.5) {(viii)}; 
    \node[gray] at (5,4.5) {(viii)'};
    \node[gray] at (-8,-6) {(ix)}; 
    \node[gray] at (8,6) {(ix)'};
    \node[gray] at (-8,5) {(x)};
    \node[gray] at (8,-5) {(x)'}; 
    \node[gray] at (-4.25,4.75) {(xi)}; 
    \node[gray] at (4.5,-4.75) {(xi)'};
    \node[gray] at (-2.65,6.5) {(xii)}; 
    \node[gray] at (2.65,-6.5) {(xii)'};
  \end{tikzpicture}
  \end{sideways}
  \caption{Diagram proving that $T\mapsto T(I)$ preserves the Frobenius law.}
  \label{fig:froblaw}
\end{figure*}

\begin{lemma}\label{lem:strictmorphism}
  If $T$ is a strong dagger Frobenius monad on a monoidal dagger category, then $T\rho \circ \st \colon A \otimes T(I) \to T(A)$ preserves $\eta^\dag$ and $\mu^\dag$.
\end{lemma}
\begin{proof}
  To show that $\eta^\dag$ is preserved, it suffices to see that
  \[\begin{tikzpicture}[font=\small]
  \matrix (m) [matrix of math nodes,row sep=2em,column sep=4em,minimum width=2em]
  {
     A\otimes T(I) & \\
     T(A\otimes I) & A\otimes I \\ 
     T(A) &  A\\};
  \path[-stealth]
    (m-1-1) edge node [left] {$\st_{A,I}$} (m-2-1)
            edge node [right] {$\ \ \id\otimes \eta^\dag_I$} (m-2-2)
    (m-2-1) edge node [below] {$\eta^\dag_{A\otimes I}$} (m-2-2)
    (m-2-1) edge node [left] {$T(\rho_A)$} (m-3-1)
    (m-2-2) edge node [right] {$\rho_A$} (m-3-2)
     (m-3-1) edge node [below] {$\eta^\dag_A$} (m-3-2);
  \end{tikzpicture}\]
  commutes. But the rectangle commutes because $\eta^\dag$ is natural, and the triangle commutes because $T$ is a strong monad and strength is unitary.

  To see that $\mu^\dag$ is preserved, consider the following diagram:
  \[\begin{tikzpicture}[font=\small,yscale=.75,xscale=.8]
    \matrix (m) [matrix of math nodes,row sep=2em,column sep=2.5em]
    {
     A\otimes T(I) & A\otimes T^2(I) &A\otimes T(T(I)\otimes I) &A\otimes (T(I)\otimes T(I)) \\
      & A\otimes T^2(I) &A\otimes T(T(I)\otimes I) &(A\otimes T(I))\otimes T(I) \\
      &  &T(A\otimes (T(I)\otimes I)) &T((A\otimes T(I))\otimes I) \\
      &  & &T(A\otimes T(I)) \\
     T(A\otimes I) &  & &T^2(A\otimes I) \\
     T(A) &  & &T^2(A) \\};
    \path[->]
    (m-1-1) edge node [left] {$\st$} (m-5-1)
            edge node [above] {$\id\otimes\mu^\dag$} (m-1-2)
    (m-1-2) edge node [above] {$\id\otimes T(\rho^{-1})$} (m-1-3)
    (m-1-3) edge node [above] {$\id\otimes \st^\dag$} (m-1-4)
    (m-2-3) edge node [above] {$\id\otimes T(\rho)$} (m-2-2)
    (m-1-2) edge node [right] {$\id$} (m-2-2)
    (m-1-4) edge node [right] {$\alpha$} (m-2-4)
    (m-2-4) edge node [right] {$\st$} (m-3-4)
    (m-3-4) edge node [right] {$T(\rho)$} (m-4-4)
    (m-4-4) edge node [right] {$T(\st)$} (m-5-4)
    (m-5-4) edge node [right] {$T^2(\rho)$} (m-6-4)
    (m-5-1) edge node [left] {$T(\rho)$} (m-6-1) 
    (m-6-1) edge node [below] {$\mu^\dag$} (m-6-4)
    (m-2-3) edge node [right] {$\st$} (m-3-3) 
    (m-1-4) edge node [below] {$\quad\id\otimes \st$} (m-2-3) 
    (m-3-3) edge node [above] {$T(\alpha)$} (m-3-4)
    (m-3-3) edge node [below] {$T(\id\otimes\rho)\quad\quad$} (m-4-4)
    (m-5-1) edge node [below] {$\mu^\dag$} (m-5-4)
    (m-2-2) edge [bend right=30] (m-4-4);
    \node at (-.2,-0.7) {$\st$};
    \node[gray] at (-4,1) {(i)};
    \node[gray] at (0,4) {(ii)};
    \node[gray] at (0,2) {(iii)};
    \node[gray] at (4,2) {(iv)};
    \node[gray] at (4,0.5) {(v)}; 
    \node[gray] at (0,-4) {(vi)};
  \end{tikzpicture}\]
  Commutativity of region (i) is a consequence of strength being unitary, (ii) commutes by definition, (iii) commutes as strength is natural, (iv) because $T$ is a strong functor, (v) by coherence and finally (vi) by naturality of $\mu^\dag$. 
\end{proof}

\begin{theorem}\label{thm:strong}
  Let $\cat{C}$ be a monoidal dagger category. The operations $B \mapsto - \otimes B$ and $T \mapsto T(I)$ form an equivalence between dagger Frobenius monoids in $\cat{C}$ and strong dagger Frobenius monads on $\cat{C}$.
\end{theorem} 
\begin{proof}
  It is well-known that $B \mapsto - \otimes B$ is left adjoint to $T \mapsto T(I)$, when considered as maps between ordinary monoids and ordinary strong monads, see~\cite{wolff:monads}; the unit of the adjunction is $\lambda \colon I \otimes B \to B$, and the counit is determined by $T\rho \circ \st \colon A \otimes T(I)\to T(A)$. 

  Example~\ref{example:tensor} already showed that $B\mapsto -\otimes B$ preserves the Frobenius law. 
  Moreover, Lemma~\ref{lem:froblaw} shows that $T \mapsto T(I)$ preserves the Frobenius law as well.
  It remains to prove that they form an equivalence. 
  Clearly the unit of the adjunction is a natural isomorphism. 
  To see that the counit is also a natural isomorphism, combine Lemmas~\ref{lem:strictmorphism} and~\ref{lem:strictmorphismsareiso}.
\end{proof}

It follows from the previous theorem that not every dagger Frobenius monad is strong: as discussed in Section~\ref{sec:monads}, the monad induced by the dagger adjunction of Example~\ref{ex:imdagadj} is not of the form $- \otimes B$ for fixed $B$.

Let us now sketch the higher-level viewpoint promised at the start of this section in order to justify our assumption that $\st$ is unitary. If $A$ is an object of a monoidal category \cat{C}, then $-\otimes A$ is an endofunctor on \cat{C}. Conversely, any endofunctor $F$ on \cat{C} gives an object $F(I)$ of \cat{C}. The category of endofunctors is monoidal (with composition of functors as $\otimes$) and the functor $A\mapsto -\otimes A$ is strong monoidal. However, for arbitrary endofunctors, there are no obvious maps between $F(I)\otimes G(I)$ and $FG(I)$, so the converse isn't in general monoidal unless one imposes further conditions on the endofunctors. Similarly, while $A\mapsto -\otimes A\mapsto I\otimes A$ results in a functor isomorphic to $\id[\cat{C}]$, there is no obvious natural transformation between $-\otimes F(I)$ and $F$. A reasonable way of getting a natural transformation $-\otimes F(I)\to F$ is by requiring $F$ to be strong: \ie asking for a natural transformation $\st_{A,B} \colon A \otimes F(B) \to F(A\otimes B)$ subject to some further conditions. Hence, if one restricts to strong functors and transformations of strong functors, the map $F\mapsto F(I)$ becomes (lax) monoidal and in fact it is the right adjoint to $A\mapsto -\otimes I$. 

Once this has been verified, it is not hard obtain the adjunction between monoids in \cat{C} and and strong monads on \cat{C}: monoidal categories, lax monoidal functors and monoidal transformations organize themselves into a 2-category. If $1$ is the terminal monoidal category, then $\hom(1,-)$ is a 2-functor into the 2-category of categories. Any 2-functor preserves adjunctions inside a 2-category, so applying $\hom(1,-)$ to the adjunction between \cat{C} and strong endofunctors on \cat{C} results in an adjunction. Moreover, a lax monoidal functor $1\to \cat{C}$ is the same thing as a monoid in \cat{C} (and a monoid homomorphism is the same thing as a monoidal transformation between the functors). Hence the resulting adjunction is the adjunction between monoids in \cat{C} and monoids in the category of strong endofunctors on \cat{C}, \ie the category of strong monads on \cat{C}.

Dualizing this, one obtains an adjunction between comonoids in \cat{C} and costrong comonads on \cat{C}, the functor $A\mapsto -\otimes A$ now being the right adjoint. Since a Frobenius algebra is both a monoid and a comonoid, one might then expect a similar story, resulting in an ambijunction. For such a story to work, one first needs to find the correct category of monoidal categories, in which a functor $1\to \cat{C}$ corresponds to a Frobenius algebra in \cat{C}. This is given by the notion of a Frobenius monoidal functor~\cite{mccurdystreet:frobeniusmonoidal,day:noteonfrob} -- a functor that is both lax monoidal and oplax monoidal and satisfies a coherence law. Then to get an ambijunction between Frobenius algebras and (strong) Frobenius monads, one only needs to find an ambijunction between \cat{C} and suitable endofunctors on it, so that both adjoints are Frobenius monoidal. Since $A\mapsto -\otimes A$ is strong monoidal there is nothing to worry about. For $F\mapsto F(I)$ to be Frobenius monoidal, one needs to consider functors that are both strong, costrong and satisfy a coherence law. A morphism of such functors will then at the very least be a morphism of the underlying strong and costrong functors. But this is enough to ensure that applying $\hom (1,-)$ to the adjunction results in an equivalence: the counit $-\otimes F(I)\to F(-)$ has to be a morphism of strong and costrong functors, and this is enough to ensure that Lemma~\ref{lem:strictmorphism} goes through. Hence the counit is a strict morphism of Frobenius structures, and thus an isomorphism by Lemma~\ref{lem:strictmorphismsareiso}. Since the counit is built from the monoidal unitor and strength, the strength has to be an isomorphism as well. The meaning of this is that, if an endofunctor $F$ is strong, costrong and satisfies the required coherence law, and moreover admits a Frobenius monad structure in which the (co)multiplication and (co)unit are morphisms of strong and costrong functors, then the strength and costrength transformations are necessarily inverses to each other. Hence, if one just wants to obtain the resulting equivalence and doesn't care too much about the adjunction between \cat{C} and coherently strong and costrong endofunctors on it, one might as well start by restricting to those functors for which the strength natural transformation is an isomorphism. Since we are working in the dagger setting, it is natural to assume that the strength map is in fact unitary. 

A Frobenius monoid in a monoidal dagger category is \emph{special} when 
$\smash{\tinymult \circ\tinycomult=\begin{pic} \draw (0,0) to (0,.45); \end{pic}}$.
Theorem~\ref{thm:strong} restricts to an equivalence between special dagger Frobenius monoids and special strong dagger Frobenius monads.

For symmetric monoidal dagger categories, there is also a notion of commutativity for strong monads~\cite{kock:strong,jacobs:weakening}. 
Given a strong dagger Frobenius monad $T$, one can define a natural transformation $\st'_{A,B} \colon T(A) \otimes B \to T(A \otimes B)$ by $T(\sigma_{B,A})\circ st_{B,A}\circ \sigma_{T(A),B}$, and
\begin{align*} 
  dst_{A,B} & = \mu_{A\otimes B}\circ T(st'_{A,B})\circ st_{T(A),B}, \\ 
  dst_{A,B}' & =\mu_{A\otimes B}\circ T(st_{A,B})\circ st'_{A,T(B)}.
\end{align*} 
The strong dagger Frobenius monad is \emph{commutative} when these coincide.
Theorem~\ref{thm:strong} restricts to an equivalence between commutative dagger Frobenius monoids and commutative strong dagger Frobenius monads.
Kleisli categories of commutative monads on symmetric monoidal categories are again symmetric monoidal~\cite{day:kleisli}. This extends to dagger categories.

\begin{theorem}
  If $T$ is a commutative strong dagger Frobenius monad on a symmetric monoidal dagger category $\cat{C}$, then $\Kl(T)$ is a symmetric monoidal dagger category.
\end{theorem}
\begin{proof} 
  The monoidal structure on $\Kl(T)$ is given by $A\otimes_T B=A\otimes B$ on objects and by $f\otimes_T g=\dst\circ (f\otimes g)$ on morphisms. 
  The coherence isomorphisms of $\Kl(T)$ are images of those in $\cat{C}$ under the functor $\cat{C}\to\Kl(T)$. 
  This functor preserves daggers and hence unitary isomorphisms, making all coherence isomorphisms of $\Kl(T)$ unitary. 
  It remains to check that the dagger on $\Kl(T)$ satisfies $(f\otimes_T g)^\dag=f^\dag\otimes_T g^\dag$. Theorem~\ref{thm:strong} makes $T$ isomorphic to $S=-\otimes T(I)$, and this induces an isomorphism between the respective Kleisli categories that preserves daggers and monoidal structure on the nose. Thus it suffices to check that this equation holds on $\Kl(S)$:
  \[
    (f \otimes_S g)^\dag
    \quad = \quad
    \begin{pic}[scale=.5]
      \node[morphism,hflip] (f) at (0,0) {$f$};
      \node[morphism,hflip] (g) at (2,0) {$g$};
      \node[dot] (d) at (1.25,-.85) {};
      \node[dot] (e) at (2.5,-1.5) {};
      \draw (f.north) to +(0,1.25);
      \draw (g.north) to +(0,1.25);
      \draw ([xshift=-1mm]f.south west) to +(0,-2.5);
      \draw ([xshift=-1mm]g.south west) to +(0,-2.5);
      \draw ([xshift=1mm]f.south east) to[out=-90,in=180] (d.west);
      \draw ([xshift=1mm]g.south east) to[out=-90,in=0] (d.east);
      \draw (d.south) to[out=-90,in=180] (e.west);
      \draw (e.south) to +(0,-.5) node[dot]{};
      \draw (e.east) to[out=0,in=-90] +(1,1) to +(0,2.25) node[right] {$T(I)$};
    \end{pic}
    = \quad
    \begin{pic}[scale=.5]
      \node[morphism,hflip] (f) at (0,0) {$f$};
      \node[morphism,hflip] (g) at (2.5,0) {$g$};
      \node[dot] (d) at (2.75,1.2) {};
      \node[dot] (df) at (.75,-.85) {};
      \node[dot] (dg) at (3.25,-.85) {};
      \draw (f.north) to +(0,2);
      \draw ([xshift=-1mm]g.north) to[out=90,in=-90,looseness=.7] +(-1.5,2);
      \draw ([xshift=-1mm]f.south west) to +(0,-2);
      \draw ([xshift=-1mm]g.south west) to +(0,-2);
      \draw ([xshift=1mm]f.south east) to[out=-90,in=180] (df.west);
      \draw ([xshift=1mm]g.south east) to[out=-90,in=180] (dg.west);
      \draw (df.south) to +(0,-.5) node[dot]{};
      \draw (dg.south) to +(0,-.5) node[dot]{};
      \draw (df.east) to[out=0,in=180,looseness=.7] (d.west);
      \draw (dg.east) to[out=0,in=0] (d.east);
      \draw (d.north) to +(0,1) node[right] {$T(I)$};
    \end{pic}
    \quad = \quad
    f^\dag \otimes_S g^\dag\text{.}
  \]
  But this is a straightforward graphical argument
  
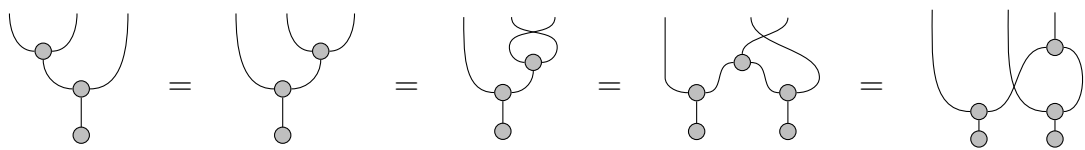
\[
    \begin{pic}[scale=.5]
      \node[dot] (l) at (0,0) {};
      \node[dot] (r) at (1,-1) {};
      \draw (r.south) to +(0,-1) node[dot]{};
      \draw (l.south) to[out=-90,in=180] (r.west);
      \draw (r.east) to[out=0,in=-90] +(1,2);
      \draw (l.east) to[out=0,in=-90] +(.66,1);
      \draw (l.west) to[out=180,in=-90] +(-.66,1);
    \end{pic}
    \quad = \quad
    \begin{pic}[scale=.5]
      \node[dot] (l) at (1,0) {};
      \node[dot] (r) at (0,-1) {};
      \draw (r.south) to +(0,-1) node[dot]{};
      \draw (l.south) to[out=-90,in=0] (r.east);
      \draw (r.west) to[out=180,in=-90] +(-1,2);
      \draw (l.west) to[out=180,in=-90] +(-.66,1);
      \draw (l.east) to[out=0,in=-90] +(.66,1);
    \end{pic}
    \quad = \quad    
    \begin{pic}[scale=.4]
      \node[dot] (l) at (1,0) {};
      \node[dot] (r) at (0,-1) {};
      \draw (r.south) to +(0,-1) node[dot]{};
      \draw (l.south) to[out=-90,in=0] (r.east);
      \draw (r.west) to[out=180,in=-90] +(-1,2.5);
      \draw (l.west) to[out=180,in=-90] +(-.5,.5) to[out=90,in=-90] +(1.5,1);
      \draw (l.east) to[out=0,in=-90] +(.5,.5) to[out=90,in=-90] +(-1.5,1);
    \end{pic}
    \quad = \quad
    \begin{pic}[yscale=.4,xscale=.6]
      \node[dot] (l) at (0,0) {};
      \node[dot] (m) at (1,1) {};
      \node[dot] (r) at (2,0) {};
      \draw (l.south) to +(0,-1) node[dot]{};
      \draw (r.south) to +(0,-1) node[dot]{};
      \draw (l.east) to[out=0,in=180] (m.west);
      \draw (m.east) to[out=0,in=180] (r.west);
      \draw (l.west) to[out=180,in=-90] +(-.5,.5) to +(0,2);
      \draw (m.north) to[out=90,in=-90] +(1,1.2);
      \draw (r.east) to[out=0,in=-90] +(.5,.5) to[out=90,in=-90] +(-1.5,2);
    \end{pic}
    \quad = \quad
    \begin{pic}[scale=.5]
      \node[dot] (l) at (0,0) {};
      \node[dot] (r) at (2,0) {};
      \node[dot] (t) at (2,1.7) {};
      \draw (l.south) to +(0,-.5) node[dot]{};
      \draw (r.south) to +(0,-.5) node[dot]{};
      \draw (l.east) to[out=0,in=180] (t.west);
      \draw (r.east) to[out=0,in=0] (t.east);
      \draw (t.north) to +(0,.7);
      \draw (l.west) to[out=180,in=-90] +(-1,2.7);
      \draw (r.west) to[out=180,in=-90] +(-1,2.7);
    \end{pic}
  \]
  using associativity, commutativity, the unit law, and Lemma~\ref{lem:extendedfrobeniuslaw}.
\end{proof}

\section{Closure}\label{sec:coherence}

This final section justifies the Frobenius law from first principles, by explaining it as a coherence property between daggers and closure. In a monoidal dagger category that is closed, monoids and daggers interact in two ways.
First, any monoid picks up an involution by internalizing the dagger.
Second, any monoid embeds into an endohomset by closure, and the dagger is an involution on the endohomset.
The Frobenius law is equivalent to the property that these two canonical involutions coincide.

We start by giving an equivalent formulation of the Frobenius law.

\begin{lemma}\label{lem:equivalentformofFrob}
  A monoid $(A,\tinymult,\tinyunit)$ in a monoidal dagger category is a dagger Frobenius monoid if and only if it satisfies the following equation.
  \begin{equation}\label{eq:equivalentfrobeniuslaw}
      \begin{pic}[scale=.45]
        \node[dot] (t) at (0,1) {};
        \draw (t) to +(0,1.5);
        \draw (t) to[out=0,in=90] (1,0) to (1,-1);
        \draw (t) to[out=180,in=90] (-1,0) to (-1,-1);
      \end{pic}
      \quad 
      =
      \quad
      \begin{pic}[scale=.85]
      \node[dot] (m) at (0,-.5) {};
      \node[dot] (r) at (1,.5) {};
      \draw (r.north) to (1,1) node[dot] {};
      \draw (r.east) to[out=0,in=90,looseness=.5] (1.5,-1);
      \draw (m.south) to (0,-1);
      \draw (m.east) to[out=0,in=180] (r.west);
      \draw (m.west) to[out=180,in=-90] +(-.35,.5) to +(0,.7);
    \end{pic}
    \end{equation}
\end{lemma}
\begin{proof}
  The Frobenius law~\eqref{eq:frobeniuslaw} directly implies~\eqref{eq:equivalentfrobeniuslaw}.
  Conversely, \eqref{eq:equivalentfrobeniuslaw} gives
  \[
    \begin{pic}[yscale=0.75]
          \node (0a) at (-0.5,0) {};
          \node (0b) at (0.5,0) {};
          \node[dot] (1) at (0,1) {};
          \node[dot] (2) at (0,2) {};
          \node (3a) at (-0.5,3) {};
          \node (3b) at (0.5,3) {};
          \draw[out=90,in=180] (0a) to (1.west);
          \draw[out=90,in=0] (0b) to (1.east);
          \draw (1.north) to (2.south);
          \draw[out=180,in=270] (2.west) to (3a);
          \draw[out=0,in=270] (2.east) to (3b);
    \end{pic}
    \;=\;
    \begin{pic}
      \node[dot] (l) at (-.75,.5) {};
      \node[dot] (m) at (0,-.5) {};
      \node[dot] (r) at (1,.5) {};
      \draw (r.north) to (1,1) node[dot] {};
      \draw (r.east) to[out=0,in=90,looseness=.5] (1.5,-1);
      \draw (m.south) to (0,-1);
      \draw (m.east) to[out=0,in=180] (r.west);
      \draw (m.west) to[out=180,in=-90] (l.south);
      \draw (l.east) to[out=0,in=-90] (-.25,1);
      \draw (l.west) to[out=180,in=-90] (-1.25,1);
    \end{pic}
    \;=\;
    \begin{pic}[xscale=.75]
      \node[dot] (l) at (-.75,-.5) {};
      \node[dot] (m) at (0,0) {};
      \node[dot] (r) at (1,.5) {};
      \draw (l.south) to (-.75,-1);
      \draw (l.west) to[out=180,in=-90] (-1.5,1);
      \draw (l.east) to[out=0,in=-90] (m.south);
      \draw (m.west) to[out=180,in=-90] (-.5,1);
      \draw (m.east) to[out=0,in=180] (r.west);
      \draw (r.north) to (1,1) node[dot] {};
      \draw (r.east) to[out=0,in=90,looseness=.5] (1.5,-1);
    \end{pic}
    \;=\;
    \begin{pic}[yscale=0.75,xscale=-1]
          \node (0) at (0,0) {}; 
          \node (0a) at (0,1) {};
          \node[dot] (1) at (0.5,2) {};
          \node[dot] (2) at (1.5,1) {};
          \node (3) at (1.5,0) {};
          \node (4) at (2,3) {};
          \node (4a) at (2,2) {};
          \node (5) at (0.5,3) {};
          \draw (0) to (0a.center);
          \draw[out=90,in=180] (0a.center) to (1.east);
          \draw[out=0,in=180] (1.west) to (2.east);
          \draw[out=0,in=270] (2.west) to (4a.center);
          \draw (4a.center) to (4);
          \draw (2.south) to (3);
          \draw (1.north) to (5);
    \end{pic}
  \]
  by associativity. But since the left-hand side is self-adjoint, so is the right-hand side, giving the Frobenius law~\eqref{eq:frobeniuslaw}.
\end{proof}

Any dagger Frobenius monoid forms a duality with itself in the following sense.

\begin{definition}
  Morphisms $\eta \colon I \to A \otimes B$ and $\varepsilon \colon B \otimes A \to I$ in a monoidal category \emph{form a duality} when they satisfy the following equations.
  \[
    \begin{pic}
    \node[morphism] (u) at (0,0) {$\eta$};
    \node[morphism] (c) at (.75,.75) {$\varepsilon$};
    \draw ([xshift=1mm]u.north east) to[out=90,in=-90] ([xshift=-1mm]c.south west);
    \draw ([xshift=-1mm]u.north west) to +(0,1);
    \draw ([xshift=1mm]c.south east) to +(0,-1);
    \end{pic}
    \; = \; 
    \begin{pic}
      \draw (0,0) to (0,1.7);
    \end{pic}
    \qquad\qquad
    \begin{pic}[xscale=-1]
    \node[morphism] (u) at (0,0) {$\eta$};
    \node[morphism] (c) at (.75,.75) {$\varepsilon$};
    \draw ([xshift=1mm]u.north west) to[out=90,in=-90] ([xshift=-1mm]c.south east);
    \draw ([xshift=-1mm]u.north east) to +(0,1);
    \draw ([xshift=1mm]c.south west) to +(0,-1);
    \end{pic}
    \; = \; 
    \begin{pic}
      \draw (0,0) to (0,1.7);
    \end{pic}
  \]
\end{definition}

In the categories $\cat{Rel}$ and $\cat{FHilb}$, every object $A$ is part of a duality $(A,B,\eta,\varepsilon)$: they are \emph{compact} categories. Moreover, in those categories we may choose $\varepsilon = \eta^\dag \circ \sigma$: they are \emph{compact dagger categories}.

Let \cat{C} be a monoidal dagger category that is also a closed monoidal category, so that there is a correspondence of morphisms $A \otimes B \to C$  and $A \to [B, C]$ called \emph{currying}. Write $A^*$ for $[A,I]$, and write $\ev$ for the counit $A^* \otimes A \to I$. If $(A,\tinymult,\tinyunit)$ is a monoid in \cat{C}, then $(A,\tinycomult,\tinycounit)$ is a comonoid in $\cat{C}$ and $A^*$ becomes a monoid with unit and multiplication given by currying $\tinycounit \colon I \otimes A \to I$ and 
\[
    \begin{pic}[yscale=.7]
      \node[morphism] (a)  at (0,0.75) {ev}; 
      \node[morphism] (c)  at (-.25,2) {ev}; 
      \node[dot] (b) at (.5,0) {};
      \draw ([xshift=-1mm]a.south west) to +(0,-1);
      \draw ([xshift=-1mm]c.south west) to[out=-90,in=90] +(-.25,-2.25);
      \draw (b.east) to[out=0,in=-90] +(.25,.75) to[out=90,in=-90] (c.south east);
      \draw (a.south east) to[out=-90,in=180]  (b.west);
      \draw (b) to +(0,-.5);
    \end{pic} \;\colon (A^* \otimes A^*) \otimes A \to I \text{.}
\]
$A$ and $A^*$ are related by a map $i\colon A\to A^*$ obtained by currying 
$\begin{pic}[scale=.3]
      \node[dot] (b) at (1,0) {};
      \node[dot] (c) at (1,1) {};
      \draw (b) to  (c);
      \draw (b) to[out=180,in=90] (0,-1);
      \draw (b) to[out=0,in=90] (2,-1);
\end{pic}$.

\begin{proposition}\label{prop:closedfrobenius}
  A monoid $(A,\tinymult,\tinyunit)$ in a monoidal dagger category that is also a closed monoidal category is a dagger Frobenius monoid if and only if $i\colon A\to A^*$ is a monoid homomorphism and $\ev \colon A^* \otimes A \to I$ forms a duality with
  \begin{equation}\label{eq:frobeniusduality}
      \begin{pic}[scale=.4]
          \node[dot] (a) at (0,0) {};
          \node[dot] (b) at (0,-1) {};
          \node[morphism] (c) at (1,1.5) {$i$};
          \draw (b.north) to  (a.south);
          \draw (a.east) to[out=0, in=-90]  (c.south);
          \draw (a.west) to[out=180,in=-90,looseness=.8] (-1,3);
          \draw ([yshift=.5mm]c.north) to  (1,3);
      \end{pic}\;\colon I \to A \otimes A^* \text{.}
  \end{equation}
\end{proposition}
\begin{proof}
  The morphism $i$ always preserves units: $\ev \circ (i \otimes \id) \circ (\tinyunit \otimes \id) = \tinycounit \circ \tinymult \circ (\tinyunit \otimes \id) = \tinycounit$.
  It preserves multiplication precisely when:
  \begin{equation}\label{eq:altfrob1}
      \begin{pic}[scale=.4]
        \node[dot] (t) at (0,1) {};
        \node[dot] (b) at (1,0) {};
        \node[dot] (c) at (0,2) {};
        \draw (c) to (t);
        \draw (t) to[out=0,in=90] (b);
        \draw (t) to[out=180,in=90] (-1,0) to (-1,-1);
        \draw (b) to[out=180,in=90] (0,-1);
        \draw (b) to[out=0,in=90] (2,-1);
      \end{pic}
      \quad = \quad
      \begin{pic}[yscale=.4,xscale=-.4]
        \node[dot] (t) at (0,1) {};
        \node[dot] (b) at (1,0) {};
        \node[dot] (c) at (0,2) {};
        \draw (c) to (t);
        \draw (t) to[out=0,in=90] (b);
        \draw (t) to[out=180,in=90] (-1,0) to (-1,-1);
        \draw (b) to[out=180,in=90] (0,-1);
        \draw (b) to[out=0,in=90] (2,-1);
      \end{pic}
      \quad = \quad
      \begin{pic}[scale=.4]
        \node[dot] (b) at (1.25,-.25) {};
        \node[morphism] (c) at (1.25,1) {$i$};
        \node[morphism] (d) at (2.5,3.2) {ev};
        \draw (b.north) to  (c.south);
        \draw (b.west) to[out=180,in=90] +(-.75,-1);
        \draw (b.east) to[out=0,in=90] +(.75,-1);
        \draw ([xshift=-1.5mm]d.south west) to[out=-90,in=90] (c.north);
        \draw ([xshift=1mm]d.south east) to +(0,-4);
      \end{pic}
      \quad = \quad
      \begin{pic}[scale=.4]
        \node[dot] (a) at (0,1) {};
        \node[morphism] (b) at (-.8,-.5) {$i$};
        \node[morphism] (c) at (.8,-.5) {$i$};
        \node[morphism] (d) at (1.5,2.75) {ev};
        \draw (a.west) to[out=180,in=90] (b.north);
        \draw (a.east) to[out=0,in=90] (c.north);
        \draw (a.north) to[out=90,in=-90] ([xshift=-1mm]d.south west);
        \draw (b.south) to +(0,-.75);
      \draw (c.south) to +(0,-.75);
        \draw ([xshift=1mm]d.south east) to[out=-90,in=90] +(0,-4);
      \end{pic}
      \quad = \quad
      \begin{pic}[scale=.75]
      \node[dot] (a)  at (-.5,0.75) {};
      \node[dot] (d)  at (-.5,1.25) {}; 
      \node[dot] (c)  at (-.5,1.75) {}; 
      \node[dot] (e)  at (-.5,2.25) {};
      \node[dot] (b) at (.2,.2) {};
      \draw (d) to (a);
      \draw (e) to (c);
      \draw (a.west) to[out=180,in=90,looseness=.8] +(-.25,-1);
      \draw (c.west) to[out=180,in=90,looseness=.8] +(-.75,-2);
      \draw (b.east) to[out=0,in=0] (c.east);
      \draw (a) to[out=0,in=180]  (b);
      \draw (b) to +(0,-.45);
      \end{pic}
  \end{equation}
  Furthermore, $\ev \colon A^* \otimes A \to I$ and~\eqref{eq:frobeniusduality} form a duality precisely when:
  \begin{equation}\label{eq:altfrob2}
      \begin{pic}
      \draw (0,0) to (0,2);
    \end{pic}
    \; = \;
      \begin{pic}[scale=.4]
        \node[dot] (a) at (0,0) {};
        \node[dot] (b) at (0,-1) {};
        \node[morphism] (c) at (1,1.5) {$i$};
        \node[morphism] (d) at (1.6,3) {ev};
        \draw (b) to  (a);
        \draw (a.east) to[out=0, in=-90]  (c.south);
        \draw (a.west) to[out=180,in=-90,looseness=.6] (-1,3.5);
        \draw ([xshift=-2mm]d.south west) to  (c.north);
        \draw ([xshift=2mm]d.south east) to  +(0,-4);
      \end{pic}
    \; = \;
      \begin{pic}[scale=.4]
        \node[dot] (a) at (0,0) {};
        \node[dot] (b) at (0,-1) {};
        \node[dot] (c) at (1.75,1.5) {};
        \node[dot] (d) at (1.75,2.5) {};
        \draw (b) to  (a);
        \draw (d) to  (c);
        \draw (a) to[out=0, in=180]  (c);
        \draw (a.west) to[out=180,in=-90,looseness=.6] (-1,3.5);
        \draw (c.east) to[out=0, in=90,looseness=.8]  +(1,-3);
      \end{pic}
      \qquad \qquad \qquad
      \begin{pic}
      \draw (0,0) to (0,2);
    \end{pic}
      \; = \;
      \begin{pic}[scale=.4]
        \node[dot] (a) at (.25,0) {};
        \node[dot] (b) at (.25,-1) {};
        \node[morphism] (c) at (1,1.5) {$i$};
        \node[morphism] (d) at (-1,1.5) {ev};
        \draw (b) to  (a);
        \draw (a) to[out=0, in=-90]  (c.south);
        \draw ([xshift=1mm]d.south east) to[out=-90,in=180] (a);
        \draw ([xshift=-1mm]d.south west) to +(0,-2.5);
        \draw ([yshift=.5mm]c.north) to  (1,3);
      \end{pic}
  \end{equation}
  By evaluating both sides, it is easy to see that the left equation implies the right one.
  
    Now, assuming the Frobenius law~\eqref{eq:frobeniuslaw}, Lemma~\ref{lem:extendedfrobeniuslaw} and the unit law guarantee that (the left equation of)~\eqref{eq:altfrob2} is satisfied, as well as~\eqref{eq:altfrob1}:
    \[  
      \begin{pic}[scale=.75]
      \node[dot] (a)  at (-.5,0.75) {};
      \node[dot] (d)  at (-.5,1.25) {}; 
      \node[dot] (c)  at (-.5,1.75) {}; 
      \node[dot] (e)  at (-.5,2.25) {};
      \node[dot] (b) at (.2,.2) {};
      \draw (d) to (a);
      \draw (e) to (c);
      \draw (a.west) to[out=180,in=90,looseness=.8] +(-.25,-1);
      \draw (c.west) to[out=180,in=90,looseness=.8] +(-.75,-2);
      \draw (b.east) to[out=0,in=0] (c.east);
      \draw (a) to[out=0,in=180]  (b);
      \draw (b) to +(0,-.45);
      \end{pic}
      \quad = \quad
      \begin{pic}[scale=.75]
    \node[dot] (b) at (0,3) {};
    \node[dot] (c) at (0,2) {};
    \node[dot] (d) at (0,1.5) {};
    \draw (d.east) to[out=0,in=90] +(.3,-.5);
    \draw (d.west) to[out=180,in=90] +(-.3,-.5);
    \draw (d.north) to (c.south);
    \draw (c.west) to[out=180,in=-90] +(-.25,.5) node[dot]{};
    \draw (c.east) to[out=0,in=0] (b.east);
    \draw (b.north) to +(0,.5) node[dot]{};
    \draw (b.west) to[out=180,in=90] +(-1,-2);
      \end{pic}
      \quad = \quad
      \begin{pic}[scale=.4]
        \node[dot] (t) at (0,1) {};
        \node[dot] (b) at (1,0) {};
        \node[dot] (c) at (0,2) {};
        \draw (c) to (t);
        \draw (t) to[out=0,in=90] (b);
        \draw (t) to[out=180,in=90] (-1,0) to (-1,-1);
        \draw (b) to[out=180,in=90] (0,-1);
        \draw (b) to[out=0,in=90] (2,-1);
      \end{pic}
    \]
    Conversely, equations~\eqref{eq:altfrob1} and~\eqref{eq:altfrob2} imply:
    \[
      \begin{pic}[scale=.9]
        \node[dot] (a)  at (-.2,0.75) {};
        \node[dot] (d)  at (-.2,1.15) {}; 
        \node[dot] (b) at (.5,0) {};
        \draw (d) to (a);
        \draw (a) to[out=180,in=90,looseness=.8] +(-.5,-1.25);
        \draw (b) to[out=0,in=-90,looseness=.8] +(.5,1.5);
        \draw (a) to[out=0,in=180]  (b);
        \draw (b) to +(0,-.45);
      \end{pic}
      \quad = \quad
      \begin{pic}[scale=.5]
      \node[dot] (a) at (0,2.2) {};
      \node[dot] (b) at (1.5,3) {};
      \node[dot] (c) at (2,.2) {};
      \node[dot] (d) at (1,1) {};
      \draw (a.south) to +(0,-.5) node[dot]{};
      \draw (b.north) to +(0,.5) node[dot]{};
      \draw (d.north) to +(0,.5) node[dot]{};
      \draw (a.east) to[out=0,in=180] (b.west);
      \draw (b.east) to[out=0,in=0,looseness=.8] (c.east);
      \draw (c.west) to[out=180,in=0] (d.east);
      \draw (c.south) to +(0,-.5);
      \draw (d.west) to[out=180,in=90,looseness=.7] +(-.5,-1.5);
      \draw (a.west) to[out=180,in=-90,looseness=.7] +(-.5,2);
      \end{pic}
      \quad = \quad
      \begin{pic}[scale=.5]
        \node[dot] (l) at (0,0) {};
        \node[dot] (m) at (1.25,1) {};
        \node[dot] (r) at (2,0) {};
        \draw (l.south) to +(0,-.5) node[dot]{};
        \draw (m.north) to +(0,.5) node[dot]{};
        \draw (l.west) to[out=180,in=-90,looseness=.8] +(-.5,2);
        \draw (l.east) to[out=0, in=180] (m.west);
        \draw (m.east) to[out=0,in=90] (r.north);
        \draw (r.west) to[out=180,in=90] +(-.5,-1.5);
        \draw (r.east) to[out=0,in=90] +(.5,-1.5);
      \end{pic}
      \quad = \quad
      \begin{pic}[scale=.5]
        \node[dot] (t) at (0,1) {};
        \draw (t) to +(0,1.5);
        \draw (t) to[out=0,in=90] (1,0) to (1,-1);
        \draw (t) to[out=180,in=90] (-1,0) to (-1,-1);
      \end{pic}
    \]
    Lemma~\ref{lem:equivalentformofFrob} now finishes the proof.
\end{proof}

In any closed monoidal category, $[A,A]$ is canonically a monoid. The `way of the dagger' suggests that there should be interaction between the dagger and closure in categories that have both.

\begin{definition}
  A \emph{sheathed dagger category}\index[word]{dagger category!sheathed} is a monoidal dagger category that is also closed monoidal, such that
  \[
    \begin{pic}[scale=.3]
      \node[morphism, hflip] (a) at (3.5,-2) {$\ev_{[A,A]}$};
      \node[dot] (b) at (1,-.25) {};
      \node[dot] (c) at (1,1) {};
      \draw (b) to  (c);
      \draw (b) to[out=180,in=90] (0,-1) to (0,-3.5);
      \draw (b) to[out=0,in=90] (a.north west);
      \draw (a.south) to +(0,-1);
      \draw (a.north east) to +(0,3);
    \end{pic}
    \quad = \quad
    \ev_{[A,A]} 
    \;\colon [A,A] \otimes A \to A
  \]
  for all objects $A$, and for all morphisms $f,g \colon B \to C \otimes [A,A] \colon$
  \[
    \begin{pic}[yscale=.6]
      \node[morphism] (a) at (0,0) {$f$};
      \node[morphism] (b) at (0.63,1) {$\ev_{[A,A]}$};
      \draw (a.south) to +(0,-.5);
      \draw ([xshift=-1mm]a.north west) to +(0,1.5);
      \draw ([xshift=1mm]a.north east) to (b.south west);
      \draw ([xshift=-1mm]b.north) to +(0,.5);
      \draw ([xshift=-1mm]b.south east) to +(0,-1.5);
    \end{pic}
    \; = \;
    \begin{pic}[yscale=.6]
      \node[morphism] (a) at (0,0) {$g$};
      \node[morphism] (b) at (0.62,1) {$\ev_{[A,A]}$};
      \draw (a.south) to +(0,-.5);
      \draw ([xshift=-1mm]a.north west) to +(0,1.5);
      \draw ([xshift=1mm]a.north east) to (b.south west);
      \draw ([xshift=-1mm]b.north) to +(0,.5);
      \draw ([xshift=-1mm]b.south east) to +(0,-1.5);
    \end{pic}
    \qquad \implies \qquad 
    \begin{pic}
      \node[morphism] (a) at (0,0) {$f$};
      \draw (a.south) to +(0,-.5);
      \draw ([xshift=-1mm]a.north west) to +(0,.5);
      \draw ([xshift=1mm]a.north east) to +(0,.5);
    \end{pic}
    \; = \;
    \begin{pic}
      \node[morphism] (a) at (0,0) {$g$};
      \draw (a.south) to +(0,-.5);
      \draw ([xshift=-1mm]a.north west) to +(0,.5);
      \draw ([xshift=1mm]a.north east) to +(0,.5);
    \end{pic}
  \]
\end{definition}

Any compact dagger category is a sheathed dagger category: the first axiom there says that $A^* \otimes A$ with its canonical monoid structure is a dagger Frobenius monoid, and the second axiom then holds because the evaluation morphism is invertible. 
In principle, the definition of sheathed dagger categories is much weaker. Although we have no uncontrived examples of sheathed dagger categories that are not compact dagger categories, we will work with the more general sheathed dagger categories because they are the natural home for the following arguments.
The second axiom merely says that partial evaluation is faithful, which is the case in any well-pointed monoidal dagger category.
In any closed monoidal category, the evaluation map canonically makes $A$ into an algebra for the monad $-\otimes [A,A]$. The first axiom merely says that $A$ is $-\otimes [A,A]$-self-adjoint, as in \eqref{eq:self-adjoint}. It does not assume $[A,A]$ is a dagger Frobenius monoid, nor that $A$ is a FEM-algebra. 
No other plausible conditions are imposed, such as the bifunctor $[-,-]$ being a dagger functor, which does hold in compact dagger categories.
Nevertheless, the following example shows that being a sheathed dagger category is an essentially monoidal notion that degenerates for cartesian categories.

\begin{example}
  If a Cartesian closed category has a dagger, every homset is a singleton. 
\end{example}
\begin{proof}
  Because the terminal object is in fact a zero object, there are natural bijections
  $\hom (A,B)\cong \hom(0\times A,B)\cong \hom (0,B^A)\cong \{*\}$.
\end{proof}

Currying the multiplication of a monoid $(A,\tinymult,\tinyunit)$ in a closed monoidal category gives a monoid homomorphism $R \colon A \to [A,A]$. This is the abstract version of Cayley's embedding theorem, which states that any group embeds into the symmetric group on itself.
If the category also has a dagger, there is also a monoid homomorphism $R_* = [R^\dag,\id[I]] \colon A^* \to [A,A]^*$.

\begin{theorem}\label{thm:sheathedfrobenius}
  In a sheathed dagger category, $(A,\tinymult,\tinyunit)$ is a dagger Frobenius monoid if and only if the following diagram commutes.
  \[
    \begin{tikzpicture}
     \matrix (m) [matrix of math nodes,row sep=2em,column sep=4em,minimum width=2em]
     {
      A & {[A,A]} \\
      A^* & {[A,A]^*} \\};
     \path[->]
     (m-1-1) edge node [left] {$i_A$} (m-2-1)
             edge node [above] {$R$} (m-1-2)
     (m-2-1) edge node [below] {$R_*$} (m-2-2)
     (m-1-2) edge node [right] {$i_{[A,A]}$} (m-2-2);
    \end{tikzpicture}
    \]
\end{theorem}
\begin{proof}
  Evaluating both sides shows that $R$ and $i$ commute precisely when:
  \[
    \begin{pic}[scale=.4]
      \node[dot] (b) at (1,0) {};
      \node[dot] (c) at (1,1) {};
      \node[morphism] (a) at (-.25,-1.5) {$R$};
      \draw (b) to  (c);
      \draw (b) to[out=180,in=90] (a.north);
      \draw (a.south) to +(0,-1);
      \draw (b) to[out=0,in=90] (2,-1) to (2,-3);
    \end{pic}
    \quad 
    = 
    \quad 
    \begin{pic}[scale=.4]
      \node[dot] (b) at (1,0) {};
      \node[dot] (c) at (1,1) {};
      \node[morphism, hflip] (a) at (1.75,-1.5) {$R$};
      \draw (b) to  (c);
      \draw (a.south) to +(0,-1);
      \draw (b) to[out=180,in=90] (0,-1) to (0,-3);
      \draw (b) to[out=0,in=90] (a.north);
    \end{pic}
  \]
  But this is equivalent to
  \[
    \begin{pic}[scale=.4]
        \node[dot] (t) at (0,1) {};
        \draw (t) to +(0,1.5);
        \draw (t) to[out=0,in=90] (1,0) to (1,-1);
        \draw (t) to[out=180,in=90] (-1,0) to (-1,-1);
    \end{pic}
    \quad = \quad
    \begin{pic}
      \node[morphism] (a) at (0,0) {$R$};
      \node[morphism] (b) at (0.44,1) {$\ev_{[A,A]}$};
      \draw (a.south) to +(0,-.5);
      \draw (a.north) to (b.south west);
      \draw ([xshift=-1mm]b.north) to +(0,.5);
      \draw ([xshift=-1mm]b.south east) to +(0,-1.5);
    \end{pic}
    \quad = \quad
    \begin{pic}[scale=.4]
      \node[dot] (b) at (1,0) {};
      \node[dot] (c) at (1,1) {};
      \node[morphism] (a) at (-.25,-1.5) {$R$};
      \node[morphism, hflip] (d) at (2.75,-3.5) {$\ev_{[A,A]}$};
      \draw (b) to  (c);
      \draw (b) to[out=180,in=90] (a.north);
      \draw (a.south) to +(0,-3);
      \draw (b) to[out=0,in=90] (2,-1) to (2,-3);
      \draw (d.north east) to +(0,5);
      \draw (d.south) to +(0,-1);
    \end{pic}
    \quad 
    = 
    \quad 
    \begin{pic}[scale=.4]
      \node[dot] (b) at (1,0) {};
      \node[dot] (c) at (1,1) {};
      \node[morphism, hflip] (a) at (1.75,-1.5) {$R$};
      \node[morphism, hflip] (d) at (2.75,-3.5) {$\ev_{[A,A]}$};
      \draw (b) to  (c);
      \draw (a.south) to +(0,-1);
      \draw (d.north east) to +(0,5);
      \draw (d.south) to +(0,-1);
      \draw (b) to[out=180,in=90] (0,-1) to (0,-5);
      \draw (b) to[out=0,in=90] (a.north);
    \end{pic}
    \quad = \quad 
    \begin{pic}[scale=.4]
      \node[dot] (b) at (1,0) {};
      \node[dot] (c) at (1,1) {};
      \node[dot] (d) at (2.75,-2) {};
      \draw (b) to  (c);
      \draw (d) to[out=0, in=-90] +(1,1) to +(0,2);
      \draw (d.south) to +(0,-1);
      \draw (b) to[out=180,in=90] (0,-1) to (0,-3.5);
      \draw (b) to[out=0,in=180] (d);
    \end{pic} 
  \]
  Lemma~\ref{eq:equivalentfrobeniuslaw} now finishes the proof.
\end{proof}

\begin{corollary} 
  The following are equivalent for a monoid $(A,\tinymult,\tinyunit)$ in a compact dagger category:
  \begin{itemize}
    \item $(A,\tinymult,\tinyunit)$ is a dagger Frobenius monoid;
    \item the canonical morphism $i \colon A \to A^*$ is an involution: $i_* \circ i = \id[A]$;
    \item the canonical Cayley embedding is involutive: $i \circ R = R_* \circ i$.
  \end{itemize}
\end{corollary}
\begin{proof}
  Combine Proposition~\ref{prop:closedfrobenius} and Theorem~\ref{thm:sheathedfrobenius}.
\end{proof}

\section{Limits of algebras}

In this section we show that the forgetful functor $\FEM(T)\to\cat{C}$ creates dagger limits. This is analogous to the non-dagger case. However, this then will imply that the functor also creates dagger colimits, whereas in the non-dagger case only those colimits preserved by $T$ and $T^2$ are created by the forgetful functor. In a sense, this is not very striking since a dagger monad $T$ is automatically both continuous and cocontinuous due to it factoring as an ambidextrous adjunction. However, it is still pleasing to see that the notion of dagger limit cooperates well with dagger monads, providing evidence that both notions are on the right track.

\begin{theorem}\index[symb]{$\FEM(T)$, the category of $\FEM$-algebras of $T$} Let $D\colon \cat{J}\to \FEM(T)$ be a diagram and $\Omega\subset \cat{J}$ be weakly initial. Write $U\colon \FEM(T)\to \cat{C}$ for the forgetful functor. If $(UD,\Omega)$ has a dagger limit, so does $(D,\Omega)$.
\end{theorem}

\begin{proof}
  For each $A\in\cat{J}$ write $D(A)$ as $(D(A),a_A)$, and let $(L,l_A)$ be the dagger limit of $(UD,\Omega)$. The maps $TL\sxto{Tl_A} TD(A)\sxto{a_A} D(A)$ make $TL$ into a cone for $UA$; let $l\colon TL\to L$ be the unique arrow factorizing this cone. We will prove that $(L,l)$ is the desired dagger limit. Note that the normalization and independence properties of $L$ imply those of $(L,l)$. Hence we only need to prove that $(L,l)$ is a limiting FEM-algebra.

  That $(L,l)$ is a limiting EM-algebra is proved as in ordinary category theory, see \eg~\cite[4.3]{borceux:vol2} 
  We reproduce the proof for the sake of completeness: For the unit law, note that the square on the left of the diagram
    \[
  \begin{tikzpicture}
     \matrix (m) [matrix of math nodes,row sep=2em,column sep=4em,minimum width=2em]
     {
      L & T(L) & L \\
      D(A) & TD(A) & D(A) \\};
     \path[->]
     (m-1-1) edge node [left] {$l_A$} (m-2-1)
             edge node [above] {$\eta$} (m-1-2)
     (m-2-1) edge node [below] {$\eta$} (m-2-2)
             edge[out=-45,in=225] node [below] {$\id$} (m-2-3)
     (m-2-2) edge node [below] {$a_A$} (m-2-3)
     (m-1-2) edge node [right] {$Tl_A$} (m-2-2)
             edge node [above] {$l$} (m-1-3)
      (m-1-3) edge node [right] {$l_A$} (m-2-3);
  \end{tikzpicture}
  \]
  commutes by naturality of $\eta$ and the square on the right by definition of $l$. The bottom part of the diagram is just the unit law of $(DA,a_A)$. Since $A$ was arbitrary, the unit axiom is satisfied by $(L,l)$. The associative law is proved similarly. 

  To prove the Frobenius law, note first that the monad $T$ factors as the composite of two dagger adjoints and hence is continuous, so that $(TL,Tl_A)$ is the dagger limit of $TUD$ with support $\Omega$. Next we show that the diagram
    \[
    \begin{tikzpicture}
     \matrix (m) [matrix of math nodes,row sep=2em,column sep=4em,minimum width=2em]
     {
      TD(A) & T(L)& L \\ 
       &TD(B)& \\
      D(A) & T(A) & D(B) \\};
     \path[->]
     (m-1-1) edge node [left] {$a_A$} (m-3-1)
             edge node [above] {$Tl_A^\dag$} (m-1-2)
     (m-3-1) edge node [below] {$l_A^\dag$} (m-3-2)
     (m-3-2) edge node [below] {$l_B$} (m-3-3)
     (m-1-2) edge node [right] {$Tl_B $} (m-2-2)
             edge node [above] {$l$} (m-1-3)
     (m-2-2) edge node [above] {$a_B$} (m-3-3)
     (m-1-3) edge node [right] {$l_B$} (m-3-3);
    \end{tikzpicture}
    \]
  commutes for any $A,B\in\Omega$. The right half commutes by definition of $l$. For the left half, note that the algebra maps induce a natural transformation $TUD\to UD$. Since each of these algebras is FEM, this transformation is adjointable. Hence we can apply Lemma~\ref{lem:connectingmaps} to conclude that the left half commutes. Hence the whole diagram commutes. As it does so for each $B$, we conclude that $l_A^\dag$ is an algebra homomorphism $(D(A),a_A)\to (L,l)$ for each $A$. Hence region (i) in the diagram
        
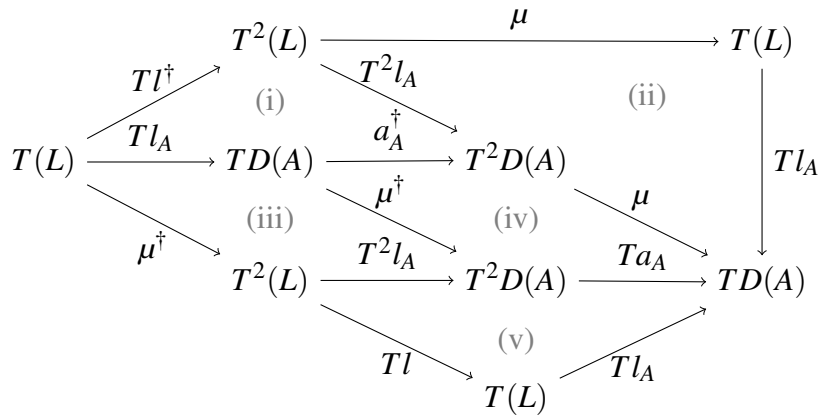
\[
    \begin{tikzpicture}
     \matrix (m) [matrix of math nodes,row sep=2em,column sep=4em,minimum width=2em]
     {
       & T^2(L) && T(L) \\
     T(L)& TD(A) & T^2D(A) & \\ 
     & T^2(L) & T^2D(A) & TD(A) \\ 
     &  & T(L) \\};
     \path[->]
     (m-2-1) edge node [above] {$Tl_A$} (m-2-2)
             edge node [above] {$Tl^\dag$} (m-1-2)
            edge node [below] {$\mu^\dag$} (m-3-2)
     (m-1-2) edge node [above] {$\mu$} (m-1-4)
             edge node [above] {$T^2 l_A$} (m-2-3)
     (m-1-4) edge node [right] {$Tl_A$} (m-3-4)
     (m-2-2) edge node [above] {$a_A^\dag$} (m-2-3)
             edge node [above] {$\mu^\dag$} (m-3-3)
     (m-2-3) edge node [above] {$\mu$} (m-3-4)
     (m-3-2) edge node [above] {$T^2l_A$} (m-3-3)
             edge node [below] {$Tl$} (m-4-3)
     (m-3-3) edge node [above] {$Ta_A$} (m-3-4)
     (m-4-3) edge node [below] {$Tl_A$} (m-3-4);
     \path
     (m-1-2) to node[gray] {(i)} (m-2-2) 
     (m-2-3) to node[gray] {(ii)} (m-1-4)
     (m-2-2) to node[gray] {(iii)} (m-3-2)
     (m-2-3) to node[gray] {(iv)} (m-3-3)
     (m-3-3) to node[gray] {(v)} (m-4-3);
    \end{tikzpicture}
    \]
  commutes. Regions (ii) and (iii) commute by naturality of $\mu$, and region (iv) commutes by the Frobenius law of $(DA,a_A)$. Finally, region (v) commutes by definition of $l$. Hence the whole diagram commutes. Since $(TL,Tl_A)$ is a limit of $TUD$, this implies the Frobenius law of $(L,l)$.

  Since $\FEM (T)\to\cat{C}$ is faithful, $(L,l)$ is a cone for $D$. We conclude by proving that it is a limiting cone. If $m_A\colon (M,m)\to (D(A),a_A)$ then $m_A\colon M\to D(A)$ is a cone for $UD$ and hence factors through $L$ via an unique $f\colon M\to L$. It suffices to prove that $f$ is an algebra homomorphism $(M,m)\to (L,l)$. This follows once we show that the diagram 
     
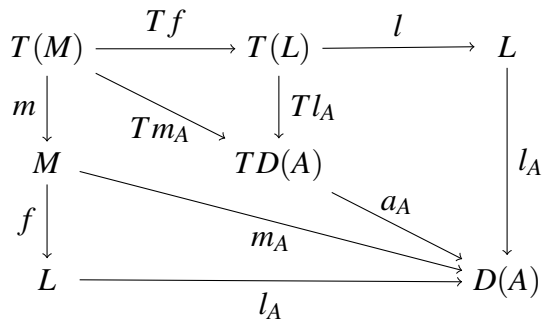
\[
    \begin{tikzpicture}
     \matrix (m) [matrix of math nodes,row sep=2em,column sep=4em,minimum width=2em]
     {
      T(M) & T(L) & L \\
      M & TD(A) &  \\
      L  &  & D(A) \\};
     \path[->]
     (m-1-1) edge node [left] {$m$} (m-2-1)
             edge node [above] {$Tf$} (m-1-2)
             edge node [below] {$Tm_A$} (m-2-2)
     (m-2-1) edge node [below] {$m_A$} (m-3-3)
             edge node [left] {$f$} (m-3-1)
     (m-2-2) edge node [above] {$a_A$} (m-3-3)
     (m-1-2) edge node [right] {$Tl_A$} (m-2-2)
             edge node [above] {$l$} (m-1-3)
      (m-1-3) edge node [right] {$l_A$} (m-3-3)
      (m-3-1) edge node [below] {$l_A$} (m-3-3);
    \end{tikzpicture}
    \]
  commutes for each $A$. The triangles on the top and bottom commute by definition of $f$ and the triangle in the middle commutes because each $m_A$ is an algebra homomorphism. The trapezoid 
  on the right commutes by definition of $l$, concluding the proof.
\end{proof}

\section{Monadicity}\label{sec:monadicity}

Next we will prove a version of Beck's monadicity theorem~\cite{beck:triples} 
for dagger monads. 

\begin{definition}\index[word]{dagger functor!dagger monadic}  A dagger functor $G\colon \cat{D}\to\cat{C}$ is \emph{dagger monadic} if there is a dagger Frobenius monad $(T,\mu,\eta)$ on \cat{C} and a dagger equivalence $\cat{D}\to \FEM (T)$ such that $G$ equals the composite $D\to \FEM (T)\to \cat{C}$. 
\end{definition}

Beck's monadicity theorem characterizes monadic functors. The usual proof 
rests on the fact that an algebra $(A,a)$ can be canonically viewed as the coequalizer of
 \[(\mu_A,Ta)\colon (T^2(A),\mu_{TA})\rightrightarrows (T(A),\mu_A)\]
and this coequalizer splits in \cat{C}. Moreover, one can show that the forgetful functor $\cat{C}^T\to C$ creates those coequalizers that split in \cat{C}, and this in fact characterizes \cat{C^T} up to equivalence. In the dagger setting, it turns out that there is no need to discuss split coequalizers. The central observation is that the Frobenius law gives additional properties to a FEM-algebra when viewed as a coequalizer.

\begin{definition}\index[word]{Frobenius coequalizer}
A coequalizer $e\colon B\to C$ of $f,g\colon A\rightrightarrows B$ is a \emph{Frobenius} coequalizer if $g^\dag f=e^\dag e$.
\end{definition}

Note that Frobenius coequalizers are unique up to unitary isomorphism by Theorem~\ref{thm:equivalenttounitary} whenever they exist. However, a consequence of the Frobenius law for a coequalizer is that $f^\dag g=g^\dag f$, so Frobenius coequalizers might not exist even when $f,g$ has dagger coequalizers. Moreover, if say $f=0$, then the existence of a Frobenius coequalizer implies that $e=0$ for the coequalizer, whence $g=0$. Hence Frobenius coequalizers are rare and do not seem to provide an interesting class of dagger (co)limits apart from their use in dagger monadicity. The name is explained by the canonical example below. 

\begin{example} Let $T$ be a dagger Frobenius monad on \cat{C}. The associativity law of an $EM$-algebra says that $a$ is an algebra homomorphism $a\colon(T(A),\mu_A)\to (A,a)$. In fact, it is a coequalizer of $(\mu_A,Ta)\colon (T^2(A),\mu_{TA})\rightrightarrows (T(A),\mu_A)$. This coequalizer is Frobenius precisely when $(A,a)$ is. 

\end{example}

\begin{theorem}\label{thm:monadicity} Let $G\colon\cat{D}\to\cat{C}$ be a dagger functor. Then $G$ is dagger monadic iff the following conditions hold:
  \begin{enumerate}[(i)]
   \item $G$ has a dagger adjoint.
   \item $G$ reflects unitary isomorphisms
   \item If $f,g\colon A\rightrightarrows B$ in $D$ is a pair such that $(U(f),U(g))$ has Frobenius coequalizer in \cat{C}, then $(f,g)$ has a Frobenius coequalizer in \cat{D} preserved by $G$.
  \end{enumerate}
\end{theorem}

\begin{proof} 
Note that conditions (i)-(iii) are preserved by dagger equivalence. Hence to prove the implication from left to right, it suffices to prove that $\FEM (T)\to \cat{C}$ satisfies (i)-(iii) for any dagger Frobenius monad $T$. Conditions (i) and (ii) are clear, so we move on to (iii). So, consider a pair $f,g\colon (A,a)\rightrightarrows (B,b)$ in $\FEM(T)$ that has a Frobenius coequalizer in \cat{C} given by $e\colon B\to C$. As $T$ factors as a dagger adjunction, it is cocontinuous so that $Te$ is the coequalizer of $Tf,Tg$. This implies the existence of a map $c\colon T(C)\to C$ making the square 
    \[
  \begin{tikzpicture}
     \matrix (m) [matrix of math nodes,row sep=2em,column sep=4em,minimum width=2em]
     {
      T(B) & T(C) \\
      B & C \\};
     \path[->]
     (m-1-1) edge node [left] {$b$} (m-2-1)
             edge node [above] {$Te$} (m-1-2)
     (m-2-1) edge node [below] {$e$} (m-2-2)
     (m-1-2) edge node [right] {$c$} (m-2-2);
  \end{tikzpicture}
  \]

commute. Next, we observe that since $T$ is the composite of two dagger adjoints, it preserves all limits. Hence $Te$ is a coequalizer and in particular an epimorphism, and ditto for $T^2e$.\footnote{The usual proof assumes that $e$ is a split coequalizer and use this to infer that $Te$ and $T^2e$ are epi.} Using this, the FEM-properties of $(C,c)$ can be deduced from the corresponding properties of $(B,b)$ and $(A,a)$. 

For associativity $(C,c)$, it suffices to prove that the diagram 
  \[
  \begin{tikzpicture}
     \matrix (m) [matrix of math nodes,row sep=2em,column sep=4em,minimum width=2em]
     {
      T^2(C) & & & T(C) \\
      & T^2(B)& T(B) \\
      & T(B)& B \\
      T(C) &&& C \\};
     \path[->]
     (m-1-1) edge node [above] {$Tc$} (m-1-4)
         edge node [left] {$\mu $} (m-4-1)
     (m-2-2) edge node [below] {$T^2e$} (m-1-1)
        edge node [above] {$Tb$} (m-2-3)
        edge node [left] {$\mu $} (m-3-2)
     (m-2-3) edge node [below] {$Te$} (m-1-4)
         edge node [right] {$b$} (m-3-3)
     (m-1-4) edge node [right] {$c$} (m-4-4)
     (m-3-2) edge node [above] {$Te$} (m-4-1)
        edge node [below] {$b$} (m-3-3)
     (m-3-3) edge node [above] {$e$} (m-4-4)
     (m-4-1) edge node [below] {$c$} (m-4-4);
  \end{tikzpicture}
  \]
commutes. The trapezoid on the left commutes by naturality of $\mu$ and the square in the center is the associative law of $(B,b)$. The other trapezoids commute by definition of $c$. The unit law is proved similarly.

To prove that $(C,c)$ satisfies the Frobenius law, consider first the diagram:
  
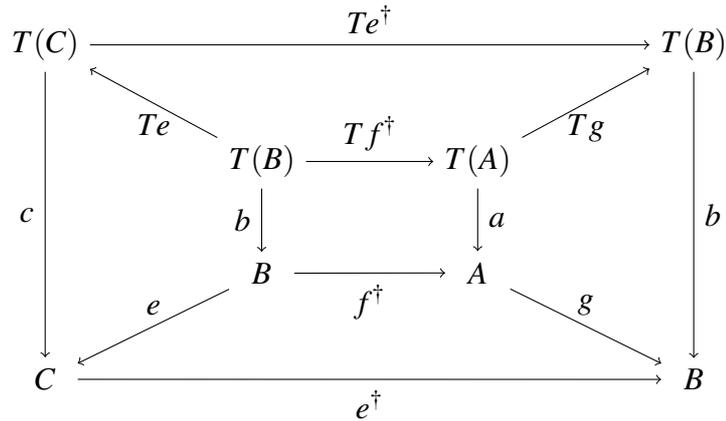
\[
  \begin{tikzpicture}
     \matrix (m) [matrix of math nodes,row sep=2em,column sep=4em,minimum width=2em]
     {
      T(C) & & & T(B) \\
      & T(B)& T(A) \\
      & B& A \\
      C &&& B \\};
     \path[->]
     (m-1-1) edge node [above] {$Te^\dag$} (m-1-4)
     		 edge node [left] {$c$} (m-4-1)
     (m-2-2) edge node [below] {$Te$} (m-1-1)
     		edge node [above] {$Tf^\dag$} (m-2-3)
     		edge node [left] {$b$} (m-3-2)
     (m-2-3) edge node [below] {$Tg$} (m-1-4)
     		 edge node [right] {$a$} (m-3-3)
     (m-1-4) edge node [right] {$b$} (m-4-4)
     (m-3-2) edge node [above] {$e$} (m-4-1)
     		edge node [below] {$f^\dag$} (m-3-3)
     (m-3-3) edge node [above] {$g$} (m-4-4)
     (m-4-1) edge node [below] {$e^\dag$} (m-4-4);
  \end{tikzpicture}
  \]
The trapezoids on the top and bottom commute because the coequalizer is Frobenius. The trapezoid on the left commutes because $e$ is an algebra homomorphism and the one on the right because $g$ is. The square in the middle commutes because $f^\dag$ is the dagger of a FEM-algebra homomorphism. As $Te$ is epi the outer rectangle commutes. In other words, $e^\dag$ is a homomorphism $(C,c)\to (B,b)$. This is the reason the trapezoid on the top of the diagram 
  
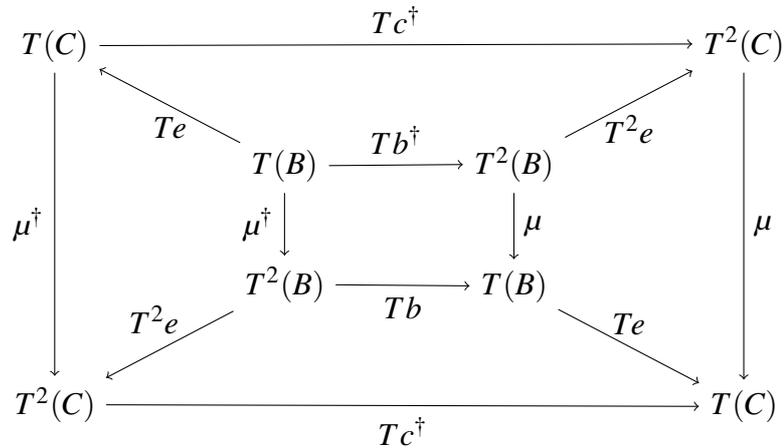
\[
  \begin{tikzpicture}
     \matrix (m) [matrix of math nodes,row sep=2em,column sep=4em,minimum width=2em]
     {
      T(C) & & & T^2(C) \\
      & T(B)& T^2(B) \\
      & T^2(B) & T(B) \\
      T^2(C) &&& T(C) \\};
     \path[->]
     (m-1-1) edge node [above] {$Tc^\dag$} (m-1-4)
     		 edge node [left] {$\mu^\dag$} (m-4-1)
     (m-2-2) edge node [below] {$Te$} (m-1-1)
     		edge node [above] {$Tb^\dag$} (m-2-3)
     		edge node [left] {$\mu^\dag$} (m-3-2)
     (m-2-3) edge node [below] {$T^2 e$} (m-1-4)
     		 edge node [right] {$\mu$} (m-3-3)
     (m-1-4) edge node [right] {$\mu$} (m-4-4)
     (m-3-2) edge node [above] {$T^2e\quad$} (m-4-1)
     		edge node [below] {$Tb$} (m-3-3)
     (m-3-3) edge node [above] {$Te$} (m-4-4)
     (m-4-1) edge node [below] {$Tc^\dag$} (m-4-4);
  \end{tikzpicture}
  \]
commutes,and the bottom one commutes because $e$ is an algebra homomorphism. The left and right trapezoids commute because $\mu^\dag$ and $\mu$ are natural. The diagram in the middle commutes because $(B,b)$ is FEM. Since $Te$ is epi, this implies that the outer rectangle commutes, giving us the Frobenius law for $(C,c)$. Hence \[f,g\colon (A,a)\rightrightarrows (B,b)\sxto{e} (C,c)\] is a diagram in $FEM(T)$, and it is easy to see that it is a Frobenius coequalizer in $FEM(T)$.

For the other direction, let $G$ satisfy (i)-(iii) and let $F\colon\cat{C}\to \cat{D}$ be its dagger adjoint. Write $(T,\mu,\eta)$ for the dagger Frobenius monad on \cat{C} generated by this dagger adjunction, and let $J\colon\cat{D}\to \FEM (T)$ be the comparison functor from Theorem~\ref{thm:comparison}. We wish to show that $J$ forms a part of a dagger equivalence, so by Lemma~\ref{lem:halfequiv} it suffices to show that it is full, faithful and unitarily essentially surjective. We already know that it is full. 

To show that $J$ is unitarily essentially surjective, fix a FEM-algebra $(A,a)$. Now $\epsilon_{FA},F(a) \colon FGF(A)\rightrightarrows F(A)$ is a pair in \cat{D} which $G$ maps to $\mu_A,Ta\colon T^2(A)\rightrightarrows T(A)$ in \cat{C}. The latter pair has a Frobenius coequalizer in \cat{C}, so by (iii) the pair \[\epsilon_{FA},F(a) \colon FGF(A)\rightrightarrows F(A)\] has a Frobenius coequalizer $e\colon F(A)\to B$ in \cat{D} preserved by $G$. As seen in the first part of the proof $GB$ has a unique FEM-algebra structure $b\colon T(B)\to B$ such that $UJe$ is a homomorphism, and this in fact makes $(GB,b)=JB$ into a Frobenius coequalizer of $(\mu_A,Ta)\colon (T^2(A),\mu_{TA})\rightrightarrows (T(A),\mu_A)$ in $\FEM(T)$. As Frobenius coequalizers are unique up to unitary iso, this implies that $JB$ is unitarily isomorphic to $(A,a)$, as desired. 

It remains to prove faithfulness of $J$, for which it suffices to prove faithfulness of $G$. So let $f,g\colon A\rightrightarrows B$ be a pair in \cat{D} with $G(f)=G(g)$. This is done exactly as in the non-dagger case: by naturality of the counit $\epsilon\colon FG\to \id[\cat{D}]$ we have
\[f\epsilon_A=\epsilon_A FG(f)=\epsilon_A FG(g)=g\epsilon_A\] 
so it suffices to prove that $\epsilon_A$ is epi. The pair $\epsilon_{FGA},FG\epsilon_A\colon FGFG(A)\rightrightarrows FG(A)$ in \cat{D} is mapped by $G$ to the pair $G\epsilon_{FGA},GFG\epsilon_A\colon GFGFG(A)\rightrightarrows GFG(A)$ in \cat{C} having a Frobenius coequalizer $G\epsilon_A\colon GFG(A)\to G(A)$. 
Hence (iii) implies that \[\epsilon_{FGA},FG\epsilon_A\colon FGFG(A)\rightrightarrows FG(A)\] has a Frobenius coequalizer $q\colon GFA\to Q$ in \cat{C} preserved by $G$, and hence ${\epsilon_A\colon GF(A)\to A}$ factors through $q$ as $\epsilon_A=fq$ for some $F$. As this coequalizer is preserved by $G$, the morphism $G(f)$ is a unitary and hence by (ii) $f$ is unitary in $D$, which implies that $\epsilon_A$ is epic, completing the proof.


\end{proof}

\begin{corollary} Dagger monadic functors are closed under composition
\end{corollary}

\begin{proof} Conditions (i)-(iii) from Theorem~\ref{thm:monadicity} are clearly closed under composition.
\end{proof}

Here dagger category theory deviates from ordinary category theory: ordinary monadic functors are not closed under composition (an example is that Torsion-free abelian groups are monadic over abelian groups which are monadic over sets, but the composite is not monadic~\cite[Counterexample 4.6.4]{borceux:vol2}) but dagger monadic functors are.
\chapter{Arrows}\label{chp:arrows}

\section{Introduction}\label{sec:introduction}

Reversible computing studies settings in which all processes can be reversed: programs can be run backwards as well as forwards.
Its history goes back at least as far as 1961, when Landauer formulated his physical principle that logically irreversible manipulation of information costs work.
This sparked the interest in developing reversible models of computation as a means to making them more energy efficient.
Reversible computing has since also found applications in high-performance computing~\cite{schordanetal:parallel}, process calculi~\cite{cristescuetal:picalculus}, probabilistic computing~\cite{stoddartlynas:probabilistic}, quantum computing~\cite{selinger:completelypositive}, and robotics~\cite{schultzbordignonstoy:reconfiguration}.

There are various theoretical models of reversible computations. The most well-known ones are perhaps Bennett's reversible Turing machines~\cite{bennett:reversibility} and Toffoli's reversible circuit model~\cite{toffoli:reversible}. There are also various other models of reversible automata~\cite{morita:twoway,kutribwendlandt:automata} and combinator calculi~\cite{abramsky:structrc,jamessabry:infeff}.

We are interested in models of reversibility suited to functional programming languages. Functional languages are interesting in a reversible setting for two reasons. First, they are easier to reason and prove properties about, which is a boon if we want to understand the logic behind reversible programming. Second, they are not stateful by definition, which eases reversing programs.
It is fair to say that existing reversible functional programming languages~\cite{jamessabry:theseus,yokoyamaetal:rfun} still lack various desirable constructs familiar from the irreversible setting. 

Irreversible functional programming languages like Haskell naturally take semantics in categories. The objects interpret types, and the morphisms interpret functions. 
Functional languages are by definition not stateful, and their categorical semantics only models pure functions. However, sometimes it is useful to have non-functional side-effects, such as exceptions, input/output, or indeed even state. Irreversible functional languages can handle this elegantly using monads~\cite{moggi:monads} or more generally arrows~\cite{hughes:programmingarrows}.

A word on terminology. We call a computation $a \colon X \to Y$ \emph{reversible} when it comes with a specified partner computation $a^\dag \colon Y \to X$ in the opposite direction. This implies nothing about possible side-effects.
Saying that a computation is \emph{partially invertible} is stronger, and requires $a \circ a^\dag \circ a = a$. Saying that it is \emph{invertible} is even stronger, and requires $a \circ a^\dag$ and $a^\dag \circ a$ to be identities.
We call this partner of a reversible computation its \emph{dagger}. 
In other words, reversible computing for us concerns dagger arrows on dagger categories, and is modelled using involutions~\cite{heunenkarvonen:daggermonads}. 
In an unfortunate clash of terminology, categories of partially invertible maps are called inverse categories~\cite{cockettlack:restrictioncategories}, and categories of invertible maps are called groupoids~\cite{gabbaykropholler:lazy}. 
Thus, inverse arrows on inverse categories concern partially invertible maps.

We develop \emph{dagger arrows} and \emph{inverse arrows}, which are useful in two ways:
\begin{itemize}
  \item We illustrate the reach of these notions by exhibiting many fundamental reversible computational side-effects that are captured (in Section~\ref{sec:arrows}), including: pure reversible functions, information effects, reversible state, serialization, vector transformations, 
  dagger Frobenius monads~\cite{heunenkarvonen:reversiblemonads,heunenkarvonen:daggermonads}, recursion~\cite{kaarsgaardaxelsengluck:joininversecategories}, and superoperators. To remain within the scope of the thesis, we treat each example informally from the perspective of programming languages, but formally from the perspective of category theory.   
  \item We prove that these notions behave well mathematically (in Section~\ref{sec:arrowscategorically}): whereas arrows are monoids in a category of profunctors~\cite{jacobs2009categorical}, dagger arrows and inverse arrows are involutive monoids.
\end{itemize}

This chapter aims to inform design principles of sound reversible programming languages. 
The main contribution is to match desirable programming concepts to precise category theoretic constructions.
As such, it is written from a theoretical perspective. To make examples more concrete for readers with a more practical background, we adopt the syntax of a typed first-order reversible functional programming language with type classes. 
We begin with preliminaries on reversible base categories (in Section~\ref{sec:inversecategories}).

\section{Inverse categories}\label{sec:inversecategories}

\begin{definition}\index[word]{inverse category!monoidal} 
   A \emph{(monoidal) inverse category} is a (monoidal) dagger category of partial isometries where positive maps commute: $f \circ f^\dag \circ f=f$ and $f^\dag \circ f \circ g^\dag \circ g = g^\dag\circ g \circ f^\dag\circ f$ for all maps $f \colon X \to Y$ and $g \colon X \to Z$. 
\end{definition}

Every groupoid is an inverse category. Another example of an inverse category is $\cat{PInj}$, whose objects are sets, and morphisms $X \to Y$ are partial injections: $R \subseteq X \times Y$ such that for each $x \in X$ there exists at most one $y \in Y$ with $(x,y) \in R$, and for each $y \in Y$ there exists at most one $x \in X$ with $(x,y) \in R$. It is a monoidal inverse category under either Cartesian product or disjoint union.

\begin{definition}\index[word]{inverse products} 
  A dagger category is said to have \emph{inverse products}~\cite{giles:thesis} if it is a symmetric monoidal dagger category with a natural transformation $\Delta_X \colon X \to X
  \otimes X$ making the following diagrams commute:
  \[
    \begin{aligned}\begin{tikzpicture}
        \matrix (m) [matrix of math nodes,row sep=2em,column sep=4em,minimum width=2em]
        {X  & X\otimes X \\
          & X\otimes X \\};
        \path[->]
        (m-1-1) edge node [below] {$\Delta_X$} (m-2-2)
            edge node [above] {$\Delta_X$} (m-1-2)
        (m-1-2) edge node [right] {$\sigma_{X,X}$} (m-2-2);
    \end{tikzpicture}\end{aligned}
    \begin{aligned}\begin{tikzpicture}
        \matrix (m) [matrix of math nodes,row sep=2em,column sep=3em,minimum width=2em]
        {X  & & X\otimes X  \\
        X\otimes X & X\otimes (X\otimes X) & (X\otimes X)\otimes X \\};
        \path[->]
        (m-1-1) edge node [right] {$\Delta_X$} (m-2-1)
            edge node [above] {$\Delta_X$} (m-1-3)
        (m-2-1) edge node [below] {$\id\otimes\Delta_X$} (m-2-2)
        (m-2-2) edge node [below] {$\alpha$} (m-2-3)
        (m-1-3) edge node [left] {$\Delta_X\otimes\id$} (m-2-3);
    \end{tikzpicture}\end{aligned}
  \]
  \[
    \begin{aligned}\begin{tikzpicture}
        \matrix (m) [matrix of math nodes,row sep=2em,column sep=4em,minimum width=2em]
        {X  & X\otimes X \\
          & X \\};
        \path[->]
        (m-1-1) edge node [below] {$\id$} (m-2-2)
            edge node [above] {$\Delta_X$} (m-1-2)
        (m-1-2) edge node [right] {$\Delta_X^\dag$} (m-2-2);
    \end{tikzpicture}\end{aligned}
    \begin{aligned}\begin{tikzpicture}
        \matrix (m) [matrix of math nodes,row sep=1.5em,column sep=4em,minimum width=2em]
        {X\otimes X  && X\otimes(X\otimes X) \\
         & X& \\
         (X\otimes X)\otimes X & & X\otimes X \\};
        \path[->]
        (m-1-1) edge node [right] {$\Delta\otimes\id$} (m-3-1)
            edge node [above] {$\id\otimes\Delta_X$} (m-1-3)
            edge node [right=3mm] {$\Delta_X^\dag$} (m-2-2)
        (m-2-2) edge node [below] {$\Delta_X$} (m-3-3)
        (m-3-1) edge node [below] {$(\id\otimes\Delta_X^\dag)\circ\alpha^\dag$} (m-3-3)
        (m-1-3) edge node [left,yshift=4mm] {$(\Delta_X^\dag\otimes \id)\circ \alpha$} (m-3-3);
    \end{tikzpicture}\end{aligned}
  \]
  These diagrams express cocommutativity, coassociativity, speciality and the Frobenius law. 
\end{definition} 

Another useful monoidal product, here on inverse categories, is a disjointness tensor, defined in the following way (see \cite{giles:thesis}):

\begin{definition}\label{def:disjtensor}\index[word]{disjointness tensor} 
  An inverse category is said to have a \emph{disjointness tensor} if it is
  equipped with a symmetric monoidal tensor product $- \oplus -$ such that its
  unit $0$ is a zero object, and the canonical \emph{quasi-injections}
  \begin{equation*}
    i_1 := X \xrightarrow{\rho^{-1}_X} X \oplus 0 \xrightarrow{X \oplus 
    0_{0,Y}} X \oplus Y \qquad
    \i_2 := Y \xrightarrow{\lambda^{-1}_Y} 0 \oplus Y \xrightarrow{0_{0,X}
    \oplus Y} X \oplus Y
  \end{equation*}
  are jointly epic.
\end{definition}

For example, $\cat{PInj}$ has inverse products $\Delta_X \colon X \to X \otimes
X$ with $x \mapsto (x,x)$, and a disjointness tensor where $X \oplus Y$ is
given by the tagged disjoint union of $X$ and $Y$ (the unit of which is
$\emptyset$).

Inverse categories can also be seen as certain instances of restriction categories. Informally, a restriction category models partially defined morphisms, by assigning to each $f\colon A\to B$ a morphism $\bar{f}\colon A\to A$ that is the identity on the domain of definition of $f$ and undefined otherwise. For more details, see~\cite{cockettlack:restrictioncategories}.

\begin{definition}\label{def:restrictioncategory}\index[word]{restriction category} 
  A \emph{restriction category} is a category equipped with an operation that assigns to each $f\colon A\to B$ a morphism $\bar{f}\colon A\to A$ such that:
  \begin{itemize} 
    \item $f\circ\bar{f}=f$ for every $f$;
    \item $\bar{f}\circ\bar{g}=\bar{g}\circ\bar{f}$ whenever $\dom f=\dom g$;
    \item $\overline{g\circ\bar{f}}=\bar{g}\circ\bar{f}$ whenever $\dom f=\dom g$;
    \item $\bar{g}\circ f=f\circ\overline{g\circ f}$ whenever $\dom g=\cod f$.
  \end{itemize}
  A \emph{restriction functor} is a functor $F$ between restriction categories with $F(\bar{f})=\overline{F(f)}$. A \emph{monoidal restriction category} is a restriction category with a monoidal structure for which $\otimes\colon \cat{C}\times\cat{C}\to\cat{C}$ is a restriction functor.

  A morphism $f$ in a restriction category is a \emph{partial isomorphism}\index[word]{partial isomorphism}  if there is a morphism $g$ such that $g \circ f=\bar{f}$ and $f \circ g=\bar{g}$. Given a restriction category \cat{C}, define $\Inv(\cat{C})$ to be the wide subcategory of \cat{C} having all partial isomorphisms of \cat{C} as its morphisms.
\end{definition}

An example of a monoidal restriction category is $\cat{PFn}$, whose objects are sets, and whose morphisms $X \to Y$ are partial functions: $R \subseteq X \times Y$ such that for each $x \in X$ there is at most one $y \in Y$ with $(x,y) \in R$. The restriction $\bar{R}$ is given by $\{(x,x) \mid \exists y \in Y \colon (x,y) \in R\}$.

\begin{remark}
  Inverse categories could equivalently be defined as either categories in which every morphism $f$ satisfies $f=f \circ g \circ f$ and $g=g \circ f \circ g$ for a unique morphism $g$, or as restriction categories in which all morphisms are partial isomorphisms~\cite[Theorem~2.20]{cockettlack:restrictioncategories}. It follows that functors between inverse categories automatically preserve daggers and that $\Inv(\cat{C})$ is an inverse category.

  It follows, in turn, that an inverse category with inverse products is a monoidal inverse category: because $X \otimes -$ and $- \otimes Y$ are endofunctors on an inverse category, they preserve daggers, so that by bifunctoriality $-\otimes -$ does as well.
\end{remark}

\section{Arrows as an interface for reversible effects}\label{sec:arrows}

Arrows are a standard way to encapsulate computational side-effects in a functional (irreversible) programming language~\cite{hughes:arrows,hughes:programmingarrows}. This section extends the definition to reversible settings, namely to dagger arrows and inverse arrows. We argue that these notions are ``right'', by exhibiting a large list of fundamental reversible side-effects that they model.
We start by recalling irreversible arrows.

\begin{definition}\label{def:arrow}\index[word]{arrow} 
  An \emph{arrow} on a symmetric monoidal category $\cat{C}$ is a functor $A \colon \cat{C}\op \times \cat{C} \to \cat{Set}$ with operations
  \begin{align*}
    \arr &: (X \to Y)\to A~X~Y \\
    (\acmp) & : A~X~Y \to A~Y~Z \to A~X~Z\\
    \first_{X,Y,Z} &: A~X~Y \to A~(X\otimes Z)~(Y\otimes Z) 
  \end{align*}
  that satisfy the following laws:
  \begin{align}
    (a \acmp b) \acmp c &= a \acmp (b \acmp c) \label{eq:arrow1}\\
    \arr (g\circ f) &= \arr f \acmp \arr g \label{eq:arrow2}\\
    \arr \id \acmp a =&\;a = a \acmp\arr \id  \label{eq:arrow3}\\
    \first_{X,Y,I} a \acmp \arr \rho_Y &= \arr \rho_X \acmp a  \label{eq:arrow4}\\
    \first_{X,Y,Z} a \acmp \arr (\id[Y]\otimes f) &=  \arr(\id[X]\otimes f) \acmp \first_{X,Y,Z} a \label{eq:arrow5}\\
    (\first_{X,Y,Z\otimes V} a) \acmp \arr \alpha_{Y,Z,V} &= \arr \alpha_{X,Z,V} \acmp \first (\first a) \label{eq:arrow6}\\
    \first (\arr f)&= \arr (f\otimes\id) \label{eq:arrow7}\\
    \first (a \acmp b)&=(\first a) \acmp (\first b) \label{eq:arrow8}
  \end{align}
  where we use the functional programming convention to write $A~X~Y$ for $A(X,Y)$ and $X\to Y$ for $\hom (X,Y)$
  The \emph{multiplicative fragment} consists of above data except $\first$, satisfying all laws except those mentioning $\first$; we call this a \emph{weak arrow}.

  Define $\second(a)$ by $\arr(\sigma) \acmp \first(a) \acmp \arr(\sigma)$, using the symmetry, so analogues of~\eqref{eq:arrow4}--\eqref{eq:arrow8} are satisfied. Arrows makes sense for (nonsymmetric) monoidal categories if we add this operation and these laws.
\end{definition}

\cut{Here come our central definitions.}

\begin{definition}
  A \emph{dagger arrow} is an arrow on a monoidal dagger category with an additional operation
  $\inv : A~X~Y \to A~Y~X$\index[symb]{$\inv (a)$, the dagger of an arrow $a$}
  satisfying the following laws:
  \begin{align}
    \inv (\inv a)&= a \label{eq:daggerarrow1}\\
    \inv a \acmp \inv b &= \inv (b\acmp a) \label{eq:daggerarrow2}\\ 
    \arr (f^\dag)&=\inv (\arr f) \label{eq:daggerarrow3}\\
    \inv (\first a )&=\first (\inv a) \label{eq:daggerarrow4}
  \intertext{An \emph{inverse arrow}  is a dagger arrow on a monoidal inverse category such that:}
    (a \acmp \inv a) \acmp a &= a \label{eq:inversearrow1}\\
    (a \acmp \inv a) \acmp (b \acmp \inv b) &= (b \acmp \inv b) \acmp (a \acmp \inv a) \label{eq:inversearrow2}
  \end{align}
  The \emph{multiplicative fragment} consists of above data except $\first$, satisfying all laws except those mentioning $\first$.
\end{definition}

\begin{remark}
  There is some redundancy in the definition of an inverse arrow:
  \eqref{eq:inversearrow1} and~\eqref{eq:inversearrow2} imply \eqref{eq:daggerarrow3} and~\eqref{eq:daggerarrow4}; and~\eqref{eq:daggerarrow3} implies $\inv (\arr \id) = \arr \id$.   
\end{remark}

Like the arrow laws~\eqref{eq:arrow1}--\eqref{eq:arrow8}, in a programming language with inverse arrows, the burden is on the programmer to guarantee~\eqref{eq:daggerarrow1}--\eqref{eq:inversearrow2} for their implementation. If that is done, the language guarantees arrow inversion.

\begin{remark}
  Now follows a long list of examples of inverse arrows, described in a typed first-order reversible functional pseudocode with type classes, inspired by Theseus~\cite{jamessabry:theseus,jamessabry:infeff}, the revised version of Rfun (briefly described in~\cite{kaarsgaardthomsen:rfun}), and Haskell. Type classes are a form of interface polymorphism: A type class is defined by a class specification containing the signatures of functions that a given type must implement in order to be a member of that type class (often, type class membership also informally requires the programmer to ensure that certain equations are required of their implementations). For example, the \<Functor\> type class (in Haskell) is given by the class specification

\begin{haskell}
\hskwd{class} Functor f \hskwd{where} \\
\quad\hsalign{
  fmap &:&\ (a\to{b})\to{f a}\to{f b}
}
\end{haskell}
with the additional informal requirements that \<fmap id = id\> and \<fmap (g\circ{f}) = (fmap g)\circ(fmap f)\> must be satisfied for all instances. For example, lists in Haskell satisfy these equations when defining \<fmap\> as the usual \<map\> function, \ie:
\begin{haskell*}
  \hskwd{instance} Functor List \hskwd{where} \\
  \quad\hsalign{
    fmap &:&\ (a\to{b})\to{List a}\to{List b} \\
    fmap f [] &=&\ [] \\
    fmap f (x{::}xs) &=&\ (f x){::}(fmap f xs)
  }
\end{haskell*}

While higher-order reversible functional programming is fraught, aspects of this can be mimicked by means of parametrized functions. A parametrized function is a function that takes parts of its input statically (\ie, no later than at compile time), in turn lifting the first-order requirement on these inputs. To separate static and dynamic inputs from one another, two distinct function types are used: $a \to b$ denotes that $a$ must be given statically, and $a \leftrightarrow b$ (where $a$ and $b$ are first-order types) denotes that $a$ is passed dynamically. As the notation suggests, functions of type $a \leftrightarrow b$ are reversible. For example, a parametrized variant of the reversible map function can be defined as a function \<map : (a\leftrightarrow{b})\to([a]\leftrightarrow[b])\>. Thus, \<map\> itself is \emph{not} a reversible function, but given statically any reversible function \<f : a\leftrightarrow{b}\>, the parametrized \<map f : ([a]\leftrightarrow[b])\> is.

  Given this distinction between static and dynamic inputs, the signature of $\arr$ becomes $(X \leftrightarrow Y) \to A~X~Y$. 
  We will see later that Arrows on $\cat{C}$ can be modelled categorically as monoids in the functor category $\prof$~\cite{jacobs2009categorical}.
  Definition~\ref{def:arrow} uses the original signature, because this distinction is not present in the irreversible case. Fortunately, the semantics of arrows remain the same whether or not this distinction is made.
\end{remark}

\begin{example}{\emph{(Pure functions)}}\label{ex:pure}
  A trivial example of an arrow is the identity arrow $\hom(-,+)$ which adds
  no computational side-effects at all. This arrow is not as boring as it may look at first. If the identity arrow is an inverse arrow, then the programming language 
  in question is both \emph{invertible} and \emph{closed under program 
  inversion}: any program $p$ has a semantic inverse $\llbracket p \rrbracket^\dag$ (satisfying certain equations), and the semantic inverse coincides with the
  semantics $\llbracket \inv(p) \rrbracket$ of another program $\inv(p)$. As 
  such, $\inv$ must be a sound and complete \emph{program inverter} (see also 
  \cite{kawabeglueck:lrinv}) on pure functions; not a trivial matter at all.
\end{example}

\begin{example}{\emph{(Information effects)}}\label{ex:informationeffects}
  James and Sabry's \emph{information effects}~\cite{jamessabry:infeff} explicitly expose creation and erasure of information as effects.
  This type-and-effect system captures irreversible computation inside a pure reversible setting.

  Categorically, they start by working in a model of the language $\Pi$, a fragment of the language $\Pi^o$ studied in Section~\ref{sec:rigcatslanguage} without recusive types or traces. Since the language can only denote isomorphisms, we may assume that the model is a rig groupoid $\cat{C}$.

  If $\cat{C}$ is small, it carries an arrow, where $A(X,Y)$ is the disjoint union of $\hom(X \otimes H, Y \otimes G)$ where $G$ and $H$ range over all objects, and morphisms $X \otimes H \to Y \otimes G$ and $X \otimes H' \to Y \otimes G'$ are identified when they are equal up to coherence isomorphisms.
  This is an inverse arrow, where $\inv(a)$ is simply $a^{-1}$. 
  It supports the following additional operations:
  \begin{haskell}
    erase &=&\ [ \sigma \colon X \otimes 1 \to 1\otimes X ]_\simeq &  \in A(X,1)\text, \\
    create_X &=&\ [ \sigma\colon 1\otimes X \to X \otimes 1 ]_\simeq & \in A(1,X)\text. 
  \end{haskell}
  James and Sabry show how a simply-typed first order functional irreversible language translates into $\pi$ using this inverse arrow to build implicit communication with a global heap $H$ and garbage dump $G$. They then augment the irreversible language with iteration and natural numbers, and extend the original translation, now translating into $\Pi^o$.
\end{example}

\begin{example}{\emph{(Reversible state)}}\label{ex:reversiblestate}
  Perhaps the prototypical example of an effect is computation with a mutable 
  store of type $S$. In the irreversible case, such computations are performed
  using the state monad \<State S X = S \multimap (X \otimes 
  S)\>, where $S \multimap -$ is the right adjoint to $- \otimes S$, and can
  be thought of as a function type. Morphisms in the
  corresponding Kleisli category are morphisms of the form $X \to S \multimap
  (Y \otimes S)$ in the ambient monoidal closed category. In this 
  formulation, the current state is fetched by \<get : 
  State S S\> defined as \<get s = (s,s)\>, while the state is (destructively) 
  updated by \<put : S\to\mathit{State} S 1\> defined as \<put x s = ((),x)\>.
  
  Such arrows  can not be used as-is in inverse categories, however, as
  canonical examples (such as \PInj{}) fail to be monoidal closed.
  To get around this, note that it follows from monoidal closure that $\hom(X, S
  \multimap (Y \otimes S)) \simeq \hom(X \otimes S, Y \otimes S)$, so that
  $\hom(-\otimes S, -\otimes S)$ is an equivalent arrow that does not depend on
  closure. With this is mind, we define the \emph{reversible state arrow} with a
  store of type $S$:
  \begin{haskell*}
    \hskwd{type} RState S X Y = 
    X\otimes\mathit{S}\leftrightarrow\mathit{Y}\otimes\mathit{S} \\\\\\
    \hskwd{instance} Arrow (RState S) \hskwd{where} \\
    \quad\hsalign{
    arr f (x,s) &=&\ (f x, s) \\
    (a\acmp\mathit{b}) (x,s) &=&\ b (a (x,s)) \\
    first a ((x,z),s) &=&\ \hskwd{let} (x',s')=a (x,s) \hskwd{in} ((x',z),s') 
    } 
    \\\\\\
    \hskwd{instance} InverseArrow (RState S) \hskwd{where} \\
    \quad\hsalign{
    inv a (y,s) &=&\ a^\dagger (y,s)
    }
  \end{haskell*}
  This satisfies the inverse arrow laws. To access the state, we use reversible 
  duplication of values (categorically, this requires the monoidal product to 
  have a natural diagonal $\Delta_X : X \to X \otimes X$, as inverse products 
  do). Syntactically, this corresponds to the following arrow:
  \begin{haskell}
    get &:&\ RState S X (X\otimes\mathit{S}) \\
    get (x,s) &=&\ ((x,s),s)
  \end{haskell}
  The inverse to this arrow is 
  \<assert : RState S (X\otimes\mathit{S}) X\>, 
  which asserts that the current state is
  precisely what is given in its second input component; if this fails, the
  result is undefined. For changing the state, while we cannot destructively
  update it reversibly, we \emph{can} reversibly update it by a given
  reversible function with signature $S \leftrightarrow S$. This gives:
  \begin{haskell}
    update &:&\ (S\leftrightarrow\mathit{S})\to\mathit{RState} S X X \\
    update f (x,s) &=&\ (x, f s)
  \end{haskell}
  This is analogous to how variable assignment works in the reversible 
  programming language Janus~\cite{yokoyamaglueck:janus}: Since destructive
  updating is not permitted, state is updated by means of built-in
  reversible update operators, \eg, updating a variable by adding a constant 
  or the contents of another variable to it, etc.
  
  \cut{$\mathtt{x}~{+}{=}~\mathtt{y}$ (add
  the contents $\mathtt{y}$ to the contents of $\mathtt{x}$),
  $\mathtt{y}~{-}{=}~\mathtt{2}$ (subtract two from the contents of 
  $\mathtt{y}$), and so on.}


\end{example}

\begin{example}{\emph{(Computation in context)}}
  Related to computation with a mutable store is computation with an immutable
  one; that is, computation within a larger context that remains invariant
  across execution. In an irreversible setting, this job is typically handled
  by the \emph{reader monad} (with context of type $C$), defined as \<Reader C
  X = C\rightarrow{X}\>. This approach is fundamentally irreversible, however,
  as the context is ``forgotten'' whenever a value is computed by supplying it
  with a context. Even further, it relies on the reversibly problematic
  notion of monoidal closure.
  
  A reversible version of this idea is one that remembers the context, 
  giving us the reversible Reader arrow:
  \begin{haskell}
    \hskwd{type} Reader C X Y = X\otimes{C}\leftrightarrow\mathit{Y}\otimes{C}
  \end{haskell}
  This is precisely the same as the state arrow -- indeed, the instance
  declarations for \<arr\>, \<(\acmp)\>, \<first\>, and \<inv\> are the same --
  save for the fact that we additionally require \emph{all} \<Reader\> arrows
  $r$ to satisfy \<c=c'\> whenever \<r (x,c) = (y,c')\>. We notice that \<arr
  f\> satisfies this property for all \<f\>, whereas \<(\acmp)\>, \<first\>, and
  \<inv\> all preserve it. This resembles the ``slice'' construction on inverse
  categories with inverse products; see \cite[Sec.~4.4]{giles:thesis}.
  
  As such, while we can provide access to the context via a function defined
  exactly as \<get\> for the reversible state arrow, we cannot provide an
  update function without (potentially) breaking this property -- as intended.
  In practice, the property that the context is invariant across execution can
  be aided by appropriate interface hiding, \ie\ exposing the \<Reader\> type
  and appropriate instance declarations and helpers (such as \<get\> and
  \<assert\>) but leaving the \emph{constructor} for \<Reader\> arrows hidden.
\end{example}

\begin{example}{\emph{(Rewriter)}}
  A particularly useful special case of the reversible state arrow is when the
  store $S$ forms a group. \cut{In the irreversible case, this is related to
  the \emph{Writer} monad typically used to perform tasks such as logging.}
  While group multiplication if seen as a function
  \<G\otimes{G}\leftrightarrow{G}\> is invertible only in degenerate cases, we
  can use parametrization to fix the first argument of the multiplication,
  giving it a much more reasonable signature of
  \<G\to({G}\leftrightarrow{G})\>. In this way, groups can be expressed as
  instances of the type class
  \begin{haskell}
    \hskwd{class} Group G \hskwd{where} \\
    \quad\hsalign{
      gunit &:&\ G\\
      gmul &:&\ G\to({G}\leftrightarrow{G}) \\
      ginv &:&\ G\leftrightarrow{G}
    }
  \end{haskell}
  subject to the usual group axioms. This gives us an arrow of the form
  \begin{haskell}
    \hskwd{type} Rewriter G X Y = X\otimes{G}\leftrightarrow{Y}\otimes{G}
  \end{haskell}
  with instance declarations identical to that of \<RState G\>, save that we 
  require \<G\> to be an instance of the \<Group\> type class.
  With this, adding or removing elements from state of type $G$ can then be
  performed by
  \begin{haskell}
    rewrite &:&\ G\to{Rewriter} G X X\\
    rewrite a (x, b) &=&\ (x, gmul a b)
  \end{haskell}
  which ``rewrites'' the state by the value $a$ of type $G$. Note that while
  the name of this arrow was chosen to be evocative of the \emph{Writer} monad
  known from irreversible functional programming, as it may be used for similar
  practical purposes, its construction is substantially different (\ie,
  irreversible \emph{Writer} arrows are maps of the form $X \to Y \times M$
  where $M$ is a monoid).
  \cut{filling a similar role as the \<tell\> function does for the \<Writer\>
  monad in the irreversible case.}
\end{example}

\begin{example}{\emph{(Vector transformation)}}
  Vector transformations, that is, functions on lists that preserve the length
  of the list, form another example of inverse arrows. The \<Vector\> arrow is 
  defined as follows:
  \begin{haskell*}
    \hskwd{type} Vector X Y = 
    [X]\leftrightarrow[Y]\\\\\\
    \hskwd{instance} Arrow (Vector) \hskwd{where} \\
    \quad\hsalign{
    arr f xs &=&\ map f xs \\
    (a\acmp{b}) xs &=&\ b (a xs) \\
    first a ps &=&\ \hskwd{let} (xs,zs) = zip^\dagger ps \hskwd{in} zip (a xs,
    zs)
    } \\\\\\
    \hskwd{instance} InverseArrow (Vector) \hskwd{where} \\
    \quad\hsalign{
    inv a ys &=&\ a^\dagger ys
    }
  \end{haskell*}
  The definition of \<first\> relies on the usual \<map\> and \<zip\> 
  functions, which are defined as follows:\\
\begin{tabular}{ll}
\begin{minipage}{.4\textwidth}
\begin{haskell}
map &:&\ (a\leftrightarrow{b})\to([a]\leftrightarrow[b])\\
map f [] &=&\ []\\
map f (x{::}xs) &=&\ (f x){::}(map f xs)
\end{haskell}
\end{minipage}
&
\begin{minipage}{.4\textwidth}
\begin{haskell}
zip &:&\ ([a],[b])\leftrightarrow[(a,b)]\\
zip ([],[]) &=&\ []\\
zip (x{::}xs,y{::}ys) &=&\ (x,y){::}(zip (xs,ys))\\
\end{haskell}
\end{minipage}
\end{tabular}\\
  Notice that preservation of length is required for \<first\> to work: if the 
  arrow \<a\> does not preserve the length of \<xs\>, then \<zip (a xs, 
  zs)\> is undefined. However, since \<arr\> lifts a pure function $f$ to 
  a \<map\> (which preserves length), and \<(\acmp)\> and \<inv\> are 
  given by the usual composition and inversion, the interface maintains this property.
\end{example}

\begin{example}{\emph{(Reversible error handling)}}
  An inverse \emph{weak arrow} comes from reversible computation
  with a possibility for failure. The weak \<Error\> arrow is defined using disjointness tensors as
  follows:
  \begin{haskell*}
    \hskwd{type} Error E X Y = 
    X\oplus{E}\leftrightarrow{Y}\oplus{E}\\\\\\
    \hskwd{instance} WeakArrow (Error E) \hskwd{where} \\
    \quad\hsalign{
    arr f (InL x) &=&\ InL (f x) \\
    arr f (InR e) &=&\ InR e \\
    (a\acmp{b}) x &=&\ b (a x) \\
    } \\\\\\
    \hskwd{instance} InverseWeakArrow (Error E) \hskwd{where} \\
    \quad\hsalign{
    inv a y &=&\ a^\dagger y
    }
  \end{haskell*}
  In this definition, we think of the type \<E\> as the type of \emph{errors}
  that could occur during computation. As such, a pure function $f$ lifts to a
  weak arrow which always succeeds with value $f(x)$ when given a
  non-erroneous input of $x$, and always propagates errors that may have occurred
  previously.
  
  Raising an error reversibly requires more work than in the irreversible case,
  as the effectful program that produces an error must be able to
  \emph{recover} from it in the converse direction. In this way, a reversible
  \<raise\> requires two pieces of data: a function
  \<f:X\leftrightarrow{E}\> that transforms problematic inputs into appropriate
  errors; and a choice function \<p:E\leftrightarrow{E}\oplus{E}\> that
  decides if the error came from this site, injecting it to the left if it did,
  and to the right if it did not. This can be formalized in an inverse category with a disjointness tensor by saying that the morphisms $f$ and $p$ should satisfy the equations
   \[p f =i_1 f \qquad   i_2^\dag p=p^\dag i_2i_2^\dag p \qquad p^\dag p=\id\]
  The choice function $p$ is critical, as in  the converse direction it decides whether the error should be handled immediately or later.  Thus we define \<raise\> as follows:
  \begin{haskell}
    raise &:&\ 
    (X\leftrightarrow{E})\to(E\leftrightarrow{E}\oplus{E})\to{Error} E X Y \\
    raise f p x &=&\ InR (p^\dagger (arr f x)))
  \end{haskell}
  The converse of \<raise\> is \<handle\>, an (unconditional) error handler that
  maps matching errors back to successful output values. Since unconditional
  error handling is seldom required, this can be combined with control flow
  (see Example~\ref{ex:control_flow}) to perform conditional error handling,
  \ie\ to only handle errors if they occur.
\end{example}

\begin{example}{\emph{(Serialization)}}
  When restricting our attention, as we do here, to only first-order reversible
  functional programming languages, another example of inverse arrows arises in
  the form of \emph{serializers}. A serializer is a function that transforms an
  internal data representation into one more suitable for storage, or for
  transmission to other running processes. To transform serialized data back
  into an internal representation, a suitable deserializer is used. 
  
  When restricting ourselves to the first-order case, it seems reasonable to
  assume that all types are serializable, as we thus avoid the problematic 
  case of how to serialize data of function type. As such, assuming that all 
  types \<X\> admit a function \<serialize : X\leftrightarrow{Serialized} X\>
  (where \<Serialized X\> is the type of serializations of data of type X), we
  define the \<Serializer\> arrow as follows:
\begin{haskell*}
\hskwd{type} Serializer X Y = 
X\leftrightarrow{Serialized} Y\\\\\\
\hskwd{instance} Arrow (Serializer) \hskwd{where} \\
\quad\hsalign{
arr f x &=&\ serialize (f x) \\
(a\acmp{b}) x &=&\ b (serialize^\dagger (a x)) \\
first a (x,z) &=&\ serialize (serialize^\dagger (a x), z)
} \\\\\\
\hskwd{instance} InverseArrow (Serializer) \hskwd{where} \\
\quad\hsalign{
inv a y &=&\ serialize (a^\dagger (serialize y))
}
\end{haskell*}
  Notice how \<serialize^\dagger : Serialized X\leftrightarrow{X}\> takes the
  role of a (partial) deserializer, able to recover the internal representation
  from serialized data as produced by the serializer. A deserializer of the
  form \<serialize^\dagger\> will often only be partially defined, since many
  serialization methods allow many different serialized representations of the
  same data (for example, many textual serialization formats are whitespace
  insensitive). Of course, all this serializing and unserializing is somewhat inefficient.
  In spite of these shortcomings, partial deserializers produced by 
  inverting serializers are sufficient for the above definition to satisfy the 
  inverse arrow laws.
\end{example}

\begin{example}{\emph{(Dagger Frobenius monads)}}\label{ex:frobeniusmonads}
  Monads are also often used to capture computational side-effects. Arrows are more general.
  If $T$ is a strong monad, then $A=\hom(-,T(+))$ is an arrow: $\arr$ is given by the unit, $\acmp$ is given by Kleisli composition, and $\first$ is given by the strength maps.
  When the base category is a dagger or inverse category modelling reversible pure functions, we can use dagger Frobenius monads from chapter~\ref{chp:monads}.

  The Kleisli category of such a monad is again a dagger category by Lemma~\ref{lem:kleislidagger}, giving rise to an operation $\inv$ satisfying~\eqref{eq:daggerarrow1}--\eqref{eq:daggerarrow3}. 
  A dagger Frobenius monad is strong when the strength maps are unitary. In this case~\eqref{eq:daggerarrow4} also follows. If the underlying category is an inverse category, then $\mu \circ \mu^\dag \circ \mu = \mu$, whence $\mu \circ \mu^\dag = \id$, and~\eqref{eq:inversearrow1}--\eqref{eq:inversearrow2} follow. Thus, if $T$ is a strong dagger Frobenius monad on a dagger/inverse category, then $A$ is a dagger/inverse arrow.
\end{example}

\begin{example}{\emph{(Restriction monads)}}
  There is a notion in between the dagger and inverse arrows of the previous example.
  A \emph{(strong) restriction monad} is a (strong) monad on a (monoidal) restriction category whose underlying endofunctor is a restriction functor.
  The Kleisli-category of a restriction monad $T$ has a natural restriction structure: just define the restriction of $f\colon X\to T(Y)$ to be $\eta_X\circ\bar{f}$. The functors between the base category and the Kleisli category then become restriction functors.
  If $T$ is a strong restriction monad on a monoidal restriction category \cat{C}, then $\Inv(\cat{C})$ has an inverse arrow $(X,Y)\mapsto (\Inv(\Kl(T)))(X,Y)$.
\end{example}

\begin{example}{\emph{(Control flow)}}\label{ex:control_flow}
  While only trivial inverse categories have coproducts~\cite{giles:thesis},
  less structure suffices for reversible control structures. When the domain
  and codomain of an inverse arrow both have disjointness tensors (see
  Definition~\ref{def:disjtensor}), it can often be used to implement
  $\mathit{ArrowChoice}$, \ie an arrow with an additional operation 
   \[left : A~X~Y\to{A} (X\oplus{Z}) (Y\oplus{Z})\]
 subject to various laws enabling conditional constructs. For a simple example, the pure arrow on an inverse
  category with disjointness tensors implements \<left : A~X~Y\to{A}
  (X\oplus{Z}) (Y\oplus{Z})\> as
  \begin{haskell}
    left f (x, z) = (f x, z)
  \end{haskell}
  The laws of \<ArrowChoice\>~\cite{hughes:arrows} simply reduce to $- \oplus
  -$ being a bifunctor with natural quasi-injections. More generally, the laws
  amount to preservation of the disjointness tensor. For the reversible state
  arrow (Example~\ref{ex:reversiblestate}), this hinges on $\otimes$
  distributing over $\oplus$.
  
  From these libraries for  \<ArrowChoice\>~ derive additional operations. The splitting combinator 
  \[(+\!\!\!\!+\!\!\!\!+) : A~X~Y \to A~Z~W\to A~(X\otimes Z)~(Y\oplus W) \]
   is unproblematic for reversibility, but the fan-in combinator 
   \[(|||) : A~X~Y\to A~Z~Y\to A~(X\oplus Z)\to Y\] cannot be
  defined reversibly, as it explicitly deletes information about which branch 
  was chosen.  Reversible conditionals thus require two 
  predicates: one determining the branch to take, and one asserted to join the branches after execution. The branch-joining predicate must be chosen carefully to ensure that it is always true after the
  \emph{then}-branch, and false after the \emph{else}-branch. This
  is a standard way of handling branch joining
  reversibly~\cite{yokoyamaglueck:janus,yokoyamaetal:rfun,glueckkaarsgaard:rfcl}.
\end{example}

\begin{example}{\emph{(Superoperators)}}\label{ex:cpm}
  Quantum information theory has to deal with environments.
  The basic category $\cat{FHilb}$ is that of finite-dimensional Hilbert spaces and linear maps. 
  But because a system may be entangled with its environment, the only morphisms that preserve states are the so-called {superoperators}, or \emph{completely positive} maps~\cite{selinger:completelypositive,coeckeheunen:completepositivity}: they are not just positive, but stay positive when tensored with an arbitrary ancillary object. In a sense, information about the system may be stored in the environment without breaking the (reversible) laws of nature. This leads to the so-called CPM construction. It is infamously known \emph{not} to be a monad. But it \emph{is} a dagger arrow on $\cat{FHilb}$, where $A~X~Y$ is the set of completely positive maps $X^* \otimes X \to Y^* \otimes Y$, $\arr f = f_* \otimes f$, $a \acmp b = b \circ a$, $\first_{X,Y,Z} a = a \otimes \id[Z^* \otimes Z]$, and $\inv a = a^\dag$.
\end{example}

Aside from these, other examples do fit the interface of inverse arrows, though
they are less syntactically interesting as they must essentially be ``built
in'' to a particular programming language. These include reversible IO, which
functions very similarly to irreversible IO, and reversible recursion, which
could be used to give a type-level separation between terminating and
potentially non-terminating functions, by only allowing fixed points of
parametrized functions between arrows rather than between (pure) functions.

\section{Inverse arrows, categorically}\label{sec:arrowscategorically}

This section explicates the categorical structure of inverse arrows. Arrows on $\cat{C}$ can be modelled categorically as monoids in the functor category $\prof$~\cite{jacobs2009categorical}. They also correspond to certain identity-on-objects functors $J\colon \cat{C}\to\cat{D}$. The category $\cat{D}$ for an arrow $A$ is built by $\cat{D}(X,Y)=A~X~Y$, and $\arr$ provides the functor $J$. We will only consider the multiplicative fragment, apart from remark~\ref{rem:strength}. The operation $\first$ can be incorporated in a standard way using strength~\cite{jacobs2009categorical,asada2010arrows}, and poses no added difficulty in the reversible setting.


Clearly, dagger arrows correspond to $\cat{D}$ being a dagger category and $J$ a dagger functor, whereas inverse arrows correspond to both $\cat{C}$ and $\cat{D}$ being inverse categories and $J$ a (dagger) functor. This section addresses the following question: which monoids correspond to dagger arrows and inverse arrows? In the dagger case, the answer is quite simple: the dagger makes $\prof$ into an involutive monoidal category, and then dagger arrows correspond to involutive monoids. Inverse  arrows furthermore require certain diagrams to commute. 

\begin{definition}\index[word]{involutive monoidal category} 
  An \emph{involutive monoidal category} is a monoidal category $\cat{C}$ equipped with an \emph{involution}: a functor $\overline{(\ )}\colon \cat{C}\to \cat{C}$  satisfying $\overline{\overline{f}}=f$ for all morphisms $f$, together with a  natural isomorphism
  $
    \chi_{X,Y}\colon \overline{X}\otimes\overline{Y}\to \overline{Y\otimes X} 
  $
  that makes the following diagrams commute\footnote{There is a more general definition allowing a natural isomorphism $\overline{\overline{X}}\to X$ (see~\cite{egger:involutive} for details), but we only need the strict case.}:
  \[
      \begin{aligned}\begin{tikzpicture}
        \matrix (m) [matrix of math nodes,row sep=2em,column sep=4em,minimum width=2em]
          {\overline{X}\otimes(\overline{Y}\otimes \overline{Z}) & (\overline{X}\otimes\overline{Y})\otimes \overline{Z} \\
            \overline{X}\otimes \overline{Z\otimes Y} &\overline{Y\otimes X}\otimes \overline{Z} \\
            \overline{(Z\otimes Y)\otimes X} & \overline{Z\otimes(Y\otimes X)} \\};
          \path[->]
          (m-1-1) edge node [left] {$\id\otimes\chi$} (m-2-1)
                edge node [above] {$\alpha$} (m-1-2)
          (m-1-2) edge node [right] {$\chi\otimes\id$} (m-2-2)
          (m-2-2) edge node [right] {$\chi$} (m-3-2)
          (m-3-2) edge node [below] {$\overline{\alpha}$} (m-3-1)
          (m-2-1) edge node [left] {$\alpha$} (m-3-1);
      \end{tikzpicture}\end{aligned}
      \qquad\qquad
      \begin{aligned}\begin{tikzpicture}
          \matrix (m) [matrix of math nodes,row sep=2em,column sep=4em,minimum width=2em]
          {\overline{\overline{X}}\otimes \overline{\overline{Y}}  & \overline{\overline{Y}\otimes\overline{X}}  \\
            X\otimes Y& \overline{\overline{X\otimes Y}}  \\};
          \path[->]
          (m-1-1) edge node [left] {$\id$} (m-2-1)
                edge node [above] {$\chi$} (m-1-2)
          (m-1-2) edge node [right] {$\overline{\chi}$} (m-2-2)
          (m-2-1) edge node [below] {$\id$} (m-2-2);
      \end{tikzpicture}\end{aligned}
  \]
\end{definition}

Just like monoidal categories are the natural setting for monoids, involutive monoidal categories are the natural setting for involutive monoids. Any involutive monoidal category has a canonical isomorphism $\phi\colon I\to\overline{I}$~\cite[Lemma~2.3]{egger:involutive}:
\[\begin{tikzpicture}[xscale=3]
  \node (1) at (0,0) {$I=\overline{\overline{I}}$};
  \node (2) at (1,0) {$\overline{\overline{I} \otimes I}$};
  \node (3) at (2,0) {$\overline{I} \otimes \overline{\overline{I}} = \overline{I} \otimes I$};
  \node (4) at (3,0) {$\overline{I}$};
  \draw[->] (1) to node[above]{$\overline{\rho_{\overline{I}}}^{-1}$} (2);
  \draw[->] (2) to node[above]{$\chi_{I,\overline{I}}^{-1}$} (3);
  \draw[->] (3) to node[above]{$\rho_{\overline{I}}$} (4);
\end{tikzpicture}\]
Moreover, any monoid $M$ with multiplication $m$ and unit $u$ induces a monoid on $\overline{M}$ with multiplication $\overline{m}\circ \chi_{M,M}$ and unit $\overline{u}\circ\phi$. This monoid structure on $\overline{M}$ allows us to define involutive monoids.

\begin{definition}\index[word]{monoid!involutive} 
  An \emph{involutive monoid} is a monoid $(M,m,u)$ together with a monoid homomorphism $i\colon \overline{M}\to M$ satisfying $i\circ\overline{i}=\id$. A \emph{morphism} of involutive monoids is a monoid homomorphism $f\colon M\to N$ making the following diagram commute:
  \[
      \begin{aligned}\begin{tikzpicture}
        \matrix (m) [matrix of math nodes,row sep=1.5em,column sep=4em,minimum width=2em]
        {\overline{M}  & \overline{N}  \\
          M& N  \\};
        \path[->]
        (m-1-1) edge node [left] {$i_M$} (m-2-1)
            edge node [above] {$\overline{f}$} (m-1-2)
        (m-1-2) edge node [right] {$i_N$} (m-2-2)
        (m-2-1) edge node [below] {$f$} (m-2-2);
      \end{tikzpicture}\end{aligned}
  \]
\end{definition}

Our next result lifts the dagger on $\cat{C}$ to an involution on the category $[\cat{C}\op \times \cat{C},\cat{Set}]$ of profunctors. First we recall the monoidal structure on that category. It categorifies the dagger monoidal category $\cat{Rel}$ of relations of Section~\ref{sec:inversecategories}~\cite[Section 7.8]{borceux:vol1}.
\begin{definition}
  If \cat{C} is small, then $\prof$ has a monoidal structure, which is a special case of composition of profunctors. In terms of coends~\cite{borceux:vol1,loregian:coend}, the monoidal product is
      \begin{equation*}
        F\otimes G (X,Z)= \int^Y F(X,Y)\times G(Y,Z)\text;
      \end{equation*}
  concretely, $F\otimes G (X,Z)= \coprod_{Y\in \cat{C}} F(X,Y)\times G(Y,Z)/\approx$, 
  where $\approx$ is the equivalence relation generated by $(y,F(f,\id)(x))\approx (G(\id,f)(y),x)$, and the action on morphisms is given by $F\otimes G (f,g):= [y,x]_\approx \mapsto [F(f,\id)x,G(\id,g)y]$. The unit of the tensor product is $\hom_\cat{C}$.
\end{definition}

\begin{proposition}\label{prop:swapdaggeronprof}
  If \cat{C} is a dagger category, then $\prof$ is an involutive monoidal category when one defines the involution on objects $F$ by $\overline{F}(X,Y) = F(Y,X)$, $\overline{F}(f,g)=F(g^\dag,f^\dag)$ and on morphisms $\tau\colon F\to G$ by $\overline{\tau}_{X,Y}=\tau_{Y,X}$.
\end{proposition}
\begin{proof}
  First observe that $\overline{(\ )}$ is well-defined: For any natural transformation of profunctors $\tau$, $\overline{\tau}$ is natural, and $\tau\mapsto \overline{\tau}$ is functorial.
  Define $\chi_{F,G}$ by the following composite of natural isomorphisms: 
    \begin{align*}
      \overline{F}\otimes \overline{G} (X,Z)
      &\cong \textstyle\int^Y \overline{F}(X,Y)\times\overline{G}(Y,Z)\text{ by definition of }\otimes \\
      &= \textstyle\int^Y F(Y,X)\times G(Z,Y)\text{ by definition of }\overline{(\ )} \\
      &\cong \textstyle\int^Y G(Z,Y)\times F(Y,X) \text{ by symmetry of }\times\\
      &\cong G\otimes F (Z,X)\text{ by definition of }\otimes \\
      &=\overline{G\otimes F} (X,Z)\text{ by definition of }\overline{(\ )}
     \end{align*}
  Checking that $\chi$ make the relevant diagrams commute is routine.
\end{proof}

\begin{theorem} 
  If \cat{C} is a dagger category, the multiplicative fragments of dagger arrows on \cat{C} correspond exactly to involutive monoids in \prof.
\end{theorem}
\begin{proof} 
  It suffices to show that the dagger on an arrow corresponds to an involution on the corresponding monoid $F$. But this is easy: an involution on $F$ corresponds to giving, for each $X,Y$ a map $F(X,Y)\to F(Y,X)$ subject to some axioms. That this involution is a monoid homomorphism amounts to it being a contravariant identity-on-objects-functor, and the other axiom amounts to it being involutive.
\end{proof}

\begin{remark}\label{rem:strength}
  If the operation $\first$ is modelled categorically as (internal) strength, axiom~\eqref{eq:daggerarrow4} for dagger arrows can be phrased in $\prof$ as follows: for each object $Z$ of $\cat{C}$, and each dagger arrow $M$, the profunctor $M_Z=M((-)\otimes Z,(+)\otimes Z)$ is also a dagger arrow, and $\first_{-,+,Z}$ is a natural transformation $M\Rightarrow M_Z$. The arrow laws~\eqref{eq:arrow7} and~\eqref{eq:arrow8} imply that it is a monoid homomorphism, and the new axiom just states that it is in fact a homomorphism of involutive monoids. For inverse arrows this law is not needed, as any functor between inverse categories is automatically a dagger functor and thus every monoid homomorphism between monoids corresponding to inverse arrows preserves the involution.
\end{remark}

Next we set out to characterize which involutive monoids correspond to inverse arrows. Given an involutive monoid $M$, the obvious approach would be to just state that  the map $M\to M$ defined by $a\mapsto a \circ a^\dag \circ a$ is the identity. However, there is a catch: for an arbitrary involutive monoid, the map $a\mapsto a \circ a^\dag \circ a$ is not natural transformation and therefore not a morphism in $\prof$. To circumvent this, we first require some conditions guaranteeing naturality. These conditions concern endomorphisms, and to discuss them we introduce an auxiliary operation on $\prof$. \cut{It categorifies the operation of a restriction category of Definition~\ref{def:restrictioncategory} to profunctors on dagger categories. }

\begin{definition} 
  Let $\cat{C}$ be a dagger category. Given a profunctor $M \colon \cat{C}\op \times \cat{C} \to\cat{Set}$, define $LM \colon \cat{C}\op \times \cat{C} \to \cat{Set}$ by 
  \begin{align*}
    LM(X,Y)&=M(X,X)\text, \\
    LM(f,g)&=f^\dag \circ (-) \circ f\text.
  \intertext{If $M$ is an involutive monoid in $\prof$, define a subprofunctor of $LM$:}
    L^+M(X,Y) &= \{a^\dag \circ a \in M(X,X) \mid a \in M(X,Z)\text{ for some }Z\}\text.
  \end{align*} 
\end{definition}

\begin{remark}
  The construction $L$ is a functor $\prof \to \prof$. 
  There is an analogous construction $RM(X,Y)=M(Y,Y)$ and $R^+M$, and furthermore $RM=\overline{LM}$.
  For any monoid $M$ in $\prof$, $LM$ is a right $M$-module (and $RM$ a left $M$-module).
  Compare Example~\ref{ex:cpm}.
\end{remark}

For the rest of this section, assume the base category \cat{C} to be an inverse category. This lets us multiply positive arrows by positive pure morphisms. If $M$ is an involutive monoid in $\prof$, then the map $LM\times L^+(\hom_\cat{C})\to LM$ defined by $(a,g^\dag\circ g)\mapsto a\circ g^\dag \circ g$ is natural:
\begin{align*}
    &LM\times L^+(\hom )(f,\id[Y])(a,g^\dag \circ g) \\
    &=(f^\dag \circ a \circ f,f^\dag \circ g^\dag \circ g\circ f) \\
    &\mapsto f^\dag\circ a\circ f\circ f^\dag \circ g^\dag \circ g\circ f \\
    &=f^\dag\circ a\circ g^\dag\circ g\circ f\circ f^\dag\circ f&\text{ because \cat{C} is an inverse category}\\
    &=f^\dag \circ a\circ g^\dag\circ g\circ f&\text{ because \cat{C} is an inverse category} \\
    &=LM(f,\id[Y])(a\circ g^\dag\circ g)
\end{align*}

Similarly there is a map $L^+(\hom)\times LM\to LM$ defined by $(g^\dag \circ g,a)\mapsto g^\dag\circ g\circ a$. Now the category corresponding to $M$ satisfies $a^\dag \circ a \circ g^\dag\circ  g=g^\dag \circ g\circ  a^\dag\circ  a$ for all $a$ and pure $g$ if and only if the following diagram commutes:
\begin{equation}\label{diag:step0}
    \begin{aligned}\begin{tikzpicture}
        \matrix (m) [matrix of math nodes,row sep=1.5em,column sep=4em,minimum width=2em]
        {L^+M\times L^+(\hom)  && LM\times L^+(\hom)  \\
          L^+(\hom)\times L^+M&   L^+(\hom)\times LM& LM\\};
        \path[->]
        (m-1-1) edge node [left] {$\sigma $} (m-2-1) 
            edge node [above] {$$} (m-1-3) 
        (m-1-3) edge node [right] {$$} (m-2-3) 
        (m-2-1) edge node [below] {$$} (m-2-2) 
        (m-2-2) edge node [below] {$$} (m-2-3); 
    \end{tikzpicture}\end{aligned}
\end{equation}
If this is satisfied for an involutive monoid $M$ in $\prof$, then positive arrows multiply. In other words, the map $L^+M\times L^+M\to LM$ defined by $(a^\dag \circ a,b^\dag \circ b)\mapsto a^\dag\circ a\circ b^\dag\circ b$ is natural:
\begin{align*}
  &D_M(f,g)(a,a^\dag,a)\\
  &=(g \circ a\circ f,f^\dag \circ a^\dag \circ g^\dag,g\circ a\circ f) \\
  &\mapsto g\circ a\circ f\circ f^\dag\circ  a^\dag\circ  g^\dag \circ g\circ a\circ f \\
  &=g\circ a\circ a^\dag\circ  g^\dag\circ  g\circ a\circ  f\circ f^\dag\circ  f &\text{ by~\eqref{diag:step0}}\\
  &=g\circ a\circ a^\dag\circ  g^\dag\circ  g\circ a\circ f&\text{ because \cat{C} is an inverse category} \\
  &=g\circ g^\dag \circ g\circ  a\circ a^\dag \circ a\circ  f &\text{ by~\eqref{diag:step0}} \\
  &=g\circ a\circ a^\dag\circ  a\circ  f&\text{ because \cat{C} is an inverse category} \\
  &=M(f,g)(a\circ a^\dag\circ  a)
  \end{align*}
This multiplication is commutative iff the following diagram commutes:
\begin{equation}\label{diag:step1 (positives commute)}
    \begin{aligned}\begin{tikzpicture}
        \matrix (m) [matrix of math nodes,row sep=2em,column sep=4em,minimum width=2em]
        {L^+M\times L^+M  & L^+M\times L^+M  \\
          & LM  \\};
        \path[->]
        (m-1-1) edge node [left] {$$} (m-2-2)
            edge node [above] {$\sigma$} (m-1-2)
        (m-1-2) edge node [right] {$$} (m-2-2);
    \end{tikzpicture}\end{aligned}
\end{equation}

Finally, let $D_M\hookrightarrow M\times\overline{M}\times M$ be the diagonal
$
  D_M(X,Y)=\{(a,a^\dag,a)\mid a\in M(X,Y)\}
$.

If $M$ satisfies~\eqref{diag:step0}, then the map $D_M\to M$ defined by $(a,a^\dag, a)\mapsto a\circ a^\dag\circ a$ is natural:

\begin{align*}
  &D_M(f,g)(a,a^\dag,a)\\
  &=(g \circ a\circ f,f^\dag \circ a^\dag \circ g^\dag,g\circ a\circ f) \\
  &\mapsto g\circ a\circ f\circ f^\dag\circ  a^\dag\circ  g^\dag \circ g\circ a\circ f \\
  &=g\circ a\circ a^\dag\circ  g^\dag\circ  g\circ a\circ  f\circ f^\dag\circ  f &\text{ by~\eqref{diag:step0}}\\
  &=g\circ a\circ a^\dag\circ  g^\dag\circ  g\circ a\circ f&\text{ because \cat{C} is an inverse category} \\
  &=g\circ g^\dag \circ g\circ  a\circ a^\dag \circ a\circ  f &\text{ by~\eqref{diag:step0}} \\
  &=g\circ a\circ a^\dag\circ  a\circ  f&\text{ because \cat{C} is an inverse category} \\
  &=M(f,g)(a\circ a^\dag\circ  a)
  \end{align*}

 Thus $M$ satisfies $a\circ a^\dag\circ a=a$ if and only if the following diagram commutes:
\begin{equation}\label{diag:step2 (aa^daga=a)}
    \begin{aligned}\begin{tikzpicture}
        \matrix (m) [matrix of math nodes,row sep=2em,column sep=4em,minimum width=2em]
        {M  & D_M \\
          & M \\};
        \path[->]
        (m-1-1) edge node [below] {$\id$} (m-2-2)
            edge node [above] {$$} (m-1-2)
        (m-1-2) edge node [right] {$$} (m-2-2);
      \end{tikzpicture}\end{aligned}
\end{equation}
Hence we have established the following theorem. 

\begin{theorem}\label{thm:characterizing_inverse_arrows}
  Let \cat{C} be an inverse category. Then the multiplicative fragments of inverse arrows on \cat{C} correspond exactly to involutive monoids in $\prof$ making the diagrams~\eqref{diag:step0}--\eqref{diag:step2 (aa^daga=a)} commute.
  \qed
\end{theorem}

\section{Applications and related work}
As we have seen, inverse arrows capture a variety of fundamental reversible
effects. An immediate application of our results would be to retrofit existing
typed reversible functional programming languages (\eg,
Theseus~\cite{jamessabry:theseus}) with inverse arrows to accommodate 
reversible effects while maintaining a type-level separation between pure and
effectful programs. Another approach could be to design entirely new such
programming languages, taking inverse arrows as the fundamental representation
of reversible effects.
While the Haskell approach to arrows uses typeclasses~\cite{hughes:arrows},
these are not a priori necessary to reap the benefits of inverse arrows. For
example, special syntax for defining inverse arrows could also be used, either
explicitly, or implicitly by means of an effect system that uses inverse arrows
``under the hood''.

\cut{To aid programming with ordinary arrows, a handy notation due to
Paterson~\cite{paterson:notation,paterson:computation} may be used. If the underlying monoidal dagger category has natural coassociative diagonals, for example when it has inverse products, a similar $\mathtt{do}$-notation can be implemented for inverse and dagger arrows.}

To aid programming with ordinary arrows, a handy notation due to
Paterson~\cite{paterson:notation,paterson:computation} may be used. The
simplest form of this notation is based on process combinators, the central one
being
\begin{equation*}
  p \to e_1 \prec e_2 = \left\{\begin{array}{ll}
    \mathit{arr}(\lambda p. e_2) \acmp e_1 & \quad\text{if $p$ is fresh for $e_1$,} \\
    \mathit{arr}(\lambda p. (e_1,e_2)) \acmp \text{app} & \quad\text{otherwise.}
  \end{array} \right.
\end{equation*}
Note that if the second branch is used, the arrow must
additionally be an instance of $\mathit{ArrowApply}$ (so that it is, in fact,
a monad). Though we only know of degenerate examples where inverse arrows are
instances of $\mathit{ArrowApply}$, this definition is conceptually
unproblematic (from the point of view of guaranteeing reversibility) so long as the pure function $\lambda p. e_2$ is first-order and
reversible. A more advanced style of this notation is the
\emph{do}-notation for arrows, which additionally relies on the arrow
combinator
\begin{haskell*}
  bind &:& A~X~Y\to{A~(X\otimes{Y})~Z}\to{A~X~Z} \\
  f~'bind'~g &=& (arr(id)\afanout{f})\acmp{g}\enspace.
\end{haskell*}
If the underlying monoidal dagger category has natural coassociative diagonals,
for example when it has inverse products, this combinator does exist: the arrow
combinator $(\afanout)$ can be defined as
\begin{haskell*}
  (\afanout) &:& A~X~Y\to{A~X~Z}\to{A~X~(Y\otimes{Z})} \\
  f\afanout{g} &=& arr(copy) \acmp first(f) \acmp second(g)
\end{haskell*}
where \<copy : X\to{X}\otimes{X}\> is the natural diagonal (given in pseudocode
by \<copy x = (x,x)\>), and the combinator \<second\> is derived from \<first\> in the
usual way, \ie, as
\begin{haskell*}
  second &:& A~X~Y\to{A}~(Z\otimes{X})~(Z\otimes{Y}) \\
  second f &=& arr(swap)\acmp{first(f)}\acmp{arr(swap)}
\end{haskell*}
with \<swap : X\otimes{Y}\leftrightarrow{Y}\otimes{X}\> given by \<swap (x,y) = (y,x)\>. This allows \emph{do}-notation of the form
\begin{haskell*}
  \hskwd{do} \{p\leftarrow{c}\scolon{A}\} \equiv c 'bind' 
  (\kappa{p}. \hskwd{do} \{A\})\text,
\end{haskell*}
so soon as the $\kappa$-calculus~\cite{hasegawa:kappa_calc}
expression \<\kappa{p}. \hskwd{do} \{A\}\> is reversible. Note, however, that
\emph{do}-expressions of the form \<\hskwd{do} \{c\scolon A\}\> (\ie, where 
the output of $c$ is discarded entirely) will fail to be reversible in all but
the most trivial cases. 
Since \<\hskwd{do} \{p \leftarrow c\scolon A\}\> produces a value of an inverse arrow type, closure under program inversion provides a program we might call
\begin{haskell*}
  \hskwd{undo} \{p\leftarrow{c}\scolon{A}\} \equiv inv(\hskwd{do} 
  \{p\leftarrow{c}\scolon{A}\}) \enspace\text.
\end{haskell*}
Inverse arrow law~\eqref{eq:inversearrow1} then guarantees that \emph{do}ing, then \emph{undo}ing, and then \emph{do}ing the same operation is the same as \emph{do}ing it once.

A pleasant consequence of the semantics of inverse arrows is that
inverse arrows are safe: as long as the inverse arrow laws are satisfied, fundamental properties guaranteed by reversible functional
programming languages (such as invertibility and closure under program
inversion) are preserved. In this way, inverse arrows provide reversible
effects as a conservative extension to pure reversible functional programming.

A similar approach to invertibility using arrows is given by bidirectional
arrows~\cite{alimarineetal:biarrows}. However, while the goal of inverse arrows
is to add effects to already invertible languages, bidirectional arrows arise
as a means to add invertibility to an otherwise uninvertible language. As such,
bidirectional arrows have different concerns than inverse arrows, and notably
do not guarantee invertibility in the general case.



\bibliographystyle{plain}


\bibliography{thesis}

\printindex[symb]\label{index:symb} 
\printindex[word]\label{index:word}
\end{document}